\documentclass[reqno]{amsart}

\pdfoutput=1

\usepackage[alphabetic]{amsrefs}
\usepackage{amsmath}
\usepackage{amsfonts}
\usepackage{amssymb}
\usepackage{enumerate}
\usepackage{amstext}
\usepackage{amsbsy}
\usepackage{amsopn}
\usepackage{bbm,amsthm}
\usepackage{amscd}
\usepackage[pdftex]{color}
\usepackage[pdftex]{graphicx}
\usepackage[all]{xy}
\usepackage{mathtools}

\usepackage{tikz}
\usetikzlibrary{intersections,positioning,patterns,arrows,decorations.markings}

\usepackage{todonotes}

\usepackage{amsxtra}
\usepackage{upref}
\usepackage{epstopdf}
\usepackage{graphicx,color}
\usepackage{hyperref}
\usepackage{longtable}
\usepackage[francais, english]{babel}

\makeatletter
\newsavebox{\@brx}
\newcommand{\llangle}[1][]{\savebox{\@brx}{\(\m@th{#1\langle}\)}%
  \mathopen{\copy\@brx\kern-0.5\wd\@brx\usebox{\@brx}}}
\newcommand{\rrangle}[1][]{\savebox{\@brx}{\(\m@th{#1\rangle}\)}%
  \mathclose{\copy\@brx\kern-0.5\wd\@brx\usebox{\@brx}}}
\makeatother

\DeclareFontFamily{OML}{rsfs}{\skewchar\font'177}
\DeclareFontShape{OML}{rsfs}{m}{n}{ <5> <6> rsfs5 <7> <8> <9> rsfs7
  <10> <10.95> <12> <14.4> <17.28> <20.74> <24.88> rsfs10 }{}
\DeclareMathAlphabet{\mathfs}{OML}{rsfs}{m}{n}

\newtheorem{theorem}{Theorem}
\newtheorem{lemma}[theorem]{Lemma}
\newtheorem{proposition}[theorem]{Proposition}
\newtheorem{corollary}[theorem]{Corollary}

\theoremstyle{definition}

\theoremstyle{remark}
\newtheorem{remark}[theorem]{\bf Remark}

\numberwithin{equation}{section}
\numberwithin{theorem}{section}

\newcommand{\intav}[1]{\mathchoice {\mathop{\vrule width 6pt height 3 pt depth  -2.5pt
\kern -8pt \intop}\nolimits_{\kern -6pt#1}} {\mathop{\vrule width
5pt height 3  pt depth -2.6pt \kern -6pt \intop}\nolimits_{#1}}
{\mathop{\vrule width 5pt height 3 pt depth -2.6pt \kern -6pt
\intop}\nolimits_{#1}} {\mathop{\vrule width 5pt height 3 pt depth
-2.6pt \kern -6pt \intop}\nolimits_{#1}}}

\newcommand{\intavl}[1]{\mathchoice {\mathop{\vrule width 6pt height 3 pt depth  -2.5pt
\kern -8pt \intop}\limits_{\kern -6pt#1}} {\mathop{\vrule width 5pt
height 3  pt depth -2.6pt \kern -6pt \intop}\nolimits_{#1}}
{\mathop{\vrule width 5pt height 3 pt depth -2.6pt \kern -6pt
\intop}\nolimits_{#1}} {\mathop{\vrule width 5pt height 3 pt depth
-2.6pt \kern -6pt \intop}\nolimits_{#1}}}

\newcommand{\norm}[1]{{\left\| #1 \right\|}}

\newcommand{\vertiii}[1]{{\left\vert\kern-0.2ex\left\vert\kern-0.2ex\left\vert #1 
    \right\vert\kern-0.2ex\right\vert\kern-0.2ex\right\vert}}

\newcommand{\un}{\underline}

\newcommand{\ve}{\varepsilon}
\newcommand{\wt}{\widetilde}
\newcommand{\wh}{\widehat}
\newcommand{\vt}{\vartheta}
\newcommand{\vf}{\varphi}

\newcommand{\R}{\mathbb{R}}
\newcommand{\N}{\mathbb{N}}

\newcommand{\Z}{\mathbb{Z}}

\newcommand{\F}{\mathfs{F}}
\newcommand{\M}{\mathfs{M}}
\newcommand{\wM}{\widetilde{\mathfs{M}}}

\renewcommand{\exp}[1]{{\rm exp}_{#1}}

\newcommand{\Hol}[1]{{\rm Hol}_{#1}}

\newcommand{\Lip}[1]{{\rm Lip}\left({#1}\right)}
\newcommand{\Sas}{d_{\rm Sas}}
\newcommand{\inj}{{\rm inj}}

\begin{document}

\title[Symbolic dynamics for maps with singularities in high dimension]{Symbolic dynamics for nonuniformly hyperbolic maps with singularities in high dimension}

\author{Ermerson Araujo, Yuri Lima, and Mauricio Poletti}

\address{Departamento de Matem\'atica, Universidade Federal do Cear\'a (UFC), Campus do Pici,
Bloco 914, CEP 60455-760. Fortaleza -- CE, Brasil}
\email{ermersonaraujo@gmail.com, yurilima@gmail.com, mpoletti@impa.br}
\address{Laboratoire de Math\'ematiques d'Orsay \\ CNRS - UMR 8628, Univ. Paris-Saclay,
Orsay 91405, France}
\email{mpoletti@impa.br}
\address{Departamento de Matem\'atica/CCET, Universidade Federal do Maranh\~ao (UFMA), 
Cidade Universit\'aria Dom Delgado, CEP 65080-805, S\~ao Lu\'is -- MA, Brazil}
\email{ermersonaraujo@gmail.com}

\date{\today}
\thanks{The authors thank Viviane Baladi, Snir Ben Ovadia, J\'er\^ome Buzzi, Gianluigi Del Magno,
Mark Demers, Fran\c cois Ledrappier, Roberto Markarian, Omri Sarig, Domokos Sz\'asz
for useful explanations and suggestion. Araujo is supported by CAPES.
Lima is supported by CNPq and Instituto Serrapilheira, grant ``Jangada Din\^amica:
Impulsionando Sistemas Din\^amicos na Regi\~ao Nordeste''.
During the preparation of this manuscript, Poletti was supported by the ERC project 692925 NUHGD}

\begin{abstract}
We construct Markov partitions for non-invertible and/or singular
nonuniformly hyperbolic systems defined on higher dimensional Riemannian manifolds.
The generality of the setup covers classical examples not treated so far, such as geodesic
flows in closed manifolds, multidimensional billiard maps, and Viana maps, and includes all the recent
results of the literature. We also provide a wealth of applications.
\end{abstract}

\maketitle

\tableofcontents

\section{Introduction}\label{Section-introduction}

In 2013, Sarig constructed countable Markov partitions with full topological entropy for $C^{1+\beta}$
surface diffeomorphisms \cite{Sarig-JAMS}. Since then, his method is being further developed to
other classes of systems, including classes where the classical literature of Markov partitions 
was not able to handle, such as billiard maps.
Here are some of the recent developments:
\begin{enumerate}[$\circ$]
\item Lima and Sarig for three dimensional flows without fixed points \cite{Lima-Sarig}.
\item Lima and Matheus for surface maps with discontinuities \cite{Lima-Matheus}.
\item Ben Ovadia for diffeomorphisms in any dimension \cite{Ben-Ovadia-2019}.
\item Lima for one-dimensional maps \cite{Lima-IHP}.
\end{enumerate}

The goal of the present work is to go far beyond the above results, constructing Markov structures
for maps on higher dimensional Riemannian manifolds, with the novelty of allowing the map
to be non-invertible and have a singular set, and the manifold to be disconnected and
non-compact. This unified framework allows us to cover many classical examples
of nonuniform hyperbolic dynamics, such as geodesic flows in closed manifolds, billiard maps, and Viana maps.
Also, this wider framework includes the four works cited above, in some cases
providing even better results (in comparison to \cite{Lima-IHP}). Not only this: using the recent work of
Ben Ovadia \cite{Ben-Ovadia-2020}, we are able to identify the points coded
by recurrent sequences. 

Before stating the main result, let us mention some of its applications.
Let $N$ be a closed $C^\infty$ Riemannian manifold,
let $X$ be a $C^{1+\beta}$ vector field on $N$ s.t. $X_x\neq 0$ for all $x\in N$,
and let $g=\{g^t\}_{t\in\R}:N\to N$ be the flow  determined by $X$.
Assume that $g$ has positive topological entropy, call it $h$.
Let $\mu$ be a $g$--invariant Borel probability measure on $N$.
We call $\mu$ {\em hyperbolic} if all of its Lyapunov exponents are non-zero, except for
the zero exponent in the flow direction. A {\em simple closed orbit} of length $\ell$ is a parametrized
curve $\gamma(t)=g^t(x)$, $0\leq t\leq\ell$, s.t. $\gamma(0)=\gamma(\ell)$ and
$\gamma(0)\neq \gamma(t)$ for $0<t<\ell$. Denoting the {\em trace} of $\gamma$ by the set
$\{\gamma(t):0\leq t\leq\ell\}$, we let $[\gamma]$ denote the equivalence class of the relation
$\gamma_1\sim\gamma_2\Leftrightarrow\gamma_1,\gamma_2$
have equal lengths and traces, and let ${\rm Per}_T(g):=\#\{[\gamma]:\gamma\textrm{ is simple and }\ell(\gamma)\leq T\}$.

\begin{theorem}\label{Thm-Flows-periodic}
Let $g$ be as above. If $g$ has a hyperbolic measure of maximal entropy,
then $\exists C>0$ s.t. ${\rm Per}_T(g)\geq C{e^{hT}}/{T}$ for all $T$ large enough.
\end{theorem}

The second application deals with equilibrium measures, see Section \ref{Section-equilibrium}
for its definition and Section \ref{Section-Flows} for the definition of the geometric potential $J$ of a flow
and its equilibrium measures.

\begin{theorem}\label{Thm-Flows-Bernoulli}
For $g$ as above, let $\vf:N\to\R$ be a H\"older continuous
potential or $\vf=tJ$ for $t\in\R$. Then every equilibrium measure of $\vf$
has at most countably many ergodic hyperbolic components, and each of them is Bernoulli up to a period.
\end{theorem}

The above result is the higher dimensional extension of Theorems 1.1 and 1.4 of \cite{LLS-2016}.
%It would be interesting to prove that, as in \cite{LLS-2016}, when $g$ is a Reeb flow then each ergodic
%hyperbolic component is indeed Bernoulli.
Observe that, by Theorem \ref{Thm-Flows-Bernoulli}, if we know that
the equilibrium measure is mixing, then it is necessarily Bernoulli.
This observation applies for geodesic flows of closed rank one manifolds, where
Burns et al proved that if the pressure gap $P({\rm Sing},g)<P(g)$ holds then $\vf$ has a unique
equilibrium measure and it is hyperbolic \cite{BCFT-2018}, and Call \& Thompson proved that this
measure has the K property and that the unique measure of maximal entropy is
Bernoulli \cite{Call-Thompson-2019}. Here, ${\rm Sing}$ is the set of vectors
with rank larger than one. We thus obtain the following corollary.

\begin{corollary}\label{Coro-Flows-Rank-1}
Let $g$ be the geodesic flow of a closed rank one manifold $N$, and assume that
$P({\rm Sing},g)<P(g)$. If $\vf:T^1N\to\R$ is H\"older continuous or $\vf=tJ$ for $t\in\R$,
then the unique equilibrium state of $\vf$ is Bernoulli.
\end{corollary}

In the above context, Knieper proved that the measure of maximal entropy is unique
\cite{Knieper-Rank-One-Entropy}, and Babillot proved that it is mixing \cite{Babillot-Mixing}.
Using these results, we obtain a proof of Corollary \ref{Coro-Flows-Rank-1}
for $\vf=0$ that is different from and does not use \cite{Call-Thompson-2019}. 
The proofs of Theorems \ref{Thm-Flows-periodic} and \ref {Thm-Flows-Bernoulli} are in Section \ref{Section-Flows},
which also contains a theorem that provides a symbolic coding for the considered flows. 

The next application deals with multidimensional billiards.
Let $f$ be the billiard map associated to one of the following contexts:
dispersing billiards with finitely many disjoint scatterers, billiards of hard balls systems, or
a three-dimensional Bunimovich stadium as considered in \cite{Bunimovich-delMagno-hyperbolicity},
see Figure \ref{figure-billiards-Intro}.
In all of these cases, it is known that the natural smooth $f$--invariant measure $\mu_{\rm SRB}$
is hyperbolic, mixing, and has Kolmogorov-Sina{\u\i} entropy $h>0$.
Let ${\rm Per}_n(f)$ denote the number of periodic orbits of period $n$ for $f$.
\begin{center}
\begin{figure}[hbt!]
\centering
\def\svgwidth{12.3cm}
%% Creator: Inkscape 1.1-dev (f9311a1, 2019-12-25), www.inkscape.org
%% Accompanies image file 'billiards.pdf' (pdf, eps, ps)
\begingroup%
  \makeatletter%
  \providecommand\color[2][]{%
    \errmessage{(Inkscape) Color is used for the text in Inkscape, but the package 'color.sty' is not loaded}%
    \renewcommand\color[2][]{}%
  }%
  \providecommand\transparent[1]{%
    \errmessage{(Inkscape) Transparency is used (non-zero) for the text in Inkscape, but the package 'transparent.sty' is not loaded}%
    \renewcommand\transparent[1]{}%
  }%
  \providecommand\rotatebox[2]{#2}%
  \newcommand*\fsize{\dimexpr\f@size pt\relax}%
  \newcommand*\lineheight[1]{\fontsize{\fsize}{#1\fsize}\selectfont}%
  \ifx\svgwidth\undefined%
    \setlength{\unitlength}{656.62351441bp}%
    \ifx\svgscale\undefined%
      \relax%
    \else%
      \setlength{\unitlength}{\unitlength * \real{\svgscale}}%
    \fi%
  \else%
    \setlength{\unitlength}{\svgwidth}%
  \fi%
  \global\let\svgwidth\undefined%
  \global\let\svgscale\undefined%
  \makeatother%
  \begin{picture}(1,0.39584813)%
    \lineheight{1}%
    \setlength\tabcolsep{0pt}%
    \put(0,0){\includegraphics[width=\unitlength,page=1]{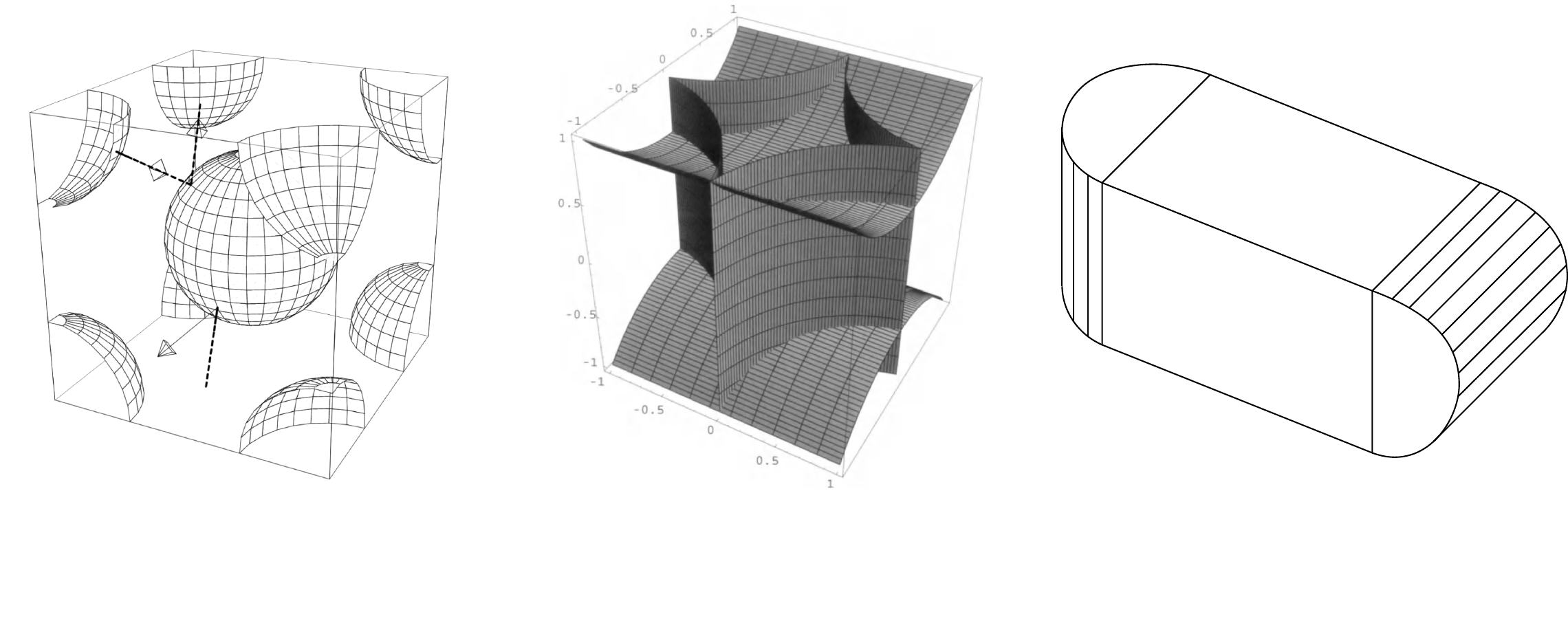}}%
    \put(0.11604656,0.00919117){\color[rgb]{0,0,0}\makebox(0,0)[lt]{\lineheight{1.25}\smash{\begin{tabular}[t]{l}(a)\end{tabular}}}}%
    \put(0.45180016,0.00919117){\color[rgb]{0,0,0}\makebox(0,0)[lt]{\lineheight{1.25}\smash{\begin{tabular}[t]{l}(b)\end{tabular}}}}%
    \put(0.76471193,0.00919117){\color[rgb]{0,0,0}\makebox(0,0)[lt]{\lineheight{1.25}\smash{\begin{tabular}[t]{l}(c)\end{tabular}}}}%
  \end{picture}%
\endgroup%

\caption{(a) Dispersing billiard in $\mathbb T^3$ (credits to Imre P\'eter T\'oth).
(b) Billiard of a hard balls system in $\mathbb T^3$ (credits to \cite{Posch-Hirschl}).
(c) Three-dimensional Bunimovich stadium.}
\label{figure-billiards-Intro}
\end{figure}
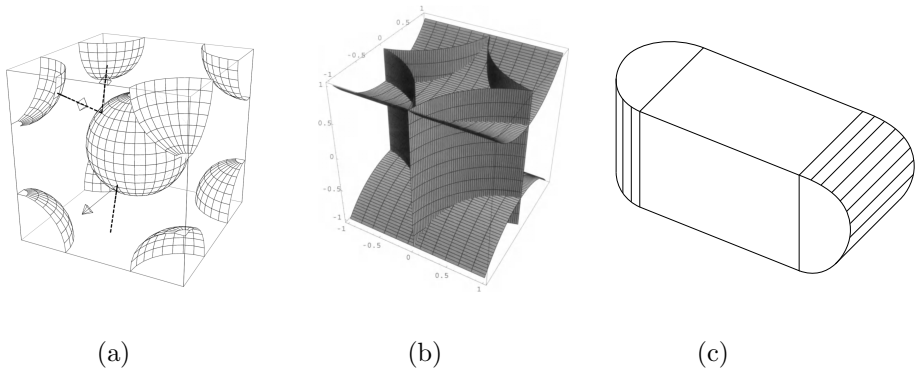
\end{center}

\begin{theorem}\label{Thm-Billiards-periodic}
For the above billiards, there exists $C>0$ s.t. ${\rm Per}_n(f)\geq Ce^{hn}$ for all $n$ large enough.
\end{theorem}

The proof of Theorem \ref{Thm-Billiards-periodic} and a detailed discussion on multidimensional
billiards is in Section \ref{Section-Billiards}. The last applications deal with {\em Viana maps},
introduced in  \cite{Viana-maps}. These maps are defined as follows. Let $a_0\in (1,2)$ be a parameter
s.t. $t=0$ is pre-periodic for the quadratic map $t\mapsto a_0-t^2$. For fixed $d\geq 16$ and
$\alpha>0$, the associated Viana map is the skew product 
$f=f_{a_0,d,\alpha}:\mathbb S^1\times \mathbb{R}\to \mathbb S^1\times \mathbb{R}$ defined by 
$f(\theta,t)=(d\theta,a_0+\alpha \sin(2\pi \theta)-t^2)$. 
If $\alpha>0$ is small enough then there is a compact interval $I_0\subset (-2,2)$ such that 
$f(\mathbb S^1\times I_0)\subset \textrm{int}(\mathbb S^1\times I_0)$. Hence $f$ has 
an attractor inside $\mathbb S^1\times I_0$, and so we consider the restriction of $f$ to
$\mathbb S^1\times I_0$. Observe that $f$ has a critical
set $\mathbb S^1\times\{0\}$, where $df$ is non-invertible. When $\alpha>0$ is small,
$f$ possesses a unique invariant probability measure $\mu_{\rm SRB}$
that is absolutely continuous with respect to Lebesgue \cite{Alves-SRB-2000,Alves-Viana-2002},
nonuniformly expanding \cite{Viana-maps}, and mixing \cite{Alves-Luzzato-Pinheiro}.
As a matter of fact, all these properties hold for a $C^3$ neighborhood of $f$.
Letting ${\rm Per}_n(f)$ as before and $h=h_{\mu_{\rm SRB}}(f)>0$, our main theorem
implies the following results.

\begin{theorem}\label{Thm-Viana-periodic}
If $\alpha>0$ is small, then $\exists C>0$ s.t. ${\rm Per}_n(f)\geq Ce^{hn}$ for all $n$ large enough.
The same applies to a $C^3$ neighborhood of $f$. 
\end{theorem}

\begin{theorem}\label{Thm-Viana-Bernoulli}
If $\alpha>0$ is small, there exists a $C^3$ neighborhood $\mathfs U$ of $f$ such that 
every map in $\mathfs U$ has at most countably many ergodic 
measures of maximal entropy, and the lift to the natural extension 
of each of them is Bernoulli up to a period.
\end{theorem}

%Above, $J$ is the geometric potential of $f$, see Section \ref{Section-J} for the definition.
The proofs of Theorems \ref{Thm-Viana-periodic} and \ref{Thm-Viana-Bernoulli} are given in
Section \ref{Section-Viana-maps}.

%Theorem~\ref{Thm-Main} allows to deal with non invertible maps, even with singularities.
%It is interesting to see how can this be applied to study measures of maximal entropy of this kind of maps.
%Outside the uniformly hyperbolic endomorphisms an interesting class of maps to analyze are the ones that exhibit some non-uniformly hyperbolic behavior Lebesgue almost everywhere.
%
%
%In \cite{Viana-maps} Viana gave an example of an open neighborhood of endomorphisms that has non-uniformly hyperbolic attractators with multidimensional sensitiveness on the initial condition, this maps are known as \emph{Viana maps}. The Viana maps are a given by a $C^3$ neighborhood of the skew product of a uniformly expanding map with quadratic maps, this maps are endomorphisms with critical points, and it is proved on \cite{Viana-maps} that, even though they have singularities, this maps exhibit expansion for Lebesgue almost every point stably.
%
%There have been many results about the physical and SRB measures of this maps, 
%for example \cite{Alves-SRB-2000}, \cite{Alves-Viana-2002}, \cite{Alves-Schnellman-2013},
%\cite{Pinheiro-SRB-2006}. We want to study the measures with high entropy of this kind of maps. 
%
%Let $\varphi_\alpha:S^1\times \mathbb{R}\to S^1\times \mathbb{R}$ be given by 
%$\varphi_\alpha(\theta,x)=(d\theta,a_0+\alpha \sin(2\pi \theta)-x^2)$, with $d\geq 16$. We are only interested on the dynamics of the attractor inside $S^1\times I_0$, where $I_0=[-2,2]$. 
%
%So we have the following result:

\subsection{Statement of main theorem}

To cover all the above contexts at the same time, our framework requires seven assumptions to be satisfied:
the first four are assumptions on the manifold, and the last three are on the map. 
We remark that the first four assumptions coincide with those in \cite{Lima-Matheus},
and the last three are similar to those in \cite{Lima-IHP}.
In the sequel we specify them. Let $M$ be a smooth, possibly disconnected and/or with boundary and/or not compact,
Riemannian manifold with finite diameter. We will always assume that the diameter of $M$
is smaller than one (just multiply the metric by a small constant). We allow $M$ to have infinitely
many connected components, as long as it is contained in a larger Riemannian manifold\footnote{This avoids 
pathological situations in which natural conditions, e.g. precompactness, are lost.}.
We fix hence and for all a closed set $\mathfs D\subset M$, which will denote the set
of {\em discontinuities} of the map. Given $x\in M$, let $T_xM$ denote
the tangent space of $M$ at $x$. For $r>0$, let $B_x[r]\subset T_xM$ denote the open ball with center 0
and radius $r$. Given $x\in M\backslash\mathfs D$, let $\inj(x)$ be the {\em injectivity radius} of $M$ at $x$,
and let $\exp{x}:B_x[\inj(x)]\to M$ be the {\em exponential map} at $x$.
%Write $\rho(x):= d(x,\mathfs D)$.

The Riemannian metric on $M$ induces a Riemannian metric on the tangent bundle $TM$,
called the {\em Sasaki metric}, see e.g. \cite[\S2]{Burns-Masur-Wilkinson}.
Denote the Sasaki metric by $\Sas(\cdot,\cdot)$, and 
denote the Sasaki metric on $TB_x[r]$ by the same notation (the context
will be clear in which space we are). For nearby small vectors, the Sasaki metric is
almost a product metric as follows. Given a geodesic $\gamma$ joining $y$ to $x$,
let $P_\gamma:T_yM\to T_xM$ be the parallel transport along $\gamma$.
If $v\in T_xM$, $w\in T_yM$ then
$\Sas(v,w)\asymp d(x,y)+\|v-P_\gamma w\|$ as $\Sas(v,w)\to 0$, see
\cite[Appendix A]{Burns-Masur-Wilkinson}. 
The properties required on $M$ are about
the exponential maps. Here are the first two. For $x\in M$ and $r>0$,
let $B(x,r)\subset M$ denote the open ball with center $x$ and radius $r$.

\medskip
\noindent
{\sc Regularity of $\exp{x}$:} $\exists a>1$ s.t. for all
$x\in M\backslash\mathfs D$ there is $d(x,\mathfs D)^a<\mathfrak d(x)<1$
s.t. for $\mathfrak B_x:=B(x,2\mathfrak d(x))$ the following holds:
\begin{enumerate}[ii]
\item[(A1)] If $y\in \mathfrak B_x$ then $\inj(y)\geq 2\mathfrak d(x)$, $\exp{y}^{-1}:\mathfrak B_x\to T_yM$
is a diffeomorphism onto its image, and
$\tfrac{1}{2}(d(x,y)+\|v-P_{y,x}w\|)\leq \Sas(v,w)\leq 2(d(x,y)+\|v-P_{y,x} w\|)$ for all $y\in \mathfrak B_x$ and
$v\in T_xM,w\in T_yM$ s.t. $\|v\|,\|w\|\leq 2\mathfrak d(x)$, where 	
$P_{y,x}:=P_\gamma$ for the length minimizing geodesic $\gamma$ joining $y$ to $x$.
\item[(A2)] If $y_1,y_2\in \mathfrak B_x$ then
$d(\exp{y_1}v_1,\exp{y_2}v_2)\leq 2\Sas(v_1,v_2)$ for $\|v_1\|$, $\|v_2\|\leq 2\mathfrak d(x)$,
and $\Sas(\exp{y_1}^{-1}z_1,\exp{y_2}^{-1}z_2)\leq 2[d(y_1,y_2)+d(z_1,z_2)]$
for $z_1,z_2\in \mathfrak B_x$ whenever the expression makes sense.
In particular $\|d(\exp{x})_v\|\leq 2$ for $\|v\|\leq 2\mathfrak d(x)$,
and $\|d(\exp{x}^{-1})_y\|\leq 2$ for $y\in \mathfrak B_x$.
\end{enumerate}

\medskip
The next two assumptions are on the regularity of the derivative $d\exp{x}$.
For $x,x'\in\ M\backslash\mathfs D$, let $\mathfs L _{x,x'}:=\{A:T_xM\to T_{x'}M:A\text{ is linear}\}$
and $\mathfs L _x:=\mathfs L_{x,x}$. 
Observe that the parallel transport $P_{y,x}$ considered in (A1) is in $\mathfs L_{y,x}$.
Given $y\in \mathfrak B_x,z\in \mathfrak B_{x'}$ and $A\in \mathfs L_{y,z}$,
let $\widetilde{A}\in\mathfs L_{x,x'}$, $\widetilde{A}:=P_{z,x'} \circ A\circ P_{x,y}$.
By definition, $\widetilde{A}$ depends on $x,x'$ but different basepoints define
a map that differs from $\widetilde{A}$ by pre and post composition with isometries.
In particular, $\|\widetilde{A}\|$ does not depend on the choice of $x,x'$.
Similarly, if $A_i\in\mathfs L_{y_i,z_i}$ then $\|\widetilde{A_1}-\widetilde{A_2}\|$ does
not depend on the choice of $x,x'$.
Define the map $\tau=\tau_x:\mathfrak B_x\times \mathfrak B_x\to \mathfs L_x$
by $\tau(y,z)=\widetilde{d(\exp{y}^{-1})_z}$, where we use the identification
$T_v(T_{y}M)\cong T_{y}M$ for all $v\in T_yM$.

\medskip
\noindent
{\sc Regularity of $d\exp{x}$:}
\begin{enumerate}[ii]
\item[(A3)] If $y_1,y_2\in \mathfrak B_x$ then
$
\|\widetilde{d(\exp{y_1})_{v_1}}-\widetilde{d(\exp{y_2})_{v_2}}\|
\leq d(x,\mathfs D)^{-a}\Sas(v_1,v_2)
$
for all $\|v_1\|,\|v_2\|\leq 2\mathfrak d(x)$, and 
$\|\tau(y_1,z_1)-\tau(y_2,z_2)\|\leq d(x,\mathfs D)^{-a}[d(y_1,y_2)+d(z_1,z_2)]$
for all $z_1,z_2\in \mathfrak B_x$.
%$
%\|\widetilde{d(\exp{y_1}^{-1})_{z_1}}-\widetilde{d(\exp{y_2}^{-1})_{z_2}}\|
%\leq d(x,\mathfs D)^{-a}[d(y_1,y_2)+d(z_1,z_2)]$ for all $z_1,z_2\in D_x$.
\item[(A4)] If $y_1,y_2\in \mathfrak B_x$ then the map 
$\tau(y_1,\cdot)-\tau(y_2,\cdot):\mathfrak B_x\to \mathfs L_x$
has Lipschitz constant $\leq d(x,\mathfs D)^{-a}d(y_1,y_2)$.
\end{enumerate}

\medskip
Conditions (A1)--(A2) say that the exponential maps and their inverses
are well-defined and have uniformly bounded Lipschitz constants in balls
of radii $d(x,\mathfs D)^a$.
Condition (A3) controls the Lipschitz constants of the derivatives of these maps,
and condition (A4) controls the Lipschitz constants of their second derivatives.
Here are some situations when (A1)--(A4) are satisfied:
\begin{enumerate}[$\circ$]
\item The curvature tensor $R$ of $M$ is globally bounded, e.g. when $M$ is the
phase space of a billiard map.
\item $R,\nabla R,\nabla^2 R,\nabla^3R$ grow at most polynomially
fast with respect to the distance to $\mathfs D$, e.g. when $M$ is a moduli space
of curves equipped with the Weil-Petersson metric \cite{Burns-Masur-Wilkinson}.
\end{enumerate}

Now we discuss the assumptions on the map. Consider a map
$f:M\backslash\mathfs D\to M$. 
We assume that $f$ is differentiable at every point $x\in M\backslash\mathfs D$, and we let
$\mathfs C=\{x\in M\backslash\mathfs D: df_x\text{ is not invertible}\}$ be the {\em critical set} of $f$.
We assume that $\mathfs C$ is a closed subset of $M$.

\medskip
\noindent
{\sc Singular set $\mathfs S$:} The {\em singular set of $f$} is
$\mathfs S:=\mathfs C\cup \mathfs D$. 

\medskip
The singular set $\mathfs S$ is closed. We assume that $f$ satisfies the following properties.

\medskip
\noindent
{\sc Regularity of $f$:} $\exists \mathfrak K>1$ s.t. for all $x\in M$ with 
$x,f(x)\notin\mathfs S$ there is $\min\{d(x,\mathfs S)^a,$ $d(f(x),\mathfs S)^a\}<\mathfrak r(x)<1$
s.t. for $D_x:=B(x,2\mathfrak r(x))$ and $E_x:=B(f(x),2\mathfrak r(x))$ the following holds:
\begin{enumerate}[.......]
\item[(A5)] The restriction of $f$ to $D_x$ is a diffeomorphism onto its image;
the inverse branch of $f$ taking $f(x)$ to $x$ is a well-defined diffeomorphism from 
$E_x$ onto its image.
%For every $y\in D_x$, the restriction of $f$ to $D_y$ is a diffeomorphism onto its image;
%if $g$ is the inverse branch of $f$ taking $f(y)$ to $y$, then $g$ is a well-defined diffeomorphism from 
%$E_y$ onto its image.
%
%If $f(x)\notin \mathfs S$ then the restriction $f\restriction_{D_x}$ is a diffeomorphism onto its image;
%if $y\in M\backslash \mathfs S$ with $f(y)=x$ then the inverse branch $f^{-1}\restriction_{D_x}$
%sending $x$ to $y$  is a diffeomorphism onto its image.
\item[(A6)] For all $y\in D_x$ it holds $d(x,\mathfs S)^a\leq \|df_y\|\leq d(x,\mathfs S)^{-a}$; for
all $z\in E_x$ it holds $d(x,\mathfs S)^a\leq \|dg_z\|\leq d(x,\mathfs S)^{-a}$,
where $g$ is the inverse branch of $f$ taking $f(x)$ to $x$.
\item[(A7)] For all $y,z\in D_x$ it holds $\|\wt{df_y}-\wt{df_z}\|\leq\mathfrak Kd(y,z)^\beta$;
for all $y,z\in E_x$ it holds $\|\wt{dg_y}-\wt{dg_z}\|\leq\mathfrak Kd(y,z)^\beta$,
where $g$ is the inverse branch of $f$ taking $f(x)$ to $x$.
\end{enumerate}

\medskip
Although technical, conditions (A5)--(A7) hold in most cases of interest, e.g.
if $\|df^{\pm 1}\|,\|d^2f^{\pm 1}\|$ grow at most polynomially fast with respect to
the distance to $\mathfs S$.

\begin{remark}\label{remark-mult-constants}
In practice, it is enough to check assumptions (A3), (A4) and (A6) up to multiplicative constants.
As a matter of fact, if they hold with multiplicative constants, then we can multiply the metric by a
small parameter and increase $a$ to obtain (A3), (A4) and (A6) as stated above. 
A similar reduction occurs in (A7): if 
$\|\wt{df_y}-\wt{df_z}\|\leq\mathfrak Kd(x,\mathfs S)^{-c}d(y,z)^\beta$ for some 
$c>0$, then we can replace $a$ by $\max\{a,2c/\beta\}$, $\beta$ by $\beta/2$,
and $\mathfrak K$ by $4^{\beta/2}\mathfrak K$ and obtain the first inequality in (A7).
\end{remark}

Finally, we define the measures that we are able to code.

\medskip
\noindent
{\sc $f$--adapted measure:} An $f$--invariant probability
measure $\mu$ on $M$ is said to be {\em $f$--adapted}
if $\log d(x,\mathfs S)\in L^1(\mu)$. A fortiori, $\mu(\mathfs S)=0$.

\medskip
Due to assumption (A6), if $\mu$ is $f$--adapted then the conditions of the
non-invertible version of the Oseledets theorem are satisfied. Therefore, if $\mu$ is $f$--adapted
then for $\mu$--a.e. $x\in M$ the Lyapunov exponent
$\chi(v)=\lim_{n\to+\infty}\tfrac{1}{n}\log\|df^n(v)\|$ exists for all $v\in T_xM$.
Among the adapted measures, we consider the hyperbolic ones. Fix $\chi>0$.

\medskip
\noindent
{\sc $\chi$--hyperbolic measure:} An $f$--invariant probability measure $\mu$ is called
{\em $\chi$--hyperbolic} if for $\mu$--a.e. $x\in M$ we have $|\chi(v)|>\chi$
for all $v\in T_xM$.

\medskip
The next theorem is the main result of this paper. For that, let $\widehat{f}:\widehat{M}\to\widehat{M}$
be the natural extension of $f$, and let $\widehat\mu$ be the lift of $\mu$, see Subsection
\ref{Section-natural-extension} for the definitions. 

\begin{theorem}[Main theorem]\label{Thm-Main}
Assume that $M,f$ satisfy assumptions {\rm (A1)--(A7)}. For all $\chi>0$, there is a set
${\rm NUH}^\#\subset \wh M$, a locally compact countable topological Markov shift $(\Sigma,\sigma)$ and
a H\"older continuous map $\pi:\Sigma\to\wh M$ s.t.:
\begin{enumerate}[{\rm (1)}]
\item $\pi\circ\sigma=\wh f\circ\pi$.
\item $\wh\mu[{\rm NUH}^\#]=1$ for all $f$--adapted $\chi$--hyperbolic measures $\mu$. 
\item $\pi\restriction_{\Sigma^\#}:\Sigma^\#\to{\rm NUH}^\#$ is a surjective
finite-to-one map.
\end{enumerate}
\end{theorem}

Above, $\Sigma^\#$ is the {\em recurrent set} of $\Sigma$; it carries all $\sigma$--invariant
probability measures, see Section \ref{Section-preliminaries} for the definition.
In general, $\pi$ might not be bounded-to-one. In summary, Theorem \ref{Thm-Main}
provides a single symbolic extension that is finite-to-one and onto a set that carries
all $f$--adapted $\chi$--hyperbolic measures.

%\medskip
%It is important to make some comments on the assumption of $f$--adaptedness.
%Ledrappier considered this for interval maps with critical points \cite{Ledrappier-acip},
%where he used the terminology {\em non-degenerate measure}. Katok and Strelcyn implicitly used
%that the Lebesgue measure is adapted to billiard maps \cite{Katok-Strelcyn}.
%For an invariant measure of a three dimensional flow with positive speed,
%Lima and Sarig constructed a Poincar\'e section for which the measure is adapted to 
%the respective Poincar\'e return map \cite{Lima-Sarig}, and Lima and Matheus
%used this assumption in their coding of billiard maps \cite{Lima-Matheus}. 
%For one dimensional maps satisfying (A1)--(A3), if $\mathfs D=\emptyset$
%and $\mathfs C$ is finite then every ergodic invariant measure is $f$--adapted \cite{Przytycki-Lyapunov},
%see also \cite[Appendix A]{Rivera-Letelier} for a proof that works under weaker assumptions.
%It would be interesting to obtain the
%same conclusion when $\mathfs S$ is finite, in which case $f$--adaptedness would be automatic.

\subsection{Method of proof}

Let $m$ denote the dimension of $M$.
The proof of Theorem \ref{Thm-Main} can be summarized in the following steps:
\begin{enumerate}[......]
\item[(1)] The derivative cocycle $df$ induces an invertible cocycle $\wh{df}$, defined on a fiber bundle over
the natural extension space $\wh M$, with the same spectrum as $df$.
\item[(2)] For a fixed $\chi>0$, define a nonuniformly hyperbolic locus ${\rm NUH}\subset \wh M$
of points with nonuniform hyperbolicity at least $\chi>0$.
\item[(3)] For each $\wh x\in{\rm NUH}$, define a Pesin chart $\Psi_{\wh x}$ from a subset of $\R^m$ to
$M$ for which $f$ and its inverse branche defined
by $\wh x$ are small perturbations of hyperbolic matrices.
\item[(4)] Define a set ${\rm NUH}^*\subset{\rm NUH}$
in which there are values $Q(\wh x)$ s.t. the restriction of $\Psi_{\wh x}$
to the ball with radius $Q(\wh x)$ is well-defined and $\lim\limits_{n\to\infty}\tfrac{1}{n}\log Q(\wh f^n(\wh x))=0$.
\item[(5)] Introduce the $\ve$--double chart $\Psi_{\wh x}^{p^s,p^u}$ as a pair of Pesin
charts with the same center but different domains. The parameters $p^s/p^u$ give definite sizes
for the stable/unstable manifolds at $\wh x$.
\item[(6)] Define abstract counterparts of stable/unstable manifolds in $\wh M$,
which we call stable/unstable sets.
\item[(7)] Construct a countable collection of $\ve$--double charts that are dense
in the space of all $\ve$--double charts, where denseness is defined in terms of
finitely many parameters of $\wh x$.
\item[(8)] Draw an edge $\Psi_{\wh x}^{p^s,p^u}\rightarrow\Psi_{\wh y}^{q^s,q^u}$
if $f$ and its respective inverse branch, when represented in the Pesin charts,
are small perturbations of hyperbolic matrices, and if the parameters $p^s,q^u$ are as large as possible.
\item[(9)] Using the stable/unstable sets, each path of $\ve$--charts defines an element of $\wh M$,
and this coding induces a {\em locally finite} countable cover on $\wh M$.
\item[(10)] Applying an abstract refinement procedure to this cover, the resulting partition defines
$(\Sigma,\sigma)$ and $\pi$ that satisfy Theorem \ref{Thm-Main}.
\end{enumerate}
In step (9), local finiteness means that each element of the cover intersects finitely many others.
It is important to stress the need of this property, since it guarantees that the partition in step (10)
is countable. If not, then we could end up with an uncountable partition
(e.g. the refinement of the cover of $\R$ by intervals with rational endpoints is uncountable).
Local finiteness is essential, and it is the reason that the whole proof is tailored in a way to provide
a precise control on all parameters involved in the construction.

The idea of considering double charts is due to Sarig \cite{Sarig-JAMS}, and it allows to treat the
hyperbolicity in the stable and unstable directions separately. We believe that this is one of the main
reasons for the method to work, together with a precise control of all the parameters involved in the construction.
The method has been further refined and simplified \cite{Ben-Ovadia-2019,Lima-Sarig,Lima-Matheus,Lima-IHP}.
Since here we allow derivatives to explode, an even finer control of the parameters is required.
For instance, while consisting on the general idea that stable/unstable subspaces improve smoothness
when backward/forward iterated, the improvement lemma of Section \ref{Section-improvement}
needs to be implemented at a scale that is small enough to have a control on derivatives and large
enough to code all relevant orbits. 

Additionally, neither $\wh M$ nor $\wh f$ are smooth objects.
This is a main issue in the proof, specially when dealing with graph transforms. One of the novelties
in this paper is to consider invariant sets instead of invariant manifolds and treat them
using both a geometrical (their zeroth positions are submanifolds) and an abstract perspective
(they are subsets of $\wh M$). This approach is what allows us to construct a Markov partition
for the natural extension.

We also mention a novelty in comparison with \cite{Lima-IHP}.
Even when working with nonuniformly expanding maps, we will use the parameter $p^s$.
Recalling that $p^s$ represents a choice of size for stable manifolds, it has no geometrical
meaning in this case, but its use provides a symmetrization between future and past and,
more importantly, allows to obtain a locally compact $(\Sigma,\sigma)$ in Theorem \ref{Thm-Main}.
This is better than \cite{Lima-IHP}, where the graph might have vertices with infinite degree.

\subsection{Related literature}\label{Section-related-literature}

Symbolic models for systems with hyperbolicity have a longstanding history that can be traced back 
at least to the 19th century with the study of geodesic flows on surfaces
with constant negative curvature by Hadamard \cite{Hadamard-1898}, see also \cite{Katok-Ugarcovici-Symbolic}.
During the late sixties and early seventies of the 20th century, the area saw a great deal of development when 
symbolic models were constructed for uniformly hyperbolic diffeomorphisms and flows, via the works of
of Adler \& Weiss \cite{Adler-Weiss-PNAS,Adler-Weiss-Similarity-Toral-Automorphisms},
Sina{\u\i} \cite{Sinai-Construction-of-MP,Sinai-MP-U-diffeomorphisms},
Bowen \cite{Bowen-MP-Axiom-A,Bowen-Symbolic-Flows}, and
Ratner \cite{Ratner-MP-three-dimensions,Ratner-MP-n-dimensions}.
Since then, other approaches were created, each of them treating different settings. 
In the sequel we discuss some of them. 

\medskip
\noindent
{\sc Hofbauer towers:} Takahashi created a combinatorial method
to construct an isomorphism between a large subset $X$ of the natural extension of $\beta$--shifts
and countable topological Markov shifts \cite{Takahashi-isomorphisms}.
Hofbauer showed that $X$ carries all measures of positive entropy and so
$\beta$--shifts have a unique measure of maximal entropy  \cite{Hofbauer-beta-shifts}.
Later, Hofbauer extended his construction to piecewise continuous interval maps
\cite{Hofbauer-intrinsic-I,Hofbauer-intrinsic-II}. The symbolic models obtained by his methods
are called {\em Hofbauer towers}, and they have been extensively used to understand ergodic properties
of one-dimensional maps.

\medskip
\noindent
{\sc Higher dimensional Hofbauer towers:} Buzzi constructed Hofbauer towers for piecewise
expanding affine maps in any dimension \cite{Buzzi-affine-maps}, for perturbations of fibered
products of one-dimensional maps \cite{Buzzi-produits-fibres}, and for arbitrary piecewise invertible
maps whose entropy generated by the boundary of some dynamically relevant partition is {\em less} than
the topological entropy of the map \cite{Buzzi-multidimensional}. 
These towers carry all invariant measures with entropy close enough to the topological
entropy of the system. 
%We observe that, contrary to our setting, Buzzi makes no assumption on the 
%nonuniform hyperbolicity of the system.

\medskip
\noindent
{\sc Inducing schemes:} Many systems, although not hyperbolic, do have sets on which
it is possible to define a (not necessarily first) return map that is uniformly hyperbolic.
This process is known as {\em inducing}. Hofbauer towers can actually be regarded as inducing schemes
to obtain a uniformly expanding map, see \cite{Bruin-inducing-Hofbauer} for this relation.
Some invariant measures of the original system lift to the induced one and, when this occurs,
the uniform hyperbolicity can be explored to provide ergodic theoretical properties 
of the original measure. This approach was done for one-dimensional maps \cite{Pesin-Senti},
higher dimensional maps that do not have full ``boundary entropy'' \cite{Pesin-Senti-Zhang},
and expanding measures \cite{Pinheiro-expanding}.

\medskip
\noindent
{\sc Yoccoz puzzles:} Yoccoz constructed Markov structures for quadratic maps of the complex
plane, now called {\em Yoccoz puzzles}, and used them to prove the MLC conjecture
for finitely renormalizable parameters, see e.g. \cite{Hubbard-Yoccoz-puzzles}.

\medskip
\noindent
{\sc Billiards:} Bunimovich, Chernov, and Sina{\u\i} constructed
countable Markov partitions for two dimensional
dispersing billiard maps \cite{Bunimovich-Chernov-Sinai}.
Young construced an inducing scheme for certain two dimensional dispersing billiard maps and
used it to prove exponential decay of correlations \cite{Young-towers}. Young's methods
are nowadays called {\em Young towers}, and are known to exist
for various billiards and in many other contexts, see e.g.
\cite{Chernov-1999,Markarian-polynomial} and references therein.
%Kr\"uger and Troubetzkoy constructed
%countable Markov partitions for non-uniformly hyperbolic billiard maps
%which include Bunimovich billiards \cite{Kruger-Troubetzkoy}.
Lima and Matheus constructed countable
Markov partitions for two dimensional billiard maps and hyperbolic
measures that are adapted to the billiard map \cite{Lima-Matheus}.

\medskip
\noindent
{\sc Symbolic extensions:} a symbolic extension is an extension by a subshift over a finite alphabet s.t.
the projection map is continuous. They reveal a great deal on the entropy
structure of topological systems, and are known to exist in various contexts, see 
\cite{Boyle-Downarowicz-Inventiones,Downarowicz-Newhouse,Downarowicz-Maass,Burguet-Inventiones,Downarowicz-Book}
and references therein.

\medskip
\noindent
{\sc Nonuniformly hyperbolic diffeomorphisms and flows:} Katok constructed horseshoes
of large topological entropy for $C^{1+\beta}$ diffeomorphisms \cite{KatokIHES}.
Such horseshoes usually have zero measure for measures of maximal entropy.
Sarig constructed countable Markov partitions with full topological entropy for $C^{1+\beta}$
surface diffeomorphisms \cite{Sarig-JAMS}.
Ben Ovadia extended the work of Sarig to any dimension \cite{Ben-Ovadia-2019}.
Lima and Sarig constructed symbolic models for nonuniformly hyperbolic three dimensional
flows with positive speed \cite{Lima-Sarig}.

\medskip
It is important noticing that the projection map $\pi$ given by Theorem \ref{Thm-Main} is {\em H\"older continuous},
which is usually not the case for Hofbauer towers neither for symbolic extensions. 
Also, its restriction to $\Sigma^\#$ is {\em finite-to-one} and hence preserves entropy.
Finally, as mentioned above, while Hofbauer towers only work in specific higher dimensional 
situations, our setting is broad enough to cover various contexts that were not 
known to have countable Markov partitions neither finite-to-one symbolic extensions by countable Markov shifts.

\section{Preliminaries}

\subsection{Basic definitions}\label{Section-preliminaries}

Let $\mathfs G=(V,E)$ be an oriented graph, where $V=$ vertex set and $E=$ edge set.
We denote edges by $v\to w$, and we assume that $V$ is countable.

\medskip
\noindent
{\sc Topological Markov shift (TMS):} A {\em topological Markov shift} (TMS) is a pair $(\Sigma,\sigma)$
where
$$
\Sigma:=\{\text{$\Z$--indexed paths on $\mathfs G$}\}=
\left\{\un{v}=\{v_n\}_{n\in\Z}\in V^{\Z}:v_n\to v_{n+1}, \forall n\in\Z\right\}
$$
%is the space of $\Z$--indexed paths on $\mathfs G$,
and $\sigma:\Sigma\to\Sigma$ is the left shift, $[\sigma(\un v)]_n=v_{n+1}$. 
The {\em recurrent set} of $\Sigma$ is
$$
\Sigma^\#:=\left\{\un v\in\Sigma:\exists v,w\in V\text{ s.t. }\begin{array}{l}v_n=v\text{ for infinitely many }n>0\\
v_n=w\text{ for infinitely many }n<0
\end{array}\right\}.
$$
We endow $\Sigma$ with the distance $d(\un v,\un w):={\rm exp}[-\min\{|n|\in\Z:v_n\neq w_n\}]$.
This choice is not canonical, and affects the H\"older regularity of
$\pi$ in Theorem \ref{Thm-Main}.

\medskip
Throughout this text, we will consider subsets of euclidean spaces $\mathbb R^n$.
All calculations in $\R^n$ will be made with respect to the 
canonical euclidean norm $\norm{\cdot}$ induced by the canonical inner product
$\langle \cdot,\cdot\rangle_{\R^n}$
of $\R^n$.
Given a linear transformation $T:\R^n\to \R^m$, let $\|T\|:=\sup_{v\in\R^n\atop{v\neq 0}}\tfrac{\|Tv\|}{\|v\|}$.
%where $\|v\|, \|Tv\|$ represent the supremum norms of $v,Tv$ in the spaces $\R^d,\R^n$ respectively.
Given an open bounded set $U\subset \R^n$ and $h:U\to \R^m$,
let $\|h\|_{C^0}:=\sup_{x\in U}\|h(x)\|$ denote the $C^0$ norm of $h$.
In most cases the domain $U$ will be clear, but when it is not we will write $\|h\|_{C^0(U)}$.

For $0<\delta\leq 1$, let $\Hol{\delta}(h):=\sup\frac{\norm{h(x)-h(y)}}{\|x-y\|^\delta}$ 
where the supremum ranges over distinct elements $x,y\in U$. 
Observe that $\Hol{1}(h)$ is the best Lipschitz constant for $h$, and so we denote
it by ${\rm Lip}(h)$. We call $h$ a {\em contraction} if ${\rm Lip}(h)<1$,
and an {\em expansion} if $h$ is invertible and ${\rm Lip}(h^{-1})<1$.
When $h$ is a linear transformation, then $\|h\|={\rm Lip}(h)$.
We also define $\norm{h}_{C^\delta}:=\norm{h}_{C^0}+\Hol{\delta}(h)$.
Now assume that $h$ is differentiable. By the mean value inequality,
${\rm Lip}(h)=\|dh\|_{C^0}$. In this case, define its $C^1$ norm by
$\|h\|_{C^1}:=\|h\|_{C^0}+\|dh\|_{C^0}=\|h\|_{C^0}+{\rm Lip}(h)$. Finally,
for $0<\delta\leq 1$ let
$$
\|h\|_{C^{1+\delta}}:=\|h\|_{C^1}+\Hol{\delta}(dh)=\|h\|_{C^0}+\|dh\|_{C^\delta}=
\|h\|_{C^0}+\|dh\|_{C^0}+\Hol{\delta}(dh)
$$
denote its $C^{1+\delta}$ norm.
%Below we collect some basic properties of these norms.
%
%
%\begin{lemma}\label{Lemma-Holder-norm}
%The following hold:
%\begin{enumerate}[{\rm (a)}]
%\item If $\vf,\psi$ are maps s.t. $\vf\circ\psi$ is well-defined, then
%$$
%\left\{
%\begin{array}{l}
%\Hol{\delta}(\vf\circ\psi)\leq {\rm Lip}(\vf)\Hol{\delta}(\psi)\\
%\\
%\Hol{\delta}(\vf\circ\psi)\leq \Hol{\delta}(\vf){\rm Lip}(\psi)^\delta.
%\end{array}
%\right.
%$$
%In particular, if $\psi$ is a contraction then $\norm{\vf\circ\psi}_{C^\delta}\leq \norm{\vf}_{C^\delta}$.
%\item (Perturbation of identity) If $\norm{\vf}_{C^\delta}<1$, then $I+\vf$ is invertible and
%$$
%\norm{(I+\vf)^{-1}}_{C^\delta}\leq \tfrac{1}{1-\norm{\vf}_{C^\delta}}\cdot
%$$
%Similarly, if $\vf$ is linear and $\norm{\vf}<1$, then $I+\vf$ is invertible and
%$$
%\norm{(I+\vf)^{-1}}\leq \tfrac{1}{1-\norm{\vf}}\cdot
%$$
%\end{enumerate}
%\end{lemma}

The euclidean norm is equivalent to any other norm in $\R^n$, but the 
notions of contraction and expansion differ from the choice of the norm. 
On the other hand, statements claiming that a norm is small are not too sensitive
to the choice of the norm. Since it is the euclidean norm that will be related to the
definition of the Lyapunov inner product (see Section \ref{Section-reduction}),
we prefer to work only with the euclidean norm.
Given $r>0$, denote the ball of $\R^n$ with center 0 and radius $r$ by $B^n[r]$. 
Observe that if $0\leq d\leq n$ then $B^n[r]\subset B^d[r]\times B^{n-d}[r]\subset B^n[\sqrt{2}r]$.

We fix a smooth Riemannian manifold $(M,\langle\cdot,\cdot\rangle)$ of dimension $m$ and finite diameter.
Multiplying the metric by a small constant if necessary, we assume that the diameter of $M$ is smaller than one.
We do not require $M$ to be connected neither compact.
For $r>0$, let $B_x[r]\subset T_xM$ denote the open ball with center 0 and radius $r$. 
Since $M$ will be fixed throughout the majority of the paper, we will denote the ball of $\R^m$
with center 0 and radius $r$ simply by $B[r]$. 
Write $a=e^{\pm\ve}b$ when $e^{-\ve}\leq \frac{a}{b}\leq e^\ve$,
and $a=\pm b$ when $-|b|\leq a\leq |b|$. We use $a\wedge b$ to denote $\min\{a,b\}$.

\subsection{Natural extensions}\label{Section-natural-extension}

Most of the discussion in this section is classical, see e.g. \cite{Rohlin-Exactness} or
\cite[\S 3.1]{Aaronson-book}.
Given a map $f:M\to M$, let
$$
\wh M:=\{\wh x=(x_n)_{n\in\Z}:f(x_{n-1})=x_n, \forall n\in\Z\}.
$$
We will write $\wh x=(\ldots,x_{-1};x_0,x_1,\ldots)$ where ; denotes the separation between
the positions $-1$ and 0.
Although $\wh M$ does depend on $f$, we will not write this dependence. Endow $\wh M$ with the distance
$\wh d(\wh x,\wh y):=\sup\{2^nd(x_n,y_n):n\leq 0\}$; then $\wh M$ is
a compact metric space. As for TMS, 
the definition of $\widehat d$ is not canonical and affects the
H\"older regularity of $\pi$ in Theorem \ref{Thm-Main}.
For each $n\in\Z$, let $\vartheta_n:\wh M\to M$ be the projection into the $n$--th
coordinate, $\vartheta_n[\wh x]=x_n$. Consider the 
sigma-algebra in $\wh M$ generated by $\{\vartheta_n:n\leq 0\}$, i.e.
the smallest sigma-algebra that makes all $\vartheta_n$, $n\leq 0$, measurable.

\medskip
\noindent
{\sc Natural extension of $f$:} The {\em natural extension} of $f$ is the 
map $\wh f:\wh M\to\wh M$ defined by
$\wh f(\ldots,x_{-1};x_0,\ldots)=(\ldots,x_0;f(x_0),\ldots)$.
It is an invertible map, with inverse $\wh f^{-1}(\ldots,x_{-1};x_0,\ldots)=(\ldots,x_{-2};x_{-1},\ldots)$.

\medskip
Note that $\wh f$ is indeed an extension
of $f$, since $\vt_0\circ \wh f=f\circ\vt_0$. It is the
smallest invertible extension of $f$: any other invertible extension of $f$ is an 
extension of $\wh f$. The benefit of considering the natural extension is that,
in addition to having an invertible map explicitly defined,
there is a bijection between $f$--invariant and $\wh f$--invariant probability measures,
as follows.

\medskip
\noindent
{\sc Projection of a measure:} If $\wh\mu$ is an $\wh f$--invariant probability measure, then
$\mu=\wh\mu\circ \vt_0^{-1}$ is an $f$--invariant probability measure.

\medskip
\noindent
{\sc Lift of a measure:} If $\mu$ is an $f$--invariant probability measure,
let $\wh\mu$ be the unique probability measure on $\wh M$ s.t.
$\wh\mu[\{\wh x\in\wh M:x_n\in A\}]=\mu[A]$ for all $A\subset M$ Borel and all $n\leq 0$.

\medskip
It is clear that $\wh\mu$ is $\wh f$--invariant. It is also clear that if $\mu,\wh\mu$ are related
either by the projection or lift procedures, then for every measurable $\vf:M\to\R$ it
holds $\int_M\vf d\mu=\int_{\wh M}(\vf\circ\vt_0)d\wh\mu$.
What is less clear is that the projection and lift procedures are inverse
operations, and that they preserve ergodicity and the Kolmogorov-Sina{\u\i} entropy, see \cite{Rohlin-Exactness}.
Here is one consequence of these facts: $\mu$ is an ergodic equilibrium measure for a potential $\varphi:M\to\R$ iff
$\wh\mu$ is an ergodic equilibrium measure for $\varphi\circ\vt_0:\wh M\to\R$. 
In particular, the topological entropies of $f$ and $\wh f$ coincide, and 
$\mu$ is a measure of maximal entropy for $f$ iff $\wh\mu$ is a measure
of maximal entropy for $\wh f$.

%These procedures are one the inverse of the other. Furthermore,
%they preserve entropy. A motivation for this fact is the following: if $\alpha$ is a generating partition
%of $f$ then $\vt^{-1}\alpha$ is a generating partition of $\wh f$, and the entropy of $\alpha$
%with respect to $f$ is the same as the entropy of $\vt^{-1}\alpha$ with respect to $\wh f$.
%Although $\wh M$ is much bigger that $M$ and $\wh f$ seems more
%complicated than $f$, for the purpose of many ergodic theoretical properties they
%are the same, since there is an association between $f$--invariant and $\wh f$--invariant
%measures that preserve the Kolmogorov-Sina{\u\i} entropy. 

Now let $N=\bigsqcup_{x\in M}N_x$ be a vector bundle over $M$, and let $A:N\to N$ measurable
s.t. for every $x\in M$ the restriction $A\restriction_{N_x}$ is a linear
isomorphism $A_x:N_x\to N_{f(x)}$. The map $A$ defines a (possibly non-invertible) cocycle $(A^{(n)})_{n\geq 0}$
over $f$ by $A^{(n)}_x=A_{f^{n-1}(x)}\cdots A_{f(x)}A_x$ for $x\in M$, $n\geq 0$.
There is a way of extending $(A^{(n)})_{n\geq 0}$ to an invertible cocycle over $\wh f$.
For $\wh x\in\wh M$, let $N_{\wh x}:=N_{\vt_0[\wh x]}$ and let
$\wh N:=\bigsqcup_{\wh x\in\wh M} N_{\wh x}$, a vector bundle over $\wh M$.
Define the map $\wh A:\wh N\to\wh N$, $\wh A_{\wh x}:=A_{\vt_0[\wh x]}$.
For $\wh x=(x_n)_{n\in\Z}$, define 
$$
\wh A^{(n)}_{\wh x}:=
\left\{
\begin{array}{ll}
A^{(n)}_{x_0}&,\text{ if }n\geq 0\\
A_{x_{-n}}^{-1}\cdots A_{x_{-2}}^{-1}A_{x_{-1}}^{-1}&,\text{ if }n\leq 0.
\end{array}
\right.
$$
By definition,
$\wh A^{(m+n)}_{\wh x}=\wh A^{(m)}_{\wh f^n(\wh x)}\wh A^{(n)}_{\wh x}$
for all $\wh x\in\wh M$ and all $m,n\in\Z$, hence $(\wh A^{(n)})_{n\in\Z}$
is an invertible cocycle over $\wh f$.

Let $M,f$ satisfying (A1)--(A7), and take $N=TM$ and $A=df$. 
Define $\rho:\wh M\to \R$ by $\rho(\wh x):=d\left(\{\vt_{-1}[\wh x],\vt_{0}[\wh x],\vt_{1}[\wh x]\},\mathfs S\right)$.
We note that $\mu$ is $f$--adapted iff $\log\rho\in L^1(\wh\mu)$:
\begin{enumerate}[$\circ$]
\item If $\mu$ is $f$--adapted then by $\wh f$--invariance the functions
$\log d(\vt_i[\wh x],\mathfs S)$, $i=-1,0,1$, are in $L^1(\wh\mu)$ and so is their minimum $\log\rho$. 
\item Reversely,
if $\log\rho\in L^1(\wh\mu)$ then $\log d(\vt_0[\wh x],\mathfs S)\in L^1(\wh\mu)$ and so
$$
\int_M |\log d(x,\mathfs S)|d\mu(x)=\int_{\wh M} |\log d(\vt_0[\wh x],\mathfs S)|d\wh\mu(\wh x)<\infty,$$
thus proving that $\log d(x,\mathfs S)\in L^1(\mu)$.
\end{enumerate}
For simplicity, we will write $\vt$ to represent $\vartheta_0$.
If $\theta$ is a function on $M$, we will sometimes represent the composition
$\theta\circ\vt$ also by $\theta$.
For example, $\mathfrak d(\wh x):=\mathfrak d(\vt[\wh x])$ and
$\mathfrak B_{\wh x}:=\mathfrak B_{\vt[\wh x]}$.

\part{Proof of the main theorem}\label{Part-Main-Thm}

\section{Pesin theory}\label{Section-Pesin-theory}

Consider the derivative cocycle $(df^n_x)_{n\geq 0}$. Since $df_x$ is not necessarily invertible for
all $x\in M$, we apply the method described in the previous section to extend it to an invertible
cocycle, defined in a subset of $\wh M$.

\medskip
\noindent
{\sc Invertible cocycle:} The cocycle $(df^n_x)_{n\geq 0}$ induces an invertible cocycle
$(\wh{df}^{(n)}_{\wh x})_{n\in\Z}$, defined for all
$\wh x=(x_n)_{n\in\Z}\in \wh M\backslash \bigcup_{n\in\Z}{\wh f}^n(\vt^{-1}[\mathfs S])$ by
$$
\wh{df}^{(n)}_{\wh x}:=
\left\{
\begin{array}{ll}
df^{n}_{x_0}&,\text{ if }n\geq 0\\
(df_{x_{-n}})^{-1}\cdots (df_{x_{-2}})^{-1}(df_{x_{-1}})^{-1}&,\text{ if }n\leq 0.
\end{array}
\right.
$$

\medskip
Since $\wh x\not\in \bigcup_{n\in\Z}{\wh f}^n(\vt^{-1}[\mathfs S])$, each $df_{x_n}$ is invertible
and so the above definition makes sense. The vector bundle in which $(\wh{df}^{(n)}_{\wh x})_{n\in\Z}$
acts is $\wh{TM}$, where $\wh{TM}_{\wh x}=T_{x_0}M$. 
We consider a parameter $\chi >0$, fixed throughout all of Part \ref{Part-Main-Thm}.

\medskip
\noindent
{\sc The nonuniformly hyperbolic locus ${\rm NUH}$:} It is defined as the set of points \label{Def-NUH}
$\wh x\in \wh M\backslash \bigcup_{n\in\Z}{\wh f}^n(\vt^{-1}[\mathfs S])$ for which there is a splitting
$\wh{TM}_{\wh x}=E^s_{\wh x}\oplus E^u_{\wh x}$ s.t.:
\begin{enumerate}[(NUH1)]
\item Every $v\in E^s_{\wh x}$ contracts in the future at least $-\chi$ and expands in the past:
$$\limsup_{n\to+\infty}\tfrac{1}{n}\log \|\wh{df}^{(n)} v\|\leq -\chi\ \text{ and } 
\ \liminf_{n\to+\infty}\tfrac{1}{n}\log \|\wh{df}^{(-n)} v\|>0.
$$ 
\item Every $v\in E^u_{\wh x}$ contracts in the past at least $-\chi$ and expands in the future:
$$\limsup_{n\to+\infty}\tfrac{1}{n}\log \|\wh{df}^{(-n)} v\|\leq -\chi\text{ and } 
\liminf_{n\to+\infty}\tfrac{1}{n}\log \|\wh{df}^{(n)}v\|>0.$$
\item \label{Def-NUH3} The parameters $s(\wh x)=\sup\limits_{v\in E^s_{\wh x}\atop{\|v\|=1}}S(\wh x,v)$
and $u(\wh x)=\sup\limits_{w\in E^u_{\wh x}\atop{\|w\|=1}}U(\wh x,w)$ are finite, where:
\begin{align*}
S(\wh x,v)&=\sqrt{2}\left(\sum_{n\geq 0}e^{2n\chi}\|\wh{df}^{(n)} v\|^2\right)^{1/2},\\
U(\wh x,w)&=\sqrt{2}\left(\sum_{n\geq 0}e^{2n\chi}\|\wh{df}^{(-n)}w\|^2\right)^{1/2}.
\end{align*}
\end{enumerate}

The above conditions are weaker than Lyapunov regularity, hence
${\rm NUH}$ contains all Lyapunov regular points with exponents greater than $\chi$
in absolute value. Moreover, a periodic point $\wh x$ is in ${\rm NUH}$ iff all of its
exponents are greater than $\chi$ in absolute value. 
But ${\rm NUH}$ might contain points with some Lyapunov exponents equal to $\pm\chi$,
and even non-regular points, where the contraction rates oscillate infinitely often. 
Formally, ${\rm NUH}={\rm NUH}_{\chi}$ depends on $\chi>0$, but since
this parameter is fixed, we will omit such dependence. When we need to
consider ${\rm NUH}$ for different values of $\chi$, we will be explicit.
Clearly, ${\rm NUH}$ is $\wh f$--invariant. By the Oseledets theorem, 
if $\mu$ is $\chi$--hyperbolic then (NUH1)--(NUH3) hold $\wh\mu$--a.e.

\subsection{Oseledets-Pesin reduction}\label{Section-reduction}

We provide a diagonalization for the cocycle $(\wh{df}^{(n)}_{\wh x})_{n\in\Z}$
inside $ {\rm NUH}$.
Write $E^s\oplus E^u$ to represent the splitting on ${\rm NUH}$.

\medskip
\noindent {\sc Lyapunov inner product:} We define an inner product $\llangle \cdot,\cdot\rrangle$ on
${\rm NUH}$, which we call {\em Lyapunov inner product}, by the following identities:
\begin{enumerate}[$\circ$]
\item For $v_1^s,v_2^s\in E^s$:
$$
\llangle v^s_1,v^s_2\rrangle=2\sum_{n\geq 0}e^{2n\chi}\left\langle \wh{df}^{(n)}v^s_1,\wh{df}^{(n)}v^s_2\right\rangle.
$$
\item For $v_1^u,v_2^u\in E^u$:
$$
\llangle v^u_1,v^u_2\rrangle=2\sum_{n\geq 0}e^{2n\chi}\left\langle \wh{df}^{(-n)}v^u_1,\wh{df}^{(-n)}v^u_2\right\rangle.
$$
\item For $v^s\in E^s$ and $v^u\in E^u$:
$$
\llangle v^s,v^u\rrangle=0.
$$
\end{enumerate}

By conditions (NUH1)--(NUH3), the infinite sums above are finite.
Let $\vertiii{\cdot}$ denote the norm induced by $\llangle \cdot,\cdot\rrangle$.
As we know, $\vertiii{\cdot}$ uniquely defines $\llangle \cdot,\cdot\rrangle$, so 
we also call $\vertiii{\cdot}$ the Lyapunov inner product.
Because of the term $n=0$ that appears in the sums defining $\llangle\cdot,\cdot\rrangle$,
we have $\vertiii{v}\geq 2\|v\|$ for all $v\in E^s\oplus E^u$.
Furthermore, if $v^s\in E^s_{\wh x}$ then $\vertiii{v^s}=S(\wh x,v^s)$, and if
$v^u\in E^u_{\wh x}$ then $\vertiii{v^u}=U(\wh x,v^u)$.

For $\wh x\in {\rm NUH}$, let $d_s(\wh x),d_u(\wh x)\in\N$ be the dimensions
of $E^s_{\wh x},E^u_{\wh x}$ respectively. Since the splitting $E^s\oplus E^u$ is $\wh{df}$--invariant,
the functions $d_s,d_u$ are $\wh f$--invariant. Recall that $\langle \cdot,\cdot\rangle_{\R^n}$
denotes the canonical inner product in $\R^n$.

\medskip
\noindent
{\sc Linear map $C(\wh x):$} For $\wh x\in {\rm NUH}$, define $C(\wh x): \R^n\to \wh{TM}_{\wh x}$ to be a linear map
satisfying the following conditions:
\begin{enumerate}[$\circ$]
\item $C(\wh x)$ sends the subspace $\R^{d_s(\wh x)}\times\{0\}$ to $E^s_{\wh x}$
and $\{0\}\times \R^{d_u(\wh x)}$ to $E^u_{\wh x}$.
\item $\langle v,w\rangle_{\R^n}=\llangle C(\wh x)v,C(\wh x)w\rrangle$ for all $v,w\in\R^n$, i.e.
$C(\wh x)$ is an isometry between $(\R^n,\langle \cdot,\cdot\rangle_{\R^n})$ and
$(\wh{TM}_{\wh x},\llangle\cdot,\cdot\rrangle)$.
\end{enumerate}

\medskip
The map $C(\wh x)$ is not uniquely defined (for instance, rotations inside $E^s_{\wh x},E^u_{\wh x}$
preserve both properties above). What is important for us is to define $C(\wh x)$ s.t. the map
$\wh x\in{\rm NUH}\mapsto C(\wh x)$ is measurable, and this can be done as in
\cite[Footnote at page 48]{Ben-Ovadia-2019}. Below we list the main properties of $C(\wh x)$.

\begin{lemma}\label{Lemma-linear-reduction}
For all ${\wh x}\in{\rm NUH}$, the following holds.
\begin{enumerate}[{\rm (1)}]
\item $\|C({\wh x})\|\leq 1$ and 
$$
\|C(\wh x)^{-1}\|^2=
\sup_{v^s+v^u\in E^s_{\wh x}\oplus E^u_{\wh x}\atop{\|v^s+v^u\|\neq 0}}\frac{\vertiii{v^s}^2+\vertiii{v^u}^2}{\|v^s+v^u\|^2}\cdot
%\leq\left[\frac{s(\wh x)+u(\wh x)}{|\sin \alpha(\wh x)|}\right]^2,
$$
%where $P^{s/u}:\wh{TM}\to E^{s/u}$ is the oblique projection with respect to $E^{u/s}$. 
\item $D({\wh x})=C(\wh f({\wh x}))^{-1}\circ \wh{df}_{\wh x}\circ C({\wh x})$ is a block matrix 
$$
D({\wh x})=\left[\begin{array}{cc}D_s({\wh x}) & 0 \\
0 & D_u({\wh x})
\end{array}\right]
$$
where $D_s({\wh x})$ is a $d_s(\wh x)\times d_s(\wh x)$ matrix
with $\|D_s({\wh x})v_1\|< e^{-\chi}$ for all unit vectors $v_1\in \R^{d_s(\wh x)}$,
and  $D_u({\wh x})$ is a $d_u(\wh x)\times d_u(\wh x)$ matrix with
$\|D_u({\wh x})v_2\|>e^{\chi}$ for all unit vectors $v_2\in \R^{d_u(\wh x)}$.
\item If $v_1\in \R^{d_s(\wh x)}$ is unitary then 
$$
\|D_s({\wh x})v_1\|\geq d(\vt[\wh x],\mathfs S)^a(1+e^{2\chi})^{-1/2},
$$
and if $v_2\in \R^{d_u(\wh x)}$ is unitary then
$$
\|D_u({\wh x})v_2\|\leq  d(\vt[\wh x],\mathfs S)^{-a}(1+e^{2\chi})^{1/2}.
$$
\item $d(\vt[\wh x],\mathfs S)^{2a}(1+e^{2\chi})^{-1/2}\leq 
\tfrac{\norm{C(\wh f(\wh x))^{-1}}}{\norm{C(\wh x)^{-1}}}\leq d(\vt[\wh x],\mathfs S)^{-2a}(1+e^{2\chi})^{1/2}$.
\end{enumerate}
\end{lemma}

\begin{proof}
Parts (1) and (2) can be proved as in \cite[Lemma~2.9]{Ben-Ovadia-2019}, but for later use on some
estimates we will prove them. Let $v=v_1+v_2$, where $v_1\in\R^{d_s(\wh x)}\times\{0\}$
and $v_2\in\{0\}\times\R^{d_u(\wh x)}$. By the Cauchy-Schwarz inequality and the
arithmetic-geometric mean inequality, we have
$$
2\langle C(\wh x)v_1,C(\wh x)v_2\rangle \leq 2\|C(\wh x)v_1\|\cdot\|C(\wh x)v_2\|\leq \|C(\wh x)v_1\|^2+\|C(\wh x)v_2\|^2,
$$
hence 
\begin{align*}
&\, \|C(\wh x)v\|^2=\|C(\wh x)v_1\|^2+\|C(\wh x)v_2\|^2+2\langle C(\wh x)v_1,C(\wh x)v_2\rangle\\
&\leq 2(\|C(\wh x)v_1\|^2+\|C(\wh x)v_2\|^2)
\leq \vertiii{C(\wh x)v_1}^2+\vertiii{C(\wh x)v_2}^2\\
&=\|v_1\|^2+\|v_2\|^2=\|v\|^2,
\end{align*}
thus proving that $\|C(\wh x)\|\leq 1$. The formula for $\|C(\wh x)^{-1}\|$ is almost automatic:
if $v=v^s+v^u\in E^s_{\wh x}\oplus E^u_{\wh x}$ then 
$\|C(\wh x)^{-1}v\|^2=\llangle v,v\rrangle = \vertiii{v^s}^2+\vertiii{v^u}^2$ and so
$$
\|C(\wh x)^{-1}\|^2=\sup_{v\in\wh{TM}_{\wh x}\atop{v\neq 0}}\frac{\vertiii{v^s}^2+\vertiii{v^u}^2}{\|v\|^2}\cdot
$$
This proves part (1). Now we prove part (2). 
By the definition of $C(\cdot)$, we have
$$
\R^{d_s(\wh x)}\times\{0\} \xrightarrow[]{C(\wh x)} E^s_{\wh x} \xrightarrow[]{\ \wh{df}_{\wh x}\ }
E^s_{\wh f(\wh x)} \xrightarrow[]{C(\wh f(\wh x))^{-1}} \R^{d_s(\wh f(\wh x))}\times \{0\}.
$$
Similarly,
$$
\{0\}\times\R^{d_u(\wh x)} \xrightarrow[]{C(\wh x)} E^u_{\wh x} \xrightarrow[]{\ \wh{df}_{\wh x}\ }
E^u_{\wh f(\wh x)} \xrightarrow[]{C(\wh f(\wh x))^{-1}} \{0\}\times\R^{d_u(\wh f(\wh x))},
$$
and so $D({\wh x})$ has the block form 
$$
D({\wh x})=\left[\begin{array}{cc}D_s({\wh x}) & 0 \\
0 & D_u({\wh x})
\end{array}\right],
$$
where $D_s(\wh x)$ has dimension $d_s(\wh x)$ and $D_u(\wh x)$ has dimension $d_u(\wh x)$.
To obtain the estimates of $\|D_s(\wh x)\|$ and $\|D_u(\wh x)\|$, we give general estimates
for the ratios $\tfrac{\vertiii{\wh{df} v^s}^2}{\vertiii{v^s}^2}$ and $\tfrac{\vertiii{\wh{df} v^u}^2}{\vertiii{v^u}^2}$.\\

\noindent
\medskip
{\sc Claim:} For all nonzero $v^s\in E^s, v^u\in E^u$ it holds
\begin{align*}
&d(\vt[\wh x],\mathfs S)^a(1+e^{2\chi})^{-1/2}
<\tfrac{\vertiii{\wh{df} v^s}}{\vertiii{v^s}}< e^{-\chi} \ \text{ and }\\
&e^\chi<\tfrac{\vertiii{\wh{df} v^u}}{\vertiii{v^u}}<
d(\vt[\wh x],\mathfs S)^{-a}(1+e^{2\chi})^{1/2}.
\end{align*}

\begin{proof}[Proof of the claim.]
Write $\wh x=(x_n)_{n\in\Z}$. Observe that for
$v^s\in E^s$ we have
\begin{align}\label{relation-s}
\vertiii{v^s}^2=2\|v^s\|^2+2e^{2\chi}\sum_{n\geq 0}e^{2n\chi}\|\wh{df}^{(n)}(\wh{df}v^s)\|^2=
2\|v^s\|^2+e^{2\chi}\vertiii{\wh{df}v^s}^2
\end{align}
and similarly for $v^u\in E^u$ we have
\begin{align}\label{relation-u}
\vertiii{v^u}^2=2\|v^u\|^2+2e^{2\chi}\sum_{n\geq 0}e^{2n\chi}\left\|\wh{df}^{(-n)}(\wh{df}^{(-1)}v^u)\right\|^2=
2\|v^u\|^2+e^{2\chi}\vertiii{\wh{df}^{(-1)}v^u}^2.
\end{align}
%Actually, the first two estimates are proved exactly as in the proof of Lemma \ref{Lemma-linear-reduction}(3),
%but for completeness we provide the details.
Equation (\ref{relation-s}) implies that
$$
\tfrac{\vertiii{\wh{df} v^s}^2}{\vertiii{v^s}^2}=e^{-2\chi}\left(1-\tfrac{2\|v^s\|^2}{\vertiii{v^s}^2}\right)\cdot
$$
On one hand, $\tfrac{\vertiii{\wh{df} v^s}^2}{\vertiii{v^s}^2}< e^{-2\chi}$. On
the other hand, since
$\vertiii{v^s}^2> 2\|v^s\|^2+2e^{2\chi}\|\wh{df} v^s\|^2\geq 2[1+e^{2\chi}d(x_0,\mathfs S)^{2a}]\|v^s\|^2$,
we have $\tfrac{2\|v^s\|^2}{\vertiii{v^s}^2}< \tfrac{1}{1+e^{2\chi}d(x_0,\mathfs S)^{2a}}$ and so
$$
\tfrac{\vertiii{\wh{df} v^s}^2}{\vertiii{v^s}^2}> e^{-2\chi}\left(1-\tfrac{1}{1+e^{2\chi}d(x_0,\mathfs S)^{2a}}\right)
=\tfrac{d(x_0,\mathfs S)^{2a}}{1+e^{2\chi}d(x_0,\mathfs S)^{2a}}
>\tfrac{d(x_0,\mathfs S)^{2a}}{1+e^{2\chi}}\cdot
$$
Hence $d(x_0,\mathfs S)^a(1+e^{2\chi})^{-1/2}<\tfrac{\vertiii{\wh{df} v^s}}{\vertiii{v^s}}< e^{-\chi}$.

The second estimate is proved similarly. Using equation (\ref{relation-u}) with $\wh{df}v^u$ instead of $v^u$,
we have $\vertiii{\wh{df}v^u}^2=2\|\wh{df}v^u\|^2+e^{2\chi}\vertiii{v^u}^2$ and so
$$
\tfrac{\vertiii{v^u}^2}{\vertiii{\wh{df}v^u}^2}=e^{-2\chi}\left(1-\tfrac{2\|\wh{df}v^u\|^2}{\vertiii{\wh{df}v^u}^2}\right)\cdot
$$
Then $\tfrac{\vertiii{v^u}^2}{\vertiii{\wh{df}v^u}^2}<e^{-2\chi}$ and, since
$\vertiii{\wh{df}v^u}^2>2\|\wh{df}v^u\|^2+2e^{2\chi}\|\wh{df}^{(-1)}(\wh{df}v^u)\|^2\geq
2[1+e^{2\chi}d(x_0,\mathfs S)^{2a}]\|\wh{df}v^u\|^2$, we also have 
$$
\tfrac{\vertiii{v^u}^2}{\vertiii{\wh{df}v^u}^2}> e^{-2\chi}\left(1-\tfrac{1}{1+e^{2\chi}d(x_0,\mathfs S)^{2a}}\right)
>\tfrac{d(x_0,\mathfs S)^{2a}}{1+e^{2\chi}},
$$
thus implying that  $e^\chi<\tfrac{\vertiii{\wh{df} v^u}}{\vertiii{v^u}}<d(x_0,\mathfs S)^{-a}(1+e^{2\chi})^{1/2}$.
\end{proof}

The claim actually implies parts (2) and (3), since if
$v^s\in E^s$ and $v^u\in E^u$ are nonzero vectors then the definition of $C(\wh x)$ implies that
$\tfrac{\|D_s(\wh x)v^s\|}{\|v^s\|}=\tfrac{\vertiii{\wh{df}\circ C(\wh x)v^s}}{\vertiii{C(\wh x)v^s}}$
and $\tfrac{\|D_u(\wh x)v^u\|}{\|v^u\|}=\tfrac{\vertiii{\wh{df}\circ C(\wh x)v^u}}{\vertiii{C(\wh x)v^u}}$.
To prove part (4), observe that since $\wh {df}_{\wh x}$ is invertible, part (1) implies
that
$$
\norm{C(\wh f(\wh x))^{-1}}^2=
\sup_{v^s+v^u\in E^s_{\wh x}\oplus E^u_{\wh x}\atop{\|v^s+v^u\|\neq 0}}
\tfrac{\vertiii{\wh{df} v^s}^2+\vertiii{\wh{df} v^u}^2}{\norm{\wh{df}(v^s+v^u)}^2}\cdot 
$$
Now apply the claim and (A6) to conclude part (4).
\end{proof}

In summary, what we did above was to define, inside ${\rm NUH}$,
a Lyapunov inner product that depends on $\chi>0$. The construction
works because of conditions (NUH1)--(NUH3). For $\chi'<\chi$, 
we can perform the same construction inside ${\rm NUH}_{\chi'}$, obtaining
a different Lyapunov inner product $\vertiii{\cdot}_{\chi'}$.
Since ${\rm NUH}_{\chi'}\supset{\rm NUH}$, the Lyapunov inner
product $\vertiii{\cdot}_{\chi'}$ is also defined inside ${\rm NUH}$.
Notice that $\vertiii{\cdot}_{\chi'}\leq \vertiii{\cdot}$, so we say that 
$\vertiii{\cdot}$ is stronger than $\vertiii{\cdot}_{\chi'}$.
For $\wh x\in{\rm NUH}_{\chi'}$, define a linear map $C_{\chi'}(\wh x)$ using $\vertiii{\cdot}_{\chi'}$.
Then $C_{\chi'}(\wh x)$ satisfies an analogue of Lemma \ref{Lemma-linear-reduction}
for a block matrix $D_{\chi'}(\wh x)$, with $\chi$ replaced by $\chi'$.
Note that $D(\wh x)$ has better hyperbolicity rates than $D_{\chi'}(\wh x)$.
Also, since $\vertiii{\cdot}_{\chi'}\leq \vertiii{\cdot}$ then part (1) implies that
$\norm{C_{\chi'}(\wh x)^{-1}}\leq \norm{C(\wh x)^{-1}}$.

The Lyapunov inner product $\vertiii{\cdot}$ will be used in the majority of the paper,
but there is one subtle point that we will need to use $\vertiii{\cdot}_{\chi'}$ (see Section \ref{Section-finite}),
as successfully implemented by Ben Ovadia \cite{Ben-Ovadia-2020}.

\subsection{Pesin charts}

The next step is to define charts for which
$\wh f$ is a small perturbation of a hyperbolic matrix. Recall
that $\mathfrak d(\wh x):=\mathfrak d(\vt[\wh x])$.
%Remember assumptions (A1)--(A2) on the maps $\exp{x}:[-\mathfrak r(x),\mathfrak r(x)]^2\to M$,
%$x\in M\backslash\mathfs D$.
%there is a constant $\mathfrak K>0$ s.t. for every $x\in M\backslash\mathfs S$, the map
%$\exp{x}:[-\mathfrak r(x),\mathfrak r(x)]^2\to M$ satisfies assumptions (A1)--(A2) for some
%$\mathfrak r(x)> d(x,\mathfs D)$.
%We use these maps to define the Pesin charts.

\medskip
\noindent
{\sc Pesin chart $\Psi_{\wh x}$:} The {\em Pesin chart} at $\wh x\in{\rm NUH}$ is
$\Psi_{\wh x}: B[\mathfrak d (\wh x)]\to M$ defined by
$\Psi_{\wh x}:=\exp{\vt[\wh x]}\circ C(\wh x)$. 

\begin{lemma}\label{Lemma-Pesin-chart}
The Pesin chart $\Psi_{\wh x}$ is a diffeomorphism onto its image s.t.:
\begin{enumerate}[{\rm (1)}]
\item $\Lip{\Psi_{\wh x}}\leq 2$ and $\Lip{\Psi_{\wh x}^{-1}}\leq 2\|C({\wh x})^{-1}\|$.
\item
$\|\widetilde{d(\Psi_{\wh x})_{v_1}}-\widetilde{d(\Psi_{\wh x})_{v_2}}\|\leq d(\vt[\wh x],\mathfs S)^{-a}\|v_1-v_2\|$
for all $v_1,v_2\in B[\mathfrak d(\wh x)]$.
\end{enumerate}
\end{lemma}

\begin{proof}
The proof is very similar to the proof of \cite[Lemma~3.1]{Lima-Matheus}.
Write $\wh x=(x_n)_{n\in\Z}$.
Since $C(\wh x)$ is a contraction, we have $C(\wh x)(B[\mathfrak d(x_0)])\subset B_{x_0}[\mathfrak d(x_0)]$
and so by (A1) it follows that $\Psi_{\wh x}$ is a well-defined differentiable map 
with $\Psi_{\wh x}(B[\mathfrak d(x_0)])\subset\mathfrak B_{x_0}$.
Its inverse is $C(\wh x)^{-1}\circ \exp{x_0}^{-1}$ which,
again by (A1), is a well-defined differentiable map from $\Psi_{\wh x}(B[\mathfrak d(x_0)])$
onto $B[\mathfrak d(x_0)]$. This proves that $\Psi_{\wh x}$ is a diffeomorphism onto its image.

\medskip
\noindent
(1) By Lemma \ref{Lemma-linear-reduction}(1) and (A2), $\Lip{\Psi_{\wh x}}\leq 2$
and $\Lip{\Psi_{\wh x}^{-1}}\leq 2\|C(\wh x)^{-1}\|$.

\medskip
\noindent
(2) Since $\|C(\wh x)v_i\|\leq \|v_i\|\leq \mathfrak d(x_0)$, assumption (A3) implies that
\begin{align*}
&\, \|\widetilde{d(\Psi_{\wh x})_{v_1}}-\widetilde{d(\Psi_{\wh x})_{v_2}}\|=
\|\widetilde{d(\exp{x_0})_{C(\wh x)v_1}}\circ C(\wh x)-
\widetilde{d(\exp{x_0})_{C(\wh x)v_2}}\circ C(\wh x)\|\\
&\leq d(x_0,\mathfs S)^{-a}\|C(\wh x)v_1-C(\wh x)v_2\|\leq d(x_0,\mathfs S)^{-a}\|v_1-v_2\|,
\end{align*}
thus completing the proof.
\end{proof}

Now we introduce a parameter for which the restriction of $\Psi_{\wh x}$ to a smaller domain
with this size has better properties.
Given $\ve>0$, let $I_\ve:=\{e^{-\frac{1}{3}\ve n}:n\geq 0\}$. 
Recall that  $\rho(\wh x)=d(\{\vartheta_{-1}[\wh x],\vartheta_0[\wh x],\vartheta_1[\wh x]\},\mathfs S)$.

\medskip
\noindent
{\sc Parameter $Q(\wh x)$:} For $\wh x\in{\rm NUH}$,
define $Q(\wh x)=\max\{Q\in I_\ve:Q\leq \widetilde Q(\wh x)\}$,
where
$$
\widetilde Q(\wh x)= \ve^{6/\beta}\min\left\{\norm{C(\wh x)^{-1}}^{-48/\beta},\rho(\wh x)^{96a/\beta}\right\}.
$$

\medskip
This parameter depends on $\ve>0$, but for simplicity we will suppress this dependence.
It is important noting that most of the results will hold for $\ve>0$ small enough. 
The term $\ve^{6/\beta}$ will allow to absorb multiplicative constants, while 
the minimum takes into account the two situations where 
we will have to take a smaller domain for the Pesin chart:
\begin{enumerate}[$\circ$]
\item The rate of nonuniform hyperbolicity is measured by $\norm{C(\wh x)^{-1}}$:
the longer it takes to see hyperbolicity, the bigger $\norm{C(\wh x)^{-1}}$ is.
\item The term $\rho(\wh x)$ controls how small the Pesin charts needs to be
for its image not to intersect $\mathfs S$.
\end{enumerate}
We observe that this definition is stronger than those in \cite{Sarig-JAMS,Lima-Sarig,Lima-Matheus}.
In particular, calculations in these works that only depend
on the definition of $Q(\wh x)$ can be used here. From its definition
and Lemma \ref{Lemma-linear-reduction}(4), we have the following bounds for $Q(\wh x)$
for $\ve>0$ small enough:
\begin{align}\label{estimates-Q}
\begin{array}{l}
Q(\wh x)\leq \ve^{6/\beta}, \ \|C(\wh x)^{-1}\|Q(\wh x)^{\beta/48}\leq \ve^{1/8},
\  \|C(\wh f(\wh x))^{-1}\|Q(\wh x)^{\beta/12}\leq \ve^{1/4},\\
\\
\rho(\wh x)^{-a}Q(\wh x)^{\beta/96}<\ve^{1/16}.
\end{array}
\end{align}

We can also define, for each $\wh x\in{\rm NUH}_{\chi'}$, a
Pesin chart $\Psi_{\chi',\wh x}=\exp{\vt[\wh x]}\circ C_{\chi'}(\wh x)$,
for which Lemma \ref{Lemma-Pesin-chart} holds with $\chi$ replaced by $\chi'$.
We could also define a parameter $Q_{\chi'}(\wh x)$, where $\|C(\wh x)^{-1}\|$ is replaced by
$\|C_{\chi'}(\wh x)^{-1}\|$. In this case, for $\wh x\in{\rm NUH}$
we have $Q(\wh x)\leq Q_{\chi'}(\wh x)$. Since all calculations will be made with the
smaller value $Q(\wh x)$, the parameter $Q_{\chi'}$ has no essential role in this paper.

\subsection{The map $f$ and its inverse branches in Pesin charts}

We start introducing a notation for the inverse branches of $f$.

\medskip
\noindent
{\sc Inverse branches of $f$:} If $x,f(x)\not\in\mathfs S$, let $f_x^{-1}$ be the inverse branch of $f$
that sends $f(x)$ to $x$ and is defined in $E_x=B(f(x),2\mathfrak r(x))$.
Given $\wh x=(x_n)_{n\in\Z}\in\wh M$, write $f_{\wh x}^{-1}:=f_{x_0}^{-1}$.

\medskip
By assumption (A5), $f_x$ is a well-defined diffeomorphism from $E_x$ onto its image.
For $x\in M$, define
$\mathfrak m(x):=\tfrac{1}{2}\min\{d(x,\mathfs S)^{2a},d(f(x),\mathfs S)^{2a}\}$,
and $\mathfrak m(\wh x):=\mathfrak m(\vt[\wh x])$.

\medskip
\noindent
{\sc The maps $F_{\wh x}$ and $F_{\wh x}^{-1}$:} For $\wh x\in{\rm NUH}$,
define $F_{\wh x}:B[\mathfrak m(\wh x)]\to \R^m$ by the composition
$F_{\wh x}= \Psi_{\wh f(\wh x)}^{-1}\circ f\circ\Psi_{\wh x}$; similarly, define
$F_{\wh x}^{-1}:B[\mathfrak m(\wh x)]\to \R^m$ by
$F_{\wh x}^{-1}= \Psi_{\wh x}^{-1}\circ f_{\wh x}^{-1}\circ\Psi_{\wh f(\wh x)}$.

\begin{theorem}\label{Thm-non-linear-Pesin}
The following holds for all $\ve>0$ small enough. If ${\wh x}\in{\rm NUH}$
then $F_{\wh x},F_{\wh x}^{-1}$ are well-defined diffeomorphisms onto their images.
Furthermore, the restrictions of $F_{\wh x}$ and $F_{\wh x}^{-1}$ to
$B[20Q(\wh x)]$ and $B[20Q(\wh f(\wh x))]$ respectively satisfy:
\begin{enumerate}[{\rm (1)}]
\item $F_{\wh x}^{-1}\circ F_{\wh x}$ and $F_{\wh x}\circ F_{\wh x}^{-1}$ are the identity map.
\item $d(F_{\wh x}^{\pm 1})_0=D(\wh x)^{\pm 1}$, cf. Lemma {\rm \ref{Lemma-linear-reduction}}.
\item $F_{\wh x}=D(\wh x)+H^+$ and
$F_{\wh x}^{-1}=D(\wh x)^{-1}+H^-$ where: 
\begin{enumerate}[{\rm (a)}]
\item $H^{\pm}(0)=0$ and $d(H^{\pm})_0=0$.
\item $\norm{H^{\pm}}_{C^{1+\frac{\beta}{2}}}<\ve$.
\end{enumerate}
\item $\norm{dF^{\pm 1}_{\wh x}}_{C^0}<2d(\vt[\wh x],\mathfs S)^{-a}(1+e^{2\chi})^{1/2}$. 
\end{enumerate}
\end{theorem}

\begin{proof}
The proof is an adaptation of the proof of \cite[Theorem~3.3]{Lima-Matheus} to our context.
We start showing that $F_{\wh x}^{\pm 1}$ are well-defined. Write $\wh x=(x_n)_{n\in\Z}$.
Given $v\in B[\mathfrak m(\wh x)]$, Lemma \ref{Lemma-Pesin-chart}(1) implies
that $\Psi_{\wh x}(v)\in B(x_0,2\mathfrak m(\wh x))\subset D_{x_0}$ and by (A6) we get that
$(f\circ\Psi_{\wh x})(v)\in B(f(x_0),2\mathfrak m(\wh x)d(x_0,\mathfs S)^{-a})\subset \mathfrak B_{f(x_0)}$,
where in the latter inclusion we used that
$2\mathfrak m(\wh x)d(x_0,\mathfs S)^{-a}\leq d(f(x_0),\mathfs S)^a<\mathfrak d(f(x_0))$.
By (A1), the composition $(\Psi_{\wh f(x)}^{-1}\circ f\circ\Psi_{\wh x})(v)$ is well-defined.
We proceed similarly for $F_{\wh x}^{-1}$: given
$v\in B[\mathfrak m(\wh x)]$, we have
$\Psi_{\wh f(\wh x)}(v)\in B(f(x_0),2\mathfrak m(\wh x))\subset E_{x_0}$
and so (A6) gives that 
$(f_{\wh x}^{-1}\circ\Psi_{\wh f(\wh x)})(v)\in B(x_0,2\mathfrak m(\wh x)d(x_0,\mathfs S)^{-a})\subset \mathfrak B_{x_0}$,
since $2\mathfrak m(\wh x)d(x_0,\mathfs S)^{-a}\leq d(x_0,\mathfs S)^a<\mathfrak d(x_0)$.
By (A1), the composition $(\Psi_{\wh x}^{-1}\circ f_{\wh x}^{-1}\circ\Psi_{\wh f(\wh x)})(v)$ is well-defined.
Now we prove parts (1)--(4).

\medskip
\noindent
(1) It is clear that, whenever defined, the compositions
$F_{\wh x}^{-1}\circ F_{\wh x}$ and $F_{\wh x}\circ F_{\wh x}^{-1}$ are the identity. 
We start showing that $F_{\wh x}^{-1}\circ F_{\wh x}$ is
well-defined in $B[20Q(\wh x)]$. By the beginning of the proof, it is enough to show that
$F_{\wh x}(B[20Q(\wh x)])\subset B[\mathfrak m(\wh x)]$.
Indeed, assumption (A6) and Lemma \ref{Lemma-Pesin-chart}(1) imply that
$$
F_{\wh x}(B[20Q(\wh x)])\subset B[80d(x_0,\mathfs S)^{-a}\|C(\wh f(\wh x))^{-1}\|Q(\wh x)]
$$
and this latter radius is
$80d(x_0,\mathfs S)^{-a}\|C(\wh f(\wh x))^{-1}\|Q(\wh x)<\tfrac{1}{2}\rho(\wh x)^{2a}\leq \mathfrak m(\wh x)$
for $\ve>0$ sufficiently small, since by (\ref{estimates-Q}) it holds
$\rho(\wh x)^{-3a}Q(\wh x)^{1/2}<\rho(\wh x)^{-3a}Q(\wh x)^{\beta/32}<\ve^{3/16}$ and 
$\|C(\wh f(\wh x))^{-1}\| Q(\wh x)^{1/2}<\|C(\wh f(\wh x))^{-1}\|Q(\wh x)^{\beta/12}<\ve^{1/4}$. 

Similarly, $F_{\wh x}^{-1}(B[20Q(\wh f(\wh x))])\subset 
B[80d(x_0,\mathfs S)^{-a}\|C(\wh x)^{-1}\|Q(\wh f(\wh x))]$.
By Lemma \ref{Lemma-linear-reduction}(4), we have
$$
80d(x_0,\mathfs S)^{-a}\|C(\wh x)^{-1}\|Q(\wh f(\wh x))\leq
80(1+e^{2\chi})^{1/2}d(x_0,\mathfs S)^{-3a}\|C(\wh f(\wh x))^{-1}\|Q(\wh f(\wh x))
$$
which is smaller than $\tfrac{1}{2}\rho(\wh f(\wh x))^{2a}$ since by (\ref{estimates-Q}) we have
$\rho(\wh f(\wh x))^{-5a}Q(\wh f(\wh x))^{1/2}<\ve^{5/16}$ and
$\|C(\wh f(\wh x))^{-1}\|Q(\wh f(\wh x))^{1/2}<\ve^{1/8}$.

\medskip
\noindent
(2) This follows from the identities $d(\Psi_{\wh x})_0=C(\wh x)$, $df_{x_0}=\wh{df}_{\wh x}$,
$d(f_{\wh x}^{-1})_{x_1}=(df_{x_0})^{-1}=(\wh{df}_{\wh x})^{-1}$
and from Lemma \ref{Lemma-linear-reduction}(2).

\medskip
\noindent
(3) Define $H^+:B[20Q(\wh x)]\to \R^m$ and $H^-:B[20Q(\wh f(\wh x))]\to \R^m$
by the equalities $F_{\wh x}=D(\wh x)+H^+$ and $F_{\wh x}^{-1}=D(\wh x)^{-1}+H^-$.
Property (a) follows from part (2) above, so it remains to prove (b). We show the following fact.\\

\noindent
\medskip
{\sc Claim:} $\|d(F_{\wh x})_{w_1}-d(F_{\wh x})_{w_2}\|\leq \frac{\ve}{3}\|w_1-w_2\|^{\beta/2}$
for every $w_1,w_2\in B[20Q(\wh x)]$, and 
 $\|d(F_{\wh x}^{-1})_{w_1}-d(F_{\wh x}^{-1})_{w_2}\|\leq \frac{\ve}{3}\|w_1-w_2\|^{\beta/2}$
for every $w_1,w_2\in B[20Q(\wh f(\wh x))]$.

\medskip
Before proving the claim, let us show how to conclude (b).
Applying the claim with $w_2=0$,
we get $\|d(H^\pm)_w\|\leq \frac{\ve}{3}\|w\|^{\beta/2}<\tfrac{\ve}{3}$ and so 
$\norm{dH^\pm}_{C^0}<\tfrac{\ve}{3}$.
By the mean value inequality,
$\|H^\pm(w)\|\leq \tfrac{\ve}{3}\|w\|<\tfrac{\ve}{3}$, thus $\norm{H^\pm}_{C^0}<\tfrac{\ve}{3}$.
The claim also implies that $\Hol{\beta/2}(dH^\pm)\leq\tfrac{\ve}{3}$.
These estimates imply that $\|H^\pm\|_{C^{1+\frac{\beta}{2}}}<\ve$.

\begin{proof}[Proof of the claim.]
We begin proving it for $F_{\wh x}$. Remind that
$\wh x=(x_n)_{n\in\Z}$. For $i=1,2$, write $u_i=C(\wh x)w_i$ and define
$$
A_i= \widetilde{d(\exp{f(x_0)}^{-1})_{(f\circ \exp{x_0})(u_i)}}\,,\
B_i=\widetilde{df_{\exp{x_0}(u_i)}}\,,\ C_i=\widetilde{d(\exp{x_0})_{u_i}}.
$$
Note that $\|u_1-u_2\|\leq \|w_1-w_2\|$. We first estimate $\|A_1 B_1 C_1-A_2 B_2 C_2\|$.
\begin{enumerate}[$\circ$]
\item By (A2), $\|A_i\|\leq 2$. By (A2), (A3), (A6):
\begin{align*}
&\, \|A_1-A_2\|\leq d(f(x_0),\mathfs S)^{-a}d((f\circ \exp{x_0})(u_1),(f\circ \exp{x_0})(u_2))\\
&\leq 2d(x_0,\mathfs S)^{-a}d(f(x_0),\mathfs S)^{-a}\|u_1-u_2\|\leq 2\rho(\wh x)^{-2a}\|w_1-w_2\|.
\end{align*}
\item By (A6), $\|B_i\|\leq \rho(\wh x)^{-a}$. By (A2) and (A7):
$$
\|B_1-B_2\|\leq \mathfrak K d(\exp{x_0}(u_1),\exp{x_0}(u_2))^{\beta}\leq 2\mathfrak K\|u_1-u_2\|^\beta
\leq 2\mathfrak K\|w_1-w_2\|^\beta.
$$
\item By (A2), $\|C_i\|\leq 2$. By (A3):
$$
\|C_1-C_2\|\leq d(x_0,\mathfs S)^{-a}\|u_1-u_2\|\leq \rho(\wh x)^{-a}\|w_1-w_2\|.
$$
\end{enumerate}
Applying some triangle inequalities, we get that
$$
\|A_1 B_1 C_1-A_2 B_2 C_2\|\leq 24\mathfrak K\rho(\wh x)^{-3a}\|w_1-w_2\|^\beta.
$$
Now we estimate $\|d(F_{\wh x})_{w_1}-d(F_{\wh x})_{w_2}\|$:
\begin{align*}
&\, \|d(F_{\wh x})_{w_1}-d(F_{\wh x})_{w_2}\|\leq
\|C(\wh f(\wh x))^{-1}\| \cdot \|A_1 B_1 C_1-A_2 B_2 C_2\|\cdot \|C(\wh x)\|\\
&\leq 24\mathfrak K\rho(\wh x)^{-3a}\|C(\wh f(\wh x))^{-1}\|\cdot\|w_1-w_2\|^\beta.
\end{align*}
Since $\|w_1-w_2\|<40Q(\wh x)$, using Lemma \ref{Lemma-linear-reduction}(4) and the definition of $Q(\wh x)$
we obtain that for $\ve>0$ small enough it holds
\begin{align*}
&\ 24\mathfrak K\rho(\wh x)^{-3a}\|C(\wh f(\wh x))^{-1}\|\|w_1-w_2\|^{\beta/2}\\
&\leq 200\mathfrak K(1+e^{2\chi})^{1/2}\rho(\wh x)^{-5a}\|C(\wh x)^{-1}\| Q(\wh x)^{\beta/2}\\
&=200\mathfrak K(1+e^{2\chi})^{1/2}\left[\rho(\wh x)^{-5a}Q(\wh x)^{5\beta/96}\right]\cdot
\left[\|C(\wh x)^{-1}\| Q(\wh x)^{\beta/48}\right]Q(\wh x)^{41\beta/96}\\
&\leq 200\mathfrak K(1+e^{2\chi})^{1/2}\ve^{5/16}\ve^{1/8}\ve^{41/16}
=200\mathfrak K(1+e^{2\chi})^{1/2}\ve^3<\tfrac{\ve}{3},
\end{align*}
which completes the proof of the claim for $F_{\wh x}$. 
Now we prove it for $F_{\wh x}^{-1}$. 
For $i=1,2$, write $u_i=C(\wh f(\wh x))w_i$ and define
$$
A_i= \widetilde{d(\exp{x_0}^{-1})_{(f_{\wh x}^{-1}\circ \exp{f(x_0)})(u_i)}}\,,\
B_i=\widetilde{d(f_{\wh x}^{-1})_{\exp{f(x_0)}(u_i)}}\,,\ C_i=\widetilde{d(\exp{f(x_0)})_{u_i}}.
$$
Again, we have $\|u_1-u_2\|\leq \|w_1-w_2\|$ and the following similar estimates:
\begin{enumerate}[$\circ$]
\item By (A2), $\|A_i\|\leq 2$. By (A2), (A3), (A6):
\begin{align*}
&\, \|A_1-A_2\|\leq d(x_0,\mathfs S)^{-a}d((f_{\wh x}^{-1}\circ \exp{f(x_0)})(u_1),(f_{\wh x}^{-1}\circ \exp{f(x_0)})(u_2))\\
&\leq 2d(x_0,\mathfs S)^{-2a}\|u_1-u_2\|\leq 2\rho(\wh f(\wh x))^{-2a}\|w_1-w_2\|.
\end{align*}
\item By (A6), $\|B_i\|\leq \rho(\wh f(\wh x))^{-a}$. By (A2) and (A7):
$$
\|B_1-B_2\|\leq \mathfrak K d(\exp{f(x_0)}(u_1),\exp{f(x_0)}(u_2))^{\beta}\leq 2\mathfrak K\|u_1-u_2\|^\beta
\leq 2\mathfrak K\|w_1-w_2\|^\beta.
$$
\item By (A2), $\|C_i\|\leq 2$. By (A3):
$$
\|C_1-C_2\|\leq d(f(x_0),\mathfs S)^{-a}\|u_1-u_2\|\leq \rho(\wh f(\wh x))^{-a}\|w_1-w_2\|.
$$
\end{enumerate}
As before, the above estimates imply that
$$
\|A_1 B_1 C_1-A_2 B_2 C_2\|\leq 24\mathfrak K\rho(\wh f(\wh x))^{-3a}\|w_1-w_2\|^\beta
$$
and then a calculation identical to the one made for $F_{\wh x}$ works:
\begin{align*}
&\ 24\mathfrak K\rho(\wh f(\wh x))^{-3a}\|C(\wh x)^{-1}\|\|w_1-w_2\|^{\beta/2}\\
&\leq 200\mathfrak K(1+e^{2\chi})^{1/2}\rho(\wh f(\wh x))^{-5a}\|C(\wh f(\wh x))^{-1}\| Q(\wh f(\wh x))^{\beta/2}\\
&\leq 200\mathfrak K(1+e^{2\chi})^{1/2}\ve^3<\tfrac{\ve}{3}\cdot
\end{align*}
Hence $\|d(F_{\wh x}^{-1})_{w_1}-d(F_{\wh x}^{-1})_{w_2}\|<\tfrac{\ve}{3}\|w_1-w_2\|^{\beta/2}$,
thus proving the claim.
\end{proof}

\noindent
(4) By Lemma \ref{Lemma-linear-reduction}(3) and part (2) above,
$$
\|d(F_{\wh x})_0\|=\|D_u(\wh x)\|\leq  d(\vt[\wh x],\mathfs S)^{-a}(1+e^{2\chi})^{1/2}.
$$
By part (3) above, if $w\in B[20Q(\wh x)]$ then
$$
\|d(F_{\wh x})_w\|\leq \ve\|w\|^{\beta/2}+d(\vt[\wh x],\mathfs S)^{-a}(1+e^{2\chi})^{1/2}<
2d(\vt[\wh x],\mathfs S)^{-a}(1+e^{2\chi})^{1/2},
$$
since $\ve\|w\|^{\beta/2}<1<d(\vt[\wh x],\mathfs S)^{-a}(1+e^{2\chi})^{1/2}$ for small $\ve>0$.
Similarly, using that $\|d(F_{\wh x}^{-1})_0\|=\|D_s(\wh x)^{-1}\|\leq d(\vt[\wh x],\mathfs S)^{-a}(1+e^{2\chi})^{1/2}$
we conclude that $\|d(F_{\wh x}^{-1})_w\|<2d(\vt[\wh x],\mathfs S)^{-a}(1+e^{2\chi})^{1/2}$
for all $w\in B[20Q(\wh f(\wh x))]$.
\end{proof}

An analogous result holds for the Lyapunov inner product $\vertiii{\cdot}_{\chi'}$:
defining $F_{\chi',\wh x}$ and $F_{\chi',\wh x}^{-1}$ as compositions 
with the Pesin charts $\Psi_{\chi',\cdot}$, then Theorem \ref{Thm-non-linear-Pesin} holds with
$D(\wh x)$ and $\chi$ replaced by $D_{\chi'}(\wh x)$ and $\chi'$, but still in the same
domain $B[20Q(\wh x)]$ (indeed, it holds in the larger domain $B[20Q_{\chi'}(\wh x)]$).

\subsection{Temperedness and the parameters $q(\wh x),q^s(\wh x),q^u(\wh x)$}\label{Section-Temperedness}

Now we will focus on the points for which
$Q$ does not converge exponentially fast to zero too fast along the trajectory of $\wh x$.
Intuitively, this condition will guarantee that hyperbolicity will prevail over the
degeneracies of nonuniform hyperbolicity, of the derivative, and of the proximity to $\mathfs S$.
Recall that $\chi>0$ is fixed a priori and $\ve>0$ is sufficiently small. Define
$\delta_\ve:=e^{-\ve n}\in I_{\ve}$ where $n$ is the unique positive integer s.t. $e^{-\ve n}<\ve\leq e^{-\ve(n-1)}$.
In particular, we have $\delta_\ve<\ve$.

\medskip
\noindent
{\sc Parameter $q(\wh x)$:} For $\wh x\in{\rm NUH}$, let
$q(\wh x):=\delta_\ve\min\{e^{\ve|n|}Q(\wh f^n(\wh x)):n\in\Z\}$.

\medskip
\noindent
{\sc The set ${\rm NUH}^*$:} It is the set of $\wh x\in{\rm NUH}$ such that $q(\wh x)>0$.

\medskip
Observe that if $\wh x\in{\rm NUH}$ satisfies
$$
\lim_{n\to\pm\infty}\tfrac{1}{|n|}\log \|C(\wh f^n(\wh x))^{-1}\|=
\lim_{n\to\pm\infty}\tfrac{1}{|n|}\log d(\vt_n[\wh x],\mathfs S)=0,
$$
then $\lim_{n\to\pm\infty}\tfrac{1}{|n|}\log Q(\wh f^n(\wh x))=0$ and so $\wh x\in{\rm NUH}^*$.
%This shows that ${\rm NUH}^*$ carries all points where the degeneracies of nonuniform hyperbolicity
%and of the derivative are subexponential. 
The next lemma states the basic properties of $q$.

\begin{lemma}\label{Lemma-q}
For all $\wh x\in{\rm NUH}^*$, $0<q(\wh x)<\ve Q(\wh x)$
and $\tfrac{q(\wh f(\wh x))}{q(\wh x)}=e^{\pm\ve}$.
\end{lemma}

This is \cite[Lemma 4.1]{Lima-Matheus}, and the proof goes through without change.
Now we introduce two more parameters, that will be regarded as sizes
of the stable and unstable manifolds at $\wh x$.

\medskip
\noindent
{\sc Parameters $q^s(\wh x),q^u(\wh x)$:}  For $\wh x\in{\rm NUH}$, define
\begin{align*}
q^s(\wh x)&:=\delta_\ve\min\{e^{\ve|n|}Q(\wh f^n(\wh x)):n\geq 0\}\\
q^u(\wh x)&:=\delta_\ve\min\{e^{\ve|n|}Q(\wh f^n(\wh x)):n\leq 0\}.
\end{align*}

\medskip
Note that $q^s(\wh x),q^u(\wh x)$ are just the separation, in the definition of $q(\wh x)$, of
the past from the future. We list the main properties of $q^s(\wh x)$ and $q^u(\wh x)$.

\begin{lemma}\label{Lemma-q^s}
For all $\wh x\in{\rm NUH}^*$, the following holds.
\begin{enumerate}[{\rm (1)}]
\item {\sc Good definition:} $0<q^s(\wh x),q^u(\wh x)<\ve Q(\wh x)$ and
$q^s(\wh x)\wedge q^u(\wh x)=q(\wh x)$.
\item {\sc Greedy algorithm:} For all $n\in\Z$ it holds
\begin{align*}
q^s(\wh f^n(\wh x))&=\min\{e^\ve q^s(\wh f^{n+1}(\wh x)),\delta_\ve Q(\wh f^n(\wh x))\}\\
q^u(\wh f^n(\wh x))&=\min\{e^\ve q^u(\wh f^{n-1}(\wh x)),\delta_\ve Q(\wh f^n(\wh x))\}.
\end{align*}
\end{enumerate}
\end{lemma}

This is \cite[Lemma 4.2]{Lima-Matheus}, and the same proof works.
Now we introduce the set that we will effectively code.

\medskip
\noindent
{\sc The set ${\rm NUH}^\#$:} It is the set of $\wh x\in{\rm NUH}^*$ s.t.
$$
\limsup_{n\to+\infty}q(\wh f^n(\wh x))>0\text{ and }\limsup_{n\to-\infty}q(\wh f^n(\wh x))>0.
$$

For a fixed $t>0$, the set $\{\wh x\in {\rm NUH}^*:q(\wh x)>t\}$
can be regarded, in the usual literature, as a Pesin block. In this sense,
$\wh x\in{\rm NUH}^\#$ iff its trajectory returns to some Pesin block
infinitely often in the past and in the future. Therefore, ${\rm NUH}^\#$ is the dynamical
counterpart of the recurrent set $\Sigma^\#$ introduced in Section \ref{Section-preliminaries}.
The next lemma shows that
${\rm NUH}^\#$ is large for all $f$--adapted $\chi$--hyperbolic measures.

\begin{lemma}\label{Lemma-adaptedness}
If $\mu$ is an $f$--adapted measure and $\wh\mu$ is carried by ${\rm NUH}$,
then $\wh\mu$ is carried by ${\rm NUH}^\#$.
In particular, ${\rm NUH}^\#$ carries all $f$--adapted $\chi$--hyperbolic measures.
\end{lemma}

\begin{proof}
Fix an $f$--adapted measure $\mu$ s.t. $\wh\mu[{\rm NUH}]=1$.
By the Poincar\'e recurrence theorem, every measure supported in ${\rm NUH}^*$ is supported in
${\rm NUH}^\#$, so it is enough to show that
$$
\lim_{n\to\pm\infty}\tfrac{1}{|n|}\log \|C(\wh f^n(\wh x))^{-1}\|=
\lim_{n\to\pm\infty}\tfrac{1}{|n|}\log \rho(\wh f^n(\wh x))=0
$$
for $\wh\mu$--a.e. $\wh x\in{\rm NUH}$. Let us prove these.
Since $\mu$ is $f$--adapted, $\log\rho\in L^1(\wh\mu)$ and so by the Birkhoff ergodic theorem
$\lim\limits_{n\to\pm\infty}\tfrac{1}{|n|}\log \rho(\wh f^n(\wh x))=0$
for $\wh\mu$--a.e. $\wh x\in{\rm NUH}$. For the other equality, define $\varphi:{\rm NUH}\to\R$ by
$$
\vf(\wh x):=\log\left[\tfrac{\|C(\wh f(\wh x))^{-1}\|}{\|C(\wh x)^{-1}\|}\right]=\log \|C(\wh f(\wh x))^{-1}\|-\log \|C(\wh x)^{-1}\|.
$$
Using Lemma \ref{Lemma-linear-reduction}(4) and again that $\log\rho\in L^1(\wh\mu)$,
we get that $\log\left[\tfrac{\|C(\wh f(\wh x))^{-1}\|}{\|C(\wh x)^{-1}\|}\right]$ is in $L^1(\wh\mu)$.
Letting $\varphi_n(\wh x)=\log\|C(\wh f^n(\wh x))^{-1}\|-\log\|C(\wh x)^{-1}\|$ denote the $n$--th Birkhoff
sum of $\varphi$, by the Birkhoff ergodic theorem
$\lim\limits_{n\to+\infty}\tfrac{\varphi_n(\wh x)}{n}$ exists $\wh\mu$--a.e.
The Poincar\'e recurrence theorem implies that
$$
\liminf\limits_{n\to+\infty}|\varphi_n(\wh x)|=
\liminf\limits_{n\to+\infty}\left|\log\|C(\wh f^n(\wh x))^{-1}\|-\log\|C(\wh x)^{-1}\|\right|<\infty
$$
for $\wh\mu$--a.e  $\wh x\in{\rm NUH}$, hence $\lim\limits_{n\to+\infty}\tfrac{\varphi_n(\wh x)}{n}=0$
for $\wh \mu$--a.e. $\wh x\in{\rm NUH}$.
Proceeding similarly for $n\to-\infty$, we get that
$\lim_{n\to\pm\infty}\tfrac{1}{|n|}\log \|C(\wh f^n(\wh x))^{-1}\|=0$ for $\wh\mu$--a.e. $\wh x\in{\rm NUH}$.
\end{proof}

\section{Overlap, double charts, and graph transforms}

We will use the Pesin charts to apply the graph transform method.
For that, we allow the size of the domain of definition to vary. Recall that $B[r]$
is the ball of center 0 and radius $r$ in $\R^m$.

\medskip
\noindent
{\sc Pesin chart $\Psi_{\wh x}^\eta$:} It is the restriction of $\Psi_{\wh x}$ to $B[\eta]$,
where $0<\eta\leq Q(\wh x)$.

\subsection{The overlap condition}\label{section-overlap}

We recall a notion, introduced in \cite{Sarig-JAMS}, that allows to change coordinates from
$\Psi_{\wh x_1}$ to $\Psi_{\wh x_2}$.

\medskip
\noindent
{\sc $\ve$--overlap:} Two Pesin charts $\Psi_{\wh x_1}^{\eta_1},\Psi_{\wh x_2}^{\eta_2}$ are said to
{\em $\ve$--overlap} if $d_{s/u}(\wh x_1)=d_{s/u}(\wh x_2)$, $\tfrac{\eta_1}{\eta_2}=e^{\pm\ve}$
and if $\exists x\in M$ s.t. $\vt[\wh x_1],\vt[\wh x_2]\in \mathfrak B_x$ and
$$
d(\vt[\wh x_1],\vt[\wh x_2])+\norm{\widetilde{C(\wh x_1)}-\widetilde{C(\wh x_2)}}<(\eta_1\eta_2)^4.
$$
When this happens, we write $\Psi_{\wh x_1}^{\eta_1}\overset{\ve}{\approx}\Psi_{\wh x_2}^{\eta_2}$.

\medskip
%Observe that we require that the dimensions of $E^{s/u}_{\wh x_1}$ and $E^{s/u}_{\wh x_2}$ are the same.
We claim that if $\ve>0$ is small enough, then
$\Psi_{\wh x_1}^{\eta_1}\overset{\ve}{\approx}\Psi_{\wh x_2}^{\eta_2}$ implies that
\begin{align}
\Psi_{\wh x_i}(B[20Q(\wh x_i)])\subset \mathfrak B_{\vt[\wh x_1]}\cap \mathfrak B_{\vt[\wh x_2]}
\text{ for }i=1,2
\end{align}
and hence we can apply (A1)--(A4) without mentioning $x$.
We prove this for $i=1$.
Firstly, since $d(\vt[\wh x_1],\vt[\wh x_2])<\ve^2 d(\vt[\wh x_2],\mathfs S)$, we have
$d(\vt[\wh x_1],\mathfs S)=d(\vt[\wh x_2],\mathfs S)\pm d(\vt[\wh x_1],\vt[\wh x_2])=
(1\pm\ve^2)d(\vt[\wh x_2],\mathfs S)=e^{\pm\ve}d(\vt[\wh x_2],\mathfs S)$.
By Lemma \ref{Lemma-Pesin-chart}(1),
$\Psi_{\wh x_1}(B[20Q(\wh x_1)])\subset B(\vt[\wh x_1],40Q(\wh x_1))$ which is 
contained in $\mathfrak B_{\vt[\wh x_1]}$ since
$40Q(\wh x_1)\ll40\ve^{6/\beta}\rho(\wh x_1)^a<2\mathfrak d(\wh x_1)$.
For the other inclusion, note that
$$
\Psi_{\wh x_1}(B[20Q(\wh x_1)])\subset B(\vt[\wh x_1],40 Q(\wh x_1))\subset
B(\vt[\wh x_2],40Q(\wh x_1)+d(\vt[\wh x_1],\vt[\wh x_2])).
$$
Since
$40Q(\wh x_1)+d(\vt[\wh x_1],\vt[\wh x_2])\leq 40\ve^{6/\beta}e^{a\ve}d(\vt[\wh x_2],\mathfs S)^a+
d(\vt[\wh x_2],\mathfs S)^a<2\mathfrak d(\wh x_2)$ for $\ve>0$ small enough, it follows that
$\Psi_{\wh x_1}(B[20Q(\wh x_1)])\subset \mathfrak B_{\vt[\wh x_2]}$.

The next result quantifies the closeness of Pesin charts when an $\ve$--overlap occurs.

\begin{proposition}\label{Proposition-overlap}
The following holds for $\ve>0$ small enough.
If $\Psi_{\wh x_1}^{\eta_1}\overset{\ve}{\approx}\Psi_{\wh x_2}^{\eta_2}$ and
$C_i=\widetilde{C(\wh x_i)}$ for $i=1,2$ then:
\begin{enumerate}[{\rm (1)}]
\item {\sc Control of $\rho$:}  $\tfrac{d(\vt_i[\wh x_2],\mathfs S)}{d(\vt_i[\wh x_1],\mathfs S)}=e^{\pm\ve}$
for $i=-1,0,1$. In particular, $\tfrac{\rho(\wh x_1)}{\rho(\wh x_2)}=e^{\pm\ve}$.
\item {\sc Control of $C^{-1}$:} $\norm{C_1^{-1}-C_2^{-1}}<(\eta_1\eta_2)^3$ and
$\tfrac{\norm{C_1^{-1}}}{\norm{C_2^{-1}}}=e^{\pm (\eta_1\eta_2)^3}$.
\item {\sc Control of $Q$:} $\tfrac{Q(\wh x_1)}{Q(\wh x_2)}={\rm exp}[\pm \tfrac{97a\ve}{\beta}]$.
\item {\sc Overlap:} $\Psi_{\wh x_i}(B[e^{-2\ve}\eta_i])\subset \Psi_{\wh x_j}(B[\eta_j])$ for $i,j=1,2$.
\item {\sc Change of coordinates:} For $i,j=1,2$, the map $\Psi_{\wh x_i}^{-1}\circ\Psi_{\wh x_j}$
is well-defined in $B[d(\vt[\wh x_j],\mathfs S)^a]$,
and $\norm{\Psi_{\wh x_i}^{-1}\circ\Psi_{\wh x_j}-{\rm Id}}_{C^{1+\frac{\beta}{2}}}<\ve(\eta_1\eta_2)^2$
where the norm is taken in $B[d(\vt[\wh x_j], \mathfs{S})^{2a}]$.
\end{enumerate}
\end{proposition}

\begin{proof} Recall from (\ref{estimates-Q}) that
$\norm{C_i^{-1}}\leq \ve^{1/8}Q(\wh x_i)^{-\beta/48}\leq \ve^{1/8}\eta_i^{-\beta/48}$.

\medskip
\noindent
(1)  The case $i=0$ was treated above. Consider $i=1$.
Since $d(\vt[\wh x_1],\vt[\wh x_2])<(\eta_1\eta_2)^4<\ve^2\rho(\wh x_2)^{2a}$,
we have that $\vt[\wh x_1]\in D_{\wh x_2}$. By (A6),
$d(\vt_1[\wh x_1],\vt_1[\wh x_2])\leq d(\vt[\wh x_2],\mathfs S)^{-a}d(\vt[\wh x_1],\vt[\wh x_2])<
\ve^2\rho(\wh x_2)^a<\ve^2 d(\vt_1[\wh x_2],\mathfs S)$ and so 
$d(\vt_1[\wh x_1],\mathfs S)=d(\vt_1[\wh x_2],\mathfs S)\pm
d(\vt_1[\wh x_1],\vt_1[\wh x_2])=(1\pm\ve^2)d(\vt_1[\wh x_2],\mathfs S)=e^{\pm\ve}d(\vt_1[\wh x_2],\mathfs S)$. 
The same proof applies to $i=-1$.

\medskip
\noindent
(2) We have $C_1^{-1}-C_2^{-1}=C_1^{-1}(C_2-C_1)C_2^{-1}$, hence
\begin{align*}
\norm{C_1^{-1}-C_2^{-1}}\leq \norm{C_1^{-1}}\cdot\norm{C_2^{-1}}\cdot\norm{C_1-C_2}\leq 
\ve^{1/4}(\eta_1\eta_2)^{4-\beta/48}<\tfrac{1}{2}(\eta_1\eta_2)^3.
\end{align*}
Additionally,
\begin{align*}
\left|\tfrac{\norm{C_1^{-1}}}{\norm{C_2^{-1}}}-1\right|\leq \norm{C_1^{-1}-C_2^{-1}} \leq \tfrac{1}{2}(\eta_1\eta_2)^3
\end{align*}
and so $\tfrac{\norm{C_1^{-1}}}{\norm{C_2^{-1}}}=e^{\pm (\eta_1\eta_2)^3}$.

\medskip
\noindent
(3) By parts (1) and (2),
$\tfrac{\rho(\wh x_1)^{96a/\beta}}{\rho(\wh x_2)^{96a/\beta}}={\rm exp}\left[\pm \tfrac{96a\ve}{\beta}\right]$,
and $\tfrac{\|C(\wh x_1)^{-1}\|^{-48/\beta}}{\|C(\wh x_2)^{-1}\|^{-48/\beta}}={\rm exp}\left[\pm\tfrac{48\ve}{\beta}\right]$.
Since $\tfrac{48\ve}{\beta}<\tfrac{96a\ve}{\beta}$, we get
$\tfrac{\wt Q(\wh x_1)}{\wt Q(\wh x_2)}={\rm exp}\left[\pm \tfrac{96a\ve}{\beta}\right]$.
By the definition of $Q$, we obtain
$$
\tfrac{Q(\wh x_1)}{Q(\wh x_2)}={\rm exp}\left[\pm\left(\tfrac{96a\ve}{\beta}+\tfrac{\ve}{3}\right)\right]=
{\rm exp}\left[\pm\tfrac{97a\ve}{\beta}\right].
$$

\medskip
\noindent
(4) The proof of this and of the next part follows \cite[Prop. 3.4]{Lima-Matheus}.
We prove that $\Psi_{\wh x_1}(B[e^{-2\ve}\eta_1])\subset \Psi_{\wh x_2}(B[\eta_2])$.
If $v\in B[e^{-2\ve}\eta_1]$ then
$\norm{C(\wh x_1)v}\leq e^{-2\ve}\eta_1<2\mathfrak d(x)$, hence by (A1):
$$
\Sas(C(\wh x_1)v,C(\wh x_2)v)\leq 2\left[d(\vt[\wh x_1],\vt[\wh x_2])+\norm{C_1v-C_2v}\right]\leq 2(\eta_1\eta_2)^4.
$$
By (A2), 
$d(\Psi_{\wh x_1}(v),\Psi_{\wh x_2}(v))\leq 4(\eta_1\eta_2)^4$ and so 
$\Psi_{\wh x_1}(v)\in B(\Psi_{\wh x_2}(v),4(\eta_1\eta_2)^4)$.
By Lemma \ref{Lemma-Pesin-chart}(1),
$B(\Psi_{\wh x_2}(v),4(\eta_1\eta_2)^4)\subset \Psi_{\wh x_2}(B)$ where
$B\subset \R^m$ is the ball with center $v$ and radius $8\norm{C_2^{-1}}(\eta_1\eta_2)^4$,
hence it is enough that $B\subset B[\eta_2]$. If $w\in B$ then
$\norm{w}\leq \norm{v}+8\norm{C_2^{-1}}(\eta_1\eta_2)^4\leq (e^{-\ve}+8\ve^{6/\beta})\eta_2<\eta_2$
for $\ve>0$ small enough.

\medskip
\noindent
(5) The proof that $\Psi_{\wh x_2}^{-1}\circ \Psi_{\wh x_1}$ is well-defined in
$B[d(\vt[\wh x_1],\mathfs S)^a]$ is similar to the proof of (4), the only difference being in the last estimate:
if $\ve>0$ is small enough then for $w\in B$ it holds
\begin{align*}
&\, \norm{w}\leq \norm{v}+8\norm{C_2^{-1}}(\eta_1\eta_2)^4\leq d(\vt[\wh x_1],\mathfs S)^a+8(\eta_1\eta_2)^3\\
&\leq [e^{a\ve}+8\ve^{6/\beta}]d(\vt[\wh x_2],\mathfs S)^a<2\mathfrak d(\wh x_2).
\end{align*}
Now:
\begin{align*}
&\Psi_{\wh x_2}^{-1}\circ \Psi_{\wh x_1}-{\rm Id}=
C(\wh x_2)^{-1}\circ\exp{\vt[\wh x_2]}^{-1}\circ\exp{\vt[\wh x_1]}\circ C(\wh x_1)-{\rm Id}\\
&=[C_2^{-1}\circ P_{\vt[\wh x_2],x}]\circ[\exp{\vt[\wh x_2]}^{-1}\circ\exp{\vt[\wh x_1]}-
P_{\vt[\wh x_1],\vt[\wh x_2]}]\circ [P_{x,\vt[\wh x_1]}\circ C_1]\\
&\ \ \ \, +C_2^{-1}(C_1-C_2)\\
&=[C_2^{-1}\circ P_{\vt[\wh x_2],x}]\circ[\exp{\vt[\wh x_2]}^{-1}-P_{\vt[\wh x_1],\vt[\wh x_2]}\circ\exp{\vt[\wh x_1]}^{-1}]\circ\Psi_{\wh x_1}+C_2^{-1}(C_1-C_2).
%&=C_2^{-1}C_1+C_2^{-1}\circ[\vartheta_{x_2}^{-1}\circ\exp{x_2}^{-1}-\vartheta_{x_1}^{-1}\circ\exp{x_1}^{-1}]
%\circ\exp{x_1}\circ\vartheta_{x_1}\circ C_1\\
%&={\rm Id}+C_2^{-1}(C_2-C_1)+C_2^{-1}\circ
%[\vartheta_{x_2}^{-1}\circ\exp{x_2}^{-1}-\vartheta_{x_1}^{-1}\circ\exp{x_1}^{-1}]\circ\Psi_{x_1}.
\end{align*}
We calculate the $C^{1+\frac{\beta}{2}}$ norm of
$[\exp{\vt[\wh x_2]}^{-1}-P_{\vt[\wh x_1],\vt[\wh x_2]}\circ\exp{\vt[\wh x_1]}^{-1}]\circ\Psi_{\wh x_1}$
in the ball $B[d(\vt[\wh x_1], \mathfs{S})^{2a}]$.
By Lemma \ref{Lemma-Pesin-chart}, $\|d\Psi_{\wh x_1}\|_0\leq 2$ and
$$
\Hol{\frac{\beta}{2}}(d\Psi_{\wh x_1})\leq d(\vt[\wh x_1],\mathfs S)^{-a}2d(\vt[\wh x_1],\mathfs{S})^{2a(1-\beta/2)}
=2d(\vt[\wh x_1],\mathfs S)^{a(1-\beta)}
< 2.
$$
%\textcolor{red}{where we used that
%$d(x_1,\mathfs D)^{-a}\leq\rho(x_1)^{-a}\leq \ve^{1/2}Q_\ve(x_1)^{-\beta/6}$.}
Call $\Theta:=\exp{\vt[\wh x_2]}^{-1}-P_{\vt[\wh x_1],\vt[\wh x_2]}\circ\exp{\vt[\wh x_1]}^{-1}$.
For $\ve>0$ sufficiently small, inside $\mathfrak B_{\wh x_1}$ we have the following bounds:
\begin{enumerate}[$\circ$]
\item By (A2),
$\norm{\Theta(z)}\leq \Sas(\exp{\vt[\wh x_2]}^{-1}(z),\exp{\vt[\wh x_1]}^{-1}(z))\leq 2d(\vt[\wh x_1],\vt[\wh x_2])
\leq 2\ve^{6/\beta}(\eta_1\eta_2)^3$
thus $\norm{\Theta\circ \Psi_{\wh x_1}}_{C^0}<\ve^{2/\beta}(\eta_1\eta_2)^3$.
\item By (A3), $\norm{d\Theta_z}=\|\tau(\vt[\wh x_2],z)-\tau(\vt[\wh x_1],z)\|
\leq d(\vt[\wh x_1],\mathfs S)^{-a}d(\vt[\wh x_1],\vt[\wh x_2])<\ve^{6/\beta}(\eta_1\eta_2)^3$.
Hence $\norm{d\Theta}_{C^0}<\ve^{6/\beta}(\eta_1\eta_2)^3$ and
$\norm{d(\Theta\circ\Psi_{\wh x_1})}_{C^0}\leq 2\ve^{6/\beta}(\eta_1\eta_2)^3<\ve^{2/\beta}(\eta_1\eta_2)^3$.
\item By (A4),
\begin{align*}
&\, \norm{\widetilde{d\Theta_y}-\widetilde{d\Theta_z}}=\norm{[\tau(\vt[\wh x_2],y)-\tau(\vt[\wh x_1],y)]
-[\tau(\vt[\wh x_2],z)-\tau(\vt[\wh x_1],z)]}\\
&\leq d(\vt[\wh x_1],\mathfs S)^{-a}d([\vt[\wh x_1],\vt[\wh x_2])d(y,z)%<40\ve^{1/2}(\eta_1\eta_2)^4\|v-w\|^{\beta/2}.
\end{align*}
hence $\Lip{d\Theta}\leq d(\vt[\wh x_1],\mathfs S)^{-a}d(\vt[\wh x_1],\vt[\wh x_2])$.
%$\Hol{\beta}(d\tau)\leq 2d(x_1,\mathfs D)^{-a}40Q_\ve(x_1)^{1-\beta}<80\ve^{1/2}$.
\item Using that
\begin{align*}
\Hol{\frac{\beta}{2}}(d(\Theta_1\circ\Theta_2))&\leq \norm{d\Theta_1}_{C^0}\Hol{\frac{\beta}{2}}(d\Theta_2)+\\
&\ \ \ \ \, \Lip{d\Theta_1}\norm{d\Theta_2}_{C^0}^{2}2d(\vt[\wh x_1],\mathfs{S})^{2a(1-\beta/2)}
\end{align*}
for $\Theta_2$ defined in $B[d(\vt[\wh x_1],\mathfs{S})^{2a}]$, we obtain that
\begin{align*}
&\ \Hol{\frac{\beta}{2}}[d(\Theta\circ\Psi_{\wh x_1})]\\
&\leq \norm{d\Theta}_{C^0}\Hol{\frac{\beta}{2}}(d\Psi_{\wh x_1})+
\Lip{d\Theta}\norm{d\Psi_{\wh x_1}}_{C^0}^2 2d(\vt[\wh x_1],\mathfs{S})^{2a(1-\beta/2)}\\
&<2\ve^{6/\beta}(\eta_1\eta_2)^3+ 
d(\vt[\wh x_1],\mathfs S)^{-a}d(\vt[\wh x_1],\vt[\wh x_2])8d(\vt[\wh x_1],\mathfs{S})^{2a(1-\beta/2)}\\
&<2\ve^{6/\beta}(\eta_1\eta_2)^3+8\ve^{6/\beta}(\eta_1\eta_2)^3<\ve^{2/\beta}(\eta_1\eta_2)^3.
\end{align*}
\end{enumerate}
This implies that $\norm{\Theta\circ\Psi_{\wh x_1}}_{C^{1+\frac{\beta}{2}}}<3\ve^{2/\beta}(\eta_1\eta_2)^3$, hence
$$
\norm{C_2^{-1}\circ P_{\vt[\wh x_2],x}\circ\Theta\circ\Psi_{\wh x_1}}_{C^{1+\frac{\beta}{2}}}
\leq \|C_2^{-1}\|3\ve^{2/\beta}(\eta_1\eta_2)^3 \leq 3\ve^{2/\beta}(\eta_1\eta_2)^2.
$$
Thus
$\norm{\Psi_{\wh x_2}^{-1}\circ \Psi_{\wh x_1}-{\rm Id}}_{C^{1+\frac{\beta}{2}}}\leq
3\ve^{2/\beta}(\eta_1\eta_2)^2+\norm{C_2^{-1}}(\eta_1\eta_2)^4<
3\ve^{2/\beta}(\eta_1\eta_2)^2+\ve^{6/\beta}(\eta_1\eta_2)^3<4\ve^{2/\beta}(\eta_1\eta_2)^2<\ve(\eta_1\eta_2)^2$.
\end{proof}

\subsection{The maps $F_{\wh x,\wh y}$ and $F^{-1}_{\wh x,\wh y}$}

Let $\wh x,\wh y\in{\rm NUH}$. When $\Psi_{\wh f(\wh x)}^{\eta}\overset{\ve}{\approx}\Psi_{\wh y}^{\eta'}$,
we are able to change $\Psi_{\wh f(\wh x)}$ by $\Psi_{\wh y}$ in $F_{\wh x}$. Similarly,
when $\Psi_{\wh x}^{\eta}\overset{\ve}{\approx}\Psi_{\wh f^{-1}(\wh y)}^{\eta'}$,
we can change $\Psi_{\wh x}$ by $\Psi_{\wh f^{-1}(\wh y)}$ in $F^{-1}_{\wh x}$.
In this section, we prove that these changes define maps that satisfy a result similar
to Theorem \ref{Thm-non-linear-Pesin}.

\medskip
\noindent
{\sc The maps $F_{\wh x,\wh y}$ and $F^{-1}_{\wh x,\wh y}$:}
If $\Psi_{\wh f(\wh x)}^{\eta}\overset{\ve}{\approx}\Psi_{\wh y}^{\eta'}$,
define $F_{\wh x,\wh y}:=\Psi_{\wh y}^{-1}\circ f\circ \Psi_{\wh x}$. Similarly, if
$\Psi_{\wh x}^{\eta}\overset{\ve}{\approx}\Psi_{\wh f^{-1}(\wh y)}^{\eta'}$, define
$F^{-1}_{\wh x,\wh y}:=\Psi_{\wh x}^{-1}\circ f^{-1}_{\wh x}\circ \Psi_{\wh y}$. 

\medskip
We will write $F^{\pm}_{\wh x,\wh y}$ to represent both  $F_{\wh x,\wh y}$
and  $F^{-1}_{\wh x,\wh y}$. In the above form, the definitions are not symmetric,
since the natural would be to define $F_{\wh x,\wh y}^{-1}$ taking the inverse branch 
$f_{\wh f^{-1}(\wh y)}^{-1}$ instead of $f_{\wh x}^{-1}$.
As a matter of fact, the overlap is strong enough to guarantee that,
in our domains of definition, $f_{\wh x}^{-1}$ and $f_{\wh f^{-1}(\wh y)}^{-1}$ coincide.
We state this as a separate lemma, since it will be used a few times in the text.

\begin{lemma}\label{Lemma-inverse-branches}
Let $\wh x,\wh y\in{\rm NUH}$ and assume
$\Psi_{\wh x}^{\eta}\overset{\ve}{\approx}\Psi_{\wh f^{-1}(\wh y)}^{\eta'}$.
If $A=\Psi_{\wh y}(B[20Q(\wh y)])$, then 
$A\subset E_{\wh x}$ and $f_{\wh f^{-1}(\wh y)}^{-1}(A)\subset f_{\wh x}^{-1}(E_{\wh x})$.
\end{lemma}

\begin{proof}
Write $\wh x=(x_n)_{n\in\Z}$
and $\wh y=(y_n)_{n\in\Z}$. The map $f_{\wh x}^{-1}$ is the inverse branch of $f$
taking $x_1$ to $x_0$ and, by (A5), it is well-defined in $E_{x_0}$.
By Proposition \ref{Proposition-overlap}(1), we have $d(y_{-1},\mathfs S)=e^{\pm\ve}d(x_0,\mathfs S)$
and $d(y_{0},\mathfs S)=e^{\pm\ve}d(x_1,\mathfs S)$, then
\begin{align*}
&\, Q(\wh y)<\ve^{6/\beta}\rho(\wh y)^a\leq \ve^{6/\beta}\min\{d(y_{-1},\mathfs S)^a,d(y_0,\mathfs S)^a\}\\
&\leq \ve^{6/\beta}e^{a\ve}\min\{d(x_0,\mathfs S)^a,d(x_1,\mathfs S)^a\}<\ve\mathfrak r(x_0).
\end{align*}
Also, (A6) implies that
$d(x_1,y_0)<\rho(\wh x)^{-a}d(x_0,y_{-1})<\rho(\wh x)^{-a}Q(\wh x)\ll \ve\rho(\wh x)^a<\ve\mathfrak r(x_0)$
and so 
$A\subset B(y_0,40Q(\wh y))\subset B(x_1,40Q(\wh y)+d(x_1,y_0))\subset E_{x_0}$.
Now we prove the second inclusion. Since $A\subset B(y_0,40Q(\wh y))$,
again by (A6) we have
\begin{align*}
&\, f_{\wh f^{-1}(\wh y)}^{-1}(A)\subset B(y_{-1},40d(y_{-1},\mathfs S)^{-a}Q(\wh y))\\
&\subset B(x_0,40d(y_{-1},\mathfs S)^{-a}Q(\wh y)+d(y_{-1},x_0)).
\end{align*}
Assumption (A6) also implies that
$f_{\wh x}^{-1}(E_{x_0})\supset B(x_0,2d(x_0,\mathfs S)^a\mathfrak r(x_0))$.
Since
\begin{align*}
&\, 40d(y_{-1},\mathfs S)^{-a}Q(\wh y)+d(y_{-1},x_0)<
\min\{d(y_{-1},\mathfs S)^{3a},d(y_0,\mathfs S)^{3a}\}\\
& \leq e^{3a\ve}\min\{d(x_0,\mathfs S)^{3a},d(x_1,\mathfs S)^{3a}\}<2d(x_0,\mathfs S)^a\mathfrak r(x_0),
\end{align*}
we obtain that $f_{\wh f^{-1}(\wh y)}^{-1}(A)\subset f_{\wh x}^{-1}(E_{x_0})$.
\end{proof}

Now we prove a version of Theorem \ref{Thm-non-linear-Pesin} for $F^{\pm}_{\wh x,\wh y}$.
%In order to keep estimates of size $\ve$, we consider the $C^{1+\frac{\beta}{3}}$ norm.

\begin{theorem}\label{Thm-non-linear-Pesin-2}
The following holds for all $\ve>0$ small enough. If $\wh x,\wh y\in{\rm NUH}$
and $\Psi_{\wh f(\wh x)}^{\eta}\overset{\ve}{\approx}\Psi_{\wh y}^{\eta'}$
then $F_{\wh x,\wh y}$ is well-defined in $B[20Q(\wh x)]$ and can be written as
$F_{\wh x,\wh y}=D(\wh x)+H^+$ s.t. for all $t\in [\eta,20Q(\wh x)]$ it holds:
\begin{enumerate}[{\rm (a)}]
\item $\|H^+(0)\|<\ve \eta$.
\item $\|H^+\|_{C^1(B[t])}<\ve t^{\beta/3}$.
\item $\Hol{\frac{\beta}{3}}(dH^+)<\ve$.
\end{enumerate}
Similarly, if $\wh x,\wh y\in{\rm NUH}$ and
$\Psi_{\wh x}^{\eta}\overset{\ve}{\approx}\Psi_{\wh f^{-1}(\wh y)}^{\eta'}$
then $F^{-1}_{\wh x,\wh y}$ is well-defined in $B[20Q(\wh y)]$ and can be written as
$F^{-1}_{\wh x,\wh y}=D(\wh f^{-1}(\wh y))^{-1}+H^-$ s.t. for all $t\in [\eta,20Q(\wh y)]$ it holds:
\begin{enumerate}[{\rm(d)}]
\item[{\rm(d)}] $\|H^-(0)\|<\ve \eta$.
\item[{\rm(e)}] $\|H^-\|_{C^1(B[t])}<\ve t^{\beta/3}$.
\item[{\rm(f)}] $\Hol{\frac{\beta}{3}}(dH^-)<\ve$.
\end{enumerate}
\end{theorem}

Observe that the overlap condition is strong enough to guarantee that
$F^{\pm}_{\wh x,\wh y}$ is defined in a domain of radius $20Q(\cdot)$
and, except for the $C^1$ norm, the estimates are also uniform in this domain.

\begin{proof}
First assume that $\Psi_{\wh f(\wh x)}^{\eta}\overset{\ve}{\approx}\Psi_{\wh y}^{\eta'}$,
and let us prove items (a)--(c).
We write $F_{\wh x,\wh y}=(\Psi_{\wh y}^{-1}\circ\Psi_{\wh f(\wh x)})\circ F_{\wh x}=:G\circ F_{\wh x}$
and see it as a small perturbation of $F_{\wh x}$.
By Theorem \ref{Thm-non-linear-Pesin}(3--4),
$$
F_{\wh x}(0)=0,\|dF_{\wh x}\|_{C^0}< 2d(\vt[\wh x],\mathfs S)^{-a}(1+e^{2\chi})^{1/2},
\|d(F_{\wh x})_v-d(F_{\wh x})_w\|\leq \ve\|v-w\|^{\beta/2}
$$
for $v,w\in B[20Q(\wh x)]$, where the $C^0$ norm is taken in $B[20Q(\wh x)]$,
and by Proposition \ref{Proposition-overlap}(5) we have
$$
\|G-{\rm Id}\|<\ve(\eta\eta')^2,\ \|d(G-{\rm Id})\|_{C^0}<\ve(\eta\eta')^2,
\ \|dG_v-dG_w\|\leq\ve(\eta\eta')^2\|v-w\|^{\beta/2}
$$
for $v,w\in B[d(\vt_1[\wh x],\mathfs{S})^{2a}]$, where the $C^0$ norm is taken in this same domain.
We start proving that $F_{\wh x,\wh y}$ is well-defined in $B[20Q(\wh x)]$. We have
$$
F_{\wh x}(B[20Q(\wh x)])\subset B[40\rho(\wh x)^{-a}(1+e^{2\chi})^{1/2}Q(\wh x)]
\subset B[d(\vt_1[\wh x],\mathfs S)^{2a}]
$$
since
$40\rho(\wh x)^{-a}(1+e^{2\chi})^{1/2}Q(\wh x)<
40(1+e^{2\chi})^{1/2}\ve^{6/\beta}d(\vt_1[\wh x],\mathfs S)^{2a}<d(\vt_1[\wh x],\mathfs S)^{2a}$
for $\ve>0$ small enough. By Proposition \ref{Proposition-overlap}(5), $F_{\wh x,\wh y}$ is well-defined.
 
Let $H^+:=G\circ F_{\wh x}-d(F_{\wh x})_0=F_{\wh x,\wh y}-D(\wh x)$,
cf. Lemma \ref{Lemma-linear-reduction}.
By Proposition \ref{Proposition-overlap}(3),
we have $\|H^+(0)\|=\|G(0)\|<\ve(\eta\eta')^2<\tfrac{1}{8}\ve\eta$, which is stronger than (a).
Now we prove (c). Start observing that for $\ve>0$ small enough it holds:
\begin{align*}
&\, \|d(H^+)_0\|\leq \|dG_0\circ D(\wh x)-D(\wh x)\|\leq \|d(G-{\rm Id})_0\|\cdot\|D(\wh x)\|\\
&<\ve(\eta\eta')^2 d(\vt[\wh x],\mathfs S)^{-a}(1+e^{2\chi})^{1/2}<\ve\eta\eta' \ve^{6/\beta}(1+e^{2\chi})^{1/2}
<\tfrac{1}{8}\ve\eta^{\beta/3}.
\end{align*}
Since $F_{\wh x}(B[20Q(\wh x)])\subset B[d(\vt_1[\wh x],\mathfs S)^{2a}]$, if $\ve>0$
is small enough then for all $v,w\in B[20Q(\wh x)]$ we have
\begin{align*}
&\, \|d(H^+)_v-d(H^+)_w\|=\|dG_{F_{\wh x}(v)}\circ d(F_{\wh x})_v-dG_{F_{\wh x}(w)}\circ d(F_{\wh x})_w\|\\
&\leq \|dG_{F_{\wh x}(v)}-dG_{F_{\wh x}(w)}\|\cdot\|d(F_{\wh x})_v\|+
\|dG_{F_{\wh x}(w)}\|\cdot\|d(F_{\wh x})_v-d(F_{\wh x})_w\|\\
&\leq \ve(\eta\eta')^2\|F_{\wh x}(v)-F_{\wh x}(w)\|^{\beta/2}\|d(F_{\wh x})\|_{C^0}+\ve\|dG\|_{C^0}\|v-w\|^{\beta/2}\\
&\leq \left[\ve(\eta\eta')^2\|d(F_{\wh x})\|_{C^0}^{1+\beta/2}+80\ve Q(\wh x)^{\beta/6}\right]\|v-w\|^{\beta/3}\\
&\leq \left[4\eta d(\vt[\wh x],\mathfs S)^{-2a}(1+e^{2\chi})+ 80Q(\wh x)^{\beta/6}\right]\ve\|v-w\|^{\beta/3}\\
&\leq \left[4\ve^{6/\beta}(1+e^{2\chi})+ 80\ve\right]\ve\|v-w\|^{\beta/3}<\tfrac{1}{8}\ve\|v-w\|^{\beta/3},
\end{align*}
hence $\Hol{\frac{\beta}{3}}(dH^+)<\tfrac{1}{8}\ve$, which is stronger than (c).
Finally, we fix $\eta\leq t\leq 20Q(\wh x)$ and prove (b). For $v\in B[t]$:
\begin{enumerate}[$\circ$]
\item $\|d(H^+)_v\|\leq \|d(H^+)_0\|+\tfrac{1}{8}\ve\|v\|^{\beta/3}<
\tfrac{1}{8}\ve\eta^{\beta/3}+\tfrac{1}{8}\ve t^{\beta/3}\leq \tfrac{1}{4}\ve t^{\beta/3}$,
therefore
$\|dH^+\|_{C^0(B[t])}<\tfrac{1}{4}\ve t^{\beta/3}$.
\item $\|H^+(v)\|\leq \|H^+(0)\|+\|dH^+\|_{C^0}\|v\|<\tfrac{1}{8}\ve\eta+\tfrac{1}{4}\ve t^{\beta/3}<
\tfrac{1}{2}\ve t^{\beta/3}$, which implies that $\|H^+\|_{C^0(B[t])}<\tfrac{1}{2}\ve t^{\beta/3}$.
\end{enumerate}
Thus, $\|H^+\|_{C^1(B[t])}<\ve t^{\beta/3}$.

Now we prove the second part. Assume that
$\Psi_{\wh x}^{\eta}\overset{\ve}{\approx}\Psi_{\wh f^{-1}(\wh y)}^{\eta'}$.
By Lemma \ref{Lemma-inverse-branches}, we have
$F_{\wh x,\wh y}^{-1}=\Psi_{\wh x}^{-1}\circ f_{\wh f^{-1}(\wh y)}^{-1}\circ \Psi_{\wh y}$
in the domain $B[20Q(\wh y)]$. Writing
$F_{\wh x,\wh y}^{-1}=(\Psi_{\wh x}^{-1}\circ\Psi_{\wh f^{-1}(\wh y)})\circ F_{\wh f^{-1}(\wh y)}^{-1}
=:G\circ  L$
and seeing it as a small perturbation of $L=F_{\wh f^{-1}(\wh y)}^{-1}$, we have
by Theorem \ref{Thm-non-linear-Pesin}(3--4) that
$$
L(0)=0,\ \|dL\|_{C^0}< 2d(\vt_{-1}[\wh y],\mathfs S)^{-a}(1+e^{2\chi})^{1/2},
\ \|dL_v-dL_w\|\leq \ve\|v-w\|^{\beta/2}
$$
for $v,w\in B[20Q(\wh y)]$, where the $C^0$ norm is taken in $B[20Q(\wh y)]$,
and by Proposition \ref{Proposition-overlap}(5) that
$$
\|G-{\rm Id}\|<\ve(\eta\eta')^2,\ \|d(G-{\rm Id})\|_{C^0}<\ve(\eta\eta')^2,
\ \|dG_v-dG_w\|\leq\ve(\eta\eta')^2\|v-w\|^{\beta/2}
$$
for $v,w\in B[d(y_{-1},\mathfs{S})^{2a}]$, where the $C^0$ norm is taken in this same domain.
To prove that $F_{\wh x,\wh y}$ is well-defined in $B[20Q(\wh y)]$, notice that
$$
L(B[20Q(\wh y)])\subset B[40\rho(\wh y)^{-a}(1+e^{2\chi})^{1/2}Q(\wh y)]
\subset B[d(y_{-1},\mathfs S)^{2a}]
$$
since
$40\rho(\wh y)^{-a}(1+e^{2\chi})^{1/2}Q(\wh y)<
40(1+e^{2\chi})^{1/2}\ve^{6/\beta}d(y_{-1},\mathfs S)^{2a}<d(y_{-1},\mathfs S)^{2a}$
for $\ve>0$ small enough. By Proposition \ref{Proposition-overlap}(5),
$F_{\wh x,\wh y}^{-1}$ is well-defined.
 
Let $H^-:=G\circ L-dL_0=F_{\wh x,\wh y}^{-1}-D(\wh f^{-1}(\wh y))^{-1}$.
By Proposition \ref{Proposition-overlap}(5), we have
$\|H^-(0)\|=\|G(0)\|<\ve(\eta\eta')^2<\tfrac{1}{8}\ve\eta$, which is stronger than (d).
To prove (f), start observing that for $\ve>0$ small enough it holds:
\begin{align*}
&\, \|d(H^-)_0\|\leq \|dG_0\circ D(\wh f^{-1}(\wh y))^{-1}-D(\wh f^{-1}(\wh y))^{-1}\|
\leq \|d(G-{\rm Id})_0\|\|D(\wh f^{-1}(\wh y))^{-1}\|\\
&<\ve(\eta\eta')^2 d(\vt_{-1}[\wh y],\mathfs S)^{-a}(1+e^{2\chi})^{1/2}<\ve\eta\eta' \ve^{6/\beta}(1+e^{2\chi})^{1/2}
<\tfrac{1}{8}\ve\eta^{\beta/3}.
\end{align*}
Since $L(B[20Q(\wh y)])\subset B[d(y_{-1},\mathfs S)^{2a}]$, if $\ve>0$
is small enough then for all $v,w\in B[20Q(\wh y)]$ we have:
\begin{align*}
&\, \|d(H^-)_v-d(H^-)_w\|=\|dG_{L(v)}\circ dL_v-
dG_{L(w)}\circ dL_w\|\\
&\leq \norm{dG_{L(v)}-dG_{L(w)}}\cdot\norm{dL_v}+\norm{dG_{L(w)}}\cdot\norm{dL_v-dL_w}\\
&\leq \ve(\eta\eta')^2\|L(v)-L(w)\|^{\beta/2}\|dL\|_{C^0}+\ve\|dG\|_{C^0}\|v-w\|^{\beta/2}\\
&\leq \left[\ve(\eta\eta')^2\|dL\|_{C^0}^{1+\beta/2}+80\ve Q(\wh y)^{\beta/6}\right]\|v-w\|^{\beta/3}\\
&\leq \left[4\eta'd(\vt_{-1}[\wh y],\mathfs S)^{-2a}(1+e^{2\chi})+ 80Q(\wh y)^{\beta/6}\right]\ve\|v-w\|^{\beta/3}\\
&\leq \left[4\ve^{6/\beta}(1+e^{2\chi})+ 80\ve\right]\ve\|v-w\|^{\beta/3}<\tfrac{1}{8}\ve\|v-w\|^{\beta/3}.
\end{align*}
This implies that $\Hol{\frac{\beta}{3}}(dH^-)<\tfrac{1}{8}\ve$, which is stronger than (f).
Finally, we fix $\eta\leq t\leq 20Q(\wh y)$ and prove (e). For $v\in B[t]$:
\begin{enumerate}[$\circ$]
\item $\|d(H^-)_v\|\leq \|d(H^-)_0\|+\tfrac{1}{8}\ve\|v\|^{\beta/3}<
\tfrac{1}{8}\ve\eta^{\beta/3}+\tfrac{1}{8}\ve t^{\beta/3}\leq \tfrac{1}{4}\ve t^{\beta/3}$,
therefore
$\|dH^-\|_{C^0(B[t])}<\tfrac{1}{4}\ve t^{\beta/3}$.
\item $\|H^-(v)\|\leq \|H^-(0)\|+\|dH^-\|_{C^0}\|v\|<\tfrac{1}{8}\ve\eta+\tfrac{1}{4}\ve t^{\beta/3}<
\tfrac{1}{2}\ve t^{\beta/3}$, which implies that $\|H^-\|_{C^0(B[t])}<\tfrac{1}{2}\ve t^{\beta/3}$.
\end{enumerate}
Thus, $\|H^-\|_{C^1(B[t])}<\ve t^{\beta/3}$.
\end{proof}

\subsection{Double charts}\label{Section-double-charts}

We now define $\ve$--double charts. Recall that $\ve>0$ is a fixed small parameter,
and that $\delta_\ve=e^{-\ve n}\in I_\ve$ where $n$ is the unique positive integer s.t.
$e^{-\ve n}<\ve\leq e^{-\ve(n-1)}$.

\medskip
\noindent
{\sc $\ve$--double chart:} An {\em $\ve$--double chart} is a pair of Pesin charts
$\Psi_{\wh x}^{p^s,p^u}=(\Psi_{\wh x}^{p^s},\Psi_{\wh x}^{p^u})$ where $p^s,p^u\in I_\ve$
with $0<p^s,p^u\leq \delta_\ve Q(\wh x)$.

\medskip
When $v=\Psi_{\wh x}^{p^s,p^u}$, we call $\wh x$ the center of $v$.
The parameters $p^s/p^u$ will define sizes for the stable/unstable manifold at $\wh x$.
Hence, for $\wh x\in{\rm NUH}^*$ the prototype of $\ve$--double chart is
$\Psi_{\wh x}^{q^s(\wh x),q^u(\wh x)}$.

\medskip
\noindent
{\sc Edge $v\overset{\ve}{\rightarrow}w$:} Given $\ve$--double charts $v=\Psi_{\wh x}^{p^s,p^u}$
and $w=\Psi_{\wh y}^{q^s,q^u}$, we draw an edge from $v$ to $w$ if the following conditions are
satisfied:
\begin{enumerate}[iii\,]
\item[(GPO1)] $\Psi_{\wh f(\wh x)}^{q^s\wedge q^u}\overset{\ve}{\approx}\Psi_{\wh y}^{q^s\wedge q^u}$
and $\Psi_{\wh f^{-1}(\wh y)}^{p^s\wedge p^u}\overset{\ve}{\approx}\Psi_{\wh x}^{p^s\wedge p^u}$.
\item[(GPO2)] $p^s=\min\{e^\ve q^s,\delta_\ve Q(\wh x)\}$ and $q^u=\min\{e^\ve p^u,\delta_\ve Q(\wh y)\}$.
\end{enumerate}

\medskip
Since $d_{s/u}$ are $\wh f$--invariant, (GPO1) implies that $d_{s/u}(\wh x)=d_{s/u}(\wh y)$.
Condition (GPO1) allows to pass from an $\ve$--double chart at $\wh x$
to an $\ve$--double chart at $\wh y$ and vice-versa.
By Lemma \ref{Lemma-q^s}(2), we have 
$\Psi_{\wh x}^{q^s(\wh x),q^u(\wh x)}\overset{\ve}{\rightarrow}
\Psi_{\wh f(\wh x)}^{q^s(\wh f(\wh x)),q^u(\wh f(\wh x))}$. 
Hence we can regard (GPO2) as the counterpart of Lemma \ref{Lemma-q^s}(2)
for a general edge. Observe that (GPO2) implies that
$\tfrac{p^s\wedge p^u}{q^s\wedge q^u}=e^{\pm\ve}$. Indeed, 
we have $p^s\wedge p^u=\min\{e^\ve q^s,\delta_\ve Q(\wh x)\}\wedge p^u=\min\{e^\ve q^s,p^u\}$
and similarly $q^s\wedge q^u=\min\{q^s,e^\ve p^u\}$. Thus
$p^s\wedge p^u=e^\ve \min\{q^s,e^{-\ve}p^u\}\leq e^\ve \min\{q^s,e^\ve p^u\}=e^\ve(q_s\wedge q_u)$
and similarly $q_s\wedge q_u\leq e^\ve(p_s\wedge p_u)$. In particular, 
the ingoing and outgoing degrees of every $\ve$--double chart are finite.

\medskip
\noindent
{\sc $\ve$--generalized pseudo-orbit ($\ve$--gpo):} An {\em $\ve$--generalized pseudo-orbit ($\ve$--gpo)}
is a sequence $\un{v}=\{v_n\}_{n\in\Z}$ of $\ve$--double charts
s.t. $v_n\overset{\ve}{\rightarrow}v_{n+1}$ for all $n\in\Z$.

\subsection{The graph transform method}\label{Sec-graph-transform}

The proofs of this section are collected, in a self-contained manner, in the Appendix.
We hope this separation will both clarify the presentation as highlight the main arguments
used to employ the graph transform method, without boundedness assumptions
on the derivative.
Let $v=\Psi_{\wh x}^{p^s,p^u}$ be an $\ve$--double chart.
Recall from Lemma \ref{Lemma-linear-reduction} that
$$
D({\wh x})=\left[\begin{array}{cc}D_s({\wh x}) & 0 \\
0 & D_u({\wh x})
\end{array}\right]
$$
where $D_s(\wh x)$ is a $d_s(\wh x)\times d_s(\wh x)$ matrix s.t.
$\|D_s({\wh x})\|< e^{-\chi}$,
and $D_u(\wh x)$ is a $d_u(\wh x)\times d_u(\wh x)$ matrix s.t.
$\|D_u({\wh x})^{-1}\|<e^{-\chi}$.

\medskip
\noindent
{\sc Admissible manifolds:} We define an {\em $s$--admissible manifold at $v$} as a set
of the form $\Psi_{\wh x}\{(v_1,G(v_1)):v_1\in B^{d_s(\wh x)}[p^s]\}$ where
$G:B^{d_s(\wh x)}[p^s]\to\R^{d_u(\wh x)}$ is a $C^{1+\beta/3}$ function
s.t.:
\begin{enumerate}
\item[(AM1)] $\norm{G(0)}\leq 10^{-3}(p^s\wedge p^u)$.
\item[(AM2)] $\norm{dG_0}\leq \tfrac{1}{2}(p^s\wedge p^u)^{\beta/3}$.
\item[(AM3)] $\|dG\|_{C^0}+\Hol{\frac{\beta}{3}}(dG)\leq\tfrac{1}{2}$ where the norms
are taken in $B^{d_s(\wh x)}[p^s]$.
\end{enumerate}
Similarly, a {\em $u$--admissible manifold at $v$} is a set
of the form $\Psi_{\wh x}\{(G(v_2),v_2):v_2\in B^{d_u(\wh x)}[p^u]\}$
where $G:B^{d_u(\wh x)}[p^u]\to\R^{d_s(\wh x)}$ is a $C^{1+\beta/3}$ function
satisfying (AM1)--(AM3), where the norms are taken in $B^{d_u(\wh x)}[p^u]$.

\medskip
The function $G$ is called the {\em representing function}.
Since we are working with the euclidean norm,
admissible manifolds at $v$ are contained in $\Psi_{\wh x}(B^{d_s(\wh x)}[p^s]\times B^{d_u(\wh x)}[p^s])$.
We let $\mathfs M^s(v)$ (resp. $\mathfs M^u(v)$) denote the set of all $s$--admissible
(resp. $u$--admissible) manifolds at $v$.
Define a metric on $\M^s(v)$, comparing representing functions: for $V_1,V_2\in\M^s(v)$ with
representing functions $G_1,G_2$ and for $i\geq 0$, define
$d_{C^i}(V_1,V_2):=\|G_1-G_2\|_{C^i}$ where the norm is taken in $B^{d_s(\wh x)}[p^s]$.
A similar definition holds for $\M^u(v)$.

\begin{lemma}\label{Lemma-admissible-manifolds}
The following holds for $\ve>0$ small enough. If $v=\Psi_{\wh x}^{p^s,p^u}$ is an $\ve$--double chart, then for
every $V^s\in\mathfs M^s(v)$ and $V^u\in\mathfs M^u(v)$
the intersection $V^s\cap V^u$ consists of a single point $y=\Psi_{\wh x}(w)$,
and $\|w\|<50^{-1}(p^s\wedge p^u)$. Additionally, $y$
is a Lipschitz function of the pair $(V^s,V^u)$, with Lipschitz constant $3$.
\end{lemma}

Let $v=\Psi_{\wh x}^{p^s,p^u}$, $w=\Psi_{\wh y}^{q^s,q^u}$ be $\ve$--double charts with
$v\overset{\ve}{\rightarrow}w$. We now define the {\em graph transforms}: these are two maps
that work in different directions of the edge $v\overset{\ve}{\rightarrow}w$, one of them
sending $u$--admissible manifolds at $v$ to $u$--admissible manifolds at $w$, and the other
sending $s$--admissible manifolds at $w$ to $s$--admissible manifolds at $v$. 
In the sequel, we will write $d=d_s(\wh x)=d_s(\wh y)$ and $m-d=d_u(\wh x)=d_u(\wh y)$.

\medskip
\noindent
{\sc Graph transforms $\mathfs F_{v,w}^s$ and $\mathfs F_{v,w}^u$:} The {\em graph transform}
$\mathfs F_{v,w}^s:\mathfs M^s(w)\to\mathfs M^s(v)$ is the map that sends
an $s$--admissible manifold at $w$ with representing function $G:B^d[q^s]\to\R^{m-d}$
to the unique $s$--admissible manifold at $v$ with representing function
$\wt G:B^d[p^s]\to\R^{m-d}$ s.t.
$$
\left\{(v_1,\wt G(v_1)):v_1\in B^d[p^s]\right\} \subset 
F_{\wh x,\wh y}^{-1}\left\{(v_1,G(v_1)):v_1\in B^d[q^s]\right\}.
$$  
The {\em graph transform} $\mathfs F_{v,w}^u:\mathfs M^u(v)\to\mathfs M^u(w)$
is the map that sends a $u$--admissible manifold at $v$ with representing function
$G:B^{m-d}[p^u]\to\R^d$ to the unique $u$--admissible manifold at $w$ with
representing function $\wt G:B^{m-d}[q^u]\to\R^d$ s.t.
$$
\left\{(\wt G(v_2),v_2):v_2\in B^{m-d}[q^u]\right\} \subset 
F_{\wh x,\wh y}\left\{(G(v_2),v_2):v_2\in B^{m-d}[p^u]\right\}.
$$

\medskip
In other words, the representing functions of $s,u$--admissible manifolds
change by the application of $F_{\wh x,\wh y}^{-1},F_{\wh x,\wh y}$ respectively.

\begin{proposition}\label{Prop-graph-transform}
The following holds for $\ve>0$ small enough. If $v\overset{\ve}{\rightarrow}w$ then
$\mathfs F_{v,w}^s$ and $\mathfs F_{v,w}^u$ are well-defined. Furthermore, if
$V_1,V_2\in \mathfs M^u(v)$ then:
\begin{enumerate}[{\rm (1)}]
\item $ d_{C^0}(\mathfs F_{v,w}^u(V_1),\mathfs F_{v,w}^u(V_2))\leq e^{-\chi/2} d_{C^0}(V_1,V_2)$.
\item $ d_{C^1}(\mathfs F_{v,w}^u(V_1),\mathfs F_{v,w}^u(V_2))\leq
e^{-\chi/2}( d_{C^1}(V_1,V_2)+ d_{C^0}(V_1,V_2)^{\beta/3})$.
\item $f(V_i)$ intersects every element of $\mathfs M^s(w)$ at exactly one point.
\end{enumerate}
An analogous statement holds for $\mathfs F_{v,w}^s$, where in {\rm (3)} we consider 
$f_{\wh x}^{-1}(V_i)$.
\end{proposition}

As already mentioned, all the proofs are in the Appendix, where the required
assumptions (GT1)--(GT3) are weaker than we need.
If $\Psi_{\wh x}^{p^s,p^u}\overset{\ve}{\rightarrow}\Psi_{\wh y}^{q^s,q^u}$
then (GT1)--(GT3) are satisfied for:
\begin{enumerate}[$\circ$]
\item $p=q^s,\wt p=p^s,\eta=q^s\wedge q^u,\wt\eta=p^s\wedge p^u$ and $F=F_{\wh x,\wh y}^{-1}$,
thus $\mathfs F^s_{v,w}$ satisfies the appendix.
\item $p=p^u,\wt p=q^u,\eta=p^s\wedge p^u,\wt\eta=q^s\wedge q^u$
and $F=F_{\wh x,\wh y}$, thus $\mathfs F^u_{v,w}$ satisfies the appendix.
\end{enumerate}

\subsection{Stable/unstable manifolds of $\ve$--gpo's}

Call a sequence ${\un v}^+=\{v_n\}_{n\geq 0}$ a {\em positive $\ve$--gpo} if $v_n\overset{\ve}{\to}v_{n+1}$
for all $n\geq 0$. Similarly, a {\em negative $\ve$--gpo} is a sequence ${\un v}^-=\{v_n\}_{n\leq 0}$
s.t. $v_{n-1}\overset{\ve}{\to}v_n$ for all $n\leq 0$.

\medskip
\noindent
{\sc Stable/unstable manifold of positive/negative $\ve$--gpo:} The {\em stable manifold}
of a positive $\ve$--gpo ${\un v}^+=\{v_n\}_{n\geq 0}$ is 
$$
V^s[{\un v}^+]:=\lim_{n\to\infty}
(\mathfs F_{v_0,v_1}^s\circ\cdots\circ\mathfs F_{v_{n-2},v_{n-1}}^s\circ\mathfs F_{v_{n-1},v_n}^s)(V_n)
$$
for some (any) choice of $(V_n)_{n\geq 0}$ with $V_n\in\mathfs M^s(v_n)$.
The {\em unstable manifold} of a negative $\ve$--gpo ${\un v}^-=\{v_n\}_{n\leq 0}$ is 
$$
V^u[{\un v}^-]:=\lim_{n\to-\infty}
(\mathfs F_{v_{-1},v_0}^u\circ\cdots\circ\mathfs F_{v_{n+1},v_{n+2}}^u\circ\mathfs F_{v_n,v_{n+1}}^u)(V_n)
$$
for some (any) choice of $(V_n)_{n\leq 0}$ with $V_n\in\mathfs M^u(v_n)$.
\medskip
For an $\ve$--gpo $\un{v}=\{v_n\}_{n\in\Z}$, let $V^s[\un v]:=V^s[\{v_n\}_{n\geq 0}]$
and $V^u[\un v]:=V^u[\{v_n\}_{n\leq 0}]$.

\medskip
Before proceeding to the main properties of these manifolds, let us show that
the action of the graph transforms has a simple description in terms of positive/negative
$\ve$--gpo's.

\begin{lemma}\label{Lemma-transform-manifolds}
Let $v\overset{\ve}{\to}w$. If $\un v^-=\{v_n\}_{n\leq 0}$ is a negative $\ve$--gpo
with $v_0=v$, then 
$$
\mathfs F^u_{v,w}(V^u[\un v^-])=V^u[\un w^-]
$$
where $\un w^-=\{w_n\}_{n\leq 0}$ is the negative $\ve$--gpo with
$w_0=w$ and $w_n=v_{n+1}$ for $n<0$. Similarly, if $\un w^+=\{w_n\}_{n\geq 0}$
is a positive $\ve$--gpo with $w_0=w$, then 
$$
\mathfs F^s_{v,w}(V^s[\un w^+])=V^s[\un v^+]
$$
where $\un v^+=\{v_n\}_{n\geq 0}$ is the positive $\ve$--gpo with
$v_0=v$ and $v_n=w_{n-1}$ for $n>0$.
\end{lemma}

\begin{proof}
We prove the first statement. By definition, 
\begin{align*}
V^u[{\un w}^-]&=\lim_{n\to-\infty}
(\mathfs F_{w_{-1},w_0}^u\circ\mathfs F_{w_{-2},w_{-1}}^u\circ\cdots\circ\mathfs F_{w_n,w_{n+1}}^u)(V_n)\\
&=\lim_{n\to-\infty}
(\mathfs F_{v,w}^u\circ\mathfs F_{v_{-1},v_{0}}^u\circ\cdots\circ\mathfs F_{v_{n+1},v_{n+2}}^u)(V_n).
\end{align*}
for some (any) choice of $(V_n)_{n\leq 0}$ with $V_n\in\mathfs M^u(w_n)=\mathfs M^u(v_{n+1})$. 
By Proposition \ref{Prop-graph-transform},
$\mathfs F_{v,w}^u$ is continuous and so
\begin{align*}
V^u[{\un w}^-]&=\mathfs F_{v,w}^u\left[\lim_{n\to-\infty}
(\mathfs F_{v_{-1},v_{0}}^u\circ\cdots\circ\mathfs F_{v_{n+1},v_{n+2}}^u)(V_n)\right]=\mathfs F_{v,w}^u(V^u[\un v^-]).
\end{align*}
The second statement is proved analogously.
\end{proof}

The next proposition collects the main properties of stable/unstable manifolds.

\begin{proposition}\label{Prop-stable-manifolds}
The following holds for all $\ve>0$ small enough. For
every positive $\ve$--gpo ${\un v}^+=\{\Psi_{\wh x_n}^{p^s_n,p^u_n}\}_{n\geq 0}$
and every negative $\ve$--gpo ${\un v}^-=\{\Psi_{\wh x_n}^{p^s_n,p^u_n}\}_{n\leq 0}$:
\begin{enumerate}[{\rm (1)}]
\item {\sc Admissibility:} $V^s[{\un v}^+]$ is a well-defined admissible manifolds at $v_0$,
and if $G$ and $G_n$ are the representing functions of $V^s[{\un v}^+]$ and 
$(\mathfs F_{v_0,v_1}^s\circ\cdots\circ\mathfs F_{v_{n-1},v_n}^s)(V_n)$ 
respectively, then $dG_n\to dG$ uniformly. The same holds for $V^u[{\un v}^-]$.
\item {\sc Invariance:}
$$
f(V^s[\{v_n\}_{n\geq 0}])\subset V^s[\{v_n\}_{n\geq 1}]\text{ and }
f_{\wh x_{-1}}^{-1}(V^u[\{v_n\}_{n\leq 0}])\subset V^u[\{v_n\}_{n\leq -1}].
$$
\item {\sc Shadowing:} If $d=d_s(\wh x_0)$ then $V^s[{\un v}^+]$ equals
$$
\left\{y\in \Psi_{\wh x_0}(B^d[p^s_0]\times B^{m-d}[p^s_0]):
\begin{array}{c}f^n(y)\in \Psi_{\wh x_n}(B[20Q(\wh x_n)]),\\ \forall n\geq 0\end{array}
\right\}
$$
and $V^u[{\un v}^-]$ equals
$$
\left\{y\in \Psi_{\wh x_0}(B^d[p^u_0]\times B^{m-d}[p^u_0]):
\begin{array}{c}(f_{\wh x_n}^{-1}\circ\cdots\circ f_{\wh x_{-1}}^{-1})(y)\in \Psi_{\wh x_n}(B[20Q(\wh x_n)]),\\
\forall n< 0\end{array}
\right\}.
$$
\item {\sc Hyperbolicity:} If $y,z\in V^s[{\un v}^+]$ then
$$
d(f^n(y),f^n(z))\leq 4p^s_0e^{-\frac{\chi}{2}n}\text{ for all }n\geq 0.
$$
Similarly, if $y,z\in V^u[{\un v}^-]$ then
$$
d((f_{\wh x_n}^{-1}\circ\cdots\circ f_{\wh x_{-1}}^{-1})(y),(f_{\wh x_n}^{-1}\circ\cdots\circ f_{\wh x_{-1}}^{-1})(z))
\leq 4p^u_0e^{\frac{\chi}{2}n}\text{ for all }n\leq 0.
$$
\item {\sc H\"older property:} The map $\un v^+\mapsto V^s[\un v^+]$ is H\"older continuous,
i.e. there exists $K>0$ and $\theta<1$ s.t. for all $N\geq 0$, if $\un v^+,\un w^+$ are positive $\ve$--gpo's
with $v_n=w_n$ for $n=0,\ldots,N$
then $ d_{C^1}(V^s[\un v^+],V^s[\un w^+])\leq K\theta^N$. The same holds for the map
$\un v^-\mapsto V^u[\un v^-]$.
\end{enumerate}
\end{proposition}

\begin{proof}
When $M$ is compact and $f$ is a $C^{1+\beta}$ diffeomorphism,
this is \cite[Prop. 3.12 and 4.4(1)]{Ben-Ovadia-2019}.
The same proofs work in our case,
since they only use the hyperbolicity of $F_{\wh x_n,\wh x_{n+1}}^{\pm 1}$
(Theorem \ref{Thm-non-linear-Pesin-2}) and the contracting properties
of the graph transforms (Proposition \ref{Prop-graph-transform}).
\end{proof}

\subsection{Stable/unstable sets of $\ve$--gpo's}\label{Section-stable-sets}

Having in mind that we want to build a symbolic dynamics for $\wh f:\wh M\to\wh M$,
we now introduce the set-theoretical counterparts of stable/unstable manifolds in $\wh M$, which 
we call stable/unstable sets.
%Just like we did for manifolds, the construction will be defined for each $\ve$--gpo.

\medskip
\noindent
{\sc Stable/unstable set of positive/negative $\ve$--gpo:} The {\em stable set}
of a positive $\ve$--gpo ${\un v}^+=\{\Psi_{\wh x_n}^{p^s_n,p^u_n}\}_{n\geq 0}$ is 
$$
\wh V^s[{\un v}^+]:=\{\wh y=(y_n)_{n\in\Z} \in \wh M: y_0\in V^s[\un v^+]\}.
$$
The {\em unstable set} of a negative $\ve$--gpo ${\un v}^-=\{\Psi_{\wh x_n}^{p^s_n,p^u_n}\}_{n\leq 0}$ is 
$$
\wh V^u[{\un v}^-]:=\{\wh y=(y_n)_{n\in\Z} \in \wh M: y_0\in V^u[\un v^-]\text{ and }
y_{n}=f_{\wh x_{n}}^{-1}(y_{n+1})\text{ for all }n<0\}.
$$
For an $\ve$--gpo $\un{v}=\{v_n\}_{n\in\Z}$, let
$\wh V^s[\un v]:=\wh V^s[\{v_n\}_{n\geq 0}]$
and $\wh V^u[\un v]:=\wh V^u[\{v_n\}_{n\leq 0}]$.

\medskip
In other words, $\wh V^s[\un v^+]$ is the set of all points in $\wh M$ for which the zeroth position
belongs to $V^s[v^+]$, while $\wh V^u[\un v^-]$ is the set of all points in $\wh M$ for which
the zeroth position belongs to $V^u[v^-]$ and {\em additionally} the negative positions
are uniquely determined by the inverse branches prescribed by $\un v^-$.
By Proposition \ref{Prop-stable-manifolds}(4), $\wh V^u[\un v^-]$ has diameter
$\leq 4p^u_0$. On the other hand, $\wh V^s[\un v^+]$ might have big diameter.
Indeed, two elements $\wh y,\wh z\in \wh V^s[\un v^+]$ can have
$y_{-1},z_{-1}$ generated by different inverse branches of $f$ and then, since the distance
in $\wh M$ is a supremum over all negative positions, the distance $d(\wh y,\wh z)$ can be large.
This is somewhat contradictory to calling $\wh V^s[\un v^+]$ a stable set, but we will
nevertheless use this terminology.
In any case, the coding we will construct will consider intersections of stable and unstable sets,
which in particular will have diameter $\leq 4p^u_0$.
The next proposition justifies our terminology.

\begin{lemma}\label{Lemma-stable-set}
The following holds for all $\varepsilon>0$ small enough.
\begin{enumerate}[{\rm (1)}]
\item If ${\un v}^+=\{\Psi_{\wh x_n}^{p^s_n,p^u_n}\}_{n\geq 0}$ is a positive $\ve$--gpo 
and $d=d_s(\wh x_0)$ then
$$
\wh V^s[\un v^+]=\vartheta^{-1}(V^s[\un v^+])=\left\{(y_n)_{n\in\Z}\in \wh M:
\begin{array}{c}y_0\in \Psi_{\wh x_0}(B^d[p^s_0]\times B^{m-d}[p^s_0])\text{ and}\\
y_n\in \Psi_{\wh x_n}(B[20Q(\wh x_n)]),\forall n\geq 0\end{array}
\right\}.
$$
\item If ${\un v}^-=\{\Psi_{\wh x_n}^{p^s_n,p^u_n}\}_{n\leq 0}$ is a negative $\ve$--gpo 
and $d=d_s(\wh x_0)$ then
$$
\wh V^u[\un v^-]=\left\{(y_n)_{n\in\Z}\in \wh M:
\begin{array}{c}y_0\in \Psi_{\wh x_0}(B^d[p^u_0]\times B^{m-d}[p^u_0])\text{ and}\\
y_n=f_{\wh x_n}^{-1}(y_{n+1})\in \Psi_{\wh x_n}(B[20Q(\wh x_n)]),\forall n<0\end{array}
\right\}.
$$
\end{enumerate}
\end{lemma}

This lemma is just a rewriting of the definitions
of $\wh V^s[\un v^+]$ and $\wh V^u[\un v^-]$ in terms of Proposition \ref{Prop-stable-manifolds}(3).
We just observe that the equality $y_n=f_{\wh x_n}^{-1}(y_{n+1})$ in part (2)
follows from Lemma \ref{Lemma-inverse-branches}.

\medskip
\noindent
{\sc Stable/unstable sets at $v$:} Given an $\ve$--double chart $v$, a {\em stable set at $v$}
is any $\wh V^s[\un v^+]$ where $\un v^+$ is a positive $\ve$--gpo with $v_0=v$. We let
$\wh{\mathfs M}^s(v)$ denote the set of all stable sets at $v$. Similarly,
an {\em unstable set at $v$} is any $\wh V^u[\un v^-]$ where $\un v^-$ is a negative $\ve$--gpo
with $v_0=v$, and we let $\wh{\mathfs M}^u(v)$ denote the set of all unstable sets at $v$. 

\medskip
Now we define graph transforms for stable/unstable sets. Let $v\overset{\ve}{\to}w$.

\medskip
\noindent
{\sc Graph transforms $\wh{\mathfs F}^s_{v,w},\wh{\mathfs F}^u_{v,w}$:}
The {\em graph transform} $\wh{\mathfs F}^s_{v,w}:\wh{\mathfs M}^s(w)\to\wh{\mathfs M}^s(v)$
is the map that sends $\wh V^s[\un w^+]\in\wh{\mathfs M}^s(w)$ to
$\wh V^s[\un v^+]\in\wh{\mathfs M}^s(v)$ where $\un v^+=\{v_n\}_{n\geq 0}$ is defined
by $v_0=v$ and $v_n=w_{n-1}$ for all $n>0$. Similarly, the {\em graph transform}
$\wh{\mathfs F}^u_{v,w}:\wh{\mathfs M}^u(v)\to\wh{\mathfs M}^u(w)$
is the map that sends $\wh V^u[\un v^-]\in\wh{\mathfs M}^u(v)$ to
$\wh V^u[\un w^-]\in\wh{\mathfs M}^u(w)$ where $\un w^-=\{w_n\}_{n\leq 0}$
is defined by $w_0=w$ and $w_n=v_{n+1}$ for all $n<0$. 

\medskip
The next result collects the main properties of the graph transforms and stable/unstable sets.

\begin{proposition}\label{Prop-stable-sets}
The following holds for all $\ve>0$ small enough.
\begin{enumerate}[{\rm (1)}]
\item {\sc Product structure:} If $\wh V^s/\wh V^u$ is a stable$/$unstable set at $v$
then $\wh V^s\cap \wh V^u$ consists of  a single element of $\wh M$.
\item {\sc Action under $\wh{\mathfs F}^{s/u}$:} Let $v\overset{\ve}{\to}w$. If
$\wh V^u\in\wh{\mathfs M}^u(v)$ then $\wh{\mathfs F}^u_{v,w}[\wh V^u]\subset \wh f(\wh V^u)$.
Similarly, if $\wh V^s\in\wh{\mathfs M}^s(w)$ then
$\wh{\mathfs F}^s_{v,w}[\wh V^s]\subset \wh f^{-1}(\wh V^s)$.
\item {\sc Action under $\wh f$:}
$$
\wh f\left(\wh V^s[\{v_n\}_{n\geq 0}]\right)\subset \wh V^s[\{v_n\}_{n\geq 1}]\text{ and }
\wh f^{-1}\left(\wh V^u[\{v_n\}_{n\leq 0}]\right)\subset \wh V^u[\{v_n\}_{n\leq -1}].
$$
\item {\sc Hyperbolicity:} If $(y_n)_{n\in\Z},(z_n)_{n\in\Z}\in \wh V^s[{\un v}^+]$ where
${\un v}^+=\{\Psi_{\wh x_n}^{p^s_n,p^u_n}\}_{n\geq 0}$, then
$$
d(y_n,z_n)\leq 4p^s_0e^{-\frac{\chi}{2}n}\text{ for all }n\geq 0.
$$
Hence, taking $\theta:=\max\left\{e^{-\frac{\chi}{2}},\tfrac{1}{2}\right\}$,
if $\wh y,\wh z\in \wh V^s[{\un v}^+]$ then $d(\wh f^n(\wh y),\wh f^n(\wh z))<\theta^n$
for all $n\geq 0$.
Similarly, if $(y_n)_{n\in\Z},(z_n)_{n\in\Z}\in \wh V^u[{\un v}^-]$ where
${\un v}^-=\{\Psi_{\wh x_n}^{p^s_n,p^u_n}\}_{n\leq 0}$, then
$$
d(y_n,z_n)\leq 4p^u_0e^{\frac{\chi}{2}n}\text{ for all }n\leq 0.
$$
Hence, if $\wh y,\wh z\in \wh V^u[{\un v}^-]$ then $d(\wh f^n(\wh y),\wh f^n(\wh z))<\theta^{-n}$
for all $n\leq 0$.
\item {\sc Disjointness:} Let $v=\Psi_{\wh x}^{p^s,p^u}$ and $w=\Psi_{\wh y}^{q^s,q^u}$
with $\wh x=\wh y$.
If $\wh V^s,\wh W^s$ are stable sets at $v,w$ then they are either disjoint or one contains the
other. The same holds for unstable sets $\wh V^u,\wh W^u$ at $v,w$.
\item Let $v\overset{\ve}{\to}w$. If $\wh V^u$ is an unstable set at $v$ then $\wh f(\wh V^u)$ intersects every
stable set at $w$ at a single element; if $\wh V^s$ is a stable set at $w$ then $\wh f^{-1}(\wh V^s)$
intersects every unstable set at $v$ at a single element.
\end{enumerate}
\end{proposition}

\begin{proof}
(1) Write $v=\Psi_{\wh x_0}^{p^s_0,p^u_0}$, $\wh V^s=\wh V^s[\{\Psi_{\wh x_n}^{p^s_n,p^u_n}\}_{n\geq 0}]$,
$\wh V^u=\wh V^u[\{\Psi_{\wh x_n}^{p^s_n,p^u_n}\}_{n\leq 0}]$. Let us see the conditions for which
$\wh y=(y_n)_{n\in\Z}$ belongs to $\wh V^s\cap\wh V^u$. By the definition of stable/unstable sets,
we must have that 
$y_0\in V^s[\{\Psi_{\wh x_n}^{p^s_n,p^u_n}\}_{n\geq 0}]\cap V^u[\{\Psi_{\wh x_n}^{p^s_n,p^u_n}\}_{n\leq 0}]$.
By Lemma \ref{Lemma-admissible-manifolds}, this latter intersection consists of a single element,
hence $y_0$ is uniquely defined. In particular, if $n\geq 0$ then each $y_n=f^n(y_0)$ is uniquely defined.
Finally, by Lemma \ref{Lemma-stable-set}(2), if $n\leq 0$ then
$y_n=(f_{\wh x_n}^{-1}\circ\cdots\circ f_{\wh x_{-1}}^{-1})(y_0)$ is uniquely defined.
These conditions uniquely characterize $\wh y$.

\medskip
\noindent
(2) We prove the statement for $\wh{\mathfs F}^{u}_{v,w}$ (the proof for  $\wh{\mathfs F}^s_{v,w}$
is simpler because there is no need to control inverse branches). Fix a negative $\ve$--gpo
$\un v^-=\{v_n\}_{n\leq 0}$ with $v_0=v$, and let $\un w^-=\{w_n\}_{n\leq 0}$ defined
by $w_0=w$ and $w_n=v_{n+1}$ for $n<0$. Let $V^u=V^u[\un v^-]$, $\wh V^u=\wh V^u[\un v^-]$, 
$W^u=V^u[\un w^-]$ and $\wh W^u=\wh V^u[\un w^-]$.
By definition, $\wh{\mathfs F}^u_{v,w}(\wh V^u)=\wh W^u$ and so we need to show
that $\wh f(\wh V^u)\supset \wh W^u$. By Lemma \ref{Lemma-transform-manifolds},
we have $\mathfs F^u_{v,w}(V^u)=W^u$. Since $\mathfs F^u_{v,w}$ is a restriction
of $f$ to a subset, $f(V^u)\supset W^u$. Writing
$w_n=\Psi_{\wh x_n}^{p^s_n,p^u_n}$, the definition of unstable set gives that
$$
\wh V^u=
\left\{\wh y=(y_n)_{n\in \Z}\in\wh M: y_0\in V^u\text{ and }y_n=f^{-1}_{\wh x_{n-1}}(y_{n+1})\text{ for all }n<0\right\}
$$
and so
\begin{align*}
&\, \wh f(\wh V^u)=\left\{\wh z=(z_n)_{n\in \Z}\in\wh M: z_0\in f(V^u)\text{ and }z_n=f^{-1}_{\wh x_n}(z_{n+1})\text{ for all }n<0\right\}\\
&\supset\left\{\wh z=(z_n)_{n\in \Z}\in\wh M: z_0\in W^u\text{ and }z_n=f^{-1}_{\wh x_n}(z_{n+1})\text{ for all }n<0\right\}
=\wh W^u.
\end{align*}

\medskip
\noindent
(3) By the definition of stable sets and Proposition \ref{Prop-stable-manifolds}(2),
\begin{align*}
&\,\wh f\left(\wh V^s[\{v_n\}_{n\geq 0}]\right)=\wh f\left(\left\{\wh y\in\wh M:y_0\in V^s[\{v_n\}_{n\geq 0}]\right\}\right)\\
&=\left\{\wh y\in\wh M:y_0\in f(V^s[\{v_n\}_{n\geq 0}])\right\}\subset
\left\{\wh y\in\wh M:y_0\in V^s[\{v_n\}_{n\geq 1}]\right\}\\
&=\wh V^s[\{v_n\}_{n\geq 1}].
\end{align*}
Similarly,
\begin{align*}
&\,\wh f^{-1}\left(\wh V^u[\{v_n\}_{n\leq 0}]\right)\\
&=\wh f^{-1}\left(\left\{\wh y\in\wh M:y_0\in V^u[\{v_n\}_{n\leq 0}]\text{ and }
y_{n}=f_{\wh x_{n}}^{-1}(y_{n+1}),\forall n<0\right\}\right)\\
&=\left\{\wh y\in\wh M:y_0\in f_{\wh x_{-1}}^{-1}(V^u[\{v_n\}_{n\leq 0}])\text{ and }
y_{n}=f_{\wh x_{n-1}}^{-1}(y_{n+1}),\forall n<0\right\}\\
&\subset\left\{\wh y\in\wh M:y_0\in V^u[\{v_n\}_{n\leq -1}]\text{ and }
y_{n}=f_{\wh x_{n-1}}^{-1}(y_{n+1}),\forall n<0\right\}\\
&=\wh V^u[\{v_n\}_{n\leq -1}].
\end{align*}

\medskip
\noindent
(4) The estimates of $d(y_n,z_n)$ follow directly from Proposition \ref{Prop-stable-manifolds}(4),
so we focus on the estimates of $d(\wh f^n(\wh y),\wh f^n(\wh z))$. We
let $\theta:=\max\left\{e^{-\frac{\chi}{2}},\tfrac{1}{2}\right\}$. Start with 
$\wh y,\wh z\in\wh V^s[\un v^+]$, and fix $n\geq 0$. Recalling that the diameter of $M$ is smaller than one
and using that $\wh f^n(\wh y)=\{y_{n+k}\}_{k\in\Z}$ and $\wh f^n(\wh z)=\{z_{n+k}\}_{k\in\Z}$, we have: 
\begin{enumerate}[$\circ$]
\item If $-n\leq k\leq 0$ then $d(y_{n+k},z_{n+k})\leq 4p^s_0 e^{-\frac{\chi}{2}(n+k)}<\theta^{n+k}$.
This implies that $2^kd(y_{n+k},z_{n+k})<\theta^{-k}\theta^{n+k}=\theta^n$. 
\item If $k<-n$ then $2^kd(y_{n+k},z_{n+k})<2^k\leq \theta^{-k}<\theta^n$.
\end{enumerate}
Hence $d(\wh f^n(\wh y),\wh f^n(\wh z))<\theta^n$. The other estimate is simpler.
Let $\wh y,\wh z\in\wh V^u[\un v^-]$ and fix $n\leq 0$. For $k\leq 0$ we have
$d(y_{n+k},z_{n+k})\leq 4p^u_0 e^{\frac{\chi}{2}(n+k)}<\theta^{-n-k}$ and so
$2^kd(y_{n+k},z_{n+k})<\theta^{-k}\theta^{-n-k}=\theta^{-n-2k}\leq\theta^{-n}$.
Hence $d(\wh f^n(\wh y),\wh f^n(\wh z))<\theta^{-n}$.

\medskip
\noindent
(5) The proof is motivated by \cite[Proposition 6.4]{Sarig-JAMS}, with the modifications needed
to deal with inverse branches. We start dealing with unstable sets. Let $\wh V^u=\wh V^u[\un v^-]$
where $\un v^-=\{\Psi_{\wh x_n}^{p^s_n,p^u_n}\}_{n\leq 0}$ and
$\wh W^u=\wh V^u[\un w^-]$ where $\un w^-=\{\Psi_{\wh y_n}^{q^s_n,q^u_n}\}_{n\leq 0}$,
and assume that $\wh x_0=\wh y_0$ and $\wh V^u\cap\wh W^u\neq\emptyset$.
We will show that if $q^u_0\leq p^u_0$ then $\wh W^u\subset\wh V^u$. To do this,
we first prove that for $n<0$ the inverse branches $f_{\wh x_n}^{-1}$ and $f_{\wh y_n}^{-1}$ coincide
in the domains of interest.

\medskip
\noindent
{\sc Claim 1:} For every $n\leq 0$ it holds $\vt_n[\wh W^u]\subset E_{\wh x_{n-1}}$
and $\vt_{n-1}[\wh W^u]\subset f_{\wh x_{n-1}}^{-1}(E_{\wh x_{n-1}})$.

\begin{proof}[Proof of Claim $1$.] Fix $\wh z=(z_k)_{k\in\Z}\in\wh V^u\cap\wh W^u$ and $n\leq 0$.
Recall, by the proof of Lemma \ref{Lemma-inverse-branches}, that $Q(\wh x_n)<\ve \mathfrak r(\wh x_{n-1})$. 
Using that $\wh z\in \wh V^u$, Lemma \ref{Lemma-stable-set} implies that
$$
z_n\in B(\vt_0[\wh x_n],40Q(\wh x_n))\subset B(\vt_1[\wh x_{n-1}],50Q(\wh x_n)),
$$
since by the overlap assumption we have $d(\vt_0[\wh x_n],\vt_1[\wh x_{n-1}])\ll p^s_n\wedge p^u_n <Q(\wh x_n)$.
Now let $\wh t=(t_k)_{k\in\Z}\in \wh W^u$ be arbitrary. Using part (4) and the inequality
$p^u_0\leq e^{-\ve n}p^u_n$, if $\ve>0$ is small then
$d(z_n,t_n)\leq 4q^u_0e^{\frac{\chi}{2}n}\leq 4q^u_0e^{\ve n}\leq 4\left(\tfrac{q^u_0}{p^u_0}\right)p^u_n\leq 4p^u_n$
and so $\vt_n[\wh W^u]\subset B(\vt_1[\wh x_{n-1}],50Q(\wh x_n)+4p^u_n)\subset E_{\wh x_{n-1}}$,
where in the last passage we used that
$50Q(\wh x_n)+4p^u_n<60Q(\wh x_n)<60\ve \mathfrak r(\wh x_{n-1})<2\mathfrak r(\wh x_{n-1})$.

The second inclusion is easier. We have $z_{n-1}\in B(\vt_0[\wh x_{n-1}],40Q(\wh x_{n-1}))$
and $d(z_{n-1},t_{n-1})\leq 4p^u_{n-1}$, thus $\vt_{n-1}[\wh W^u]\subset B(\vt_0[\wh x_{n-1}],50Q(\wh x_{n-1}))$.
Now, (A6) implies
$f_{\wh x_{n-1}}^{-1}(E_{\wh x_{n-1}})\supset B(\vt_0[\wh x_{n-1}],
2d(\vt_0[\wh x_{n-1}],\mathfs S)^a\mathfrak r(\wh x_{n-1}))$. The radius of this latter ball is
$>2\rho(\wh x_{n-1})^{2a}>50Q(\wh x_{n-1})$, therefore
$\vt_{n-1}[\wh W^u]\subset f_{\wh x_{n-1}}^{-1}(E_{\wh x_{n-1}})$.
\end{proof}

Letting $d=d_s(\wh x_0)$, Claim 1 implies that 
$$
\wh W^u=\left\{(z_n)_{n\in\Z}\in \wh M:
\begin{array}{c}z_0\in \Psi_{\wh x_0}(B^d[q^u_0]\times B^{m-d}[q^u_0])\text{ and}\\
z_n=f_{\wh x_n}^{-1}(z_{n+1})\in \Psi_{\wh y_n}(B[20Q(\wh y_n)]),\forall n<0\end{array}
\right\},
$$
i.e. we can substitute $f_{\wh y_n}^{-1}$ by $f_{\wh x_n}^{-1}$ in the inverse branches.
This allows us to employ the same methods of \cite[Proposition 6.4]{Sarig-JAMS}, summarized
in our context as follows. Below, ``$n$ is small enough'' means that $-n$ is large enough.

\medskip
\noindent
{\sc Claim 2:} If $n$ is small enough then
$\vt_n[\wh V^u]\subset \Psi_{\wh x_n}\left(B^d[\tfrac{1}{2}p^u_n]\times B^{m-d}[\tfrac{1}{2}p^u_n]\right)$.

\medskip
The proof is made exactly as in \cite[Claim 1 in pp. 420]{Sarig-JAMS}.

\medskip
\noindent
{\sc Claim 3:} If $n$ is small enough then
$\vt_n[\wh W^u]\subset \Psi_{\wh x_n}\left(B^d[p^u_n]\times B^{m-d}[p^u_n]\right)$.

\medskip
The proof of this claim is made exactly as in \cite[Claim 2 in pp. 420]{Sarig-JAMS}.
Finally, to prove that $\wh W^u\subset\wh V^u$, we proceed as in
\cite[Claim 3 in pp. 102]{Ben-Ovadia-2019}. The proof for stable sets is simpler and can
be made as in \cite[Proposition 6.4]{Sarig-JAMS} or \cite[Claim 3 in pp. 102]{Ben-Ovadia-2019},
since in this case we do not have difficulties with inverse branches.

\medskip
\noindent
(6) Since $\wh f$ is invertible, it is enough to prove the first statement.
Let $v=\Psi_{\wh x}^{p^s,p^u},w=\Psi_{\wh y}^{q^s,q^u}$, $\wh V^u=\wh V^u[\un v^-]$ an
unstable set at $v$, and $\wh V^s=\wh V^s[\un w^+]$ a stable set at $w$.
Let $V^u=V^u[\un v^-]$ and $V^s=V^s[\un w^+]$.
If $(z_n)_{n\in\Z}\in \wh f(\wh V^u)\cap \wh V^s$ then $z_0\in f(V^u)\cap V^s$.
By Proposition \ref{Prop-graph-transform}(3), $z_0$ is uniquely defined.
In particular, $z_n$ is uniquely defined for all $n\geq 0$.
Furthermore, $z_{-1}=f_{\wh x}^{-1}(z_0)$ is uniquely defined and therefore,
by Lemma \ref{Lemma-stable-set}(2), $z_n$ is uniquely defined for each $n<-1$. 
This show that $\wh f(\wh V^u)\cap \wh V^s$ consists of a single element.
\end{proof}

\medskip
\noindent
{\sc Shadowing:} We say that an $\ve$--gpo $\{\Psi_{\wh x_n}^{p^s_n,p^u_n}\}_{n\in\Z}$ {\em shadows}
a point $\wh y=(y_n)_{n\in\Z}\in\wh M$ if $y_n\in \Psi_{\wh x_n}(B[20Q(\wh x_n)])$ for all $n\in\Z$.

\begin{lemma}\label{Lemma-shadowing}
Every $\ve$--gpo shadows a unique point.
\end{lemma}

The proof is the same as \cite[Theorem 4.2]{Lima-Sarig}, with the shadowed
point of an $\ve$--gpo $\un v$ being the intersection $\wh V^s[{\un v}]\cap\wh V^u[\un v]$.
In particular, by Lemma \ref{Lemma-admissible-manifolds} we obtain that
shadowing implies the stronger inclusion
$y_n\in \Psi_{\wh x_n}(B[p^s_n\wedge p^u_n])$ for all $n\in\Z$.

\subsection{Hyperbolicity along stable/unstable sets}

In this section, we establish some estimates that will play an essential role in the proof
of Theorem \ref{Thm-inverse}.
%, more specifically when proving that a point that is shadowed by a recurrent sequence
%of $\ve$--double charts belongs to ${\rm NUH}$. 
%The idea for this is to explore the hyperbolicity along stable/unstable manifolds,
%which will be made by two main ingredients: the results of this section and the improvement lemma
%(see Section \ref{Section-improvement}).
Given a positive $\ve$--gpo $\un v^+$, let
$V^s/\wh V^s$ be its $s$--admissible manifold/set.
Below, we use the notation $T\wh V^s=\wh{TV^s}$ to represent the tangent bundle
of $V^s$ as a subset of $\wh{TM}$, i.e. for every $\wh y\in\wh V^s$ we let $T_{\wh y}\wh V^s$
be the tangent subspace $T_{\vt[\wh y]}V^s$ as a subset of $\wh{TM}_{\wh y}$.

\begin{proposition}\label{Prop-hyperbolicity-manifolds}
The following holds for $\ve>0$ small enough.
\begin{enumerate}[{\rm (1)}]
\item If $\wh V^s=\wh V^s[\un v^+]$ where $\un v^+=\{\Psi_{\wh x_n}^{p^s_n,p^u_n}\}_{n\geq 0}$ is a positive $\ve$--gpo,
then for every $v\in T\wh V^s$ with $\norm{v}=1$ and every $n\geq 0$ it holds
$$
\|\wh{df}^{(n)}v\|\leq 8\norm{C(\wh x_0)^{-1}}e^{-\frac{\chi}{2}n}.
$$
Similarly, if $\wh V^u=\wh V^u[\un v^-]$ where $\un v^-=\{\Psi_{\wh x_n}^{p^s_n,p^u_n}\}_{n\leq 0}$
is a negative $\ve$--gpo, then for every $v\in T\wh V^u$ with $\norm{v}=1$ and every $n\geq 0$ it holds
$$
\|\wh{df}^{(-n)}v\|\leq 8\norm{C(\wh x_0)^{-1}}e^{-\frac{\chi}{2}n}.
$$
\item If $\un v$ is an $\ve$--gpo and $\{\wh y\}=\wh V^s[\un v]\cap \wh V^u[\un v]$,
then for all nonzero $v\in T_{\wh y}\wh V^s$ it holds
$$
\liminf_{n\to+\infty}\tfrac{1}{n}\log\|\wh{df}^{(-n)}v\|>0
$$
and for all nonzero $v\in T_{\wh y}\wh V^u$ it holds 
$$
\liminf_{n\to+\infty}\tfrac{1}{n}\log\|\wh{df}^{(n)}v\|>0.
$$
\end{enumerate}
\end{proposition}

%In view of part (1), whenever $v\in T\wh V^s$ we let $\vertiii{v}$ be defined
%as in $E^s$, even if this sum is infinite. We do the same for $v\in T\wh V^u$.
 
\begin{proof}
All statements are about the cocycle $(\wh{df}^{(n)})_{n\in\Z}$, therefore
we can prove the result in the domains of the Pesin charts and then push
it to the respective tangent spaces. As we will see, the proof is symmetric
and so we just need to prove the estimates for the stable manifolds.

\medskip
\noindent
(1) Let $d=d_s(\wh x_0)$, and let $G_k$ be the representing function of
$V^s[\{\Psi_{\wh x_n}^{p^s_n,p^u_n}\}_{n\geq k}]$. Fix $v\in T_{\wh y}\wh V^s$ with $\|v\|=1$, and
let $z=(v_0,G_0(v_0))$ with $\vt[\wh y]=y=\Psi_{\wh x_0}(z)$. Thus
$$
v=(d\Psi_{\wh x_0})_{z}\begin{bmatrix} w \\ (dG_0)_{v_0}w\end{bmatrix}
$$
for some $w\in \R^d$. In particular,
$\norm{w}\leq \norm{\begin{bmatrix} w \\ (dG_0)_{v_0}w\end{bmatrix}}\leq 2\norm{C(\wh x_0)^{-1}}$.
Using that
$\Psi_{\wh x_n}\circ  F_{\wh x_{n-1},\wh x_n}\circ\cdots\circ F_{\wh x_0,\wh x_1}=f^n\circ \Psi_{\wh x_0}$,
we have
$$
\wh{df}^{(n)}v=[d(\Psi_{\wh x_n}\circ  F_{\wh x_{n-1},\wh x_n}\circ\cdots\circ F_{\wh x_0,\wh x_1})]_z
\begin{bmatrix} w \\ (dG_0)_{v_0}w\end{bmatrix}
$$
and so
$$
\|\wh{df}^{(n)}v\|\leq 2\norm{[d( F_{\wh x_{n-1},\wh x_n}\circ\cdots\circ F_{\wh x_0,\wh x_1})]_z
\begin{bmatrix} w \\ (dG_0)_{v_0}w\end{bmatrix}}.
$$
Define inductively $z_k=(v_k,G_k(v_k))$ and $w_k\in\R^d$ by
$z_0=z$, $z_k=F_{\wh x_{k-1},\wh x_k}(z_{k-1})$, $w_0=w$ and
$$
\begin{bmatrix} w_k \\ (dG_k)_{v_k}w_k\end{bmatrix}=
(dF_{\wh x_{k-1},\wh x_k})_{z_{k-1}}\begin{bmatrix} w_{k-1} \\ (dG_{k-1})_{v_{k-1}}w_{k-1}\end{bmatrix}.
$$
In particular, $z_k\in B^d[p^s_k]\times B^{m-d}[p^s_k]\subset B[2p^s_k]$.
We want to bound $\norm{\begin{bmatrix} w_n \\ (dG_n)_{v_n}w_n\end{bmatrix}}$.
By Lemma \ref{App-Lemma-introductory}(1),
we have $\norm{\begin{bmatrix} w_k \\ (dG_k)_{v_k}w_k\end{bmatrix}}\leq 2\norm{w_k}$,
hence it is enough to estimate $\norm{w_n}$. Write
$F_{\wh x_{k-1},\wh x_k}=D(\wh x_{k-1})+H^+_{k-1}$,
where $H^+_{k-1}$ satisfies Theorem \ref{Thm-non-linear-Pesin-2}. Then
\begin{align*}
&\, \norm{w_k}\leq \norm{D_s(\wh x_{k-1})w_{k-1}}+\norm{(dH^+_{k-1})_{z_{k-1}}\begin{bmatrix} w_{k-1} \\ (dG_{k-1})_{v_{k-1}}w_{k-1}\end{bmatrix}}\\
&\leq \left(\norm{D_s(\wh x_{k-1})}+2\norm{(dH^+_{k-1})_{z_{k-1}}}\right)\norm{w_{k-1}}
\end{align*}
and so, by Lemma \ref{Lemma-linear-reduction}(2) and Theorem \ref{Thm-non-linear-Pesin-2}(b), we
obtain that 
\begin{equation}\label{estimate-w_k}
\norm{w_k}\leq \left[e^{-\chi}+4\ve(p^s_{k-1})^{\beta/3}\right]\norm{w_{k-1}}\leq e^{-\frac{\chi}{2}}\norm{w_{k-1}},
\end{equation}
where in the last passage we used that $\ve>0$ is small enough.
Therefore, we conclude that
$\norm{w_n}\leq e^{-\frac{\chi}{2}n}\norm{w_0}\leq 2\norm{C(\wh x_0)^{-1}}e^{-\frac{\chi}{2}n}$
and so $\|\wh{df}^{(n)}v\|\leq 8\norm{C(\wh x_0)^{-1}}e^{-\frac{\chi}{2}n}$.

\medskip
\noindent
(2) Write $\un v=\{v_n\}_{n\in\Z}=\{\Psi_{\wh x_n}^{p^s_n,p^u_n}\}_{n\in\Z}$ with $\eta_n=p^s_n\wedge p^u_n$,
and let $\wh y\in \wh V^s[\un v]\cap \wh V^u[\un v]$. Letting
$\wh V^s_n=\wh V^s[\{v_k\}_{k\geq n}]$ and $\wh V^u_n=\wh V^u[\{v_k\}_{k\leq n}]$,
we have $\{\wh f^n(\wh y)\}=\wh V^s_n\cap \wh V^u_n$.
% and, in particular, $y_0=y$. 
Fix a nonzero vector $v\in T_{\wh y}\wh V^s$ and $n\geq 0$.
Then $w=\wh{df}^{(-n)}v\in T_{\wh f^{-n}(\wh y)}\wh V^s_{-n}$.
By part (1), we have
$\norm{v}=\|\wh{df}^{(n)}w\|\leq 8\norm{C(\wh x_{-n})^{-1}}e^{-\frac{\chi}{2}n}\norm{w}$
and so $\|\wh{df}^{(-n)}v\|\geq \tfrac{1}{8}\norm{v}e^{\frac{\chi}{2}n}\norm{C(\wh x_{-n})^{-1}}^{-1}$.
Since $\norm{C(\wh x_{-n})^{-1}}^{-1}\geq \eta_{-n}^{\beta/48}\geq (e^{-\ve n}\eta_0)^{\beta/48}$, we get that
$$
\|\wh{df}^{(-n)}v\|\geq\tfrac{1}{8}\eta_0^{\beta/48}\norm{v}e^{\left(\frac{\chi}{2}-\frac{\beta\ve}{48}\right)n}
$$
and so $\displaystyle\liminf_{n\to+\infty}\tfrac{1}{n}\log\|\wh{df}^{(-n)}v\|\geq 
\tfrac{\chi}{2}-\tfrac{\beta\ve}{48}>0$
if $\ve>0$ is small enough.
\end{proof}

Observe that part (1) can be improved if we have an estimate that is better than the one
provided by Theorem \ref{Thm-non-linear-Pesin-2}(b), which in general is not true for 
$\ve$--gpo's. Nevertheless, if we apply the method to a real orbit, then we can apply 
Theorem \ref{Thm-non-linear-Pesin} and, the smaller $\norm{z_{k-1}}$ is,
the  closer we get to a contraction of rate $e^{-\chi}$. 
This is the content of the next result, 
which is the claim in \cite[pp. 85]{Ben-Ovadia-2019}, adapted to our context.
Fix $\wh x\in{\rm NUH}^*$, and let $\wh x_n=\wh f^n(\wh x)$, $p^s_n=q^s(\wh x_n)$
and $p^u_n=q^u(\wh x_n)$. We know that $\{\Psi_{\wh x_n}^{p^s_n,p^u_n}\}_{n\in\Z}$
is an $\ve$--gpo. Fix $\chi'\in (0,\chi)$, let $\delta=(e^{-\chi'}-e^{-\chi})^{3/\beta}<1$ and 
consider the sequence $\un v=\{\Psi_{\wh x_n}^{\delta p^s_n,\delta p^u_n}\}_{n\in\Z}$.
In practice, $\delta$ diminishes the domains of the Pesin charts.
Observe that $\un v$ is {\em not necessarily} an $\ve$--gpo, since the parameters
$\delta p^s_n,\delta p^u_n$ might not satisfy  (GPO2).
Nevertheless, they do satisfy condition (GT1) in the Appendix,
so we can define graph transforms. Hence, although $\un v$ is not an $\ve$--gpo,
this sequence still defines invariant sets, which we will denote by $\wh V^s_{\chi'}=\wh V^s_{\chi'}[\un v^+]$
and $\wh V^u_{\chi'}=\wh V^u_{\chi'}[\un v^-]$. Now, since we diminished the domains of the Pesin charts
then the estimates used in the proof of Proposition \ref{Prop-hyperbolicity-manifolds}(1)
can be improved to give a rate of contraction of the order of $e^{-\chi'}$.

\begin{corollary}\label{Corollary-better-hyperbolicity}
Let $\wh x\in{\rm NUH}^*$. For a fixed $\chi'\in (0,\chi)$,
let $\un v$ be as above.
If $\wh V^s_{\chi'}=\wh V^s_{\chi'}[\un v^+]$ then for every $v\in T\wh V^s_{\chi'}$ with $\norm{v}=1$
and every $n\geq 0$ it holds
$$
\|\wh{df}^{(n)}v\|\leq 8\norm{C(\wh x)^{-1}}e^{-\chi' n}.
$$
Similarly, if $\wh V^u_{\chi'}=\wh V^u_{\chi'}[\un v^-]$ then for every $v\in T\wh V^u_{\chi'}$ with
$\norm{v}=1$ and every $n\geq 0$ it holds
$$
\|\wh{df}^{(-n)}v\|\leq 8\norm{C(\wh x)^{-1}}e^{-\chi' n}.
$$
\end{corollary} 

\begin{proof}
Proceed exactly as in the proof of Proposition \ref{Prop-hyperbolicity-manifolds}(1), except
at inequalities (\ref{estimate-w_k}) which, by Theorem \ref{Thm-non-linear-Pesin}(3)(b),
take the stronger form
$$
\norm{w_k}\leq \left[e^{-\chi}+4\ve(\delta p^s_{k-1})^{\beta/3}\right]\norm{w_{k-1}}
\leq\left[e^{-\chi}+\delta^{\beta/3}\right]\norm{w_{k-1}}=e^{-\chi'}\norm{w_{k-1}},
$$
thus establishing the result.
\end{proof}

Observe that the closer $\chi'$ is to $\chi$, the better the estimate gets
but the domains where the estimates hold will shrink to a point.

\section{Coarse graining}\label{Section-coarse-graining}

We now construct a countable family of $\ve$--double charts that defines
a topological Markov shift that shadows all relevant orbits. Recall the definitions of the
sets ${\rm NUH}$, ${\rm NUH}^*$, and ${\rm NUH}^\#$.

\begin{theorem}\label{Thm-coarse-graining}
For all $\ve>0$ sufficiently small, there exists a countable family $\mathfs A$ of $\ve$--double charts
with the following properties:
\begin{enumerate}[{\rm (1)}]
\item {\sc Discreteness}: For all $t>0$, the set $\{\Psi_{\wh x}^{p^s,p^u}\in\mathfs A:p^s,p^u>t\}$ is finite.
\item {\sc Sufficiency:} If $\wh x\in {\rm NUH}^\#$ then there is a recurrent
sequence $\un v\in{\mathfs A}^{\Z}$ that shadows $x$.
%; if $\wh x\in {\rm NUH}^\#$ then there is a recurrent sequence $\un v\in{\mathfs A}^{\Z}$ that shadows $x$.
\item {\sc Relevance:} For all $v\in \mathfs A$ there is an $\ve$--gpo $\un{v}\in\mathfs A^\Z$
with $v_0=v$ that shadows a point in ${\rm NUH}^\#$.
\end{enumerate}
\end{theorem}

Above, $(v_n)_{n\in\Z}\in{\mathfs A}^{\Z}$ is called {\em recurrent}
if there are $v,w\in\mathfs A$ s.t. $v_n=v$ for infinitely many $n>0$ and
$v_n=w$ for infinitely many $n<0$.
Parts (1) and (3) will be crucial to prove the inverse theorem (Theorem \ref{Thm-inverse}).
Since ${\rm NUH}^\#$ carries all $f$--adapted $\chi$--hyperbolic measures
(see Lemma \ref{Lemma-adaptedness}), part (2) says that the $\ve$--gpo's in $\mathfs A$
shadow a.e. point with respect to every $f$--adapted $\chi$--hyperbolic measure.
%see Section \ref{Section-Temperedness}.

\begin{proof}
We follow the same strategy of \cite[Theorem 5.1]{Lima-Matheus}.
For $t>0$, let $M_t=\{x\in M: d(x,\mathfs S)\geq t\}$.
Since $M$ has finite diameter (remember we are even assuming it is smaller than one), each $M_t$
is precompact.
Let $\N_0=\N\cup\{0\}$. Fix a countable open cover $\mathfs P=\{\mathfrak B_i\}_{i\in\N_0}$
of $M\backslash\mathfs S$ s.t.:
\begin{enumerate}[$\circ$]
\item $\mathfrak B_i:=\mathfrak B_{z_i}=B(z_i,2\mathfrak d(z_i))$ for some $z_i\in M$.
\item For every $t>0$, $\{\mathfrak B\in\mathfs P:\mathfrak B\cap M_t\neq\emptyset\}$ is finite.
\end{enumerate}
%These are the domains of the parallelization of $M$ we will consider in the construction to come.

\medskip
Let $X:=M^3\times {\rm GL}(m,\R)^3\times (0,1]$.
For $\wh x\in{\rm NUH}^\#$, let
$\Gamma(\wh x)=(\un{\wh x},\un C,\un Q)\in X$ with
\begin{align*}
\un{\wh x}=(\vt_{-1}[\wh x],\vt_0[\wh x],\vt_1[\wh x]),\ \un C=(C(\wh f^{-1}(\wh x)),C(\wh x),C(\wh f(\wh x))),
\ \un Q=Q(\wh x).
\end{align*}
Let $Y=\{\Gamma(\wh x):\wh x\in{\rm NUH}_\chi^\#\}$. We want to construct a countable dense subset
of $Y$. Since the maps $\wh x\mapsto C(\wh x),Q(\wh x)$ are usually just measurable,
we apply a precompactness argument.
For each triple of vectors $\un{k}=(k_{-1},k_0,k_1)$, $\un{\ell}=(\ell_{-1},\ell_0,\ell_1)$,
$\un a=(a_{-1},a_0,a_1)\in\N_0^3$, $b\in\N_0$ and $d\in \{0,1,\ldots,m\}$, define
$$
Y_{\un k,\un \ell,\un a,b,d}:=\left\{\Gamma(\wh x)\in Y:
\begin{array}{cl}
e^{-k_i-1}\leq d(\vt_i[\wh x],\mathfs S)< e^{-k_i},& -1\leq i\leq 1\\
e^{\ell_i}\leq\|C(\wh f^i(\wh x))^{-1}\|<e^{\ell_i+1},&-1\leq i\leq 1\\
\vt_i[\wh x]\in \mathfrak B_{a_i},&-1\leq i\leq 1\\
e^{-b-1}\leq Q(\wh x)< e^{-b}&\\
d_s(\wh x)=d&\\
\end{array}
\right\}.
$$
Clearly, any dependence on $d\in \{0,1,\ldots,m\}$ is finite. To avoid 
an extra summand in the calculations, in the sequel we will omit the dependence on $d$.

\medskip
\noindent
{\sc Claim 1:} $Y=\bigcup_{\un k,\un\ell,\un a\in\N_0^3\atop{b,d\in\N_0}}Y_{\un k,\un\ell,\un a,b,d}$, and each
$Y_{\un k,\un\ell,\un a,b,d}$ is precompact in $X$.

\medskip
\noindent
{\em Proof of Claim $1$.}
The first statement is clear. We focus on the second.
Fix $\un k,\un\ell,\un a\in \N_0^3$, $b,d\in\N_0$. Take $\Gamma(\wh x)\in Y_{\un k,\un\ell,\un a,b,d}$. Then
$$
\un{\wh x}\in M_{e^{-k_{-1}-1}}\times M_{e^{-k_0-1}}\times M_{e^{-k_1-1}},$$
a precompact subset of $M^3$.
For $|i|\leq 1$, $C(\wh f^i(\wh x))$ is an element of ${\rm GL}(m,\R)$ with norm $\leq 1$ and
inverse norm $\leq e^{\ell_i+1}$, hence it belongs to a compact subset of ${\rm GL}(m,\R)$.
This guarantees that $\un C$ belongs to a compact subset of ${\rm GL}(m,\R)^3$. Also,
$\un Q\in [e^{-b-1},1]$, a compact subinterval of $(0,1]$. Since the product of precompact sets
is precompact, the claim is proved.

\medskip
Let $j\geq 0$. By Claim 1, there exists a finite set
$Y_{\un k,\un\ell,\un a,b,d}(j)\subset Y_{\un k,\un\ell,\un a,b,d}$
s.t. for every $\Gamma(\wh x)\in Y_{\un k,\un\ell,\un a,b,d}$
there exists $\Gamma(\wh y)\in Y_{\un k,\un\ell,\un a,b,d}(j)$
s.t.:
\begin{enumerate}[{\rm (a)}]
\item $ d(\vt_i[\wh x],\vt_i[\wh y])+\norm{\widetilde{C(\wh f^i(\wh x))}-\widetilde{C(\wh f^i(\wh y))}}<e^{-8(j+2)}$
for $-1\leq i\leq 1$.
\item $\tfrac{Q(\wh x)}{Q(\wh y)}=e^{\pm \ve/3}$.
\end{enumerate}

\medskip
\noindent
{\sc The alphabet $\mathfs A$:} Let $\mathfs A$ be the countable family of $\Psi_{\wh x}^{p^s,p^u}$ s.t.:
\begin{enumerate}[i i)]
\item[(CG1)] $\Gamma(\wh x)\in Y_{\un k,\un\ell,\un a,b,d}(j)$ for some
$(\un k,\un\ell,\un a,b,d,j)\in\N_0^3\times\N_0^3\times\N_0^3\times \N_0\times \N_0\times \N_0$.
\item[(CG2)] $0<p^s,p^u\leq \delta_\ve Q(\wh x)$ and $p^s,p^u\in I_\ve$.
\item[(CG3)] $e^{-j-2}\leq p^s\wedge p^u\leq e^{-j+2}$.
\end{enumerate}

\medskip
\noindent
{\em Proof of discreteness.}
By Lemma \ref{Lemma-linear-reduction}(4), we have:
\begin{align}\label{inequality-C}
\|C(\wh f^{\pm 1}(\wh x))^{-1}\|\leq  \rho(\wh x)^{-2a}(1+e^{2\chi})^{1/2}\|C(\wh x)^{-1}\|.
\end{align}
Fix $t>0$, and let $\Psi_{\wh x}^{p^s,p^u}\in\mathfs A$ with $p^s,p^u>t$.
Note that $\rho(\wh x)>\rho(\wh x)^{2a}>Q(\wh x)>p^s,p^u>t$.
If $\Gamma(\wh x)\in Y_{\un k,\un\ell,\un a,b,d}(j)$ then:
\begin{enumerate}[$\circ$]
\item Finiteness of $\un k$: for $|i|\leq 1$, $e^{-k_i}> d(\vt_i[\wh x],\mathfs S)\geq\rho(\wh x)>t$, hence $k_i< |\log t|$.
\item Finiteness of $\un\ell$: we have $e^{\ell_0}\leq \|C(\wh x)^{-1}\|<Q(\wh x)^{-1}<t^{-1}$,
hence $\ell_0<|\log t|$. For $i=\pm 1$, inequality (\ref{inequality-C}) implies that
$$
e^{\ell_i}\leq \|C(\wh f^{i}(\wh x))^{-1}\|<t^{-1}(1+e^{2\chi})^{1/2}t^{-1}<e^{\chi+1}t^{-2},
$$
hence $\ell_i<\chi+1+2|\log t|=:T_t$, which is bigger than $|\log t|$.
\item Finiteness of $b$: $e^{-b}>Q(\wh x)>t$, hence $b<|\log t|$.
\item Finiteness of $j$: $t<p^s\wedge p^u\leq e^{-j+2}$, hence $j\leq |\log t|+2$.
\item Finiteness of $\un a$: $\vt_i[\wh x]\in \mathfrak B_{a_i}\cap M_t$,
hence $\mathfrak B_{a_i}$ belongs to the finite set
$\{\mathfrak B\in\mathfs P:\mathfrak B\cap M_t\neq\emptyset\}$.
\item Finiteness of $(p^s,p^u)$: $t< p^s,p^u$, hence $\#\{(p^s,p^u):p^s,p^u>t\}\leq \#(I_\ve\cap (t,1])^2$ is finite.
\end{enumerate}
The first four items above give that, for $\un a\in\N_0^3$ and $t>0$,
\begin{align*}
\#\left\{\Gamma(\wh x):
\begin{array}{c}
\Psi_{\wh x}^{p^s,p^u}\in\mathfs A\text{ s.t. }p^s,p^u>t\\
\text{and }\vt_i[\wh x]\in \mathfrak B_{a_i}, |i|\leq 1
\end{array}
\right\}\leq\sum_{j=0}^{\lceil |\log t|\rceil+2}\sum_{b=0}^{\lceil |\log t|\rceil}
\sum_{k_i,\ell_i=0\atop{-1\leq i\leq 1}}^{T_t}
\# Y_{\un k,\un\ell,\un a,b,d}(j)
\end{align*}
is the finite sum of finite terms, hence finite. Together with the last two items,
we conclude that
\begin{align*}
\#\left\{\Psi_{\wh x}^{p^s,p^u}\in\mathfs A:p^s,p^u>t\right\}&\leq 
\sum_{j=0}^{\lceil |\log t|\rceil+2}\sum_{b=0}^{\lceil |\log t|\rceil}\sum_{k_i,\ell_i=0\atop{-1\leq i\leq 1}}^{T_t}
\# Y_{\un k,\un\ell,\un a,b,d}(j)\\
&\ \ \ \ \times (\#\{\mathfrak B\in\mathfs P:\mathfrak B\cap M_t\neq\emptyset\})^3\times (\#(I_\ve\cap (t,1]))^2
\end{align*}
is finite. This proves the discreteness of $\mathfs A$.

\medskip
\noindent
{\em Proof of sufficiency.}
Let $\wh x\in {\rm NUH}^\#$, and take $(k_i)_{i\in\Z}$, $(\ell_i)_{i\in\Z}$, $(a_i)_{i\in\Z}$,
$(b_i)_{i\in\Z}$, $(d_i)_{i\in\Z}$, $(j_i)_{i\in\Z}$  s.t.:
\begin{align*}
& d(\vt_i[\wh x],\mathfs S)\in [e^{-k_i-1},e^{-k_i}),\ \|C(\wh f^i(\wh x))^{-1}\|\in [e^{\ell_i},e^{\ell_i+1}),
\ \vt_i[\wh x]\in \mathfrak B_{a_i},\\
&Q(\wh f^i(\wh x))\in [e^{-b_i-1},e^{-b_i}), \ d_i=d_s(\wh f^i(\wh x)), \ q(\wh f^i(\wh x))\in[e^{-j_i-1},e^{-j_i+1}).
\end{align*}
For $n\in\Z$, define
$$
\un k^{(n)}=(k_{n-1},k_n,k_{n+1}),\ \un\ell^{(n)}=(\ell_{n-1},\ell_n,\ell_{n+1}),\ \un a^{(n)}=(a_{n-1},a_n,a_{n+1}).
$$
Then $\Gamma(\wh f^n(\wh x))\in Y_{\un k^{(n)},\un\ell^{(n)},\un a^{(n)},b_n,d_n}$.
Take $\Gamma(\wh x_n)\in Y_{\un k^{(n)},\un\ell^{(n)},\un a^{(n)},b_n,d_n}(j_n)$
s.t.:
\begin{enumerate}[aaa)]
\item[(${\rm a}_n$)] $ d(\vt_i[\wh f^n(\wh x)],\vt_i[\wh x_n])+
\norm{\wt{C(\wh f^i(\wh f^n(\wh x)))}-\wt{C(\wh f^i(\wh x_n))}}<e^{-8(j_n+2)}$
for $|i|\leq 1$.
\item[(${\rm b}_n$)] $\tfrac{Q(\wh f^n(\wh x))}{Q(\wh x_n)}=e^{\pm\ve/3}$.
\end{enumerate}
Define $p^s_n=\delta_\ve\min\{e^{\ve|k|}Q(\wh x_{n+k}):k\geq 0\}$ and
$p^u_n=\delta_\ve\min\{e^{\ve|k|}Q(\wh x_{n+k}):k\leq 0\}$.
We claim that $\{\Psi_{\wh x_n}^{p^s_n,p^u_n}\}_{n\in\Z}$ is an $\ve$--gpo
in $\mathfs A^\Z$ that shadows $x$.

\medskip
\noindent
{\sc Claim 2:} $\Psi_{\wh x_n}^{p^s_n,p^u_n}\in\mathfs A$ for all $n\in\Z$.

\medskip
\noindent
(CG1) By definition, $\Gamma(\wh x_n)\in Y_{\un k^{(n)},\un\ell^{(n)},\un a^{(n)},b_n,d_n}(j_n)$.

\noindent
(CG2) By (${\rm b}_n$), $\inf\{e^{\ve|k|}Q(\wh x_{n+k}):k\geq 0\}=
e^{\pm\ve/3}\inf\{e^{\ve|k|}Q(\wh f^{n+k}(\wh x)):k\geq 0\}$ is positive.
Since the only accumulation point of $I_\ve$ is zero, it follows that
$p^s_n$ is well-defined and positive. The same proof applies to $p^u_n$.
The other conditions are clear from the definition.

%0<p^s_n\wedge p^u_n\leq e^{\ve/3} q_\ve(f^n(x))<
%\ve e^{\ve/3}Q_\ve(f^n(x))\leq \ve e^{2\ve/3}Q_\ve(x_n)<Q_\ve(x_n)$$
%for $\ve>0$ small enough.

\noindent
(CG3) Again by (${\rm b}_n$), we have
$$
\min\{e^{\ve|k|}Q(\wh x_{n+k}):k\geq 0\}=e^{\pm\ve/3}\min\{e^{\ve|k|}Q(\wh f^{n+k}(\wh x)):k\geq 0\}
$$
hence $\tfrac{p^s_n}{q^s(\wh f^n(\wh x))}=e^{\pm\ve/3}$, and analogously 
$\tfrac{p^u_n}{q^u(\wh f^n(\wh x))}=e^{\pm\ve/3}$. By Lemma \ref{Lemma-q^s}(1),
$p^s_n\wedge p^u_n=e^{\pm\ve/3}q(\wh f^n(\wh x))\in [e^{-j_n-2},e^{-j_n+2})$.

\medskip
\noindent
{\sc Claim 3:} $\Psi_{\wh x_n}^{p^s_n,p^u_n}\overset{\ve}{\rightarrow}\Psi_{\wh x_{n+1}}^{p^s_{n+1},p^u_{n+1}}$
for all $n\in\Z$.

\medskip
\noindent
(GPO1) We have $\vt_1[\wh x_n],\vt_0[\wh x_{n+1}]\in \mathfrak B_{a_{n+1}}$, and by (${\rm a}_n$) with $i=1$
and (${\rm a}_{n+1}$) with $i=0$, we have
\begin{align*}
&\, d(\vt_1[\wh x_n],\vt_0[\wh x_{n+1}])+\norm{\wt{C(\wh f(\wh x_n))}-\wt{C(\wh x_{n+1})}}\\
&\leq  d(\vt_1[\wh f^n(\wh x)],\vt_1[\wh x_n])+
\norm{\wt{C(\wh f^{n+1}(\wh x))}-\wt{C(\wh f(\wh x_n))}}+\\
&\ \ \ \ d(\vt_1[\wh f^n(\wh x)],\vt_0[\wh x_{n+1}])+
\norm{\wt{C(\wh f^{n+1}(\wh x))}-\wt{C(\wh x_{n+1})}}\\
&<e^{-8(j_n+2)}+e^{-8(j_{n+1}+2)}\leq e^{-8}\left(q(\wh f^n(\wh x))^8+q(\wh f^{n+1}(\wh x))^8\right)\\
&\overset{!}{\leq} e^{-8}(1+e^{8\ve})q(\wh f^{n+1}(\wh x))^8\leq e^{-8+8\ve/3}(1+e^{8\ve})(p^s_{n+1}\wedge p^u_{n+1})^8
\overset{!!}{<}(p^s_{n+1}\wedge p^u_{n+1})^8,
\end{align*}
where in $\overset{!}{\leq}$ we used Lemma \ref{Lemma-q} and in $\overset{!!}{<}$
we used that $e^{-8+8\ve/3}(1+e^{8\ve})<1$ when $\ve>0$ is sufficiently small. This proves
that $\Psi_{\wh f(\wh x_n)}^{p^s_{n+1}\wedge p^u_{n+1}}\overset{\ve}{\approx}
\Psi_{\wh x_{n+1}}^{p^s_{n+1}\wedge p^u_{n+1}}$.
Similarly, we prove that
$\Psi_{\wh f^{-1}(\wh x_{n+1})}^{p^s_n\wedge p^u_n}\overset{\ve}{\approx}\Psi_{\wh x_n}^{p^s_n\wedge p^u_n}$.

\medskip
\noindent
(GPO2) The very definitions of $p^s_n,p^u_n$ guarantee that
$p^s_n=\min\{e^\ve p^s_{n+1},\delta_\ve Q(\wh x_n)\}$ and
$p^u_{n+1}=\min\{e^\ve p^u_n,\delta_\ve Q(\wh x_{n+1})\}$.

\medskip
\noindent
{\sc Claim 4:} $\{\Psi_{\wh x_n}^{p^s_n,p^u_n}\}_{n\in\Z}$ shadows $\wh x$.

\medskip
By (${\rm a}_n$) with $i=0$, we have
$\Psi_{\wh f^n(\wh x)}^{p^s_n\wedge p^u_n}\overset{\ve}{\approx}\Psi_{\wh x_n}^{p^s_n\wedge p^u_n}$, hence 
by Proposition \ref{Proposition-overlap}(4) we have
$\vt_n[\wh x]=\Psi_{\wh f^n(\wh x)}(0)\in \Psi_{\wh x_n}(B[p^s_n\wedge p^u_n])$,
thus $\{\Psi_{\wh x_n}^{p^s_n,p^u_n}\}_{n\in\Z}$ shadows $\wh x$.

\medskip
We thus proved that every $\wh x\in{\rm NUH}^\#$ is shadowed
by an $\ve$--gpo $\{\Psi_{\wh x_n}^{p^s_n,p^u_n}\}_{n\in\Z}\in \mathfs A^\Z$
s.t. $p^s_n\wedge p^u_n=e^{\pm\ve/3}q(\wh f^n(\wh x))$. Since we have $\limsup_{n\to+\infty}(p^s_n\wedge p^u_n)>0$
and  $\limsup_{n\to-\infty}(p^s_n\wedge p^u_n)>0$, the discreteness of $\mathfs A$
implies that $\{\Psi_{\wh x_n}^{p^s_n,p^u_n}\}_{n\in\Z}$ is recurrent. This completes
the proof of sufficiency.

\medskip
\noindent
{\em Proof of relevance.} The alphabet $\mathfs A$ might not a priori satisfy
the relevance condition, but we can easily reduce it to a sub-alphabet $\mathfs A'$ satisfying (1)--(3).
Call $v\in\mathfs A$ relevant if there is $\un v\in\mathfs A^\Z$ with $v_0=v$ s.t. $\un{v}$ shadows
a point in ${\rm NUH}^\#$. Since ${\rm NUH}^\#$ is $\wh f$--invariant, every $v_i$ is relevant.
Then $\mathfs A'=\{v\in\mathfs A:v\text{ is relevant}\}$ is discrete
because $\mathfs A'\subset\mathfs A$, it is sufficient and relevant by definition.
%because ${\rm NUH}_\chi^*\subset {\rm NUH}_\chi$,
%and it is relevant by definition.
\end{proof}

Let $\Sigma$ be the TMS associated to the graph with vertex set $\mathfs A$ given by
Theorem \ref{Thm-coarse-graining} and
edges $v\overset{\ve}{\to}w$. An element of $\Sigma$ is an $\ve$--gpo, hence
we define $\pi:\Sigma\to \wh M$ by
$$
\{\pi[\{v_n\}_{n\in\Z}]\}:=\wh V^s[\{v_n\}_{n\geq 0}]\cap \wh V^u[\{v_n\}_{n\leq 0}].
$$
Recall the definition of $\Sigma^\#$. Here are the main properties of the triple $(\Sigma,\sigma,\pi)$.

\begin{proposition}\label{Prop-pi}
The following holds for all $\ve>0$ small enough.
\begin{enumerate}[{\rm (1)}]
\item Each $v\in\mathfs A$ has finite ingoing and outgoing degree, hence $\Sigma$ is locally compact.
\item $\pi:\Sigma\to \wh M$ is H\"older continuous.
\item $\pi\circ\sigma=\wh f\circ\pi$.
\item $\pi[\Sigma^\#]\supset{\rm NUH}^\#$.
\end{enumerate} 
\end{proposition}

\begin{proof}
Part (1) follows from (GPO2), part (3) is obvious, and part (4) follows from Theorem \ref{Thm-coarse-graining}(2).
We prove part (2). Let $K>1,\theta\in(0,1)$ satisfying Proposition \ref{Prop-stable-manifolds}(5).
We will define constants $K_0>0$ and $\theta_0\in (0,1)$ s.t. if
$\un v=\{v_n\}_{n\in\Z},\un w=\{w_n\}_{n\in\Z}\in\Sigma$ satisfy $v_n=w_n$ for $|n|\leq N$
then $d(\pi(\un v),\pi(\un w)) \leq K_0\theta_0^N$. Recall that $M$
has diameter smaller than one.
Let $\theta_0:=\max\left\{\theta,\tfrac{1}{2}\right\}\in (0,1)$. For $\un v,\un w$ as above, write
$\pi(\un v)=\wh x=\{x_n\}_{n\in\Z}$ and $\pi(\un w)=\wh y=\{y_n\}_{n\in\Z}$. By part (3),
we have $\wh f^n(\wh x)=\pi[\sigma^n(\un v)]$ and $\wh f^n(\wh y)=\pi[\sigma^n(\un w)]$
and so $\{x_n\}=V^s[\sigma^n(\un v)]\cap V^u[\sigma^n(\un v)]$ and
$\{y_n\}=V^s[\sigma^n(\un w)]\cap V^u[\sigma^n(\un w)]$. Fix $-N\leq n\leq 0$.
Since $\sigma^n(\un v),\sigma^n(\un w)$ coincide in all positions $-(N+n),\ldots,N+n$,
Proposition \ref{Prop-stable-manifolds}(5) implies that
$d_{C^1}(V^s[\sigma^n(\un v)],V^s[\sigma^n(\un w)])\leq K\theta^{N+n}$
and $d_{C^1}(V^u[\sigma^n(\un v)],V^u[\sigma^n(\un w)])\leq K\theta^{N+n}$.
By Lemma \ref{Lemma-admissible-manifolds}, it follows that 
$d(x_n,y_n)\leq 6K\theta^{N+n}$ and so 
$$
2^nd(x_n,y_n)\leq 6K 2^n\theta^{N+n}\leq 6K \theta_0^{-n}\theta_0^{N+n}=6K \theta_0^N.
$$
Noting that if $n<-N$ then $2^nd(x_n,y_n)<2^n\leq \theta_0^{-n}<\theta_0^N$,
we conclude that $d(\wh x,\wh y)\leq 6K\theta_0^N$.
Letting $K_0:=6K$, the proof is complete.
\end{proof}

It is important noting that $(\Sigma,\sigma,\pi)$ does {\em not} satisfy Theorem \ref{Thm-Main},
since $\pi$ might be (and usually is) infinite-to-one. We use $\pi$ to induce a locally
finite cover of ${\rm NUH}^\#$, which will then be refined to a cover of ${\rm NUH}^\#$
by pairwise disjoint sets that will give a proof of Theorem \ref{Thm-Main}.

\section{The inverse problem}

Now we start investigating the non-injectivity of $\pi$. More specifically, if
$\un v=\{v_n\}_{n\in\Z}$ and $\un w=\{w_n\}_{n\in\Z}$ satisfy $\pi(\un{v})=\pi(\un{w})$,
then we want to compare $v_n$ and $w_n$ and show that they
are uniquely defined ``up to bounded error''. We do this under the additional assumption
that $\un{v},\un{w}\in\Sigma^\#$. Remind that $\Sigma^\#$ is the {\em recurrent set} of $\Sigma$:
$$
\Sigma^\#:=\left\{\un v\in\Sigma:\exists v,w\in V\text{ s.t. }\begin{array}{l}v_n=v\text{ for infinitely many }n>0\\
v_n=w\text{ for infinitely many }n<0
\end{array}\right\}.
$$
The main result is the following. 

\begin{theorem}[Inverse theorem]\label{Thm-inverse}
The following holds for $\ve>0$ small enough.
If $\{\Psi_{\wh x_n}^{p^s_n,p^u_n}\}_{n\in\Z}\in\Sigma^\#$ satisfies
$\pi[\{\Psi_{\wh x_n}^{p^s_n,p^u_n}\}_{n\in\Z}]=\wh x$ then $\wh x\in{\rm NUH}^*$,
and for all $n\in\Z$:
\begin{enumerate}[{\rm (1)}]
\item $d(\vt[\wh x_n],\vt[\wh f^n(\wh x)])<{25}^{-1}(p^s_n\wedge p^u_n)$ and
$\tfrac{\rho(\wh x_n)}{\rho(\wh f^n(\wh x))}=e^{\pm\ve}$.
\item $\tfrac{\norm{C(\wh x_n)^{-1}}}{\norm{C(\wh f^n(\wh x))^{-1}}}=e^{\pm2\sqrt{\ve}}$.
\item $\tfrac{Q(\wh x_n)}{Q(\wh f^n(\wh x))}=e^{\pm \sqrt[3]{\ve}}$.
\item $\tfrac{p^s_n}{q^s(\wh f^n(\wh x))}=e^{\pm\sqrt[3]{\ve}}$ and
$\tfrac{p^u_n}{q^u(\wh f^n(\wh x))}=e^{\pm\sqrt[3]{\ve}}$.
\item $(\Psi_{\wh f^n(\wh x)}^{-1}\circ\Psi_{\wh x_n})(v)=\delta_n+O_nv+\Delta_n(v)$ for $v\in B[10Q(\wh x_n)]$,
where $\delta_n\in\R^m$ satisfies $\norm{\delta_n}<{25}^{-1}(p^s_n\wedge p^u_n)$,
$O_n$ is an orthogonal linear map preserving the splitting $\R^d\times\R^{m-d}$,
and $\Delta_n:B[10Q(\wh x_n)]\to \R^m$ satisfies
$\Delta_n(0)=0$ and $\norm{d\Delta_n}_{C^0}<5\sqrt{\ve}$ on $B[10Q(\wh x_n)]$.
\end{enumerate}
In particular, $\wh x\in{\rm NUH}^\#$.
\end{theorem}

The above theorem compares the parameters of an $\ve$--gpo with the parameters
of its shadowed point, and it allows to identify the coded set in Theorem \ref{Thm-Main}.
Below, we state a corollary that is useful for comparing the
parameters of two $\ve$--gpo's that shadow a same point.

\begin{corollary}\label{Corollary-inverse}
The following holds for $\ve>0$ small enough.
If $\un v=\{\Psi_{\wh x_n}^{p^s_n,p^u_n}\}_{n\in\Z}\in\Sigma^\#$ 
and $\un w=\{\Psi_{\wh y_n}^{q^s_n,q^u_n}\}_{n\in\Z}\in\Sigma^\#$
satisfy $\pi(\un v)=\pi(\un w)$, then for all $n\in\Z$:
\begin{enumerate}[{\rm (1)}]
\item $d(\vt[\wh x_n],\vt[\wh y_n])<{10}^{-1}\max\{p^s_n\wedge p^u_n,q^s_n\wedge q^u_n\}$.
\item $\tfrac{p^s_n}{q^s_n}=e^{\pm 2\sqrt[3]{\ve}}$ and $\tfrac{p^u_n}{q^u_n}=e^{\pm2\sqrt[3]{\ve}}$.
\item $(\Psi_{\wh y_n}^{-1}\circ\Psi_{\wh x_n})(v)=\delta_n+O_nv+\Delta_n(v)$ for $v\in B[10Q(\wh x_n)]$,
where $\delta_n\in\R^m$ satisfies $\norm{\delta_n}<10^{-1}(q^s_n\wedge q^u_n)$,
$O_n$ is an orthogonal linear map preserving the splitting $\R^d\times\R^{m-d}$,
and $\Delta_n:B[10Q(\wh x_n)]\to \R^m$ satisfies
$\Delta_n(0)=0$ and $\norm{d\Delta_n}_{C^0}<9\sqrt{\ve}$ on $B[10Q(\wh x_n)]$.
\end{enumerate}
\end{corollary}

In the following section, we provide some preliminary results that will be used
to prove Theorem \ref{Thm-inverse}. After this preparation is done, we will proceed
to prove each part of \ref{Thm-inverse}. Clearly, we only need to prove it for $n=0$.

\subsection{Identification of invariant subspaces}

In order to analyse the inverse problem, we have to compare the parameters
of charts. One of the steps is to compare the hyperbolicity parameters, which
are related to the Lyapunov inner product, see Section \ref{Section-reduction}.
In principle, for $\wh x,\wh y\in {\rm NUH}$ with $d(\vt[\wh x],\vt[\wh y])\ll 1$,
it is not clear how to compare the norm in the Lyapunov inner product between
vectors in $E^{s/u}_{\wh x}$ and in $E^{s/u}_{\wh y}$. When $\wh x$ defines an 
$\ve$--double chart $\Psi_{\wh x}^{p^s,p^u}$ and $\wh y$ belongs to a stable/unstable
set at $\Psi_{\wh x}^{p^s,p^u}$, Ben Ovadia proposed a canonical way of making this comparison, 
using the representing function in the chart \cite{Ben-Ovadia-2019}.
The idea is to make an identification of subspaces in the chart and then push it
to the manifold. Here, we follow the same approach, adapted to our context.
Since the unstable subspace depends on inverse branches, the 
definition is made inside $\wh{TM}$. In particular, if $\wh y$ is the intersection of a stable 
and an unstable set, both at an $\ve$--double chart $\Psi_{\wh x}^{p^s,p^u}$, 
then we can compare the tangent spaces $\wh{TM}_{\wh x}$ and $\wh{TM}_{\wh y}$. 

Let $\Psi_{\wh x}^{p^s,p^u}$ be an $\ve$--double chart, let $V=V^s$ be an $s$--admissible
manifold at $\Psi_{\wh x}^{p^s,p^u}$ and $\wh V$ be the respective stable set.
In the sequel, we use again the notation $T\wh V=\wh{TV}$ to represent the tangent bundle
of $V$ as a subset of $\wh{TM}$. Writing $d=d_s(\wh x)$, recall that
$$
V=\Psi_{\wh x}\{(v_1,G(v_1)):v_1\in B^d[p^s]\}
$$
where $G:B^d[p^s]\to \R^{m-d}$ is a $C^{1+\beta/3}$ function satisfying (AM1)--(AM3).
Fix $\wh y\in \wh V$, and write $\vt[\wh y]=y=\Psi_{\wh x}(z)$ where $z=(v_1,G(v_1))$.
Denote the tangent space to the graph of $G$ at $z$ by
$T_{z}{\rm Graph}(G)$. We have
$$
T_{z}{\rm Graph}(G)=\left\{\begin{bmatrix} w \\ (dG)_{v_1}w\end{bmatrix}:w\in \R^d\right\},
$$
which is canonically isomorphic to $\R^d\times\{0\}$ via the map 
$$
\begin{array}{rcl}
\iota_s\ :\  \R^d\times\{0\}& \xrightarrow{\hspace*{8mm}}& T_{z}{\rm Graph}(G)\\
&&\\
\begin{bmatrix} w \\ 0\end{bmatrix}& \xmapsto{\hspace*{8mm}} & \begin{bmatrix} w \\ (dG)_{v_1}w\end{bmatrix}.\\
\end{array}
$$
Recalling that $E^s_{\wh x}=(d\Psi_{\wh x})_0[\R^d\times\{0\}]$ and 
$T_{\wh y}\wh V=(d\Psi_{\wh x})_z[T_{z}{\rm Graph}(G)]$, we have the following definition.

\medskip
\noindent
{\sc The map $\Theta^s_{\wh x,\wh y}$:} We define $\Theta^s_{\wh x,\wh y}:E^s_{\wh x}\to T_{\wh y}\wh V$ 
as the composition of the linear maps
$$
\Theta^s_{\wh x,\wh y}:= (d\Psi_{\wh x})_z\circ \iota_s \circ [(d\Psi_{\wh x})_0]^{-1}.
$$

\medskip
In other words, $\Theta^s_{\wh x,\wh y}$ is defined to make the diagram below commute
$$
\xymatrixcolsep{6pc}
\xymatrix{
\R^d\times\{0\} \ar[d]_{(d\Psi_{\wh x})_0} \ar[r]^{\iota_s} &T_z{\rm Graph}(G)\ar[d]^{(d\Psi_{\wh x})_z} \\
E^s_{\wh x} \ar[r]_{\Theta^s_{\wh x,\wh y}} &T_{\wh y}\wh V}
$$
so that
$$
\Theta^s_{\wh x,\wh y}\left((d\Psi_{\wh x})_{0}\begin{bmatrix} w \\ 0\end{bmatrix}\right)=
(d\Psi_{\wh x})_{z}\begin{bmatrix} w \\ (dG)_{v_1}w\end{bmatrix}.
$$
A similar definition holds when $\wh y$ belongs to an unstable set at
$\Psi_{\wh x}^{p^s,p^u}$. Let
$$
V=\Psi_{\wh x}\{(G(v_2),v_2):v_2\in B^{m-d}[p^u]\}
$$
where $G:B^{m-d}[p^u]\to \R^d$ is a $C^{1+\beta/3}$ function satisfying (AM1)--(AM3),
and let $\wh V$ be the respective unstable set. For $\wh y\in\wh V$,
write $\vt[\wh y]=y=\Psi_{\wh x}(z)$ where $z=(G(v_2),v_2)$. Now
$$
T_{z}{\rm Graph}(G)=\left\{\begin{bmatrix} (dG)_{v_2}w \\ w\end{bmatrix}:w\in \R^{m-d}\right\},
$$
which is canonically isomorphic to $\{0\}\times\R^{m-d}$ via the map
$$
\begin{array}{rcl}
\iota_u\ :\  \{0\}\times\R^{m-d}& \xrightarrow{\hspace*{8mm}}& T_{z}{\rm Graph}(G)\\
&&\\
\begin{bmatrix} 0 \\ w\end{bmatrix}& \xmapsto{\hspace*{8mm}} & \begin{bmatrix} (dG)_{v_2}w \\ w\end{bmatrix}.\\
\end{array}
$$

\medskip
\noindent
{\sc The map $\Theta^u_{\wh x,\wh y}$:} We define $\Theta^u_{\wh x,\wh y}:E^u_{\wh x}\to T_{\wh y}\wh V$ 
as the composition of the linear maps
$$
\Theta^u_{\wh x,\wh y}:= (d\Psi_{\wh x})_z\circ \iota_u \circ [(d\Psi_{\wh x})_0]^{-1}.
$$

\medskip
Similarly, $\Theta^u_{\wh x,\wh y}$ is defined to make the following diagram to commute
$$
\xymatrixcolsep{6pc}
\xymatrix{
\{0\}\times\R^{m-d} \ar[d]_{(d\Psi_{\wh x})_0} \ar[r]^{\iota_u} &T_z{\rm Graph}(G)\ar[d]^{(d\Psi_{\wh x})_z} \\
E^u_{\wh x} \ar[r]_{\Theta^u_{\wh x,\wh y}} &T_{\wh y}\wh V}
$$
and so
$$
\Theta^u_{\wh x,\wh y}\left((d\Psi_{\wh x})_{0}\begin{bmatrix} 0 \\ w\end{bmatrix}\right)=
(d\Psi_{\wh x})_{z}\begin{bmatrix} (dG)_{v_2}w \\ w\end{bmatrix}.
$$

Now assume that $\{\wh y\}=\wh V^s\cap \wh V^u$, where $\wh V^{s/u}$ is a stable/unstable set
at $\Psi_{\wh x}^{p^s,p^u}$. By the above discussion, we have two linear maps
$\Theta^{s/u}_{\wh x,\wh y}:E^{s/u}_{\wh x}\to T_{\wh y}\wh V^{s/u}$.
Since $E^s_{\wh x}\oplus E^u_{\wh x}=\wh{TM}_{\wh x}$
and $T_{\wh y}\wh V^s\oplus T_{\wh y}\wh V^u=\wh{TM}_{\wh y}$, the following definition makes sense.

\medskip
\noindent
{\sc The map $\Theta_{\wh x,\wh y}$:} We define $\Theta_{\wh x,\wh y}:\wh{TM}_{\wh x}\to \wh{TM}_{\wh y}$
as the unique linear map s.t. $\Theta_{\wh x,\wh y}\restriction_{E^s_{\wh x}}=\Theta^s_{\wh x,\wh y}$ and 
$\Theta_{\wh x,\wh y}\restriction_{E^u_{\wh x}}=\Theta^u_{\wh x,\wh y}$. More specifically, if $v=v^s+v^u$ with
$v^{s/u}\in E^{s/u}_{\wh x}$ then 
$$
\Theta_{\wh x,\wh y}(v):=\Theta^s_{\wh x,\wh y}(v^s)+\Theta^u_{\wh x,\wh y}(v^u).
$$

\medskip
We can similarly see $\Theta_{\wh x,\wh y}$ as a map defined in terms of a commuting diagram.
Let $G,H$ be the representing functions of $\wh V^s,\wh V^u$, let $\vt[\wh y]=y=\Psi_{\wh x}(z)$
with $z=(v_1,v_2)$, and let 
$$
\begin{array}{rcl}
\iota\ \ \ :\ \ \R^m& \xrightarrow{\hspace*{8mm}}& \R^m\\
&&\\
\begin{bmatrix} w_1 \\ w_2\end{bmatrix}& \xmapsto{\hspace*{8mm}} &
\begin{bmatrix} w_1+(dH)_{v_2}w_2 \\ w_2+(dG)_{v_1}w_1\end{bmatrix}.\\
\end{array}
$$
Then we obtain a commuting diagram
$$
\xymatrixcolsep{6pc}
\xymatrix{
\R^m \ar[d]_{(d\Psi_{\wh x})_0} \ar[r]^{\iota} &\R^m\ar[d]^{(d\Psi_{\wh x})_z} \\
\wh{TM}_{\wh x} \ar[r]_{\Theta_{\wh x,\wh y}} &\wh{TM}_{\wh y}}
$$
and so
$$
\Theta_{\wh x,\wh y}\left((d\Psi_{\wh x})_{0}\begin{bmatrix} w_1 \\ w_2\end{bmatrix}\right)=
(d\Psi_{\wh x})_{z}\begin{bmatrix} w_1+(dH)_{v_2}w_2 \\ w_2+(dG)_{v_1}w_1\end{bmatrix}.
$$
The next lemma proves that the maps just defined are close to parallel transports.

\begin{lemma}\label{Lemma-pi}
Let $\Psi_{\wh x}^{p^s,p^u}$ be an $\ve$--double chart with $\vt[\wh x]=x$ and $\eta=p^s\wedge p^u$.
\begin{enumerate}[\rm (1)]
\item Let $\wh V$ be a stable set at $\Psi_{\wh x}^{p^s,p^u}$,
and let $\wh y\in\wh V$. If $\vt[\wh y]=y\in \Psi_{\wh x}(B^d[\eta]\times B^{m-d}[\eta])$ then
$$
\norm{\Theta^s_{\wh x,\wh y}-P_{x,y}}\leq \tfrac{1}{2}\eta^{15\beta/48},
$$
where $P_{x,y}$ is the restriction of the parallel transport from $x$ to $y$ to the subspace
$E^s_{\wh x}$. 
In particular, $\tfrac{\norm{\Theta^s_{\wh x,\wh y}(v)}}{\norm{v}}={\rm exp}\left[\pm \eta^{15\beta/48}\right]$
for all $v\in E^s_{\wh x}$. An analogous statement holds for unstable sets.
\item If $\wh y\in \wh V^s\cap \wh V^u$ where $\wh V^{s/u}$ is a stable/unstable
set at $\Psi_{\wh x}^{p^s,p^u}$, then 
$$
\norm{\Theta_{\wh x,\wh y}-P_{x,y}}\leq \tfrac{1}{2}\eta^{15\beta/48}.
$$
In particular, $\tfrac{\norm{\Theta_{\wh x,\wh y}(v)}}{\norm{v}}={\rm exp}\left[\pm \eta^{15\beta/48}\right]$
for all $v\in \wh{TM}_{\wh x}$.
\end{enumerate}
\end{lemma}

\begin{proof}
We start proving part (2). For simplicity, write $\Theta=\Theta_{\wh x,\wh y}$ and $P=P_{x,y}$.
By Lemma \ref{Lemma-admissible-manifolds}, $\norm{v_1}\leq\norm{(v_1,v_2)}<\eta$. By (AM2)--(AM3),
$$
\norm{(dG)_{v_1}}\leq \norm{(dG)_0}+\tfrac{1}{2}\norm{v_1}^{\beta/3}\leq \tfrac{1}{2}\eta^{\beta/3}+
\tfrac{1}{2}\eta^{\beta/3}=\eta^{\beta/3}.
$$
Similarly, $\norm{(dH)_{v_2}}\leq\eta^{\beta/3}$. Hence $\iota={\rm Id}+L$, where
$\norm{L}\leq \eta^{\beta/3}$. Writing $y=\Psi_{\wh x}(z)$ for $z\in B^d[\eta]\times B^{m-d}[\eta]$,
we have $\Theta=(d\Psi_{\wh x})_z\circ\iota\circ[(d\Psi_{\wh x})_0]^{-1}$ and so
$$
\Theta=(d\Psi_{\wh x})_z\circ [(d\Psi_{\wh x})_0]^{-1} + E=(d{\rm exp}_x)_{C(\wh x)z}+E
$$
where $E=(d{\rm exp}_x)_{C(\wh x)z}\circ C(\wh x)\circ L\circ C(\wh x)^{-1}$.
Therefore $\Theta-P=(d{\rm exp}_x)_{C(\wh x)z}-P+E$.
Observe that:
\begin{enumerate}[$\circ$]
\item $(d{\rm exp}_x)_{C(\wh x)z}-P=\widetilde{(d{\rm exp}_x)_{C(\wh x)z}}-\widetilde{(d{\rm exp}_x)_0}$
and so by (A3) we have
$$
\norm{(d{\rm exp}_x)_{C(\wh x)z}-P}\leq d(x,\mathfs S)^{-a}\norm{z}\leq 2\rho(\wh x)^{-a}\eta
\leq 2\ve^{1/16}\eta^{95\beta/96},
$$
where in the last passage we used that
$\rho(\wh x)^{-a}\eta^{\beta/96}\leq \rho(\wh x)^{-a}Q(\wh x)^{\beta/96}\leq \ve^{1/16}$.
\item $\norm{E}\leq 2\norm{C(\wh x)^{-1}}\norm{L}\leq 2\norm{C(\wh x)^{-1}}\eta^{\beta/3}\leq 2\ve^{1/8}\eta^{15\beta/48}$, where in the last passage we used that
$\norm{C(\wh x)^{-1}}\eta^{\beta/48}\leq \norm{C(\wh x)^{-1}}Q(\wh x)^{\beta/48}\leq \ve^{1/8}$.
\end{enumerate}
Therefore, if $\ve>0$ is small enough then
$$
\norm{\Theta-P}\leq 2\ve^{1/16}\eta^{95\beta/96}+2\ve^{1/8}\eta^{15\beta/48}<\tfrac{1}{2}\eta^{15\beta/48}.
$$
This proves part (2). Part (1) is proved exactly in the same manner, using
that $\iota_s={\rm Id}_{\R^d\times\{0\}}+L$ and that the same estimates above hold for $L$.
\end{proof}

Compare Lemma \ref{Lemma-pi} with the estimates in the proof of \cite[Lemma 4.6]{Ben-Ovadia-2019}.
Our assumption is stronger because we require $y$ to be very close to the center of the chart
$\Psi_{\wh x}^{p^s,p^u}$, but we obtain a much better estimate.
This sharpened estimate will be used in the next section, for the following reason.
Given an edge $\Psi_{\wh x}^{p^s,p^u}\overset{\ve}{\to}\Psi_{\wh y}^{q^s,q^u}$,
we will need to compare the parameters of the two charts. The usual implementation,
under the assumption that $f$ is a diffeomorphism,
makes this comparison using that the ratio $\tfrac{Q(\wh f(\wh x))}{Q(\wh x)}$
is bounded, since the derivative is uniformly bounded.
This is not the case in our context, and there is no a priori bound on the ratio $\tfrac{Q(\wh f(\wh x))}{Q(\wh x)}$.
On the other hand, we know that $\tfrac{p^s\wedge p^u}{q^s\wedge q^u}=e^{\pm\ve}$.
Hence, provided our estimates depend on $p^s\wedge p^u$ and $q^s\wedge q^u$,
we will be able to compare the parameters of $\Psi_{\wh x}^{p^s,p^u}$ with those of $\Psi_{\wh y}^{q^s,q^u}$.

\subsection{Improvement lemma}\label{Section-improvement}

As we have seen in Lemma \ref{Lemma-linear-reduction}, the Lyapunov inner product
allows to ``transform'' a nonuniformly hyperbolic system into a uniformly hyperbolic one
and then use the classical tools of uniform hyperbolicity, such as graph transforms.
One essential feature of uniform hyperbolicity is the following: forward iterations of the system improve 
the behavior along unstable directions (the forward expansion uniformizes the behavior),
and backward iterations improve the behavior along stable directions.
This philosophy is behind the recent work on anisotropic spaces, see e.g.
\cite{Blank-Keller-Liverani}. The main result of this section is another
manifestation of this behavior. We call it an {\em improvement lemma},
as in \cite{Lima-Matheus}. Recall the functions $S,U$ introduced in (NUH3)
and the graph transforms $\wh{\mathfs F}^{s/u}$ introduced in Section \ref{Section-stable-sets}.
Throughout this section, $d$ is the dimension of the stable subspace.

\begin{lemma}[Improvement lemma]\label{Lemma-improvement}
The following holds for all $\ve>0$ small enough. Let $v\overset{\ve}{\to}w$ with
$v=\Psi_{\wh x_0}^{p^s_0,p^u_0},w=\Psi_{\wh x_1}^{p^s_1,p^u_1}$, and write
$\eta_0=p^s_0\wedge p^u_0$, $\eta_1=p^s_1\wedge p^u_1$. Fix $\wh V_1\in\wh{\mathfs M}^s(w)$
and let $\wh V_0=\wh{\mathfs F}^s_{v,w}[\wh V_1]$. Let
$\wh y_0\in \wh V_0$, $\wh y_1=\wh f(\wh y_0)\in \wh V_1$, and assume that 
$$
\vt[\wh y_0]\in\Psi_{\wh x_0}(B^d[\eta_0]\times B^{m-d}[\eta_0])\ \text{ and }\
\vt[\wh y_1]\in\Psi_{\wh x_1}(B^d[\eta_1]\times B^{m-d}[\eta_1]).
$$
Let $v_1\in T_{\wh y_1}\wh V_1$ and $v_0=\wh{df}^{(-1)}v_1\in T_{\wh y_0}\wh V_0$.
Finally, let $w_0\in \wh{TM}_{\wh x_0}$ s.t. $v_0=\Theta^s_{\wh x_0,\wh y_0}(w_0)$
and $w_1\in \wh{TM}_{\wh x_1}$ s.t.  $v_1=\Theta^s_{\wh x_1,\wh y_1}(w_1)$.
If $\tfrac{S(\wh y_1,v_1)}{S(\wh x_1,w_1)}={\rm exp}[\pm\xi]$
for $\xi\geq {\sqrt{\ve}}$, then 
$$
\frac{S(\wh y_0,v_0)}{S(\wh x_0,w_0)}={\rm exp}\left[\pm(\xi-Q(\wh x_0)^{\beta/4})\right].
$$
An analogous statement holds for unstable sets.
\end{lemma}

\begin{proof}
Let $x_0=\vt[\wh x_0]$, $y_0=\vt[\wh y_0]$, $x_1=\vt[\wh x_1]$, $y_1=\vt[\wh y_1]$.
Also let $\varrho=\wh{df}^{(-1)}w_1=(df_{\wh x_0}^{-1})_{x_1} w_1\in \wh{TM}_{\wh f^{-1}(\wh x_1)}$,
see Figure \ref{figure_improvement}.
\begin{figure}[hbt!]\centering
\begin{tikzpicture}[scale=0.8]
	\coordinate (l) at (0,2);
  \coordinate (m) at (6,2);
	\draw[color=red,line width=1pt] (l) to [bend right=15] (m);
	\coordinate (l) at (0,0);
  \coordinate (m) at (6,0);
	\draw[color=blue,line width=1pt] (l) to [bend right=15] (m);
	\coordinate (l) at (8,2);
  \coordinate (m) at (14,2);
	\draw[color=red,line width=1pt] (l) to [bend right=15] (m);
	\coordinate (l) at (8,0);
  \coordinate (m) at (14,0);
	\draw[color=blue,line width=1pt] (l) to [bend right=15] (m);
	\draw[->,>=stealth,line width=1pt] (3,1.55) -- (5,1.52);
	\fill[thick] (3,1.55) circle (1.mm);
	\draw[->,>=stealth,line width=1pt] (3.4,0.3) -- (5.4,0.4);
	\fill[thick] (3.4,0.3) circle (1.mm);	
	\draw[->,>=stealth,line width=1pt] (3,-0.45) -- (5,-0.48);
	\fill[thick] (3,-0.45) circle (1.mm);
	\draw[->,>=stealth,line width=1pt] (11,1.55) -- (13,1.52);
	\fill[thick] (11,1.55) circle (1.mm);
	\draw[->,>=stealth,line width=1pt] (11,-0.45) -- (13,-0.48);
	\fill[thick] (11,-0.45) circle (1.mm);
	\fill[thick] (3,1.55) circle (0.mm) node[above] {$y_0$};
	\fill[thick] (3.4,0.3) circle (0.mm) node[left] {$f^{-1}_{\widehat x_0}(x_1)$};
	\fill[thick] (3,-0.55) circle (0.mm) node[below] {$x_0$};
	\fill[thick] (11,1.55) circle (0.mm) node[above] {$y_1$};
	\fill[thick] (11,-0.55) circle (0.mm) node[below] {$x_1$};
	\fill[thick] (5,1.43) circle (0.mm) node[below] {$v_0$};
	\fill[thick] (5.4,0.4) circle (0.mm) node[right] {$\varrho$};
	\fill[thick] (5,-0.65) circle (0.mm) node[below] {$w_0$};
	\fill[thick] (13,1.43) circle (0.mm) node[below] {$v_1$};
	\fill[thick] (13,-0.65) circle (0.mm) node[below] {$w_1$};
	\fill[thick] (0.5,1.9) circle (0.mm) node[above] {$\vt[\wh V_0]$};
	\fill[thick] (8.5,1.9) circle (0.mm) node[above] {$\vt[\wh V_1]$};
\end{tikzpicture}
\caption{Points and vectors in the improvement lemma.}
\label{figure_improvement}	
\end{figure}
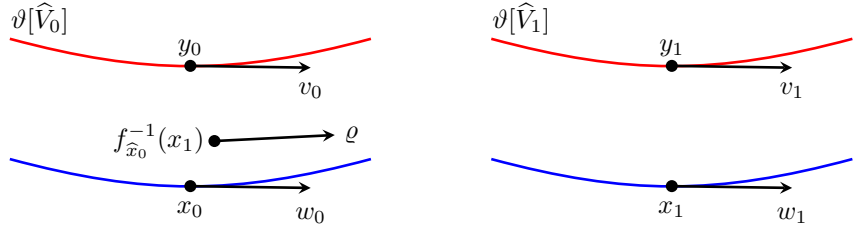
%\begin{center}
%\begin{figure}[hbt!]
%\centering
%\def\svgwidth{14cm}
%\input{improvement.pdf_tex}
%\label{figure-improvement}
%\caption{Representation of the improvement lemma}
%\end{figure}
%\end{center}
In the following, the basepoints in $\wh {TM}$ of all vectors involved in the calculations
are fixed, hence we write $S(\cdot,v)$ simply by $\vertiii{v}$.
We focus on one side of the estimate:
under the assumption that $\tfrac{\vertiii{v_1}}{\vertiii{w_1}}\leq {\rm exp}[\xi]$ for
$\xi\geq\sqrt{\ve}$ we will prove that $\tfrac{\vertiii{v_0}}{\vertiii{w_0}}\leq {\rm exp}[\xi-Q(\wh x_0)^{\beta/4}]$.
The inequality in the other direction will follow in the same way.
We are allowed to make one normalization, so we take $\norm{v_0}=1$. We write
$\tfrac{\vertiii{v_0}}{\vertiii{w_0}}=\tfrac{\vertiii{v_0}}{\vertiii{\varrho}}\cdot\tfrac{\vertiii{\varrho}}{\vertiii{w_0}}$
and estimate each of the fractions separately. Before, we compare the norms
of the fiber components of the vectors $v_0,w_0,\varrho$. The calculations are simple, 
once we pay attention to which parallel transports we apply.
\begin{enumerate}[$\circ$]
\item Comparison of $v_0,w_0$: by Lemma \ref{Lemma-pi}(1),
\begin{equation}\label{comparison-v-w}
\norm{P_{y_0,x_0}v_0-w_0}\leq \tfrac{1}{2}\eta_0^{15\beta/48}<\tfrac{1}{2}\eta_0^{14\beta/48}.
\end{equation}
\item Comparison of $v_0,\varrho$: recalling the definitions of $\varrho$ and $v_0$, we have
\begin{align*}
&\,\norm{\varrho-P_{y_0,f_{\wh x_0}^{-1}(x_1)}v_0}=
\norm{(df_{\wh x_0}^{-1})_{x_1}w_1-P_{y_0,f_{\wh x_0}^{-1}(x_1)}\circ (df_{\wh x_0}^{-1})_{y_1}v_1}\\
&=\norm{(df_{\wh x_0}^{-1})_{x_1}w_1-\widetilde{(df_{\wh x_0}^{-1})_{y_1}}[P_{y_1,x_1}{v_1}]}\\
&\leq \norm{(df_{\wh x_0}^{-1})_{x_1}}\norm{w_1-P_{y_1,x_1}{v_1}}+
\norm{\widetilde{(df_{\wh x_0}^{-1})_{x_1}}-\widetilde{(df_{\wh x_0}^{-1})_{y_1}}}\norm{v_1}. 
\end{align*}
By Lemma \ref{Lemma-pi}(1), 
$\norm{w_1-P_{y_1,x_1}{v_1}}\leq \tfrac{1}{2}\norm{v_1}\eta_1^{15\beta/48}\leq
d(x_0,\mathfs S)^{-a}\eta_0^{15\beta/48}$ for $\ve>0$ small enough,
since $\eta_1\leq e^{\ve}\eta_0$. Using assumptions (A6)--(A7) and
that $d(x_1,y_1)\leq 2\eta_1\leq 4\eta_0$, we conclude that
\begin{align*}
&\,\norm{\varrho-P_{y_0,f_{\wh x_0}^{-1}(x_1)}v_0}\leq d(x_0,\mathfs S)^{-2a}\eta_0^{15\beta/48}+
\mathfrak K(4\eta_0)^\beta d(x_0,\mathfs S)^{-a}\\
&\leq d(x_0,\mathfs S)^{-2a}\eta_0^{15\beta/48}+
4\mathfrak Kd(x_0,\mathfs S)^{-a}\eta_0^\beta \leq \ve^{1/8}\eta_0^{14\beta/48}+
4\ve^{1/16}\mathfrak K\eta_0^{95\beta/96}\\
&<\tfrac{1}{2}\eta_0^{14\beta/48}.
\end{align*}
\item Comparison of $w_0,\varrho$: the two estimates above imply that 
\begin{align*}
&\,\norm{\varrho-P_{x_0,f_{\wh x_0}^{-1}(x_1)}w_0}\leq
\norm{\varrho-P_{y_0,f_{\wh x_0}^{-1}(x_1)}v_0}+
\norm{P_{y_0,f_{\wh x_0}^{-1}(x_1)}v_0-P_{x_0,f_{\wh x_0}^{-1}(x_1)}w_0}\\
&=\norm{\varrho-P_{y_0,f_{\wh x_0}^{-1}(x_1)}v_0}+\norm{P_{y_0,x_0}v_0-w_0}<\eta_0^{14\beta/48}.
\end{align*}
\end{enumerate}
Now we estimate the fractions $\tfrac{\vertiii{v_0}}{\vertiii{\varrho}}$ and
$\tfrac{\vertiii{\varrho}}{\vertiii{w_0}}$.

\medskip
\noindent
{\sc Estimate of $\tfrac{\vertiii{\varrho}}{\vertiii{w_0}}$:} by the definition of $S$
and Lemma \ref{Lemma-pi}(1),
$$
\vertiii{w_0}=\norm{C(\wh x_0)^{-1}w_0}\geq \norm{w_0}\geq
{\rm exp}\left[-\eta_0^{15\beta/48}\right]\geq \tfrac{1}{2}\cdot
$$
Therefore, letting $A=C(\wh x_0)^{-1}$ and $B=C(\wh f^{-1}(\wh x_1))^{-1}$, we have
$$
\left|\tfrac{\vertiii{\varrho}}{\vertiii{w_0}}-1\right|\leq 2\left|\vertiii{\varrho}-\vertiii{w_0}\right|=
2\left|\norm{B\varrho}-\norm{Aw_0}\right|.
$$
Let $\widetilde B=B\circ P_{x_0,f_{\wh x_0}^{-1}(x_1)}$, so that $A,\widetilde B$
have the same domain (and codomain). Then
\begin{align*}
&\, \left|\norm{B\varrho}-\norm{Aw_0}\right|\leq
\left|\norm{B\varrho}-\|\widetilde Bw_0\|\right|+\norm{\widetilde Bw_0-Aw_0}\\
&\leq\norm{B\varrho-B\circ P_{x_0,f_{\wh x_0}^{-1}(x_1)} w_0}+
\norm{\widetilde Bw_0-Aw_0}\\
&\leq\norm{B}\norm{\varrho-P_{x_0,f_{\wh x_0}^{-1}(x_1)} w_0}+\norm{\widetilde B-A}\norm{w_0}
\end{align*}
By the overlap condition, Proposition \ref{Proposition-overlap}(2) and the comparison
of $w_0,\varrho$, the latter expression above is at most
$2\norm{A}\eta_0^{14\beta/48}+2\eta_0\leq 2\ve^{1/8}\eta_0^{13\beta/48}+2\eta_0<\tfrac{1}{4}\eta_0^{\beta/4}$.
Hence $\left|\tfrac{\vertiii{\varrho}}{\vertiii{w_0}}-1\right|<\tfrac{1}{2}\eta_0^{\beta/4}$ and so
\begin{equation}\label{1st-fraction}
\frac{\vertiii{\varrho}}{\vertiii{w_0}}={\rm exp}\left[\pm\eta_0^{\beta/4}\right].
\end{equation}

\medskip
\noindent
{\sc Estimate of $\tfrac{\vertiii{v_0}}{\vertiii{\varrho}}$:} by the latter estimate,
we need to prove that
\begin{equation}\label{2nd-fraction}
\frac{\vertiii{v_0}^2}{\vertiii{\varrho}^2}\leq{\rm exp}\left[2\xi-4Q(\wh x_0)^{\beta/4}\right].
\end{equation}
Recall our assumption that $\tfrac{\vertiii{v_1}}{\vertiii{w_1}}\leq e^{\xi}$ for
$\xi\geq\sqrt{\ve}$. By relation (\ref{relation-s}),
\begin{align*}
&\,\frac{\vertiii{v_0}^2}{\vertiii{\varrho}^2}=
\frac{2\norm{v_0}^2+e^{2\chi}\vertiii{v_1}^2}{2\norm{\varrho}^2+e^{2\chi}\vertiii{w_1}^2}
\leq\frac{2\norm{v_0}^2+e^{2\chi}e^{2\xi}\vertiii{w_1}^2}{2\norm{\varrho}^2+e^{2\chi}\vertiii{w_1}^2}\\
&=e^{2\xi}\left(\frac{\tfrac{2\norm{v_0}^2}{e^{2\xi}}+
e^{2\chi}\vertiii{w_1}^2}{2\norm{\varrho}^2+e^{2\chi}\vertiii{w_1}^2}\right)
=e^{2\xi}\left(1-\frac{2\norm{\varrho}^2-2\tfrac{\norm{v_0}^2}{e^{2\xi}}}{2\norm{\varrho}^2+e^{2\chi}\vertiii{w_1}^2}\right)\\
&=e^{2\xi}\left(1-2\frac{\norm{\varrho}^2-\tfrac{\norm{v_0}^2}{e^{2\xi}}}{\vertiii{\varrho}^2}\right)
\end{align*}
hence (\ref{2nd-fraction}) will follow once we prove that (remind that $\norm{v_0}=1$)
$$
\frac{\norm{\varrho}^2-e^{-2\xi}}{\vertiii{\varrho}^2}\geq 2Q(\wh x_0)^{\beta/4}
\iff \norm{\varrho}^2-e^{-2\xi}\geq 2\vertiii{\varrho}^2Q(\wh x_0)^{\beta/4}.
$$
We estimate the order of $\ve$ in both sides of this last inequality:
\begin{enumerate}[$\circ$]
\item By the comparison of $v_0,\varrho$, we have 
$|\norm{\varrho}-1|<\tfrac{1}{2}\eta_0^{14\beta/48}<Q(\wh x_0)^{\beta/4}<\ve^{3/2}$.
Since $\xi\geq\sqrt{\ve}$, we have $1-e^{-2\xi}\geq 1-e^{-2\sqrt{\ve}}>\ve^{1/2}$ and so
$$
\norm{\varrho}^2-e^{-2\xi}=\left(1-e^{-2\xi}\right)+\left(\norm{\varrho}^2-1\right)\geq 
\left(1-e^{-2\xi}\right)-\left|\norm{\varrho}^2-1\right|>\ve^{1/2}-3\ve^{3/2}.
$$
\item Since $\vertiii{\varrho}=\norm{B\varrho}\leq 2\norm{B}\leq 4\norm{A}$, we have
$$
2\vertiii{\varrho}^2Q(\wh x_0)^{\beta/4}\leq 32\norm{A}^2Q(\wh x_0)^{\beta/4}\leq 32\ve^{1/4}Q(\wh x_0)^{5\beta/24}
\leq 32\ve^{1/4}\ve^{5/4}=32\ve^{3/2}.
$$
\end{enumerate}
If $\ve>0$ is small enough then $\ve^{1/2}-3\ve^{3/2}>32\ve^{3/2}$. This proves (\ref{2nd-fraction})
which, together with (\ref{1st-fraction}), gives that
$\tfrac{\vertiii{v_0}}{\vertiii{w_0}}\leq {\rm exp}[\xi-Q(\wh x)^{\beta/4}]$. The proof for
stable sets is now complete.

The proof for unstable sets is the same, as we now explain. 
When dealing with the unstable direction, we usually need
to identify which inverse branches are being used. 
%For instance, that is why we defined $U$ as a function on the pair $(\wh x,v)$. 
So we would expect to
require that, for the proof of the improvement lemma, the inverse
branches associated with $\wh f(\wh x_0)$ and $\wh x_1$ are the same,
so that we can compare $U(\wh f(\wh x_0),\wh{df}w_0)$ and $U(\wh x_1,w_1)$.
This is indeed {\em not} necessary, since this comparison comes directly from the overlap
condition, which guarantees that these two values are almost the same! 
Therefore, the same proof applied for stable sets also works for
unstable sets. 
\end{proof}

As we just saw, the proof of the improvement lemma relies on the estimate
of the fractions $\tfrac{\vertiii{\varrho}}{\vertiii{w_0}}$ and
$\tfrac{\vertiii{v_0}}{\vertiii{\varrho}}$. The estimate of the first fraction is consequence
of the overlap condition, which allows to compare the matrices $A$ and $B$.
The overlap condition depends heavily on the parameter $\chi$: if we consider a different
parameter $\chi'\in(0,\chi)$, then it is not clear how the respective matrices $A$ and $B$
defined in terms of this new parameter are related (in particular, we might not have an 
overlap). On the other hand, the estimate of the second fraction is entirely dynamical, and does not depend 
on the comparison of $A$ and $B$. Because of this, if we consider a real orbit
as in Corollary \ref{Corollary-better-hyperbolicity}, then the functions $S_{\chi'},U_{\chi'}$ 
defined in terms of $\chi'$ satisfy an improvement. This fact was observed in \cite{Ben-Ovadia-2020}.
Let us recall the context of Corollary \ref{Corollary-better-hyperbolicity}.
For a fixed $\chi'\in (0,\chi)$, let $\delta=(e^{-\chi'}-e^{-\chi})^{3/\beta}<1$.
Fix $\wh x\in{\rm NUH}^*$, and let $\wh x_n=\wh f^n(\wh x)$, $p^s_n=q^s(\wh x_n)$,
$p^u_n=q^u(\wh x_n)$. The sequence $\un v=\{\Psi_{\wh x_n}^{\delta p^s_n,\delta p^u_n}\}_{n\in\Z}$
defines sets $\wh V^s_{\chi'}[\un v]$ and $\wh V^u_{\chi'}[\un v]$.
Clearly, if $\wh y\in \wh V^s_{\chi'}[\un v]$ then
$\wh f(\wh y)\in \wh f(\wh V^s_{\chi'}[\un v])\subset \wh V^s_{\chi'}[\sigma(\un v)]$.

\begin{corollary}\label{Corollary-improvement}
The following holds for all $\ve>0$ small enough. 
For $\wh y\in \wh V^s_{\chi'}[\un v]$, let $v_1\in T_{\wh f(\wh y)}\wh V^s_{\chi'}[\sigma(\un v)]$
and $v_0=\wh{df}^{(-1)}v_1\in T_{\wh y}\wh V^s_{\chi'}[\un v]$.
Also, let $w_0\in \wh{TM}_{\wh x}$ s.t. $v_0=\Theta^s_{\wh x,\wh y}(w_0)$
and $w_1\in \wh{TM}_{\wh f(\wh x)}$ s.t.  $v_1=\Theta^s_{\wh f(\wh x),\wh f(\wh y)}(w_1)$.
If $\tfrac{S_{\chi'}(\wh f(\wh y),v_1)}{S_{\chi'}(\wh f(\wh x),w_1)}={\rm exp}[\pm\xi]$
for $\xi\geq {\sqrt{\ve}}$, then 
$$
\frac{S_{\chi'}(\wh y,v_0)}{S_{\chi'}(\wh x,w_0)}={\rm exp}\left[\pm(\xi-Q(\wh x)^{\beta/4})\right].
$$
An analogous statement holds for unstable sets.
\end{corollary}

\begin{proof}
As in Lemma  \ref{Lemma-improvement}, it is enough to prove the result
for $\wh V^s_{\chi'}[{\un v}]$. We assume that $\norm{v_0}=1$ and
write $S_{\chi'}(\cdot,v)$ simply by $\vertiii{v}_{\chi'}$.
Again, we focus on one side of the estimate: assuming that
$\tfrac{\vertiii{v_1}_{\chi'}}{\vertiii{w_1}_{\chi'}}\leq {\rm exp}[\xi]$ for $\xi\geq\sqrt{\ve}$,
we will prove that $\tfrac{\vertiii{v_0}_{\chi'}}{\vertiii{w_0}_{\chi'}}\leq {\rm exp}[\xi-Q(\wh x)^{\beta/4}]$.
Note that, in the notation of Lemma \ref{Lemma-improvement}, we have $w_0=\varrho$,
and so we can obtain as in the estimate of $\tfrac{\vertiii{v_0}}{\vertiii{\varrho}}$ in
Lemma \ref{Lemma-improvement} that
\begin{align*}
&\,\frac{\vertiii{v_0}_{\chi'}^2}{\vertiii{w_0}_{\chi'}^2}=
\frac{2\norm{v_0}^2+e^{2\chi'}\vertiii{v_1}_{\chi'}^2}{2\norm{w_0}^2+e^{2\chi'}\vertiii{w_1}_{\chi'}^2}
\leq\frac{2\norm{v_0}^2+e^{2\chi'}e^{2\xi}\vertiii{w_1}_{\chi'}^2}{2\norm{w_0}^2+e^{2\chi'}\vertiii{w_1}_{\chi'}^2}\\
&=e^{2\xi}\left(1-2\frac{\norm{w_0}^2-\tfrac{\norm{v_0}^2}{e^{2\xi}}}{\vertiii{w_0}_{\chi'}^2}\right).
\end{align*}
Therefore, we just need to prove that
$$
\frac{\norm{w_0}^2-e^{-2\xi}}{\vertiii{w_0}_{\chi'}^2}\geq Q(\wh x)^{\beta/4}
\iff \norm{w_0}^2-e^{-2\xi}\geq \vertiii{w_0}_{\chi'}^2Q(\wh x)^{\beta/4}.
$$
Note that:
\begin{enumerate}[$\circ$]
\item $|\norm{w_0}-1|<\ve^{3/2}$ and so $\norm{w_0}^2-e^{-2\xi}>\ve^{1/2}-3\ve^{3/2}$.
\item Since
$\vertiii{w_0}_{\chi'}=\norm{C_{\chi'}(\wh x)^{-1}w_0}\leq \norm{C(\wh x)^{-1}w_0}\leq 2\norm{C(\wh x)^{-1}}$,
we have
$$
\vertiii{w_0}_{\chi'}^2Q(\wh x)^{\beta/4}\leq 4\norm{C(\wh x)^{-1}}^2Q(\wh x)^{\beta/4}\leq 4\ve^{1/4}Q(\wh x)^{5\beta/24}
\leq 4\ve^{3/2}.
$$
\end{enumerate}
If $\ve>0$ is small enough then $\ve^{1/2}-3\ve^{3/2}>4\ve^{3/2}$, which concludes the proof.
\end{proof}

The above corollary indeed works for the (potentially larger) invariant sets
associated to $\{\Psi_{\wh x_n}^{p^s_n,p^u_n}\}_{n\in\Z}$, but since our application will
require to diminish the Pesin charts by $\delta$ (so that Corollary \ref{Corollary-better-hyperbolicity} works),
we preferred to state it with $\delta$.

\subsection{Proof that $\wh x\in{\rm NUH}$}\label{Section-finite}

The first step in the proof of Theorem \ref{Thm-inverse} is guaranteeing that
$\wh x\in{\rm NUH}$. We start proving that the relevance of each symbol of
$\mathfs A$ implies the existence of stable/unstable sets
where $S/U$ are bounded. We continue employing the notation for tangent
bundles of stable/unstable sets. Recall the definitions of $s(\wh x),u(\wh x)$ in page \pageref{Def-NUH3}.
The result below is
\cite[Lemma 4.2 and Corollary 4.3]{Ben-Ovadia-2020}, adapted to our context.

\begin{lemma}\label{Lemma-finite-norm-1}
Let $\wh W^s=\wh V^s[\un w^+]$, where $\un w^+$ is a positive $\ve$--gpo.
If $\wh W^s\cap {\rm NUH}^\#\neq\emptyset$ then
$$
\sup_{\wh y\in \wh W^s}s(\wh y)<\infty.
$$
The same applies to negative $\ve$--gpo's, with respect to the function $u$.
\end{lemma}

\begin{proof}
We start proving the lemma for stable sets. For $\chi'\in(0,\chi)$, represent
$S_{\chi'}(\cdot,v)$ simply by $\vertiii{v}_{\chi'}$ and $S(\cdot,v)$ by $\vertiii{v}$. 

\medskip
\noindent
{\sc Claim 1:} For every $v$,
$\vertiii{v}=\sup_{\chi'<\chi}\vertiii{v}_{\chi'}$.
In particular, if $\vertiii{v}_{\chi'}\leq L$ for every $v\in T\wh W^s$ with $\|v\|=1$ and every $\chi'\in (0,\chi)$,
then $\sup_{\wh y\in \wh W^s}s(\wh y)\leq L$.

\medskip
The proof of Claim 1 is direct, and we leave the details to the reader.
By hypothesis, there is $\wh x\in\wh W^s\cap {\rm NUH}^\#$.
Write $\wh x_n=\wh f^n(\wh x)$, $p^s_n=q^s(\wh f^n(\wh x))$ and
$p^u_n=q^u(\wh f^n(\wh x))$. We have $\limsup_{n\to\infty}q(\wh x_n)>0$
and so there is $q>0$ and an increasing sequence
$\{n_k\}_{k\geq 0}$ s.t. $q(\wh x_{n_k})\geq q$ for all $k\geq 0$. By (\ref{estimates-Q})
and Lemma \ref{Lemma-q}, for every $\wh y\in{\rm NUH}^*$ we have
$\|C(\wh y)^{-1}\|<\ve^{1/8}Q(\wh y)^{-\beta/48}<q(\wh y)^{-\beta/48}$ and
so in particular $Q(\wh x_{n_k})>q$ and $\|C(\wh x_{n_k})^{-1}\|<q^{-\beta/48}$ for every $k\geq 0$.

\medskip
Fix $\chi'\in (0,\chi)$, and let $\overline{\chi}=\tfrac{\chi+\chi'}{2}$.
Also, let $\un v=\{v_n\}_{n\in\Z}$ with $v_n=\Psi_{\wh x_n}^{\delta p^s_n,\delta p^u_n}$, 
where $\delta=(e^{-\overline{\chi}}-e^{-\chi})^{3/\beta}<1$.
Write $\wh V^s_{\overline{\chi},n}=\wh V^s_{\overline{\chi}}[\{v_\ell\}_{\ell\geq n}]$.
We want to bound $\vertiii{v}_{\chi'}$, uniformly in $v$ with $\|v\|=1$ and $\chi'$.
Instead of $\wh W^s$, doing this for $\wh V^s_{\overline{\chi},0}$ is simpler,
since inside this latter set we have better estimates given by Corollary \ref{Corollary-better-hyperbolicity}.
Although $\wh W^s$ is in general not contained in $\wh V^s_{\overline{\chi},0}$, if
$n$ is large then $\wh f^n(\wh W^s)\subset \wh V^s_{\overline{\chi},n}$. The reason
for this inclusion is that, while by Proposition \ref{Prop-hyperbolicity-manifolds}(1)
the size of $\wh f^n(\wh W^s)$ is at most of the order of $e^{-\tfrac{\chi}{2}n}$,
the size of $\wh V^s_{\overline{\chi},n}$ is of the order of $q(\wh x_n)$, which decreases to zero
slower than $e^{-\ve n}$. Let us provide the details of this fact.

\medskip
\noindent
{\sc Claim 2:} If $n$ is large enough then $\wh f^n(\wh W^s)\subset \wh V^s_{\overline{\chi},n}$.

\begin{proof}[Proof of Claim $2$.] Write $\wh x=\{x_n\}_{n\in\Z}$ and let $W^s$ 
be the $s$--admissible manifold associated to $\wh W^s$.
We will check Proposition \ref{Prop-stable-manifolds}(3).
Since $B^d[\delta p^s_n]\times B^{m-d}[\delta p^s_n]\supset B[\delta p^s_n]$, it is enough to prove
that if $n$ is large enough then $f^n(W^s)\subset \Psi_{\wh x_n}(B[\delta p^s_n])$.
By Proposition \ref{Prop-stable-manifolds}(4), $f^n(W^s)\subset B(x_n,e^{-\frac{\chi}{2}n})$.
Since
$$
\Psi_{\wh x_n}(B[\delta p^s_n])\supset B\left(x_n,\tfrac{1}{2}\|C(\wh x_n)^{-1}\|^{-1}\delta p^s_n\right),
$$
it is enough to prove that for $n$ large enough
$$\tfrac{1}{2}\|C(\wh x_n)^{-1}\|^{-1}\delta p^s_n>e^{-\frac{\chi}{2}n} \iff
e^{-\frac{\chi}{2}n}\|C(\wh x_n)^{-1}\|\tfrac{1}{p^s_n}<\tfrac{\delta}{2}\cdot
$$
Since $\|C(\wh x_n)^{-1}\|(p^s_n)^{\beta/48}\leq \|C(\wh x_n)^{-1}\|Q(\wh x_n)^{\beta/48}<1$
and $p^s_n\geq e^{-\ve n}p^s_0$, we have 
\begin{align*}
&\ e^{-\frac{\chi}{2}n}\|C(\wh x_n)^{-1}\|\tfrac{1}{p^s_n}<
e^{-\frac{\chi}{2}n}\left(\tfrac{1}{p^s_n}\right)^{1+\frac{\beta}{48}}\leq 
e^{-\frac{\chi}{2}n}\left(\tfrac{e^{\ve n}}{p^s_0}\right)^{1+\frac{\beta}{48}}\\
&=\left(\tfrac{1}{p^s_0}\right)^{1+\frac{\beta}{48}}e^{-\left[\frac{\chi}{2}-\ve\left(1+\frac{\beta}{48}\right)\right]n}
\end{align*}
which, for $\ve>0$ small enough, converges to zero exponentially fast.
\end{proof}

Now we complete the proof of the lemma for $\wh W^s$. Start observing that, for every $n\geq 0$, 
Corollary \ref{Corollary-better-hyperbolicity} implies that if $v\in T\wh V^s_{\overline{\chi},n}$
with $\|v\|=1$ then
\begin{align*}
&\, \vertiii{v}_{\chi'}^2=2\sum_{\ell\geq 0}e^{2\ell \chi'}\|\wh{df}^{(\ell)}v\|^2
\leq 2\sum_{\ell\geq 0}e^{2\ell \chi'} 64e^{-2\ell \overline{\chi}}\|C(\wh x_n)^{-1}\|^2\\
&\leq
128\|C(\wh x_n)^{-1}\|^2\sum_{\ell\geq 0}e^{\ell(\chi'-\chi)}
\end{align*}
and so for every $k\geq 0$ and every $v\in T\wh V^s_{\overline{\chi},n_k}$ we obtain the bound
$$
\vertiii{v}_{\chi'}^2 \leq 128q^{-\beta/24}\sum_{\ell\geq 0}e^{\ell(\chi'-\chi)}.
$$
Call this bound $L^2$, so that $\vertiii{v}_{\chi'}\leq L$ for every $v\in T\wh V^s_{\overline{\chi},n_k}$
with $\|v\|=1$. Had $L$ been independent of $\chi'$, the proof would be complete, but unfortunately
$L\to\infty$ as $\chi'\to \chi$. The way to improve the estimate on $\vertiii{v}_{\chi'}$ to a bound
that is independent of $\chi'$ is by applying Corollary \ref{Corollary-improvement},
as we now explain. Define $\xi\geq \sqrt{\ve}$ by $e^\xi=\max\{\sqrt{2}L,e^{\sqrt{\ve}}\}$.
%For $v\in T\wh V^s_{\overline{\chi},n_k}$,
%write $\Theta^s_{n_k}(v)$ to represent $\Theta^s_{\wh x_{n_k},\wh y}(v)$. 
We claim that
$$
\tfrac{\vertiii{\Theta^s_{\wh x_{n_k},\wh y}(v)}_{\chi'}}{\vertiii{v}_{\chi'}}={\rm exp}[\pm \xi],
\ \ \text{ for all }\wh y\in\wh V^s_{\overline{\chi},n_k}\text{ and } v\in E^s_{\wh x_{n_k}}\backslash\{0\}.
$$
By a normalization, we just need to check this estimate for $\norm{v}=1$.
By Lemma \ref{Lemma-pi}(1), we have $\tfrac{1}{2}\leq \|\Theta^s_{\wh x_{n_k},\wh y}(v)\|\leq 2$
and so 
$$
(\sqrt{2}L)^{-1}=\tfrac{\sqrt{2}/2}{L}\leq\tfrac{\vertiii{\Theta^s_{\wh x_{n_k},\wh y}(v)}_{\chi'}}{\vertiii{v}_{\chi'}}
\leq\tfrac{2L}{\sqrt{2}}=\sqrt{2} L.
$$
Now fix $k\geq 1$. Apply Corollary \ref{Corollary-improvement} along the path
$\wh x_{n_{k-1}}\to\cdots\to \wh x_{n_k}$.
Since the ratio does not get worse for all transitions $\wh x_\ell\to \wh x_{\ell+1}$
and it improves a fixed amount in the last edge $\wh x_{n_{k-1}}\to \wh x_{n_{k-1}+1}$,
we conclude that
$$
\tfrac{\vertiii{\Theta^s_{\wh x_{n_{k-1}},\wh y}(v)}_{\chi'}}{\vertiii{v}_{\chi'}}={\rm exp}\left[\pm (\xi-q^{\beta/4})\right],
\ \ \text{ for all }\wh y\in\wh V^s_{\overline{\chi},n_{k-1}}\text{ and } v\in E^s_{\wh x_{n_{k-1}}}\backslash\{0\}.
$$
Repeating this procedure until reaching $\wh x_{n_0}$, we obtain
at least $k$ improvements, as long as the ratio remains outside $[{\rm exp}(-\sqrt{\ve}),{\rm exp}(\sqrt{\ve})]$.
Taking $k\to+\infty$, we conclude that
$\tfrac{\vertiii{\Theta^s_{\wh x_{n_0},\wh y}(v)}_{\chi'}}{\vertiii{v}_{\chi'}}={\rm exp}\left[\pm \sqrt{\ve}\right]$
for all $\wh y\in\wh V^s_{\overline{\chi},n_0}$ and $v\in E^s_{\wh x_{n_0}}\backslash\{0\}$.
Similarly, we obtain that for every $k\geq 0$ it holds
\begin{equation}\label{improved-estimate}
\tfrac{\vertiii{\Theta^s_{\wh x_{n_k},\wh y}(v)}_{\chi'}}{\vertiii{v}_{\chi'}}={\rm exp}[\pm \sqrt{\ve}],
\ \ \text{ for all }\wh y\in\wh V^s_{\overline{\chi},n_k}\text{ and } v\in E^s_{\wh x_{n_k}}\backslash\{0\}.
\end{equation}
\medskip
Now fix $k$ large enough s.t. $f^{n_k}(\wh W^s)\subset \wh V^s_{\overline{\chi},n_k}$.
Let $v\in T\wh W^s$ with $\norm{v}=1$.
If $w=\wh{df}^{(n_k)}v\in T\wh V^s_{\overline{\chi},n_k}$ then
\begin{align*}
&\,\vertiii{v}_{\chi'}^2=2\sum_{n\geq 0}e^{2n\chi'}\|\wh{df}^{(n)}v\|^2=
2\sum_{n=0}^{n_k-1}e^{2n\chi'}\|\wh{df}^{(n)}v\|^2+
2\sum_{n\geq n_k}e^{2n\chi'}\|\wh{df}^{(n)}v\|^2\\
&=2\sum_{n=0}^{n_k-1}e^{2n\chi'}\|\wh{df}^{(n)}v\|^2+e^{2n_k\chi'}\norm{w}^2\cdot\vertiii{\tfrac{w}{\|w\|}}_{\chi'}^2\\
&\leq 2\sum_{n=0}^{n_k-1}e^{2n\chi}\|\wh{df}^{(n)}v\|^2+e^{2n_k\chi}\norm{w}^2\cdot\vertiii{\tfrac{w}{\|w\|}}_{\chi'}^2.
\end{align*}
If $\wh z$ is the center of the zeroth chart defining $\wh W^s=\wh V^s[\un w^+]$,
then by Proposition \ref{Prop-hyperbolicity-manifolds}(1) the first term above is bounded
by $128\|C(\wh z)^{-1}\|^2\sum_{n=0}^{n_k-1}e^{n\chi}$.
To estimate the second term, write $w\in T_{\wh y}\wh V^s_{\overline\chi,n_k}$ and define
$\varrho$ by the equality $\tfrac{w}{\|w\|}=\Theta^s_{\wh x_{n_k},\wh y}(\varrho)$.
Using estimate (\ref{improved-estimate}), Lemma \ref{Lemma-pi}(1)
and Lemma \ref{Lemma-linear-reduction}(1), we get that
\begin{align*}
&\, \vertiii{\tfrac{w}{\|w\|}}_{\chi'}\leq e^{\sqrt{\ve}}\vertiii{\varrho}_{\chi'}\leq 2e^{\sqrt{\ve}}\vertiii{\tfrac{\varrho}{\|\varrho\|}}_{\chi'}\\
&\leq 2e^{\sqrt{\ve}}s(\wh x_{n_k})\leq 2e^{\sqrt{\ve}}\|C(\wh x_{n_k})^{-1}\|<2e^{\sqrt{\ve}}q^{-\beta/48}.
\end{align*}
Hence $e^{2n_k\chi}\norm{w}^2\cdot\vertiii{\tfrac{w}{\|w\|}}_{\chi'}^2\leq 256\|C(\wh z)^{-1}\|^2e^{n_k\chi+2\sqrt{\ve}}q^{-\beta/24}$,
and so
$$
\vertiii{v}_{\chi'}^2\leq 128\|C(\wh z)^{-1}\|^2\left[
\sum_{n=0}^{n_k-1}e^{n\chi }+2e^{n_k\chi+2\sqrt{\ve}}q^{-\beta/24}\right],
$$
which is independent of $v$ and $\chi'$. 
By Claim 1, the proof for $\wh W^s$ is complete.

\medskip
Now let $\wh W^u=\wh V^u[\un w^-]$ for a negative $\ve$--gpo
$\un w^-=\{\Psi_{\wh y_n}^{q^s_n,q^u_n}\}_{n\leq 0}$, and assume that
there exists $\wh x\in \wh W^u\cap{\rm NUH}^\#$. As before,
write $\wh x_n=\wh f^n(\wh x)$ and let
$\wh V^u_{\overline{\chi},n}=\wh V^u_{\overline{\chi}}[\{\Psi_{\wh x_k}^{\delta p^s_k,\delta p^u_k}\}_{k\leq n}]$
for $n\leq 0$. The proof for $\wh W^u$ is identical to the proof for $\wh W^s$,
provided the inverse branches used to define $\wh W^u$ are the same
as those used to define $\wh V^u_{\overline{\chi},n}$. In other words, we just need to verify that
$f_{\wh x_{n}}^{-1}$ and 
$f_{\wh y_{n}}^{-1}$ are the same inverse branch for all $n<0$. Fix $n<0$.
By Lemma \ref{Lemma-stable-set}(2),
$x_{n+1}\in \Psi_{\wh y_{n+1}}(B[20Q(\wh y_{n+1})])$.
Since $f_{\wh x_n}^{-1}(x_{n+1})=x_n=f_{\wh y_n}^{-1}(x_{n+1})$,
the uniqueness of inverse branches provided by assumption (A5) implies that
$f_{\wh x_{n}}^{-1}$ and $f_{\wh y_{n}}^{-1}$ are the same inverse branches.
\end{proof}

Now we proceed to show that, in the notation of Theorem \ref{Thm-inverse},
$\pi(\un v)\in {\rm NUH}$.
Recall the relevance property of each $\ve$--double chart of $\mathfs A$,
see Theorem \ref{Thm-coarse-graining}(3).
%Recall also that $\{\Psi_{\wh x_n}^{p^s_n,p^u_n}\}_{n\in\Z}\in\Sigma^\#$.

\begin{proposition}\label{Prop-finite-norm-2}
For every $v_0\in\mathfs A$, there exists a constant $L=L(v_0)$ s.t. the following holds.
If $\un v=\{v_n\}_{n\in\Z}\in\Sigma^\#$ satisfies $v_n=v_0$ for infinitely
many $n>0$ and if $\wh x=\pi(\un v)$, then $s(\wh x)<L$.
The same applies to $u(\wh x)$. In particular, $\wh x\in{\rm NUH}$. 
\end{proposition}

\begin{proof}
As usual, we represent $S(\cdot,v)$ simply by $\vertiii{v}$.
Write $\wh V^s=\wh V^s[\un v]$, $v_n=\Psi_{\wh x_n}^{p^s_n,p^u_n}$,
$\eta_n=p^s_n\wedge p^u_n$, and let $\{n_k\}_{k\geq 1}$ be an increasing
sequence s.t. $v_{n_k}=v_0$. Since $v_0$ is relevant, there is $\un w\in\Sigma$
with $w_0=v_0$ s.t. $\pi(\un w)\in {\rm NUH}^\#$. By Lemma \ref{Lemma-finite-norm-1},
if $\wh W^s=\wh V^s[\un w]$ then $\sup\limits_{\wh y\in\wh W^s}s(\wh y)<\infty$.
By Lemma \ref{Lemma-pi}(1),
$$
L_0:=\sup\left\{\tfrac{\vertiii{\Theta^s_{\wh x_0,\wh y}(v)}}{\vertiii{v}}:
\begin{array}{c}
\wh y\in \wh W^s, \vt[\wh y]\in\Psi_{\wh x_0}(B^d[\eta_0]\times B^{m-d}[\eta_0]),\\
v\in E^s_{\wh x_0}\setminus\{0\}
\end{array}
\right\}
$$
is also finite. Define $L_1=\max\{L_0,e^{\sqrt{\ve}}\}>1$.

For each $k\geq 1$, we have $\wh W^s\in {\wh{\mathfs M}}^s(v_{n_k})$.
Starting from $\wh W^s$, apply the stable graph transform along the path
$v_0\overset{\ve}{\to}v_1\overset{\ve}{\to}\cdots\overset{\ve}{\to}v_{n_k}$
to obtain some stable set at $v_0$, call it $\wh W^s_k$.
Let $G$ be the representing function of $\wh V^s$, and let
$G_k$ be the representing function of $\wh W^s_k$. By Proposition \ref{Prop-stable-manifolds}(1),
$\norm{G_k-G}_{C^1}\xrightarrow[k\to\infty]{}0$.

\medskip
\noindent
{\sc Claim:} If $\{w_k\}_{k\geq 1}\subset TM$ converges to $w\in TM$ in the Sasaki metric, then
$$
\sum_{n\geq 0}e^{2n\chi}\|\wh{df}^{(n)}w\|^2\leq \liminf_{k\to+\infty}\sum_{n\geq 0}e^{2n\chi}\|\wh{df}^{(n)}w_k\|^2.
$$

\begin{proof}[Proof of the claim.]
Fix $N>0$. For each $\delta>0$, there exists $k_0\geq 1$ s.t. 
$$
k\geq k_0\Longrightarrow \tfrac{\|\wh{df}^{(n)}w_k\|}{\|\wh{df}^{(n)}w\|}=e^{\pm\delta} \text{ for }n=0,1,\ldots,N.
$$
For $k\geq k_0$, we have 
$$
\sum_{n=0}^N e^{2n\chi}\|\wh{df}^{(n)}w\|^2\leq e^{2\delta}\sum_{n=0}^N e^{2n\chi}\|\wh{df}^{(n)}w_k\|^2
\leq e^{2\delta}\sum_{n\geq 0}e^{2n\chi}\|\wh{df}^{(n)}w_k\|^2
$$
and so 
$$
\sum_{n=0}^N e^{2n\chi}\|\wh{df}^{(n)}w\|^2\leq
e^{2\delta}\liminf_{k\to+\infty}\sum_{n\geq 0}e^{2n\chi}\|\wh{df}^{(n)}w_k\|^2.
$$
Since $\delta>0$ is arbitrary, we get that 
$$
\sum_{n=0}^N e^{2n\chi}\|\wh{df}^{(n)}w\|^2\leq
\liminf_{k\to+\infty}\sum_{n\geq 0}e^{2n\chi}\|\wh{df}^{(n)}w_k\|^2.
$$
Finally, take $N\to+\infty$ to conclude the proof of the claim.
\end{proof}

Write $\vt[\wh x]=\Psi_{\wh x_0}(z,G(z))$. Since $\wh x=\pi(\un v)$, 
by Lemma \ref{Lemma-admissible-manifolds} we have that $\|z\|<\eta_0$.
For each $k\geq 1$,
define $\wh y_k$ as the unique element of $\wh W^s_k$ s.t. $\vt[\wh y_k]=\Psi_{\wh x_0}(z,G_k(z))$.
Fix $v\in E^s_{\wh x_0}$ with $\norm{v}=1$.
If $v=(d\Psi_{\wh x_0})_0\begin{bmatrix} w \\ 0\end{bmatrix}$ then
$$
\Theta^s_{\wh x_0,\wh x}(v)=(d\Psi_{\wh x_0})_{(z,G(z))}\begin{bmatrix} w \\ (dG)_z w\end{bmatrix}\ \text{ and }\
\Theta^s_{\wh x_0,\wh y_k}(v)=(d\Psi_{\wh x_0})_{(z,G_k(z))}\begin{bmatrix} w \\ (dG_k)_z w\end{bmatrix}
$$
and so $\Theta^s_{\wh x_0,\wh y_k}(v)\to \Theta^s_{\wh x_0,\wh x}(v)$ in the Sasaki metric. By the claim,
$$
\vertiii{\Theta^s_{\wh x_0,\wh x}(v)}\leq \liminf_{k\to+\infty}\vertiii{\Theta^s_{\wh x_0,\wh y_k}(v)}
$$
and so it is enough to bound the right hand side above.

We claim that $\vt[\wh f^n(\wh y_k)]\in \Psi_{\wh x_n}(B^d[\eta_n]\times B^{m-d}(\eta_n))$
for $n=0,\ldots,n_k$. To prove this, note that $\wh f^n(\wh y_k)$ belongs to a stable set
and so it is enough to show that, in the charts representation, the first coordinate of
$\vt[\wh f^n(\wh y_k)]$ belongs to $B^d[\eta_n]$. The case $n=0$ is true
because $\|z\|<\eta_0$. The proof is by induction, so we just show how to obtain it for $n=1$.
Write $F_{\wh x_0,\wh x_1}=D(\wh x_0)+(H^+_1,H^+_2)$,
where $D(\wh x_0)$ is given by Lemma \ref{Lemma-linear-reduction}(2) and 
$H^+=(H^+_1,H^+_2)$ satisfies itens (a)--(c) in Theorem \ref{Thm-non-linear-Pesin-2}.
We have $\vt[\wh y_k]=\Psi_{\wh x_0}(z,G_k(z))$ and so
$\vt[\wh f(\wh y_k)]=\Psi_{\wh x_1}(\overline{z},*)$ where
$\overline z=D_s(\wh x_0)z+H^+_1(z,G_k(z))$. Then
\begin{align*}
&\, \norm{\overline z}\leq \norm{D_s(\wh x_0)z}+\norm{H^+_1(z,G_k(z))}\leq 
\norm{D_s(\wh x_0)z}+\norm{H^+(z,G_k(z))}\\
&\leq \norm{D_s(\wh x_0)}\norm{z}+\norm{H^+(0,0)}+
\norm{dH^+}_{C^0(B[2\eta_0])}\norm{(z,G_k(z))}\\
&\leq e^{-\chi}\eta_0+\ve \eta_0+\ve(2\eta_0)^{\beta/3}2\eta_0\leq (e^{-\chi}+3\ve)\eta_0
\leq (e^{-\chi}+3\ve)e^{\ve}\eta_1 
\end{align*}
is smaller than $\eta_1$ for $\ve>0$ small enough.
\medskip
Now, for fixed $k\geq 1$ write $\varrho_k=\Theta^s_{\wh x_0,\wh y_k}(v)$ and define $w_\ell\in E^s_{\wh x_\ell}$
by the equality $\Theta^s_{\wh x_\ell,\wh f^\ell(\wh y_k)}(w_\ell)=\wh{df}^{(\ell)}\varrho_k$,
for all $\ell\geq 0$.
Since $\wh f^{n_k}(\wh y_k)\in \wh f^{n_k}(\wh W^s_k)\subset \wh W^s$, we have
that $\tfrac{\vertiii{\wh{df}^{n_k}\varrho_k}}{\vertiii{w_{n_k}}}\leq L_1$.
By Lemma \ref{Lemma-improvement}, we get that
$\tfrac{\vertiii{\wh{df}^{n_k-1}\varrho_k}}{\vertiii{w_{n_k-1}}}\leq L_1$ and, repeating this
procedure, we  obtain that $\tfrac{\vertiii{\varrho_k}}{\vertiii{w_{0}}}\leq L_1$.
By Lemma \ref{Lemma-pi}(1), $\vertiii{\varrho_k}\leq L_1\vertiii{w_0}\leq 2L_1 s(\wh x_0)$
and so
\begin{equation}\label{estimate-finite-norm}
\vertiii{\Theta^s_{\wh x_0,\wh x}(v)}\leq 2L_1 s(\wh x_0).
\end{equation}
Defining $L=L(v_0):=3L_1s(\wh x_0)$ and applying Lemma \ref{Lemma-pi}(1) again,
it follows that $\vertiii{v}<L$ for all $v\in T_{\wh x}\wh V^s$ with $\norm{v}=1$,
and so $s(\wh x)<L$.

\medskip
The proof for $\wh V^u[\un v]$ is identical to the above since, 
as in the proof of Lemma \ref{Lemma-finite-norm-1}, the inverse branches
of the unstable sets involved in the proof coincide.
Finally, we show that $\wh x\in{\rm NUH}$.
Defining $E^s_{\wh x}=T_{\wh x}(\wh V^s[\un v])$ and
$E^u_{\wh x}=T_{\wh x}(\wh V^u[\un v])$, then what we just proved and
Proposition \ref{Prop-hyperbolicity-manifolds}(2) imply that $\wh x$ satisfies (NUH1)--(NUH3).
\end{proof}

Consequently, if $\un v\in\Sigma^\#$ then $\wh x=\pi(\un v)\in{\rm NUH}$ and so
the Lyapunov inner product $\vertiii{\cdot}$ is well-defined in $\wh{TM}_{\wh x}$,
as well as $C(\wh x),Q(\wh x)$. In the next sections we will prove parts (1)--(5) of
Theorem \ref{Thm-inverse}.

\subsection{Control of $d$ and $\rho$}

We start estimating $d(\vt[\wh x_0],\vt[\wh x])$. By Lemma \ref{Lemma-admissible-manifolds},
$\vt[\wh x]=\Psi_{\wh x_0}(w)$ for some $\norm{w}< 50^{-1}(p^s_0\wedge p^u_0)$. Since
$\vt[\wh x_0]=\Psi_{\wh x_0}(0)$ and $\Psi_{\wh x_0}$ is $2$--Lipschitz,
we get that $d(\vt[\wh x_0],\vt[\wh x])<25^{-1}(p^s_0\wedge p^u_0)$.
Now we estimate $\tfrac{\rho(\wh x_0)}{\rho(\wh x)}$. Such estimate
{\em does not} follow from Proposition \ref{Proposition-overlap}(1), since
we assume no overlap condition between $\wh x_0$ and $\wh x$. Nevertheless,
we can employ a calculation similar to that used to prove Proposition \ref{Proposition-overlap}(1).
Start observing that, by the above estimate, 
$d(\vt[\wh x_0],\vt[\wh x])<25^{-1}(p^s_0\wedge p^u_0)<\ve^2\rho(\wh x_0)^2<\ve^2d(\vt[\wh x_0],\mathfs S)$
and so 
$$
d(\vt[\wh x],\mathfs S)=d(\vt[\wh x_0],\mathfs S)\pm d(\vt[\wh x_0],\vt[\wh x])
=(1\pm\ve^2)d(\vt[\wh x_0],\mathfs S)=e^{\pm\ve}d(\vt[\wh x_0],\mathfs S).
$$
To estimate $\tfrac{d(\vt_1[\wh x],\mathfs S)}{d(\vt_1[\wh x_0],\mathfs S)}$,
start observing that, since $\vt[\wh x]\in D_{\wh x_0}$, assumption (A6) implies that
$d(\vt_1[\wh x],\vt_1[\wh x_0])<\ve^2\rho(\wh x_0)<\ve^2d(\vt_1[\wh x_0],\mathfs S)$
and so
\begin{align*}
&\ d(\vt_1[\wh x],\mathfs S)=d(\vt_1[\wh x_0],\mathfs S)\pm d(\vt_1[\wh x_0],\vt_1[\wh x])\\
&=(1\pm\ve^2)d(\vt_1[\wh x_0],\mathfs S)=e^{\pm\ve}d(\vt_1[\wh x_0],\mathfs S).
\end{align*}
We estimate of $\tfrac{d(\vt_{-1}[\wh x],\mathfs S)}{d(\vt_{-1}[\wh x_0],\mathfs S)}$
in the same way, using that $\vt[\wh x]\in E_{\wh f^{-1}(\wh x_0)}$.
Altogether, we conclude that $\tfrac{\rho(\wh x_0)}{\rho(\wh x)}=e^{\pm\ve}$.
This proves part (1) of Theorem \ref{Thm-inverse}.

\subsection{Control of $C^{-1}$}

Recall that $\wh x=\pi(\un v)$ for $\un v\in\Sigma^\#$.
We proved in Section \ref{Section-finite} that $\wh x\in {\rm NUH}$, 
i.e. there is a splitting $\wh{TM}_{\wh x}=E^s_{\wh x}\oplus E^u_{\wh x}$
satisfying (NUH1)--(NUH3). To control $C^{-1}$, we need to control the Lyapunov
inner product for vectors in $E^s_{\wh x}$ and $E^u_{\wh x}$. We explain
how to make the control in $E^s_{\wh x}$ (the control in $E^u_{\wh x}$ is analogous).
Write $\un v=\{v_n\}_{n\in\Z}$ and $\Theta_n=\Theta_{\wh x_n,\wh f^n(\wh x)}$. Without loss of generality,
assume that $v_{0}$ repeats infinitely often in the future, i.e. there is an increasing sequence
$\{n_k\}_{k\geq 1}$ s.t. $v_{n_k}=v_0$ for all $k\geq 1$.
Let $L=3L_1s(\wh x_0)$ as in Proposition \ref{Prop-finite-norm-2},
and let $\xi>0$ s.t. $L=e^{\xi}$. Since $L_1\geq e^{\sqrt{\ve}}$ and $s(\wh x_0)\geq \sqrt{2}$,
we have $\xi>\sqrt{\ve}$. We claim that
\begin{equation}\label{estimate-for-improvement}
\frac{\vertiii{v}}{\vertiii{\Theta_{n_k}(v)}}={\rm exp}[\pm \xi],
\ \text{ for all }v\in E^s_{\wh x_0}\backslash\{0\}.
\end{equation}
By a normalization, we just need to check this for $\norm{v}=1$.
We proved in Proposition \ref{Prop-finite-norm-2}
that $\vertiii{\Theta_{n_k}(v)}\leq 2L_1s(\wh x_0)$, see estimate
(\ref{estimate-finite-norm}). On one hand, applying Lemma \ref{Lemma-pi}(1) we have
$\tfrac{\vertiii{v}}{\vertiii{\Theta_{n_k}(v)}}\leq 2s(\wh x_0)<L$,
and on the other hand
$\tfrac{\vertiii{v}}{\vertiii{\Theta_{n_k}(v)}}\geq \frac{\sqrt{2}}{2L_1s(\wh x_0)}>L^{-1}$,
which proves (\ref{estimate-for-improvement}). Now fix $k\geq 1$.
Using (\ref{estimate-for-improvement}), apply
Lemma \ref{Lemma-improvement} along the path
$v_{n_{k-1}}\overset{\ve}{\to}\cdots\overset{\ve}{\to}v_{n_k}$.
Since the ratio does not get worse for all edges $v_{\ell}\overset{\ve}{\to}v_{\ell+1}$
and it improves a fixed amount in the last edge $v_{n_{k-1}}\overset{\ve}{\to}v_{n_{k-1}+1}$,
we conclude that
$\tfrac{\vertiii{v}}{\vertiii{\Theta_{n_{k-1}}(v)}}={\rm exp}\left[\pm (\xi-Q(\wh x_0)^{\beta/4})\right]$
for all $v\in E^s_{\wh x_0}\backslash\{0\}$. Repeating this procedure until reaching $v_0$, we obtain
at least $k$ improvements, as long as the ratio remains outside $[{\rm exp}(-\sqrt{\ve}),{\rm exp}(\sqrt{\ve})]$.
Taking $k\to+\infty$, we conclude that
$\frac{\vertiii{v}}{\vertiii{\Theta_0(v)}}={\rm exp}[\pm \sqrt{\ve}]$
for all $v\in E^s_{\wh x_0}\backslash\{0\}$.
Since $v_{n_k}=v_0$, we obtain similarly that
$\frac{\vertiii{v}}{\vertiii{\Theta_{n_k}(v)}}={\rm exp}[\pm \sqrt{\ve}]$
for all $v\in E^s_{\wh x_0}\backslash\{0\}$. Finally, given $n\in\Z$, let $n_k>n$ and apply
Lemma \ref{Lemma-improvement} along the path $v_{n}\overset{\ve}{\to}\cdots\overset{\ve}{\to}v_{n_k}$
to conclude that
\begin{equation}\label{comparison-C-1}
\frac{\vertiii{v}}{\vertiii{\Theta_n(v)}}={\rm exp}[\pm \sqrt{\ve}],
\ \text{ for all }v\in E^s_{\wh x_n}\backslash\{0\}.
\end{equation}
By a similar argument, we conclude that
\begin{equation}\label{comparison-C-2}
\frac{\vertiii{v}}{\vertiii{\Theta_n(v)}}={\rm exp}[\pm \sqrt{\ve}],
\ \text{ for all }v\in E^u_{\wh x_n}\backslash\{0\}.
\end{equation}

We are now ready to prove part (2) of Theorem \ref{Thm-inverse}.
For simplicity, assume $n=0$ and write $\Theta=\Theta_0$.
If $v=v^s+v^u\in E^s_{\wh x_0}\oplus E^u_{\wh x_0}$,
then $\Theta(v)=\Theta(v^s)+\Theta(v^u)\in E^s_{\wh x}\oplus E^u_{\wh x}$.
By the calculation made in the proof of Lemma \ref{Lemma-linear-reduction}(1) 
and estimates (\ref{comparison-C-1}) and (\ref{comparison-C-2}),
\begin{equation}\label{comparison-C-3}
\frac{\norm{C(\wh x_0)^{-1}v}^2}{\norm{C(\wh x)^{-1}\Theta(v)}^2}=
\frac{\vertiii{v^s}^2+\vertiii{v^u}^2}{\vertiii{\Theta(v^s)}^2+\vertiii{\Theta(v^u)}^2}={\rm exp}[\pm 2\sqrt{\ve}]
\end{equation}
and so $\frac{\|C(\wh x_0)^{-1}\|}{\norm{C(\wh x)^{-1}\Theta}}={\rm exp}[\pm\sqrt{\ve}]$.
By Lemma \ref{Lemma-pi}(2), if $\ve>0$ is small enough then
$\norm{\Theta^{\pm 1}}={\rm exp}[\pm\sqrt{\ve}]$, hence
$\frac{\|C(\wh x_0)^{-1}\|}{\|C(\wh x)^{-1}\|}={\rm exp}[\pm2\sqrt{\ve}]$.

\subsection{Control of $Q,p^s,p^u$ and proof that $\wh x\in{\rm NUH}^\#$}

Now we prove parts (3) and (4) of Theorem \ref{Thm-inverse}.
We begin controlling $Q$. As usual, let $n=0$.
Recall that $Q(\wh x)=\max\{Q\in I_\ve:Q\leq \widetilde Q(\wh x)\}$, where
$$
\widetilde Q(\wh x)= \ve^{6/\beta}\min\left\{\norm{C(\wh x)^{-1}}^{-48/\beta},\rho(\wh x)^{96a/\beta}\right\}.
$$
By part (1),
$\tfrac{\rho(\wh x_0)^{96a/\beta}}{\rho(\wh x)^{96a/\beta}}=
{\rm exp}\left[\pm\tfrac{96a\ve}{\beta}\right]={\rm exp}\left[\pm\tfrac{96\sqrt{\ve}}{\beta}\right]$
for $\ve>0$ small enough. By part (2),
$\tfrac{\|C(\wh x_0)^{-1}\|^{-48/\beta}}{\|C(\wh x)^{-1}\|^{-48/\beta}}={\rm exp}\left[\pm \tfrac{96\sqrt{\ve}}{\beta}\right]$.
Hence, $\tfrac{\widetilde Q(\wh x_0)}{\widetilde Q(\wh x)}={\rm exp}\left[\pm \tfrac{96\sqrt{\ve}}{\beta}\right]$
and so $\tfrac{Q(\wh x_0)}{Q(\wh x)}={\rm exp}\left[\pm\left( \tfrac{96\sqrt{\ve}}{\beta}+\tfrac{\ve}{3}\right)\right]$.
If $\ve>0$ is small enough then $\tfrac{Q(\wh x_0)}{Q(\wh x)}={\rm exp}\left[\pm\sqrt[3]{\ve}\right]$.

Now we prove part (4). We show how to control $p^s$ (the proof for $p^u$ is analogous).
By (GPO2), we have
$p^s_0=\delta_\ve\inf\{e^{\ve n}Q(\wh x_n):n\geq 0\}$ and so by part (3)
\begin{align*}
q^s(\wh x)=\delta_\ve\inf\{e^{\ve n}Q(\wh f^n(\wh x)):n\geq 0\}=p^s_0\times {\rm exp}\left[\pm\sqrt[3]{\ve}\right].
\end{align*}
Hence $q^s(\wh x)>0$ and $\tfrac{p^s_0}{q^s(\wh x)}= {\rm exp}\left[\pm\sqrt[3]{\ve}\right]$.

Finally, observe that we also have
$\tfrac{p^s_n\wedge p^u_n}{q(\wh f^n(\wh x))}= {\rm exp}\left[\pm\sqrt[3]{\ve}\right]$.
Since $\un v\in\Sigma^\#$, it follows that $\limsup\limits_{n\to+\infty}q(\wh f^n(\wh x))>0$
and $\limsup\limits_{n\to-\infty}q(\wh f^n(\wh x))>0$, thus proving that $\wh x\in{\rm NUH}^\#$.

\subsection{Control of $\Psi_{\wh f^n(\wh x)}^{-1}\circ\Psi_{\wh x_n}$}

Fix $n=0$, and let $x_0=\vt[\wh x_0]$, $x=\vt[\wh x]$, $\eta=p^s_0\wedge p^u_0$,
$\Theta=\Theta_{\wh x_0,\wh x}$, and $P=P_{x_0,x}$.
We start proving that $\Psi_{\wh x}^{-1}\circ\Psi_{\wh x_0}$ is well-defined in $B[10Q(\wh x_0)]$.
We have
$$
\Psi_{\wh x_0}(B[10Q(\wh x_0)])\subset B(x_0,20Q(\wh x_0))\subset B(x,20Q(\wh x_0)+d(x_0,x))
\subset\mathfrak B_{\wh x},
$$
where in the last inclusion we used parts (1) and (3) to obtain that
$20Q(\wh x_0)+d(x_0,x)<20Q(\wh x_0)+\tfrac{\eta}{25}<25Q(\wh x_0)\leq
25e^{\sqrt[3]{\ve}}Q(\wh x)<\mathfrak d(\wh x)$.
By assumption (A1), $\Psi_{\wh x}^{-1}\circ\Psi_{\wh x_0}$ is well-defined in $B[10Q(\wh x_0)]$.

Now we represent $\Psi_{\wh x}^{-1}\circ\Psi_{\wh x_0}$ as required. 
We have $\Psi_{\wh x}^{-1}\circ\Psi_{\wh x_0}=C(\wh x)^{-1}\circ \Phi\circ C(\wh x_0)$,
where $\Phi=\exp{x}^{-1}\circ \exp{x_0}$. The composition
$C(\wh x)^{-1}\circ \Theta\circ C(\wh x_0)$ has norm close to one. Indeed, 
by (\ref{comparison-C-3}) we have
$\|C(\wh x)^{-1}\circ \Theta\circ C(\wh x_0)v\|=e^{\pm 2\sqrt{\ve}}\|C(\wh x_0)^{-1}\circ C(\wh x_0)v\|=
e^{\pm 2\sqrt{\ve}}\norm{v}$. By the polar decomposition for matrices,
$C(\wh x)^{-1}\circ \Theta\circ C(\wh x_0)=OR$ where $O$ is an orthogonal matrix
and $R$ is positive symmetric with $\norm{Rv}=e^{\pm 2\sqrt{\ve}}\norm{v}$ for all $v\in\R^m$.
Since $C(\wh x)^{-1}\circ \Theta\circ C(\wh x_0)$ preserves the splitting $\R^d\times\R^{m-d}$,
the same holds for $O$. Also, diagonalizing $R$ and estimating its eigenvalues,
we get that if $\ve>0$ is small enough then $\norm{R-{\rm Id}}\leq 4\sqrt{\ve}$,
see details in \cite[pp. 100]{Ben-Ovadia-2019}.
Define $\delta=(\Psi_{\wh x}^{-1}\circ\Psi_{\wh x_0})(0)\in\R^m$, and $\Delta:B[10Q(\wh x_0)]\to \R^m$
s.t. $\Psi_{\wh x}^{-1}\circ\Psi_{\wh x_0}=\delta+O+\Delta$.
We start estimating $d\Delta$. For $z\in B[10Q(\wh x_0)]$,
\begin{align*}
&\,(d\Delta)_z=C(\wh x)^{-1}\circ (d\Phi)_{C(\wh x_0)z}\circ C(\wh x_0)-O\\
&=C(\wh x)^{-1}\circ \underbrace{\left[(d\Phi)_{C(\wh x_0)z}-\Theta\right]}_{=E}\circ C(\wh x_0)+OR-O.
\end{align*}
To estimate $E$, observe that:
\begin{enumerate}[$\circ$]
\item By lemma \ref{Lemma-pi}(2), $\norm{P-\Theta}\leq \tfrac{1}{2}\eta^{15\beta/48}$.
\item By assumptions (A2)--(A3),
\begin{align*}
&\, \norm{(d\Phi)_{C(\wh x_0)z}-P}=
\norm{\widetilde{(d\exp{x}^{-1})_{\Psi_{\wh x_0}(z)}(d\exp{x_0})_{C(\wh x_0)z}}-\widetilde{{\rm Id}}}\\
&=\norm{\widetilde{(d\exp{x}^{-1})_{\Psi_{\wh x_0}(z)}(d\exp{x_0})_{C(\wh x_0)z}}-
\widetilde{(d\exp{x_0}^{-1})_{\Psi_{\wh x_0}(z)}(d\exp{x_0})_{C(\wh x_0)z}}}\\
&\leq \norm{\widetilde{(d\exp{x}^{-1})_{\Psi_{\wh x_0}(z)}}-
\widetilde{(d\exp{x_0}^{-1})_{\Psi_{\wh x_0}(z)}}}\cdot\norm{(d\exp{x_0})_{C(\wh x_0)z}}\\
&\leq 2d(x_0,\mathfs S)^{-a}d(x,x_0)
\end{align*}
which, by part (1), is bounded by
$\tfrac{2}{25}d(x_0,\mathfs S)^{-a}\eta\leq \tfrac{2}{25}\rho(\wh x_0)^{-a}\eta
\ll \tfrac{1}{2}\eta^{15\beta/48}$.
\end{enumerate}
Hence, $\norm{E}<\eta^{15\beta/48}$ and so, by part (2), we get that
\begin{align*}
&\,\norm{C(\wh x)^{-1}\circ E\circ C(\wh x_0)}\leq \norm{C(\wh x)^{-1}}\eta^{15\beta/48}
\leq e^{2\sqrt{\ve}}\norm{C(\wh x_0)^{-1}}\eta^{15\beta/48}\\
&\leq e^{2\sqrt{\ve}}\ve^{1/8}\eta^{14\beta/48}\ll \sqrt{\ve}.
\end{align*}
Since $\norm{OR-O}=\norm{R-{\rm Id}}\leq 4\sqrt{\ve}$, we conclude that
$\norm{(d\Delta)_z}\leq 5\sqrt{\ve}$. In particular, since $\Delta(0)=0$, we have
$\norm{\Delta(z)}\leq \norm{d\Delta}_{C^0}\norm{z}\leq 5\sqrt{\ve}\norm{z}$.

Finally, we estimate $\norm{\delta}$. Let $z$ s.t. $\Psi_{\wh x_0}(z)=x$.
We have $0=(\Psi_{\wh x}^{-1}\circ\Psi_{\wh x_0})(z)=\delta+Oz+\Delta(z)$
and so $\delta=-Oz-\Delta(z)$. By Lemma \ref{Lemma-admissible-manifolds}
we have $\norm{z}< 50^{-1}\eta$, therefore for $\ve>0$ small enough
$$
\norm{\delta}\leq \norm{Oz}+\norm{\Delta(z)}\leq (1+5\sqrt{\ve})\norm{z}
\leq \tfrac{1+5\sqrt{\ve}}{50}\eta<25^{-1}\eta.
$$

\subsection{Proof of Corollary \ref{Corollary-inverse}}

Parts (1)--(2) are consequences of parts (1) and (4) of Theorem \ref{Thm-inverse}.
We prove part (3) as we proved Theorem \ref{Thm-inverse}(5), with the required
modifications.
Write $\pi(\un v)=\pi(\un w)=\wh x$. Assume $n=0$, and let
$x=\vt[\wh x]$, $x_0=\vt[\wh x_0]$, $y_0=\vt[\wh y_0]$, $\eta=p^s_0\wedge p^u_0$,
$\overline\eta=q^s_0\wedge q^u_0$. By Theorem \ref{Thm-inverse}(4),
$\eta/\overline\eta=e^{\pm 2\sqrt[3]{\ve}}$. Let
$$
\begin{array}{ccc}
A=\Theta_{\wh y_0,\wh x}^{-1},& B=\Theta_{\wh x_0,\wh x},& \Theta=A\circ B\\
&&\\
C=P_{y_0,x}^{-1},& D=P_{x_0,x},& P=C\circ D.
\end{array}
$$
The proof that $\Psi_{\wh y_0}^{-1}\circ\Psi_{\wh x_0}$ is well-defined in $B[10Q(\wh x_0)]$
is the same. Now, we have $\Psi_{\wh y_0}^{-1}\circ\Psi_{\wh x_0}=C(\wh y_0)^{-1}\circ \Phi\circ C(\wh x_0)$,
where $\Phi=\exp{y_0}^{-1}\circ \exp{x_0}$, and 
$\|C(\wh y_0)^{-1}\circ \Theta\circ C(\wh x_0)v\|=e^{\pm 4\sqrt{\ve}}\norm{v}$ for all $v$.
Hence $C(\wh y_0)^{-1}\circ \Theta\circ C(\wh x_0)=OR$ where $O$ is an orthogonal matrix 
preserving the splitting $\R^d\times\R^{m-d}$ and $R$ is positive symmetric with
$\norm{R-{\rm Id}}\leq 8\sqrt{\ve}$. Let $\delta=(\Psi_{\wh y_0}^{-1}\circ\Psi_{\wh x_0})(0)\in\R^m$ and
$\Delta$ s.t. $\Psi_{\wh y_0}^{-1}\circ\Psi_{\wh x_0}=\delta+O+\Delta$.
If $y\in B[10Q(\wh x_0)]$ then
$$
(d\Delta)_z=C(\wh y_0)^{-1}\circ E\circ C(\wh x_0)+OR-O
$$
where $E=(d\Phi)_{C(\wh x_0)z}-\Theta$. Note that:
\begin{enumerate}[$\circ$]
\item By lemma \ref{Lemma-pi}(2), $\norm{A^{-1}-C^{-1}}\leq \tfrac{1}{2}\overline\eta^{15\beta/48}$
and $\norm{B-D}\leq \tfrac{1}{2}\eta^{15\beta/48}$.
We also have $\norm{A},\norm{B},\norm{C},\norm{D}<2$.
Since $A-C=A(C^{-1}-A^{-1})C$, we have
$$
\norm{A-C}\leq 4\norm{A^{-1}-C^{-1}}\leq 2\overline\eta^{15\beta/48}
$$
and so 
\begin{align*}
&\,\norm{P-\Theta}=\norm{AB-CD}\leq \norm{A}\norm{B-D}+\norm{A-C}\norm{D}\\
&\leq \eta^{15\beta/48}+4\overline\eta^{15\beta/48}\leq (e^{2\sqrt[3]{\ve}}+4)\overline\eta^{15\beta/48}
<6\overline\eta^{15\beta/48}
\end{align*}
for $\ve>0$ small enough.
\item Proceeding similarly as before and using part (1), we have
\begin{align*}
&\,\norm{(d\Phi)_{C(\wh x_0)z}-P}\leq 2d(y_0,\mathfs S)^{-a}d(x_0,y_0)
\leq 2\rho(\wh y_0)^{-a}\tfrac{1}{10}\max\{\eta,\overline\eta\}\\
&\leq \tfrac{1}{5}e^{2\sqrt[3]{\ve}}\rho(\wh y_0)^{-a}\overline\eta\ll \overline\eta^{15\beta/48}.
\end{align*}
\end{enumerate}
Hence $\norm{E}<7\overline\eta^{15\beta/48}$ and so
\begin{align*}
&\,\norm{C(\wh y_0)^{-1}\circ E\circ C(\wh x_0)}\leq 7\norm{C(\wh y_0)^{-1}}\overline\eta^{15\beta/48}
\leq 7\ve^{1/8}\overline\eta^{14\beta/48}\ll \sqrt{\ve},
\end{align*}
which implies that $\norm{(d\Delta)_z}< 9\sqrt{\ve}$.
To estimate $\norm{\delta}$ we let $z$ s.t. $\Psi_{\wh x_0}(z)=x$.
By Lemma \ref{Lemma-admissible-manifolds} we have $\norm{z}< 50^{-1}\eta$
and $\|(\Psi_{\wh y_0}^{-1}\circ\Psi_{\wh x_0})(z)\|< 50^{-1}\overline\eta$, and so
\begin{align*}
&\,\norm{\delta}\leq \|(\Psi_{\wh y_0}^{-1}\circ\Psi_{\wh x_0})(z)\|+\norm{Oz}+\norm{\Delta(z)}\leq
\|(\Psi_{\wh y_0}^{-1}\circ\Psi_{\wh x_0})(z)\|+(1+9\sqrt{\ve})\norm{z}\\
&< 50^{-1}\overline\eta +50^{-1}(1+9\sqrt{\ve})\eta 
\leq 50^{-1}\left[1+(1+9\sqrt{\ve})e^{2\sqrt[3]{\ve}}\right]\overline\eta<10^{-1}\overline\eta
\end{align*}
for $\ve>0$ small enough.

\section{Symbolic dynamics}

We are approaching the completion of the proof of Theorem \ref{Thm-Main}.
In this section, we will employ the same methods used in
\cite{Sarig-JAMS,Ben-Ovadia-2019} to our context (the results are stated in terms of $\wh M$).
This part is more abstract than the material of the previous sections, and whenever
possible we will just invoke the results of \cite{Sarig-JAMS,Ben-Ovadia-2019}.
When needed, we will provide the complete details.
%We also wrote the previous part
%of the paper so that, now, we can more easily rely on the results that were proved,
%and use these results in a more concise way. In this way, we expect that the next 
%sections are the more self-contained as possible.
In the sequel, we fix $\chi>0$ and
for $\ve>0$ small enough we let $(\Sigma,\sigma,\pi)$ be the triple given by
Proposition \ref{Prop-pi}.
Most of the results below will use Lemma \ref{Lemma-stable-set}, Proposition \ref{Prop-stable-sets},
Theorem \ref{Thm-coarse-graining}, Proposition \ref{Prop-pi}, and Corollary \ref{Corollary-inverse}.

\subsection{A countable Markov partition}

We use Corollary \ref{Corollary-inverse} to construct a cover of
${\rm NUH}^\#$ that is locally finite and satisfies a (symbolic) Markov property.

\medskip
\noindent
{\sc The Markov cover $\mathfs Z$:} Let $\mathfs Z:=\{Z(v):v\in\mathfs A\}$, where
$$
Z(v):=\{\pi(\un v):\un v\in\Sigma^\#\text{ and }v_0=v\}.
$$

\medskip
In other words, $\mathfs Z$ is the family defined by the natural partition of $\Sigma^\#$ into
cylinder at the zeroth position. Using stable/unstable sets, we define {\em invariant sets}
inside each $Z\in\mathfs Z$. Let $Z=Z(v)$.

\medskip
\noindent
{\sc $s$/$u$--sets in $\mathfs Z$:} Given $\wh x\in Z$, let $\wh W^s(\wh x,Z):=\wh V^s[\{v_n\}_{n\geq 0}]\cap Z$
be the {\em $s$--set} of $\wh x$ in $Z$ for some (any) $\un v=\{v_n\}_{n\in\Z}\in\Sigma^\#$
s.t. $\pi(\un v)=\wh x$ and $v_0=v$. Similarly, let $\wh W^u(\wh x,Z):=\wh V^u[\{v_n\}_{n\leq 0}]\cap Z$ be
the {\em $u$--set} of $\wh x$ in $Z$.

\medskip
By Proposition \ref{Prop-stable-sets}(5), the definitions above do not depend on the choice of $\un v$, 
and any two $s$--sets ($u$--sets) either coincide or are disjoint. We also
define $\wh V^s(\wh x,Z):=\wh V^s[\{v_n\}_{n\geq 0}]$ and $\wh V^u(\wh x,Z):=\wh V^u[\{v_n\}_{n\leq 0}]$.
Here are the main properties of $\mathfs Z$.

\begin{proposition}\label{Prop-Z}
The following holds for all $\ve>0$ small enough.
\begin{enumerate}[{\rm (1)}]
\item {\sc Covering property:} $\bigcup_{Z\in\mathfs Z}Z={\rm NUH}^\#$.
\item {\sc Local finiteness:} For every $Z\in\mathfs Z$, $\#\{Z'\in\mathfs Z:Z\cap Z'\neq\emptyset\}<\infty$.
\item {\sc Product structure:} For every $Z\in\mathfs Z$ and every $\wh x,\wh y\in Z$, the intersection
$\wh W^s(\wh x,Z)\cap \wh W^u(\wh y,Z)$ consists of a single point, and this point belongs to $Z$.
\item {\sc Symbolic Markov property:} If $\wh x=\pi(\un v)$ with $\un v=\{v_n\}_{n\in\Z}\in\Sigma^\#$, then
$$
\wh f(\wh W^s(\wh x,Z(v_0)))\subset \wh W^s(\wh f(\wh x),Z(v_1))\, 
\text{ and }\, \wh f^{-1}(\wh W^u(\wh f(\wh x),Z(v_1)))\subset \wh W^u(\wh x,Z(v_0)).
$$
\end{enumerate}
\end{proposition}

\begin{proof}
By Proposition \ref{Prop-pi}(4), we have $\pi[\Sigma^\#]\supset{\rm NUH}^\#$. On the other
hand, Theorem \ref{Thm-inverse} implies that $\pi[\Sigma^\#]\subset{\rm NUH}^\#$,
hence part (1) follows.
Part (2) is proved as in \cite[Theorem 10.2]{Sarig-JAMS}, using 
Corollary \ref{Corollary-inverse}(2) and Theorem \ref{Thm-coarse-graining}(1).
Part (3) is proved as in \cite[Proposition 10.5]{Sarig-JAMS}, using
Proposition \ref{Prop-stable-sets}(1). Part (4) is proved as in \cite[Prop. 10.9]{Sarig-JAMS},
using Lemma \ref{Lemma-stable-set} and Proposition \ref{Prop-stable-sets}(3).
\end{proof}

For $\wh x,\wh y\in Z$, let $[\wh x,\wh y]_Z:=$ intersection point of $\wh W^s(\wh x,Z)$
and $\wh W^u(\wh y,Z)$, and call it the {\em Smale bracket} of $\wh x,\wh y$ in $Z$.

\begin{lemma}\label{Lemma-fibres}
The following holds for all $\ve>0$ small enough.
\begin{enumerate}[{\rm (1)}]
\item {\sc Compatibility:} If $\wh x,\wh y\in Z(v_0)$ and $\wh f(\wh x),\wh f(\wh y)\in Z(v_1)$ with
$v_0\overset{\ve}{\to} v_1$ then $\wh f([\wh x,\wh y]_{Z(v_0)})=[\wh f(\wh x),\wh f(\wh y)]_{Z(v_1)}$.
\item {\sc Overlapping charts properties:} If $Z=Z(\Psi_{\wh x_0}^{p_0^s,p_0^u}),
Z'=Z(\Psi_{\wh y_0}^{q_0^s,q_0^u})\in\mathfs Z$
with $Z\cap Z'\neq \emptyset$ then:
\begin{enumerate}[{\rm (a)}]
\item $\vt[Z]\subset \Psi_{\wh y_0}(B[q_0^s\wedge q_0^u])$.
\item If $\wh x\in Z\cap Z'$ then $\wh W^{s/u}(\wh x,Z)\subset \wh V^{s/u}(\wh x,Z')$. 
\item If $\wh x\in Z,\wh y\in Z'$ then $\wh V^s(\wh x,Z)$ and $\wh V^u(\wh y,Z')$ intersect at a unique point. 
\end{enumerate}
\end{enumerate}
\end{lemma}

\begin{proof}
Part (1) is proved as in \cite[Lemma 10.7]{Sarig-JAMS}, using Proposition \ref{Prop-stable-sets}.
The proof of part (2) is based on \cite[Lemmas 10.8 and 10.10]{Sarig-JAMS} and
\cite[Lemmas 5.8 and 5.10]{Ben-Ovadia-2019}. Since the referred proofs are not simple applications
of previous results, and since they use two norms (the euclidean and supremum norms),
we will provide the complete details below.

Let $p=p_0^s\wedge p_0^u$ and $q=q_0^s\wedge q_0^u$.
Since $Z\cap Z'\neq\emptyset$, Corollary \ref{Corollary-inverse} 
implies that $\tfrac{p}{q}=e^{\pm 2\sqrt[3]{\ve}}$ and that
$\Psi_{\wh y_0}^{-1}\circ\Psi_{\wh x_0}=\delta+O+\Delta$ on $B[10Q(\wh x_0)]$,
where $\delta\in\R^m$  satisfies $\norm{\delta}<10^{-1}q$,
$O$ is an orthogonal linear map preserving the splitting $\R^d\times\R^{m-d}$,
and $\Delta:B[10Q(\wh x_0)]\to \R^m$ satisfies
$\Delta(0)=0$ and $\norm{d\Delta}_{C^0}< 9\sqrt{\ve}$ on $B[10Q(\wh x_0)]$.
In particular, $\norm{\Delta(v)}\leq 9\sqrt{\ve}\norm{v}$ for all $v\in B[10Q(\wh x_0)]$.

\medskip
\noindent
(a) By Lemma \ref{Lemma-admissible-manifolds} and the definition of $Z$, we have 
$\vt[Z]\subset\Psi_{\wh x_0}(B[50^{-1}p])$. If $v\in B[50^{-1}p]$ then
\begin{align*}
&\, \|(\Psi_{\wh y_0}^{-1}\circ\Psi_{\wh x_0})(v)\|\leq \norm{\delta}+\norm{Ov}+\norm{\Delta(v)}
\leq \norm{\delta}+(1+9\sqrt{\ve})\norm{v}\\
&\leq 10^{-1}q+(1+9\sqrt{\ve})50^{-1}p\leq \left[10^{-1}+50^{-1}(1+9\sqrt{\ve})e^{2\sqrt[3]{\ve}}\right]q<q
\end{align*}
for $\ve>0$ small enough, and so
$\vt[Z] \subset\Psi_{\wh y_0}[(\Psi_{\wh y_0}^{-1}\circ\Psi_{\wh x_0})(B[50^{-1}p])]
\subset\Psi_{\wh y_0}(B[q])$.

\medskip
\noindent
(b) Let $\un v=\{v_n\}_{n\in\Z}=\{\Psi_{\wh x_n}^{p^s_n,p^u_n}\}_{n\in\Z}$
and $\un w=\{w_n\}_{n\in\Z}=\{\Psi_{\wh y_n}^{q^s_n,q^u_n}\}_{n\in\Z}$ in $\Sigma^\#$
s.t. $\wh x=\pi(\un v)=\pi(\un w)$. We have to show that
$\wh W^{s/u}(\wh x,Z)\subset \wh V^{s/u}(\wh x,Z')$. Let us start proving the statement
for $\wh W^s$. Fix $n\geq 0$. By Proposition \ref{Prop-pi}(3), we have
$\wh f^n(\wh x)=\pi[\sigma^n(\un v)]=\pi[\sigma^n(\un w)]$ and so $\wh f^n(\wh x)\in Z(v_n)\cap Z(w_n)$.
By item (a) above, $\vt[Z(v_n)]\subset\Psi_{\wh y_n}(B[q_n^s\wedge q_n^u])\subset\Psi_{\wh y_n}(B[20Q(\wh y_n)])$.
Now, Proposition \ref{Prop-Z}(4) implies that
$\wh f^n(\wh W^s(\wh x,Z))\subset \wh W^s(\wh f^n(\wh x), Z(v_n))\subset Z(v_n)$ and so
$$
\vt_n\left[\wh W^s(\wh x,Z)\right]=\vt\left[\wh f^n(\wh W^s(\wh x,Z))\right]\subset 
\vt[Z(v_n)]\subset \Psi_{\wh y_n}(B[20Q(\wh y_n)]).
$$
By Lemma \ref{Lemma-stable-set}, it follows that
$\wh W^s(\wh x,Z)\subset \wh V^s[\{w_n\}_{n\geq0}]=\wh V^s(\wh x, Z')$.
The proof for $\wh W^u$ is analogous, once we pay attention that the equality $\wh x=\pi(\un v)=\pi(\un w)$
implies that the inverse branches associated to $\un v$ and $\un w$ coincide.

\medskip
\noindent
(c) We apply methods very similar to those used in the appendix.
%(but instead of applying the dynamics of $f$ to admissible
%manifolds we just change their representation from one chart to another).
Let $\wh V^s(\wh x,Z)=\wh V^s[\un v^+]$ and $\wh V^u(\wh y,Z')=\wh V^u[\un w^-]$.
Any point in $\wh V^s(\wh x,Z)\cap \wh V^u(\wh y,Z')$ is defined by its zeroth position,
so we need to show that $V^s[\un v^+]$ and $V^u[\un w^-]$ intersect at a single point.
Write $V^s=V^s[\un v^+]$, and let $G:B^d[p^s_0]\to \R^{m-d}$ be the representing function of $V^s$, i.e.
$V^s=\Psi_{\wh x_0}\{(v,G(v)):v\in B^d[p^s_0]\}$. Recall that $G$ satisfies (AM1)--(AM3)
of Section \ref{Sec-graph-transform}.
By item (a), we can express $V^s$
in the chart $\Psi_{\wh y_0}$ by
\begin{align*}
V^s&=\left[\Psi_{\wh y_0}\circ(\Psi_{\wh y_0}^{-1}\circ \Psi_{\wh x_0})\right]\left\{(v,G(v)):v\in B^d[p^s_0]\right\}\\
&=\Psi_{\wh y_0}\left\{\delta+O(v,G(v))+\Delta(v,G(v)):v\in B^d[p^s_0]\right\}.
\end{align*}
We want to express $\delta+O(v,G(v))+\Delta(v,G(v))$ as a graph. With respect to the splitting
$\R^d\times\R^{m-d}$, write $\delta=(\delta_1,\delta_2)$ and $\Delta=(\Delta_1,\Delta_2)$.
Since $O$ preserves such splitting, if we write $w=Ov$ then 
\begin{align*}
&\, \delta+O(v,G(v))+\Delta(v,G(v))=(Ov+\Delta_1(v,G(v))+\delta_1,OG(v)+\Delta_2(v,G(v))+\delta_2)\\
&=(w+\overline{\Delta}_1(w)+\delta_1,OGO^{-1}w+\overline\Delta_2(w)+\delta_2)=
(\Gamma(w),OGO^{-1}w+\overline\Delta_2(w)+\delta_2),
\end{align*}
where $\overline\Delta_1(w)=\Delta_1(O^{-1}w,GO^{-1}w)$,
$\overline\Delta_2(w)=\Delta_2(O^{-1}w,GO^{-1}w)$ and $\Gamma={\rm Id}+\overline\Delta_1+\delta_1$
are defined in  $B^d[p^s_0]$. We start collecting some estimates for $\overline\Delta_1,\overline\Delta_2,\Gamma$:
\begin{enumerate}[$\circ$]
\item $\|\overline\Delta_i(0)\|=\norm{\Delta_i(0,G(0))}\leq \norm{d\Delta_i}_{C^0}\norm{G(0)}\leq
9\sqrt{\ve}10^{-3}p<10^{-2}\sqrt{\ve}p$.
\item $\|d\overline\Delta_i\|_{C^0}\leq \norm{d\Delta_i}_{C^0}(1+\norm{dG}_{C^0})<9\sqrt{\ve}(1+\ve)<10\sqrt{\ve}$.
\item $\|\Gamma(0)\|\leq \|\overline\Delta_1(0)\|+\|\delta_1\|<10^{-2}\sqrt{\ve}p+10^{-1}q<\tfrac{1}{6}(p\wedge q)$,
since $\tfrac{p}{q}=e^{\pm 2\sqrt[3]{\ve}}$.
\end{enumerate}
In particular, if $v\in B^d[p^s_0]$ then
$$\|\overline\Delta_1(v)\|\leq \|d\overline\Delta_1\|_{C^0}\|v\|+\|\overline\Delta_1(0)\|
\leq 10\sqrt{\ve}p^s_0+10^{-2}\sqrt{\ve}p<20\sqrt{\ve}p^s_0$$
and so $\|\overline\Delta_1+\delta_1\|_{C^0}\leq 20\sqrt{\ve}p^s_0+10^{-1}q<20\ve^{3/2}+10^{-1}\ve<\ve$.
By Lemma \ref{App-Lemma-Holder-norm}(2), $\Gamma$ has an inverse
$\Phi:\Gamma(B^d[p^s_0])\to B^d[p^s_0]$. We can also estimate $\|d\Gamma\|_{C^0}$
and $\|d\Phi\|_{C^0}$ in their domains as definition, as follows:
\begin{enumerate}[$\circ$]
\item $d\Gamma={\rm Id}+d\overline\Delta_1$ and so
$\|d\Gamma\|_{C^0}=1\pm\|d\overline\Delta_1\|_{C^0}=1\pm 10\sqrt{\ve}=e^{\pm 15\sqrt{\ve}}$,
since $e^{-15\sqrt{\ve}}<1-10\sqrt{\ve}<1+10\sqrt{\ve}<e^{15\sqrt{\ve}}$ for $\ve>0$ small enough.
\item Letting $A=(d\overline\Delta_1)_v$, we have
$(d\Phi)_{\Gamma(v)}=[(d\Gamma)_v]^{-1}=[{\rm Id}+A]^{-1}={\rm Id}+\sum_{k\geq 1}(-1)^k A^k$.
Noting that $\norm{\sum_{k\geq 1}(-1)^kA^k}\leq \sum_{k\geq 1}\|A\|^k<\tfrac{10\sqrt{\ve}}{1-10\sqrt{\ve}}<12\sqrt{\ve}$,
we conclude that
$\|d\Phi\|_{C^0}=1\pm 12\sqrt{\ve}=e^{\pm 15\sqrt{\ve}}$,
since $e^{-15\sqrt{\ve}}<1-12\sqrt{\ve}<1+12\sqrt{\ve}<e^{15\sqrt{\ve}}$ for $\ve>0$ small enough.
\end{enumerate}
Introducing the variable $\wt w=\Gamma(w)$, we obtain that 
$$
V^s=\Psi_{\wh y_0}\left\{(\wt w,\wt G(\wt w)):\wt w\in\Gamma(B^d[p^s_0])\right\},
$$
where $\wt G:\Gamma(B^d[p^s_0])\to \R^{m-d}$ is defined by
$\wt G=(OGO^{-1}+\overline\Delta_2+\delta_2)\circ\Phi$.
Now we estimate $\|\wt G(0)\|$ and $\|d\wt G\|_{C^0}$. Since 
$$
\|\Phi(0)\|=\|\Phi(0)-\Phi(\Gamma(0))\|\leq \|d\Phi\|_{C^0}\|\Gamma(0)\|\leq e^{15\sqrt{\ve}}\tfrac{1}{6}(p\wedge q)
$$
and
\begin{align*}
&\,\|OGO^{-1}\Phi(0)\|=\|GO^{-1}\Phi(0)\|\leq \|GO^{-1}\Phi(0)-G\underbrace{O^{-1}\Phi(\Gamma(0))}_{=0}\|+\|G(0)\|\\
&\leq \|dG\|_{C^0}\|d\Phi\|_{C^0}\|\Gamma(0)\|+\|G(0)\|\leq \ve e^{15\sqrt{\ve}}\tfrac{1}{6}(p\wedge q)+10^{-3}p\\
&\leq \left(\tfrac{1}{6}\ve e^{15\sqrt{\ve}}+10^{-3}e^{2\sqrt[3]{\ve}}\right)(p\wedge q)<10^{-2}(p\wedge q),
\end{align*}
we obtain that
\begin{align*}
&\,\|\wt G(0)\|\leq \|OGO^{-1}\Phi(0)\|+\|\overline\Delta_2(\Phi(0))\|+\|\delta_2\|\\
&\leq \|OGO^{-1}\Phi(0)\|+\|d\overline\Delta_2\|_{C^0}\|\Phi(0)\|+\|\overline\Delta_2(0)\|+\|\delta_2\|\\
&\leq 10^{-2}(p\wedge q)+10\sqrt{\ve} e^{15\sqrt{\ve}}\tfrac{1}{6}(p\wedge q)+10^{-2}\sqrt{\ve}p+10^{-1}q\\
&\leq \left(10^{-2}+2\sqrt{\ve}e^{15\sqrt{\ve}}+10^{-2}\sqrt{\ve}e^{2\sqrt[3]{\ve}}+10^{-1}e^{2\sqrt[3]{\ve}}\right)(p\wedge q)
<\tfrac{1}{6}(p\wedge q),
\end{align*}
where in the last inequality we used that $\ve>0$ is small enough.
The estimate of $d\wt G$ is simpler:
\begin{align*}
&\,\|d\wt G\|_{C^0}\leq \|O\circ dG\circ O^{-1}+d\overline\Delta_2\|_{C^0}\|d\Phi\|_{C^0}
\leq \left( \|dG\|_{C^0}+\|d\overline\Delta_2\|_{C^0}\right)\|d\Phi\|_{C^0}\\
&\leq (\ve+10\sqrt{\ve})e^{15\sqrt{\ve}}<15\sqrt{\ve},
\end{align*}
where again in the last inequality we used that $\ve>0$ is small enough.

We finally prove item (c). We proceed similarly to the proof of 
Lemma \ref{App-Lemma-admissible-graphs}(1).
Write $V^u=V^u[\wh w^-]$, and let $H:B^{m-d}[q^u_0]\to \R^d$
be its representing function. The intersection $V^s\cap V^u$
corresponds to points $(\wt w,v)$ satisfying the system 
$$
\left\{
\begin{array}{l}
\, v=\wt G(\wt w)\\
\wt w=H(v)
\end{array}
\right.
$$
for which necessarily $\wt w$ is a fixed point of $H\circ\wt G$.
Note the following:
\begin{enumerate}[$\circ$]
\item If $\wt w\in B^d[\tfrac{2}{3}p]$ then $\|\wt G(\wt w)\|\leq \|d\wt G\|_{C^0}\|\wt w\|+\|\wt G(0)\|
\leq 15\sqrt{\ve}\tfrac{2}{3}p+\tfrac{1}{6}(p\wedge q)<\tfrac{2}{3}p$.
\item If $v\in B^{m-d}[\tfrac{2}{3}p]$ then $\|H(v)\|\leq \|dH\|_{C^0}\|v\|+\|H(0)\|\leq \ve\tfrac{2}{3}p+10^{-3}q<\tfrac{2}{3}p$.
\end{enumerate}
Therefore 
$$
B^d[\tfrac{2}{3}p] \ \xrightarrow{\ \wt G\ } \ B^{m-d}[\tfrac{2}{3}p]\ \xrightarrow{\ H\ }\ B^d[\tfrac{2}{3}p],
$$
so the restriction of $H\circ\wt G$ to $B^d[\tfrac{2}{3}p]$ is well-defined. The derivative of this composition
is at most $\|dH\|_{C^0}\|d\wt G\|_{C^0}\leq 15\ve^{3/2}<1$, so it has a fixed point
$\wt w\in B^d[\tfrac{2}{3}p]$. For this fixed point, $v=\wt G(\wt w)\in B^{m-d}[\tfrac{2}{3}p]\subset B^{m-d}[q^u_0]$,
and so we obtain a point in $V^s\cap V^u$. 
It remains to prove the uniqueness. Note that $\Gamma(B^d[p^s_0])\subset B^d[2p^s_0]$.
Leting $a=(2p^s_0)\vee q^u_0$, apply the Kirszbraun theorem to extend $\wt G$ and $H$ to functions
in the domains $B^d[a]$ and $B^{m-d}[a]$ with Lipschitz constants bounded by
$15\sqrt{\ve}$ and $\ve$ respectively. The same calculation above implies that 
$$
B^d[a] \ \xrightarrow{\ \widetilde G\ } \ B^{m-d}[a] \ \xrightarrow{\ H\ }\ B^d[a]
$$
and that $H\circ\wt G$ is a contraction, hence $V^s\cap V^u$ consists of a single point.
\end{proof}

\medskip
The next step is to refine $\mathfs Z$ to destroy its non-trivial intersections. 
The result will be a partition of ${\rm NUH}^\#$ by sets with the
(geometrical) Markov property. This method, first developed by Sina{\u\i} and Bowen
for finite covers \cite{Sinai-Construction-of-MP,Sinai-MP-U-diffeomorphisms,Bowen-LNM},
also works for countable covers satisfying the local finiteness property \cite{Sarig-JAMS}.
Write $\mathfs Z=\{Z_1,Z_2,\ldots\}$.

\medskip
\noindent
{\sc The Markov partition $\mathfs R$:} For every $Z_i,Z_j\in\mathfs Z$, define a partition of $Z_i$ by:
\begin{align*}
T_{ij}^{su}&=\{\wh x\in Z_i: \wh W^s(\wh x,Z_i)\cap Z_j\neq\emptyset,
\wh W^u(\wh x,Z_i)\cap Z_j\neq\emptyset\}\\
T_{ij}^{s\emptyset}&=\{\wh x\in Z_i: \wh W^s(\wh x,Z_i)\cap Z_j\neq\emptyset,
\wh W^u(\wh x,Z_i)\cap Z_j=\emptyset\}\\
T_{ij}^{\emptyset u}&=\{\wh x\in Z_i: \wh W^s(\wh x,Z_i)\cap Z_j=\emptyset,
\wh W^u(\wh x,Z_i)\cap Z_j\neq\emptyset\}\\
T_{ij}^{\emptyset\emptyset}&=\{\wh x\in Z_i: \wh W^s(\wh x,Z_i)\cap Z_j=\emptyset,
\wh W^u(\wh x,Z_i)\cap Z_j=\emptyset\}.
\end{align*}
Let $\mathfs T:=\{T_{ij}^{\alpha\beta}:i,j\geq 1,\alpha\in\{s,\emptyset\},\beta\in\{u,\emptyset\}\}$,
and let $\mathfs R$ be the partition generated by $\mathfs T$.
Clearly $T_{ii}^{su}=Z_i$, hence $\mathfs R$ is a partition of ${\rm NUH}^\#$.
Using Lemma \ref{Lemma-fibres}(2)(b), we prove that $T^{su}_{ij}=Z_i\cap Z_j$.
Also, proceeding as in \cite[Proposition 11.2]{Sarig-JAMS},
$\mathfs R$ is the partition defined by the following equivalence relation in ${\rm NUH}^\#$:
$$
\wh x\sim \wh y\iff\text{for all }Z,Z'\in\mathfs Z:
\left\{
\begin{array}{rcl}
\wh x\in Z&\Leftrightarrow & \wh y\in Z\\
\wh W^s(\wh x,Z)\cap Z'\neq\emptyset&\Leftrightarrow &\wh W^s(\wh y,Z)\cap Z'\neq\emptyset\\
\wh W^u(\wh x,Z)\cap Z'\neq\emptyset&\Leftrightarrow &\wh W^u(\wh y,Z)\cap Z'\neq\emptyset.
\end{array}
\right.
$$
By definition, $\mathfs R$ is a refinement of $\mathfs Z$. By Corollary \ref{Corollary-inverse}(2)
and Theorem \ref{Thm-coarse-graining}(1), $\mathfs R$ possesses two local finiteness properties:
\begin{enumerate}[$\circ$]
\item For every $Z\in\mathfs Z$, $\#\{R\in\mathfs R:R\subset Z\}<\infty$.
\item For every $R\in\mathfs R$, $\#\{Z\in\mathfs Z:Z\supset R\}<\infty$.
\end{enumerate}
Now we show that $\mathfs R$ is a Markov partition in the sense of
Sina{\u\i} \cite{Sinai-MP-U-diffeomorphisms}. 

\medskip
\noindent
{\sc $s$/$u$--sets in $\mathfs R$:} Given $\wh x\in R\in\mathfs R$, we define the {\em $s$--set}
and {\em $u$--set} of $\wh x$ by:
\begin{align*}
\wh W^s(\wh x,R):=
\bigcap_{T_{ij}^{\alpha\beta}\in\mathfs T\atop{T_{ij}^{\alpha\beta}\supset R}} \wh W^s(\wh x,Z_i)\cap T_{ij}^{\alpha\beta}
\, \text{ and }\, \wh W^u(\wh x,R):=
\bigcap_{T_{ij}^{\alpha\beta}\in\mathfs T\atop{T_{ij}^{\alpha\beta}\supset R}} \wh W^u(\wh x,Z_i)\cap T_{ij}^{\alpha\beta}.
\end{align*}

We have that $\wh W^{s/u}(\wh x,R)\subset R$, $\wh W^s(\wh x,R)\cap \wh W^u(\wh x,R)=\{\wh x\}$, 
and any two $s$--sets ($u$--sets) either coincide or are disjoint. The proofs are made
as in \cite[Prop. 11.5(1)--(2)]{Sarig-JAMS}, using Proposition \ref{Prop-stable-sets}(5).

\begin{proposition}\label{Prop-R}
The following are true.
\begin{enumerate}[{\rm (1)}]
\item {\sc Product structure:} For every $R\in\mathfs R$ and every $\wh x,\wh y\in R$, the intersection
$\wh W^s(\wh x,R)\cap \wh W^u(\wh y,R)$ consists of a single point, and this
point belongs to $R$. Denote it by $[\wh x,\wh y]$.
\item {\sc Hyperbolicity:} If $\wh y,\wh z\in \wh W^s(\wh x,R)$ then
$d(\wh f^n(\wh y),\wh f^n(\wh z))\xrightarrow[n\to\infty]{}0$, and
if $\wh y,\wh z\in \wh W^u(\wh x,R)$ then $d(\wh f^n(\wh y),\wh f^n(\wh z))\xrightarrow[n\to-\infty]{}0$.
The rates are exponential.
\item {\sc Markov property:} Let $R_0,R_1\in\mathfs R$. If $\wh x\in R_0$ and
$\wh f(\wh x)\in R_1$ then 
$$
\wh f(\wh W^s(\wh x,R_0))\subset \wh W^s(\wh f(\wh x),R_1)\, \text{ and }\,
\wh f^{-1}(\wh W^u(\wh f(x),R_1))\subset \wh W^u(\wh x,R_0).
$$
\end{enumerate}
\end{proposition}

\begin{proof}
Part (1) is proved as in \cite[Proposition 11.5(3)]{Sarig-JAMS}, using Lemma \ref{Lemma-fibres}(2)(b).
Along the proof we also obtain that $[\wh x,\wh y]=[\wh x,\wh y]_Z$ for every 
$Z\in\mathfs Z$ containing $R$.
Part (2) is proved as in \cite[Proposition 11.5(4)]{Sarig-JAMS}, using
Proposition \ref{Prop-stable-sets}(4). Finally, part (3) is proved as in \cite[Proposition 11.7]{Sarig-JAMS},
using Lemma \ref{Lemma-fibres}(1), Proposition \ref{Prop-Z}(4), Lemma \ref{Lemma-fibres}(2)(b)
and the equality $[\wh x,\wh y]=[\wh x,\wh y]_Z$ for all $\wh x,\wh y\in R$ and  
all $Z\in\mathfs Z$ containing $R$.
\end{proof}

\subsection{A finite-to-one Markov extension}

Using the Markov partition $\mathfs R$, we now define a
symbolic coding that satisfies Theorem \ref{Thm-Main}.
Let $\widehat{\mathfs G}=(\widehat V,\widehat E)$ be the oriented graph with vertices
$\widehat V=\mathfs R$ and edges
$\widehat E=\{R\to S:R,S\in\mathfs R\text{ s.t. }\wh f(R)\cap S\neq\emptyset\}$,
and let $(\widehat\Sigma,\widehat\sigma)$ be the TMS induced by $\widehat{\mathfs G}$.
Since the ingoing and outgoing degrees of every vertex in $\Sigma$ is finite,
the same occurs to $\widehat\Sigma$, see \cite[Proposition 12.3]{Sarig-JAMS}.

For $\ell\in\Z$ and a path $R_m\to\cdots\to R_n$ on $\widehat{\mathfs G}$, define
$$
_\ell[R_m,\ldots,R_n]:=\wh f^{-\ell}(R_m)\cap\cdots\cap\wh f^{-\ell-(n-m)}(R_n),
$$
that represents the set of points whose itinerary
from $\ell$ to $\ell+(n-m)$ visits the rectangles $R_m,\ldots,R_n$. The crucial property that
gives the new coding is that $_\ell[R_m,\ldots,R_n]\neq\emptyset$. This is a consequence
of the Markov property of $\mathfs R$ (Proposition \ref{Prop-R}(3)), and the proof can be 
made by induction as in \cite[Lemma 12.1]{Sarig-JAMS}.
By Proposition \ref{Prop-R}(2), if $\{R_n\}_{n\in\Z}\in\wh\Sigma$ then
$\bigcap_{n\geq 0}\overline{_{-n}[R_{-n},\ldots,R_n]}$
is the intersection of a descending chain of nonempty closed sets with
diameters converging to zero. We can therefore consider the following definition.

\medskip
\noindent
{\sc The map $\widehat\pi:\widehat\Sigma\to \wh M$:} Given $\un R=\{R_n\}_{n\in\Z}\in\widehat\Sigma$,
$\widehat\pi(\un R)$ is defined by the identity
$$
\{\widehat\pi(\un R)\}:=\bigcap_{n\geq 0}\overline{_{-n}[R_{-n},\ldots,R_n]}.
$$

\medskip
As we will prove, $(\widehat\Sigma,\widehat\sigma,\widehat\pi)$ is a triple satisfying
Theorem \ref{Thm-Main}. To show this, we start relating $\pi$ and $\wh\pi$.
Consider the cylinder sets in $\Sigma$: for $\ell\in\Z$ and a path
$v_m\overset{\ve}{\to}\cdots\overset{\ve}{\to}v_n$ in $\Sigma$, let
$$
Z_\ell[v_m,\ldots,v_n]:=\{\pi(\un w):\un w\in\Sigma^\#\text{ and }w_\ell=v_m,\ldots,w_{\ell+(n-m)}=v_n\}.
$$
By Proposition \ref{Prop-pi}(2), the diameter of $Z_{-n}[v_{-n},\ldots,v_n]$ goes to
zero as $n\to\infty$.

\begin{lemma}\label{Lemma-relation-R-Z}
The following are true.
\begin{enumerate}[{\rm (1)}]
\item For each $\un R=\{R_n\}_{n\in\Z}\in\wh\Sigma$ and $Z(v)\supset R_0$, there exists
$\un v=\{v_n\}_{n\in\Z}\in\Sigma$ with $v_0=v$ s.t. $_m[R_m,\ldots,R_n]\subset Z_m[v_m,\ldots,v_n]$
for every $m\leq n$. In particular, $\wh\pi(\un R)=\pi(\un v)$.
If $\un R\in\wh\Sigma^\#$, then $\un v\in\Sigma^\#$.
\item If $R_m\to\cdots\to R_n$ is a path on $\wh\Sigma$ and
$v_m\overset{\ve}{\to}\cdots\overset{\ve}{\to}v_n$ is a path on $\Sigma$
s.t. $R_k\subset Z(v_k)$ for $k=m,\ldots,n$, then  $_m[R_m,\ldots,R_n]\subset Z_m[v_m,\ldots,v_n]$.
\end{enumerate}
\end{lemma}

\begin{proof}
The proof of part (1) follows \cite[Claim in pp. 20]{BLPV}.
Fix $\underline R=\{R_n\}_{n\in\mathbb Z}\in\wh\Sigma$ and $v$ s.t. $Z(v)\supset R_0$.
For each $k\geq 0$, fix some $\wh y_k\in{_{-k}[R_{-k}},\ldots,R_k]$. Since $\mathfs R$ and $\mathfs Z$
both cover the same set ${\rm NUH}^\#$, there is
$\underline v^{(k)}=\{v^{(k)}_\ell\}_{\ell\in\mathbb Z}\in\Sigma^\#$ with $v^{(k)}_0=v$ s.t.
$\wh y_k=\pi[\underline v^{(k)}]$.
Since the degrees of the graph defining $\Sigma$ are finite,
for each $\ell\geq 0$ there are finitely many possibilities for the tuple $(v^{(k)}_{-\ell},\ldots,v^{(k)}_{\ell})$.
By a diagonal argument, there is $\underline v=\{v_\ell\}_{\ell\in\mathbb Z}\in\Sigma$
s.t. for each $\ell\geq 0$ we have $(v_{-\ell},\ldots,v_\ell)=(v^{(k)}_{-\ell},\ldots,v^{(k)}_{\ell})$
for infinitely many $k\geq 0$. We claim that $\underline v$ satisfies part (1).
Clearly $v_0=v$. For the other statements, firstly note that
$\wh f^n(\wh y_k)=\pi[\sigma^n(\underline v^{(k)})]$,
and so $R_n\subset Z_{v^{(k)}_n}$ for all $k\geq |n|$. In particular, $R_n\subset Z_{v_n}$. Now let
$m\leq n$, and take $k\geq |m|,|n|$. Fix $\wh y\in{_{m}[R_m},\ldots,R_n]$. We need to show
that $\wh y\in Z_m[v_m,\ldots,v_n]$.
\begin{enumerate}[$\circ$]
\item Since $\wh f^m(\wh y)\in R_m\subset Z_{v_m}$, there is
$\underline w\in\Sigma^\#$ with $w_m=v_m$ s.t. $\wh f^m(\wh y)=\pi[\sigma^m(\underline w)]$.
\item Since $\wh f^n(\wh y)\in R_n\subset Z_{v_n}$, there is $\underline z\in\Sigma^\#$ with $z_n=v_n$
s.t. $\wh f^n(\wh y)=\pi[\sigma^n(\underline z)]$.
\end{enumerate}
Define $\underline a=\{a_\ell\}_{\ell\in\mathbb Z}$ by:
$$
a_\ell=\left\{
\begin{array}{ll}
w_\ell&, \ell\leq m\\
v_\ell&, m\leq \ell\leq n\\
z_\ell&, \ell\geq n.
\end{array}
\right.
$$
Since both $\underline w,\underline z\in\Sigma^\#$, also $\underline a\in\Sigma^\#$. Observe that:
\begin{enumerate}[$\circ$]
\item If $\ell\leq m$ then $\wh f^\ell(\wh y)=\pi[\sigma^\ell(\underline w)]\in Z_{w_\ell}=Z_{a_\ell}$.
\item If $m\leq \ell\leq n$ then $\wh f^\ell(\wh y)\in R_\ell\subset Z_{v_\ell}=Z_{a_\ell}$.
\item If $\ell\geq n$ then $\wh f^\ell(\wh y)=\pi[\sigma^\ell(\underline z)]\in Z_{z_\ell}=Z_{a_\ell}$.
\end{enumerate}
Hence $\vt[\wh f^\ell(\wh y)]\in \vt[Z_{a_\ell}]$ for all $\ell\in\Z$, which implies that
$\underline a$ shadows $\wh y$ and so $\pi(\underline a)=\wh y$.
In particular $\wh y\in Z_m[v_m,\ldots,v_n]$. To prove the equality $\widehat\pi(\underline R)=\pi(\underline v)$,
firstly observe that
$$\{\widehat\pi(\underline R)\}=\bigcap_{n\geq 0}\overline{_{-n}[R_{-n},\ldots,R_n]}\subset
\bigcap_{n\geq 0}\overline{Z_{-n}[v_{-n},\ldots,v_n]}.
$$
Since
$$
\{\pi(\underline v)\}=\bigcap_{n\geq 0}Z_{-n}[v_{-n},\ldots,v_n]\subset
\bigcap_{n\geq 0}\overline{Z_{-n}[v_{-n},\ldots,v_n]},
$$
the intersection $\bigcap_{n\geq 0}\overline{Z_{-n}[v_{-n},\ldots,v_n]}$ is non-empty.
But it consists of a descending chain of closed sets
with diameter converging to zero, so it equals $\{\pi(\underline v)\}$.
It follows that $\widehat\pi(\underline R)=\pi(\underline v)$. Finally, the local finiteness
properties of $\mathfs R$ and $\mathfs Z$ imply that if $\un R\in\wh\Sigma^\#$ then $\un v\in\Sigma^\#$.
This concludes the proof of part (1). 

The proof of part (2) follows exactly the same argument made above, when 
we took $\wh y\in~_{m}[R_m,\ldots,R_n]$ and proved that $\wh y\in Z_m[v_m,\ldots,v_n]$.
\end{proof}

The non-injectivity of $\wh\pi$ is analysed using the notion of affiliation \cite[\S 12.3]{Sarig-JAMS}.

\medskip
\noindent
{\sc Affiliation:} Two rectangles $R,R'\in\mathfs R$ are called {\em affiliated} if there are 
$Z,Z'\in\mathfs Z$ s.t. $Z\supset R$, $Z'\supset R'$ and $Z\cap Z'\neq\emptyset$.
If this occurs, we write $R\sim R'$.

\begin{lemma}[Bowen property]\label{Lemma-Bowen-relation}
If $\un R=\{R_n\}_{n\in\Z},\un S=\{S_n\}_{n\in\Z}\in\wh\Sigma^\#$, then 
$\wh\pi(\un R)=\wh\pi(\un S)$ iff $R_n\sim S_n$ for all $n\in\Z$.
\end{lemma}

\begin{proof}
$(\Longrightarrow)$ Asume that $\un R,\un S\in\wh\Sigma^\#$ with $\wh x=\wh\pi(\un R)=\wh\pi(\un S)$.
By Lemma \ref{Lemma-relation-R-Z}, there are $\un v=\{v_n\}_{n\in\Z},\un w=\{w_n\}_{n\in\Z}\in\Sigma^\#$ s.t.
$\wh x=\pi(\un v)=\pi(\un w)$. By Proposition \ref{Prop-pi}(3), for every $n\in\Z$ we have
$\wh f^n(\wh x)=\pi[\sigma^n(\un v)]=\pi[\sigma^n(\un w)]$ and so
$\wh f^n(\wh x)\in Z(v_n)\cap Z(w_n)$, proving that $R_n\sim S_n$.

\medskip
\noindent
$(\Longleftarrow)$ Assume that $\un R=\{R_n\}_{n\in\Z},\un S=\{S_n\}_{n\in\Z}\in\wh\Sigma^\#$
with $R_n\sim S_n$. Letting $\wh x=\wh\pi(\un R)$ and $\wh y=\wh\pi(\un S)$, we want to show that
$\wh x=\wh y$. Start applying Lemma \ref{Lemma-relation-R-Z}(1) to obtain
$\un v=\{v_n\}_{n\in\Z}=\{\Psi_{\wh x_n}^{p^s_n,p^u_n}\}\in\Sigma^\#$ s.t. $\wh x=\pi(\un v)$ and
$\un w=\{w_n\}_{n\in\Z}=\{\Psi_{\wh y_n}^{q^s_n,q^u_n}\}\in\Sigma^\#$ s.t. $\wh y=\pi(\un w)$.
Had we that $Z(v_n)\cap Z(w_n)\neq\emptyset$ for all $n\in\Z$, we could just apply Lemma \ref{Lemma-fibres}(2)(a)
to obtain that $\wh x$ is shadowed by $\un w$, thus concluding that $\wh x=\pi(\un w)=\wh y$. 
Unfortunately, we do not know if $Z(v_n)$ and $Z(w_n)$ intersect, but we do know that they cannot be
far apart. Indeed, we claim that $\vt[Z(v_n)]\subset \Psi_{\wh y_n}(B[20Q(\wh y_n)])$ for all $n\in\Z$.
Once this is proved, we get that $\wh x=\wh y$. We prove the inclusion for $n=0$.
Since $R_0\sim S_0$, there are $a=\Psi_{\wh c}^{r^s,r^u}$ and $b=\Psi_{\wh d}^{t^s,t^u}$
s.t. $Z(a)\supset R_0$, $Z(b)\supset S_0$ and $Z(a)\cap Z(b)\neq\emptyset$.
Since $Z(v_0),Z(a)\supset R_0$, we have that $Z(v_0)\cap Z(a)\neq\emptyset$.
Similarly, $Z(w_0)\cap Z(b)\neq\emptyset$. Letting $r=r^s\wedge r^u$, $t=t^s\wedge t^u$ and
$q=q^s_0\wedge q^u_ 0$, the same arguments used the proof of Lemma \ref{Lemma-fibres}(2)(a)
imply that
$$
\vt[Z(v_0)]\subset \Psi_{\wh c}(B[r])\subset \Psi_{\wh d}(B[3t])\subset \Psi_{\wh y_0}(B[9q]),
$$
which is contained in $\Psi_{\wh y_0}(B[20Q(\wh y_0)])$.
\end{proof}

The next theorem completes the proof of Theorem \ref{Thm-Main}.

\begin{theorem}\label{Thm-widehat-pi}
The following holds for all $\ve>0$ small enough.
\begin{enumerate}[{\rm (1)}]
\item $\widehat\pi:\widehat\Sigma\to\wh M$ is H\"older continuous.
\item $\widehat\pi\circ\widehat\sigma=\wh f\circ\widehat\pi$.
\item $\widehat\pi\restriction_{\widehat\Sigma^\#}:\widehat\Sigma^\#\to {\rm NUH}^\#$
is a finite-to-one surjective map.  
\end{enumerate}
\end{theorem}

\begin{proof}
(1)  We have to show that the diameter of $_{-n}[R_{-n},\ldots,R_n]$
decays exponentially fast in $n$. By Proposition \ref{Prop-pi}(2), there are
$C>0$ and $\theta\in(0,1)$ s.t. the diameter of every $Z_{-n}[v_{-n},\ldots,v_n]$
is at most $C\theta^n$. For a fixed $_{-n}[R_{-n},\ldots,R_n]$, apply 
Lemma \ref{Lemma-relation-R-Z} to obtain $v_{-n}\overset{\ve}{\to}\cdots\overset{\ve}{\to}v_n$
s.t. $_{-n}[R_{-n},\ldots,R_n]\subset Z_{-n}[v_{-n},\ldots,v_n]$. This concludes the proof.

\medskip
\noindent
(2) Fix $\un R\in \wh\Sigma$, and let $\un v\in\Sigma$ satisfying Lemma \ref{Lemma-relation-R-Z}(1).
Then $\wh\pi(\un R)=\pi(\un v)$ and $\wh\pi[\wh\sigma(\un R)]=\pi[\sigma(\un v)]$.
By Proposition \ref{Prop-pi}(3), it follows that
$$
\wh\pi[\wh\sigma(\un R)]=\pi[\sigma(\un v)]=\wh f[\pi(\un v)]=\wh f[\wh\pi(\un R)].
$$

\medskip
\noindent
(3) The proof is an adaptation of \cite[Theorem 5.6(5)]{Lima-Sarig} to our context.
We provide the details. For $R\in\mathfs R$, define
$$
N(R):=\#\{(R',v')\in\mathfs R\times\mathfs A:\textrm{$R'$ is affiliated to $R$ and $Z(v')\supset R'$}\}.
$$
By Proposition \ref{Prop-Z}(2), $N(R)$ is finite.
If $R,S\in\mathfs R$ then $N(R,S):=N(R)N(S)$ is also finite.
 
\medskip
\noindent
{\sc Claim:} If $\wh x=\wh \pi(\un{R})$ where $R_n=R$ for infinitely many $n<0$ and $R_n=S$
for infinitely many $n>0$, then $\#[\wh\pi^{-1}(\wh x)\cap\wh\Sigma^\#]\leq N(R,S)$.

\begin{proof}[Proof of the claim.]
The proof is by contradiction. Write $N:=N(R,S)$, and assume that
$\un{R}^{(0)},\ldots,$ $\un{R}^{(N)}\in\wh\Sigma^\#$ are distinct
and satisfy $\wh\pi[\un R^{(j)}]=\wh x$ for $j=0,\ldots,N$.
Write $\un{R}^{(0)}=\un{R}$ and $\un{R}^{(j)}=\{R^{(j)}_k\}_{k\in\Z}$.
By Lemma \ref{Lemma-relation-R-Z}(1), for each $j$ there is $\un{v}^{(j)}\in \Sigma^\#$
with $\wh x=\pi(\un{v}^{(j)})$. By Lemma \ref{Lemma-Bowen-relation},
for every $i\in \Z$ the sets $R^{(0)}_i,\ldots, R^{(N)}_i$ are affiliated.
Fix $n$ large enough s.t. the sequences $(R_{-n}^{(j)},\ldots,R_n^{(j)})$, $j=0,\ldots,N$,
are all distinct, and fix $k,\ell>n$ s.t. $R^{(0)}_{-k}=R$ and $R^{(0)}_{\ell}=S$.
There are at most $N=N(R)N(S)$ possibilities for the quadruple
$(R^{(j)}_{-k},v^{(j)}_{-k};R^{(j)}_\ell,v^{(j)}_\ell)$, $j=0,\ldots,N$, 
and so by the pigeonhole principle
there are $0\leq j_1,j_2\leq N$ with $j_1\neq j_2$ s.t.
$$
(R^{(j_1)}_{-k},v^{(j_1)}_{-k},R^{(j_1)}_\ell,v^{(j_1)}_\ell)=
(R^{(j_2)}_{-k},v^{(j_2)}_{-k},R^{(j_2)}_\ell,v^{(j_2)}_\ell).
$$
Write $\un{A}:=\un{R}^{(j_1)}$, $\un{B}:=\un{R}^{(j_2)}$, $\un{a}:=\un{v}^{(j_1)}$, $\un{b}:=\un{v}^{(j_2)}$,
 $A:=A_{\ell}=B_{\ell}$, $B:=A_{-k}=B_{-k}$, $a:=a_\ell=b_\ell$ and $b:=a_{-k}=b_{-k}$.
Fix $\wh x_A\in {_{-k}[}A_{-k},\ldots,A_\ell]\textrm{ and }\wh x_B\in {_{-k}[}B_{-k},\ldots,B_\ell]$,
and define $\wh z_A, \wh z_B$ by the equalities
$$
\wh f^{\ell}(\wh z_A):=[\wh f^{\ell}(\wh x_A),\wh f^{\ell}(\wh x_B)]\ \text{ and }\
\wh f^{-k}(\wh z_B):=[\wh f^{-k}(\wh x_A),\wh f^{-k}(\wh x_B)].
$$
These points are uniquely defined by Proposition \ref{Prop-R}(1),
since  $\wh f^{\ell}(\wh x_A),\wh f^{\ell}(\wh x_B)\in A$ and $\wh f^{-k}(\wh x_A), \wh f^{-k}(\wh x_B)\in B$.
We will obtain a contradiction once we show that $\wh z_A\neq\wh z_B$ and $\wh z_A=\wh z_B$.
Firstly, note that by Proposition \ref{Prop-R}(3) we have $\wh z_A\in  {_{-k}[}B_{-k},\ldots,B_\ell]$
and $\wh z_B\in {_{-k}[}A_{-k},\ldots,A_\ell]$. Since $(A_{-k},\ldots,A_\ell)\neq (B_{-k},\ldots,B_\ell)$,
we get that $\wh z_A\neq \wh z_B$.

Now we show that $\wh z_A=\wh z_B$.
Since $\wh f^{\ell}(\wh z_A)\in A_{\ell}\subset Z(a)$ and $\wh f^{-k}(\wh z_B)\in B_{-k}\subset  Z(b)$,
there are $\un{\alpha}=\{\alpha_n\}_{n\in\Z},\un{\beta}=\{\beta_n\}_{n\in\Z}\in\Sigma^\#$ with
$\alpha_\ell=a$ and $\beta_{-k}=b$ s.t. $\wh z_A=\pi(\un{\alpha})$ and $\wh z_B=\pi(\un{\beta})$.
Define $\un c=\{c_n\}_{n\in\Z}\in\Sigma$ by:
$$
c_n=\left\{
\begin{array}{ll}
\beta_n&, n\leq -k\\
a_n&, -k\leq n\leq \ell\\
\alpha_n&, n\geq\ell.
\end{array}
\right.
$$
Since both $\un \alpha,\un \beta\in\Sigma^\#$, also $\underline c\in\Sigma^\#$.
We claim that $\wh z_A=\pi(\un{c})=\wh z_B$.
Write $c_n=\Psi_{\wh x_n}^{p^s_n,p^u_n}$. By Proposition \ref{Prop-R}(3),
$\wh f^{-k}(\wh z_A),\wh f^{-k}(\wh z_B)$ both belong to $\wh W^u(\wh f^{-k}(\wh x_B),B)$.
Since this latter $s$--set is contained in $\wh V^u[\{c_n\}_{n\leq -k}]$,
Lemma \ref{Lemma-stable-set} implies that
$$
\vt_n[\wh z_A],\vt_n[\wh z_B]\in \Psi_{\wh x_n}(B[20Q(\wh x_n)]) \text{ for all }n\leq -k.
$$
Similarly, we obtain that 
$$
\vt_n[\wh z_A],\vt_n[\wh z_B]\in \Psi_{\wh x_n}(B[20Q(\wh x_n)]) \text{ for all }n\geq \ell.
$$
Now fix $-k\leq n\leq \ell$. We have
$\wh f^n(\wh z_A),\wh f^n(\wh z_B)\in A_n\cup B_n\subset Z(a_n)\cup Z(b_n)$.
Since $\pi(\un a)=\wh x=\pi(\un b)$, we have $Z(a_n)\cap Z(b_n)\neq\emptyset$
and so, recalling that $a_n=c_n$, Lemma \ref{Lemma-fibres}(2)(a) implies that
$\vt[Z(a_n)\cup Z(b_n)]\subset \Psi_{\wh x_n}(B[20Q(\wh x_n)])$.
Therefore 
$$
\vt_n[\wh z_A],\vt_n[\wh z_B]\in \Psi_{\wh x_n}(B[20Q(\wh x_n)]) \text{ for all }-k\leq n\leq \ell.
$$
But then $\un c$ shadows both $\wh z_A$ and $\wh z_B$ which, by Lemma \ref{Lemma-shadowing},
implies that $\wh z_A=\wh z_B$. We therefore obtain a contradiction.
\end{proof}
\noindent
The claim concludes the proof of the theorem.
\end{proof}

We finish this section stating a result that will be used in Part \ref{Part-Applications}.
Since $\wh\pi[\wh\Sigma^\#]={\rm NUH}^\#$ and every $\wh x\in{\rm NUH}^\#$ has
invariant subspaces $E^s_{\wh x},E^u_{\wh x}$, each $\un R\in\wh\Sigma^\#$
is associated to subspaces $E^s_{\wh\pi(\un R)},E^u_{\wh\pi(\un R)}$. 

\begin{proposition}\label{Lemma-Holder-directions}
The maps $\un R\in\wh\Sigma^\# \mapsto E^{s/u}_{\wh\pi(\un R)}$ are H\"older continuous.
\end{proposition}

The proof is made as in \cite[Lemma 12.6]{Sarig-JAMS}, using 
Lemma \ref{Lemma-relation-R-Z}(2) and Proposition \ref{Prop-stable-manifolds}(5).

\part{Applications}\label{Part-Applications}

In this part, we provide some applications of Theorem \ref{Thm-Main}:
\begin{enumerate}[$\circ$]
\item Estimating the number of closed orbits.
\item Establishing ergodic properties of equilibrium measures.
\end{enumerate}
Let $f$ satisfying conditions (A1)--(A7), and recall that $\vt:\wh M\to M$ is the
projection into the zeroth coordinate, see Subsection \ref{Section-natural-extension}.
Recall also that there are bijective projection/lift procedures between $f$--invariant
probability measures $\mu$ in $M$ and $\wh f$--invariant probability measures $\wh\mu$ in $\wh M$.
Using this bijection, call $\wh\mu$ a $\chi$--hyperbolic measure if $\mu$ is $\chi$--hyperbolic,
and $\wh\mu$ an $\wh f$--adapted measure if $\mu$ is $f$--adapted. 
In the sequel, let ${\rm NUH}_\chi^\#$ represent the set ${\rm NUH}^\#$ for the parameter $\chi>0$.
Also, let $(\Sigma,\sigma,\pi)$ be the triple given by Theorem \ref{Thm-Main} for $f$ and the parameter $\chi$.
The next result is of great relevance to our applications.

\begin{proposition}\label{Prop-measures}
Fix $f$ satisfying conditions {\rm (A1)--(A7)} and $\chi>0$.
Each $f$--adapted $\chi$--hyperbolic measure $\mu$ in $M$ can be lifted to a
$\sigma$--invariant probability measure $\nu$ in $\Sigma$ s.t.:
\begin{enumerate}[{\rm (1)}]
\item {\sc Ergodicity:} $\mu$ is ergodic iff $\nu$ is ergodic.
\item {\sc Entropy:} $h_\mu(f)=h_\nu(\sigma)$.
\item {\sc Integral:} For all $\psi:M\to \R$ measurable it holds
$\int_M\psi d\mu=\int_\Sigma(\psi\circ\vt\circ\pi)d\nu$.
\end{enumerate}
Reversely, each $\sigma$--invariant probability measure $\nu$ in $\Sigma$
projects to a $\chi$--hyperbolic probability measure $\mu$ in $M$ s.t. {\rm (1)--(3)} hold.
\end{proposition}

Note that we do not claim that the projected measure $\mu$ is $f$--adapted.
%As a matter of fact, we do not believe this is always the case.

\begin{proof}
The bijection of Subsection \ref{Section-natural-extension} between $f$--invariant probability
measures $\mu$ in $M$ and $\wh f$--invariant probability measures $\wh\mu$ in $\wh M$ preserve (1)--(3),
and restricts to a bijection between $f$--adapted $\chi$--hyperbolic measures in $M$ and
$\wh f$--adapted $\chi$--hyperbolic measures in $\wh M$.
By Lemma \ref{Lemma-adaptedness}, if $\mu$ is $f$--adapted
and $\chi$--hyperbolic, then $\wh\mu$ is carried by ${\rm NUH}^\#_\chi$.
Proceeding as in \cite[\S 13]{Sarig-JAMS}, we can lift $\wh\mu$ to a $\sigma$--invariant probability
measure $\nu$ on $\Sigma$ that satisfies (1)--(3).

Now we project measures. By the Poincar\'e recurrence theorem,
every $\sigma$--invariant probability measure $\nu$ in $\Sigma$ is carried by $\Sigma^\#$,
hence it descends to an $\wh f$--invariant probability measure $\wh\mu$ carried by ${\rm NUH}^\#_\chi$
(attention: $\wh\mu$ is not necessarily $\wh f$--adapted). This projection clearly
satisfies (1) and (3) and, since $\pi\restriction_{\Sigma^\#}:\Sigma^\#\to {\rm NUH}^\#_\chi$ is finite-to-one,
it also satisfies (2). By the definition of ${\rm NUH}^\#_\chi$, $\wh\mu$ descends
to a $\chi$--hyperbolic measure $\mu$ on $M$.
\end{proof}

\section{Equilibrium measures}\label{Section-equilibrium}

Call a measure {\em hyperbolic} if it is $\chi$--hyperbolic for some $\chi>0$.
We will restrict our attention to equilibrium measures of H\"older continuous potentials
or of multiples of the geometric potential. Let us first recall some definitions.
Let $(Y,S)$, where $Y$ is a complete metric separable space and $S:Y\to Y$ is a map,
and let $\psi:Y\to\R$ be a (not necessarily continuous) potential.

\medskip
\noindent
{\sc Topological pressure:} The {\em topological pressure} of $\psi$ is
$P_{\rm top}(\psi):=\sup\{h_\mu(S)+\int\psi d\mu\}$,
where the supremum ranges over all $S$--invariant probability measures
for which $\int\psi d\mu$ makes sense and $h_\mu(S)+\int\psi d\mu\neq \infty-\infty$.

\medskip
\noindent
{\sc Equilibrium measure:} An {\em equilibrium measure} for $\psi$ is an $S$--invariant probability 
measure $\mu$ for which $\int\psi d\mu$ makes sense, $h_\mu(S)+\int\psi d\mu\neq \infty-\infty$,
and $P_{\rm top}(\psi)=h_\mu(S)+\int\psi d\mu$.

\medskip
When $\psi\equiv 0$, equilibrium measures are called {\em measures of maximal entropy} and
$P_{\rm top}(0)$ is called the {\em topological entropy} of $S$, which we denote by $h_{\rm top}(S)$.
If $\pi:(X,T)\to (Y,S)$ is finite-to-one, then equilibrium measures for
$\psi$ lift to equilibrium measures for $\vf=\psi\circ\pi$.
Note that $\vf$ is H\"older continuous whenever $\psi,\pi$ are.
%Using this observation, we prove the results in the sequel.

\subsection{H\"older continuous potentials}

In this section we prove the following result.

\begin{theorem}\label{Thm-Application-1}
Assume that $f$ satisfies assumptions {\rm (A1)--(A7)}.
Every equilibrium measure of a bounded H\"older continuous potential with finite
topological pressure has at most countably many $f$--adapted hyperbolic ergodic components.
Furthermore, the lift to $\wh M$ of each such ergodic component is Bernoulli up to a period.
\end{theorem}

Of course, if $f$ is invertible then each measure itself is Bernoulli up to a period.
Note that if $P_{\rm top}(0)<\infty$ then every bounded H\"older continuous potential
has finite topological pressure.

\begin{proof}
Let $\psi:M\to\R$ be a bounded H\"older continuous potential with $P_{\rm top}(\psi)<\infty$.
We prove that, for each $\chi>0$, $\psi$ possesses at most countably
many $f$--adapted $\chi$--hyperbolic ergodic equilibrium measures. The first part of the theorem
will follow by taking the union of these measures for $\chi_n=\tfrac{1}{n}$, $n>0$. 
Fix $\chi>0$, and let $\pi:\Sigma\to\wh M$ be the coding given by Theorem \ref{Thm-Main}
for the parameter $\chi$. Define $\vf:\Sigma\to \R$ by $\vf=\psi\circ\vt\circ\pi$,
a bounded H\"older continuous potential. We claim that $P_{\rm top}(\vf)\leq P_{\rm top}(\psi)$.
Indeed, for each $\sigma$--invariant probability measure $\nu$ there is an $f$--invariant probability
measure $\mu$ satisfying Proposition \ref{Prop-measures}, and so
$h_\nu(\sigma)+\int \vf d\nu=h_\mu(f)+\int \psi d\mu$. Taking the supremum over all $\nu$,
we conclude that $P_{\rm top}(\vf)\leq P_{\rm top}(\psi)$.
Now assume that $\mu$ is an $f$--adapted $\chi$--hyperbolic equilibrium measure for $\psi$.
Again by Proposition \ref{Prop-measures}, there is a $\sigma$--invariant probability measure $\nu$
s.t. $h_\nu(\sigma)+\int \vf d\nu=h_\mu(f)+\int \psi d\mu=P_{\rm top}(\psi)$,
and so $P_{\rm top}(\vf)\geq P_{\rm top}(\psi)$.
Therefore,  $P_{\rm top}(\vf)=P_{\rm top}(\psi)$ and $\nu$ is an equilibrium measure for $\vf$.
By \cite[Thm. 1.1]{Buzzi-Sarig}, each topologically transitive countable topological Markov shift
carries at most one equilibrium measure for $\vf$. Hence, there are at most countably many such $\nu$,
which implies that there are at most countably many $f$--adapted $\chi$--hyperbolic
equilibrium measure for $\psi$. This concludes the proof of the first part of theorem. 

Now we prove the second part. Buzzi \& Sarig have shown how to construct, when they exist,
equilibrium measures for bounded H\"older continuous potentials in symbolic spaces,
see \cite[Thm. 1.2]{Buzzi-Sarig}. Using this, Sarig proved that such measures have local
product structure and are Bernoulli up to a period, see \cite{Sarig-Bernoulli-JMD}. 
Since these results are proved at a symbolic level, each $\nu$ in the previous paragraph 
is Bernoulli up to a period. Finally, since this latter property is preserved by finite-to-one factor maps,
we conclude that $\wh\mu$ is Bernoulli up to a period.
\end{proof}

Theorem \ref{Thm-Application-1} also holds for unbounded
H\"older continuous potentials, provided $\sup \psi<\infty$ and $P_{\rm top}(\psi)<\infty$.
Indeed, these are the conditions required for the potential in \cite[Thm. 1.1]{Buzzi-Sarig}.

\begin{corollary}\label{Coro-Application-1}
If $f$ satisfies assumptions {\rm (A1)--(A7)} and $h_{\rm top}(f)<\infty$,
then $f$ possesses at most countably many $f$--adapted hyperbolic ergodic 
measures of maximal entropy, and the lift to $\wh M$ of each of them is Bernoulli up to a period.
\end{corollary}

Again, if $f$ is invertible then each measure itself is Bernoulli up to a period.

\subsection{Multiples of the geometric potential}\label{Section-J}

The proof of Theorem \ref{Thm-Application-1} relies on the properties on $\vf$ that allow
to use \cite{Buzzi-Sarig}. We will prove, using Proposition \ref{Lemma-Holder-directions}, that
these conditions are also satisfied for multiples of the geometric potential. We start with 
some definitions.

\medskip
\noindent
{\sc Hyperbolic locus of $f$:} The {\em hyperbolic locus} of $f$ is the set $\wh M_{\rm hyp}$ defined as 
$$
\wh M_{\rm hyp}:=\bigcup_{\chi>0}{\rm NUH}_\chi.
$$

Clearly, $\wh M_{\rm hyp}$ is $\wh f$--invariant and carries all hyperbolic measures.
Observe that, by Proposition \ref{Prop-hyperbolicity-manifolds}, each $\wh x\in\wh M_{\rm hyp}$
is associated to two subspaces $E^s_{\wh x},E^u_{\wh x}$ of $\wh{TM}_{\wh x}$.

\medskip
\noindent
{\sc Geometric potential of $f$:} The {\em geometric potential} of $f$ is the function
$J:\wh M_{\rm hyp}\to\R$ defined by
$$
J(\wh x):=-\log {\rm det}\left[\wh{df}\restriction_{E^u_{\wh x}}\right].
$$

\medskip
Note that $J$ is defined in the natural extension only. Indeed, it is in this set that
unstable directions are well-defined. Let $t\in\R$.

\medskip
\noindent
{\sc Topological pressure of $tJ$:} The {\em topological pressure} of $tJ$ is defined as
$P_{\rm top}(tJ)=\sup\{h_{\wh\mu}(\wh f)+\int tJ d\wh\mu\}$
where the supremum ranges over all $\wh f$--invariant probability measures $\wh\mu$
supported on $\wh M_{\rm hyp}$ for which $\int J d\wh\mu$ makes sense and
$h_{\wh\mu}(\wh f)+\int tJ d\wh\mu\neq \infty-\infty$.

\medskip
\noindent
{\sc Equilibrium measure for $tJ$:} An {\em equilibrium measure} for $tJ$ is an
$\wh f$--invariant probability measure $\wh\mu$ supported on $\wh M_{\rm hyp}$
for which $\int J d\wh\mu$ makes sense, $h_{\wh\mu}(\wh f)+\int tJ d\wh\mu\neq \infty-\infty$,
and $P_{\rm top}(tJ)=h_{\wh\mu}(\wh f)+\int tJ d\wh\mu$.

\medskip
In the generality of (A1)--(A7), the functions $tJ$ can be unbounded (possibly approaching $\pm\infty$),
as well as $P_{\rm top}(tJ)$ might not be finite. Except for these possibilities, an analogue
of Theorem \ref{Thm-Application-1} holds. 

\begin{theorem}\label{Thm-Application-2}
Assume that $f$ satisfies assumptions {\rm (A1)--(A7)}, and fix $t\in\R$.
If $\sup(tJ)<\infty$ and $P_{\rm top}(tJ)<\infty$, then every equilibrium measure of $tJ$
has at most countably many $\wh f$--adapted hyperbolic ergodic components.
Furthermore, each such ergodic component is Bernoulli up to a period.
\end{theorem}

\begin{proof}
Each ergodic hyperbolic equilibrium measure of $tJ$ is $\chi$--hyperbolic for some $\chi>0$.
For each $\chi>0$, consider $(\Sigma,\sigma,\pi)$ given by Theorem \ref{Thm-Main} and let
$\vf:\Sigma\to\R$ be defined by $\vf=(tJ)\circ\pi$. Observe that $\vf$ is globally defined
with $\sup(\vf)<\infty$, and by Proposition \ref{Lemma-Holder-directions} it is H\"older continuous. 
Proceeding as in the proof of Theorem  \ref{Thm-Application-1}, we get that 
$P_{\rm top}(\vf)\leq P_{\rm top}(tJ)<\infty$. Hence we can apply \cite{Buzzi-Sarig} to $\vf$
and repeat the rest of the proof of Theorem \ref{Thm-Application-1}.
\end{proof}

It is important noticing that, in our setting, neither the Ruelle inequality nor the Pesin entropy formula
are known to hold, hence we do not know if an invariant Liouville measure is an equilibrium measure
of $J$.

\section{Exponential growth rate of periodic orbits}

Let $f$ satisfying assumptions {\rm (A1)--(A7)}. For $\chi>0$ and
$n\geq 1$, let ${\rm Per}_{n,\chi}(f)$ denote the number of periodic points of
period $n$ for $f$ with all Lyapunov exponents greater than $\chi$ in absolute value.
The next result provides sufficient conditions for the exponential growth rate of ${\rm Per}_{n,\chi}(f)$.
%Also, let ${\rm Per}_n(f)$ denote the number of periodic points of period $n$ for $f$.

\begin{theorem}\label{Thm-Application-3}
Assume that $f$ satisfies assumptions {\rm (A1)--(A7)}. If $f$ possesses
an $f$--adapted $\chi$--hyperbolic mixing measure with entropy $h>0$, then
$$
\liminf\limits_{n\to\infty}e^{-hn}{\rm Per}_{n,\chi}(f)\geq 1.
$$
\end{theorem}

\begin{proof}
This proof was suggested to us by J\'er\^ome Buzzi, making use of his recent work \cite{Buzzi-2019}.
We will use some standard terminology of symbolic dynamics on Markov shifts,
see e.g. \cite{Kitchens-Book}. Let $\mu$ be $f$--adapted $\chi$--hyperbolic mixing with $h_\mu(f)=h>0$.
Being mixing, $\mu$ is in particular ergodic.
Consider the triple $(\Sigma,\sigma,\pi)$ given by Theorem \ref{Thm-Main}.
Then $\pi\restriction_{\Sigma^\#}:\Sigma^\#\to{\rm NUH}^\#_\chi$ satisfies the Main Theorem of 
\cite{Buzzi-2019}: it is a Borel semiconjugacy with the Bowen property (see Lemma \ref{Lemma-Bowen-relation}).
Hence there is a TMS $(\wt\Sigma,\wt\sigma)$ and a 1--Lipschitz map $\theta:\wt\Sigma\to\Sigma^\#$
s.t. $\pi\circ\theta:\wt\Sigma\to{\rm NUH}^\#_\chi$ is an {\em injective} semiconjugacy and
$(\pi\circ\theta)[\wt\Sigma^\#]$ carries all $\wh f$--invariant measures of $\pi[\Sigma^\#]={\rm NUH}^\#_\chi$.
This allows to lift $\wh\mu$ to a {\em mixing} $\wt\sigma$--invariant probability measure $\lambda$
on $\wt\Sigma$ with entropy $h$. If $\wt\Sigma$ is defined by the oriented graph $(\wt V,\wt E)$, then $\lambda$
is carried by the topological Markov subshift $(\wt\Sigma_\lambda,\wt\sigma)$ defined by the 
graph with vertex set $\wt V_\lambda:=\{\wt v\in \wt V: \lambda(_{0}[\wt v])>0\}$ and edges accordingly.
Since $\lambda$ is mixing, $(\wt\Sigma_\lambda,\wt\sigma)$ is topologically mixing (its graph
is aperiodic).

Now we prove the theorem. Since the set of periodic points of $f$ is in bijection with the set of periodic points
of its natural extension $\wh f$ and $\pi\circ\theta$ is injective, it is enough to estimate 
${\rm Per}_n(\wt\sigma\restriction_{\wt\Sigma_\lambda})$.
We consider two cases. In the first, assume that $\lambda$ is a
measure of maximal entropy for $(\wt\Sigma_\lambda,\wt\sigma)$. By \cite[Lemma 6.2]{Buzzi-2019},
$$
\lim_{n\to\infty}e^{-hn}{\rm Per}_n(\wt\sigma\restriction_{\wt\Sigma_\lambda})\geq 1
$$
and the result follows. Now assume that $\lambda$ is not a measure of maximal entropy.
Since $h_{\rm top}(\wt\Sigma_\lambda)=\sup\{h_{\rm top}(\wt\sigma\restriction_{\overline{\Sigma}})\}$
where $\overline{\Sigma}$ runs over all subshifts of $\wt\Sigma_\lambda$ with finitely many states,
we can fix one such subshift with $h_{\rm top}(\wt\sigma\restriction_{\overline{\Sigma}})>h$.
The pair $(\overline{\Sigma},\wt\sigma)$ might not be topologically mixing, but since 
$\overline{\Sigma}$ is defined by a finite graph inside an aperiodic (in particular transitive) graph,
we can add to this finite graph finitely many closed paths with coprime lengths
and get a new finite graph that is aperiodic
and whose associated TMS, still denoted by $\overline{\Sigma}$, satisfies
$h_{\rm top}(\wt\sigma\restriction_{\overline{\Sigma}})=h'>h$. By a standard result
for TMS with finitely many states (see e.g. \cite[Observation 1.4.3]{Kitchens-Book}),
we have
$$
\lim_{n\to\infty}e^{-h'n}{\rm Per}_n(\wt\sigma\restriction_{\overline\Sigma})\geq 1
$$
and then we obtain a result stronger than the one stated in the theorem.
\end{proof}

\section{Flows with positive speed}\label{Section-Flows}

In this section we recast some definitions and results of \cite{Lima-Sarig},
which will allow us to apply Theorem \ref{Thm-Main} to code nonuniformly hyperbolic measures for flows.
We begin defining the flow counterpart of a TMS. Let $(\Sigma,\sigma)$ be a TMS
and $r:\Sigma\to\R^+$ a H\"older continuous map bounded away from zero and infinity.
Let $r_n$ denote the $n$--th Birkhoff sum of $r$ with respect to $\sigma$.

\medskip
\noindent
{\sc Topological Markov flow:} The {\em topological Markov flow} with roof function $r$ and
base map $\sigma:\Sigma\to\Sigma$ is the flow $\sigma_r:\Sigma_r\to\Sigma_r$ where
$$\Sigma_r:=\{(\un{v},t):\un{v}\in\Sigma,0\leq t<r(\un{v})\},\
\sigma_r^\tau(\un{v},t)=(\sigma^n(\un{v}),t+\tau-r_n(\un{v}))
$$
for the unique $n\in\Z$ s.t. $0\leq t+\tau-r_n(\un{v})<r(\sigma^{n}(\un{v}))$.

\medskip
Informally,  $\sigma_r$ increases the $t$ coordinate at unit speed subject to the identifications
$(\un{v}, r(\un{v}))\sim (\sigma(\un{v}),0)$. The cocycle identity guarantees that
$\sigma_r^{\tau_1+\tau_2}=\sigma_r^{\tau_1}\circ\sigma_r^{\tau_2}$.
There is a natural metric $d_r$ on $\Sigma_r$, called the {\em Bowen-Walters metric},  s.t. $\sigma_r$
is a continuous flow \cite{Bowen-Walters-Metric}. Moreover, $\exists C>0, 0<\kappa<1$ s.t. $d_r(\sigma_r^\tau(\omega_1),\sigma_r^\tau(\omega_2))\leq C d_r(\omega_1,\omega_2)^\kappa$
for all $|\tau|<1$ and $\omega_1,\omega_2\in\Sigma_r$, see \cite[Lemma 5.8(3)]{Lima-Sarig}.
Accordingly, we define the {\em recurrent set} of $\Sigma_r$ by
$\Sigma_r^\#:=\{(\un{v},t)\in\Sigma_r: \un{v}\in\Sigma^\#\}$.

Let $N$ be a closed $C^\infty$ Riemannian manifold,
let $X$ be a $C^{1+\beta}$ vector field on $N$ s.t. $X_x\neq 0$ for all $x\in N$
and $g=\{g^t\}_{t\in\R}:N\to N$ be the flow defined by $X$.
Given $\chi>0$, a $g$--invariant probability measure $\mu$ on $N$ is called
{\em $\chi$--hyperbolic} if all of its Lyapunov exponents are greater than $\chi$
in absolute value, except for the zero exponent in the flow direction.
Theorem \ref{Thm-Main} implies the following result.

\begin{theorem}\label{Thm-Flows}
Let $g$ be as above and $\chi>0$. If $\mu$ is $\chi$--hyperbolic,
then there is a topological Markov flow $(\Sigma_r,\sigma_r)$
and a H\"older continuous map $\pi_r:\Sigma_r\to N$ s.t.:
\begin{enumerate}[$(1)$]
\item $r:\Sigma\to\R^+$ is H\"older continuous and bounded away from zero and infinity.
\item $\pi_r\circ\sigma^t_r=g^t\circ\pi_r$ for all $t\in\R$.
\item $\pi_r[\Sigma^\#_r]$ has full measure with respect to $\mu$.
\item Every $x\in N$ has finitely many pre-images in $\Sigma_r^\#$.
\end{enumerate}
\end{theorem}

To prove the above theorem, we recall some facts from \cite{Lima-Sarig}.

\medskip
\noindent
{\sc Canonical transverse ball and flow box:} Given $r>0$, the {\em canonical transverse ball}
at $x$ of radius $r$ is $S_r(x):=\{\exp{x}(v):v\in T_xN, v\perp X_x,\|v\| \leq r\}$, and
the {\em canonical flow box} of $S_r(x)$ is ${\rm FB}_r(x):=\{g^t(y):y\in S_r(x),|t|\leq r\}$.

\medskip
The following lemmas are the higher dimensional counterparts of \cite[Lemmas 2.2 and 2.3]{Lima-Sarig},
and their proofs are canonical.

\begin{lemma}\label{LemmaFB_1}
There are constants $r_0,\mathfrak d\in (0,1)$ which only depend on $g$ s.t.
for every $x\in N$ the flow box ${\rm FB}_{r_0}(x)$ contains $B(x,\mathfrak d)$,
and the map $(y,t)\mapsto g^t(y)$ is a diffeomorphism from
$S_{r_0}(x)\times [-r_0,r_0]$ onto ${\rm FB}_{r_0}(x)$.
\end{lemma}

\begin{lemma}\label{LemmaFB_2}
There is a constant $\mathfrak H>1$ which only depends on $g$ s.t. for every $x\in N$
the maps $\mathfrak t_x:{\rm FB}_{r_0}(x)\to [-r_0,r_0]$ and
$\mathfrak q_x:{\rm FB}_{r_0}(x)\to S_{r_0}(x)$ defined by
$y=\vf^{\mathfrak t_x(y)}[\mathfrak q_x(y)]$ are well-defined maps with
${\rm Lip}(\mathfrak t_x),{\rm Lip}(\mathfrak q_x),\|\mathfrak t_x\|_{C^{1+\beta}},
\|\mathfrak q_x\|_{C^{1+\beta}}\leq \mathfrak H$.
\end{lemma}

\medskip
\noindent
{\sc Standard Poincar\'e section:} A {\em standard Poincar\'e section} for $g$ 
is a set $M=\biguplus_{i=1}^n S_r(x_i)$ satisfying the following conditions:
\begin{enumerate}[$\circ$]
\item For every $x\in N$ the set $\{t>0: g^t(x)\in M\}$ is a sequence tending to $+\infty$
and the set $\{t<0:g^t(x)\in M\}$ is a sequence tending to $-\infty$.
\item The {\em roof function} $R:M\to (0,\infty)$, defined by $R(x):=\min\{t>0:g^t(x)\in M\}$,
is bounded away from zero and infinity.
\end{enumerate} 

\medskip
\noindent
{\sc Poincar\'e map:} The {\em Poincar\'e map} of a standard Poincar\'e section $M$ is the map
$f:M\to M$ defined by $f(x):=g^{R(x)}(x)$.

\medskip
Above, $r$ is called the radius of $M$.
Let $\partial M$ denote the relative boundary of $M$, equal to the disjoint union of the
sets $\partial S_r(x_i):=\{\exp{x_i}(v):v\in T_{x_i}N, v\perp X_{x_i},\|v\| = r\}$.
The set $\partial M$ introduces discontinuities for $f$.

\medskip
\noindent
{\sc Singular and regular sets:} The {\em singular set} of $f$ is
$$
\mathfs S:=\left\{x\in M:\begin{array}{l} x\textrm{ does {\em not} have a relative neighborhood
$V\subset M\setminus\partial M$ s.t. }\\
\textrm{$V$ is diffeomorphic to an open ball, and $f:V\to f(V)$ and}\\
\textrm{$f^{-1}:V\to f^{-1}(V)$ are diffeomorphisms}
\end{array}\right\}.
$$
The {\em regular set} is $M'=M\setminus\mathfs S$.

\begin{lemma}\label{Lemma_Smooth_Section}
The maps $R,f,f^{-1}$ are differentiable on $M'$, and there is $\mathfrak C>0$
only depending on $g$ s.t. $\sup_{x\in M'}\|df_x\|<\mathfrak C$, $\sup_{x\in M'}\|(df_x)^{-1}\|<\mathfrak C$
and $\|f\restriction_{U}\|_{C^{1+\beta}}<\mathfrak C,\|f^{-1}\restriction_{U}\|_{C^{1+\beta}}<\mathfrak C$
for all open and connected $U\subset M'$.
\end{lemma}

This is \cite[Lemma 2.5]{Lima-Sarig}, and the proof is the same.
There is a relation between $g$--invariant and $f$--invariant measures.

\medskip
\noindent
{\sc Induced measure:\/} Every $g$--invariant probability measure $\mu$ on $N$ induces an
$f$--invariant measure $\nu$ on $M$ by the equality
$$
\int_N \vf d\mu=\frac{1}{\int_M R d\nu}\int_M\left(\int_0^{R(x)}\vf[g^t(x)]dt\right)d\nu(x),
\ \text{for all }\vf\in L^1(\mu).
$$

\medskip
\noindent
{\sc Adapted Poincar\'e section:} A standard Poincar\'e section $M$ is {\em adapted} for $\mu$ if
$\lim\limits_{n\to\pm\infty}\frac{1}{n}\log d_M(f^{n}(x),\mathfs S)=0$
for $\nu$--a.e. $x\in M$.

\medskip
A fortiori, $\nu(\mathfs S)=0$. Above, $d_M$ is the intrinsic Riemannian distance on $M$,
with the convention that the distance between different connected components of $M$ is infinite.
The next result states that any $g$--invariant probability measure has adapted Poincar\'e sections.

\begin{theorem}\label{Thm-Adapted-Section}
Every $g$--invariant probability measure $\mu$ has adapted Poincar\'e sections
with arbitrarily small radius and roof function.
\end{theorem}

This is the higher dimensional version of \cite[Theorem 2.8]{Lima-Sarig}, whose proof
consists on constructing a one-parameter family $\{M_r\}_{r\in[a,b]}$ of standard Poincar\'e sections
s.t. $M_r$ has radius $r$ and then using a Borel-Cantelli argument to show that $M_r$ is adapted
to Lebesgue a.e. $r\in[a,b]$. To construct the family of standard Poincar\'e sections, \cite{Lima-Sarig}
fixes a finite set of flows boxes that covers $N$ and then takes a fine discretization of the associated
canonical transverse balls. Then they show that it is possible to lift/decrease the points of the discretization
in the flow direction and take smaller canonical transverse balls centered at these displaced points 
so that the balls are pairwise disjoint. This argument uses Lemmas \ref{LemmaFB_1} and \ref{LemmaFB_2},
which are also true in our case. Hence the same proof applies.

Now we discuss nonuniform hyperbolicity. Consider a $g$--invariant probability measure $\mu$.
Let $\mu$ be $\chi$--hyperbolic, let $M$ be an adapted Poincar\'e section for $\mu$, and let
$\nu$ be the induced measure on $M$. The derivative cocycle
$df^n_x:T_xM\to T_{f^n(x)}M$ is well-defined $\nu$--a.e.
By Lemma \ref{Lemma_Smooth_Section}, the functions $\log\|df_x\|, \log\|df_x^{-1}\|$ are bounded,
hence the Oseledets theorem applies and so $f$ possesses
well-defined Lyapunov exponents $\nu$--a.e. The following lemma proves
that $\nu$ is hyperbolic. Let $\chi'=\chi\inf R$.

\begin{lemma}\label{Lemma_Lyap_Exp}
If $\mu$ is $\chi$--hyperbolic for $g$ then $\nu$ is $\chi'$--hyperbolic for $f$.
\end{lemma}

This is the higher dimensional version of \cite[Lemma 2.6]{Lima-Sarig}, and the proof
is the same.

\begin{proof}[Proof of Theorem \ref{Thm-Flows}]
Let $\mu$ be $\chi$--hyperbolic, let $M$ be an adapted Poincar\'e section for $\mu$ given by Theorem 
\ref{Thm-Adapted-Section}. Let $f:M\to M$ be its Poincar\'e map and $\nu$ be the measure on $M$
induced by $\mu$. By Lemma \ref{Lemma_Lyap_Exp}, $\nu$ is $\chi'$--hyperbolic.
Since $M$ is the disjoint union of smooth balls of codimension one,  
assumptions (A1)--(A4) are satisfied. By Lemma \ref{Lemma_Smooth_Section},
assumptions (A5)--(A7) are satisfied. Apply Theorem \ref{Thm-Main}
for $f$ and the parameter $\chi'>0$ to get a triple $(\Sigma,\sigma,\pi)$. Since $\nu$ is $\chi'$--hyperbolic,
it is carried by $\pi[\Sigma^\#]={\rm NUH}^\# _{\chi'}$. Now proceed as in \cite[pp. 230--232]{Lima-Sarig}
to construct a coding $\pi_r:\Sigma_r\to N$ satisfying Theorem \ref{Thm-Flows}.
\end{proof}

Concluding this section, we prove Theorems \ref{Thm-Flows-periodic} and \ref{Thm-Flows-Bernoulli}.
Start recalling the definition of the geometric potential of a flow and its equilibrium measures,
see e.g. \cite{LLS-2016}. Let $\chi>0$. The following definition is the flow version of the nonuniformly
hyperbolic locus.

\medskip
\noindent
{\sc The nonuniformly hyperbolic locus ${\rm NUH}_\chi$:} It is defined as the set of points 
$x\in N $ for which there is a splitting $T_xN=E^s_{x}\oplus X_x\oplus E^u_{x}$ s.t.:
\begin{enumerate}[(NUH1)]
\item Every $v\in E^s_{x}$ contracts in the future at least $-\chi$ and expands in the past:
$$\limsup_{t\to+\infty}\tfrac{1}{t}\log \|dg^{-t} v\|\leq -\chi\ \text{ and } 
\ \liminf_{t\to+\infty}\tfrac{1}{t}\log \|dg^{t} v\|>0.
$$ 
\item Every $v\in E^u_{x}$ contracts in the past at least $-\chi$ and expands in the future:
$$\limsup_{t\to+\infty}\tfrac{1}{t}\log \|dg^{-t} v\|\leq -\chi\text{ and } 
\liminf_{t\to+\infty}\tfrac{1}{t}\log \|dg^{t}v\|>0.$$
\item \label{Def-NUH3-flow} The parameters $s(x)=\sup\limits_{v\in E^s_{x}\atop{\|v\|=1}}S(x,v)$
and $u(x)=\sup\limits_{w\in E^u_{x}\atop{\|w\|=1}}U(x,w)$ are finite, where:
\begin{align*}
S(x,v)&=\sqrt{2}\left(\sum_{n\geq 0}e^{2n\chi}\|dg^nv\|^2\right)^{1/2},\\
U(x,w)&=\sqrt{2}\left(\sum_{n\geq 0}e^{2n\chi}\|dg^{-n}w\|^2\right)^{1/2}.
\end{align*}
\end{enumerate}
\medskip
\noindent
{\sc Hyperbolic locus of $g$:} The {\em hyperbolic locus} of $g$ is the set $N_{\rm hyp}$ defined as 
$$
N_{\rm hyp}:=\bigcup_{\chi>0}{\rm NUH}_\chi.
$$

\medskip
\noindent
{\sc Geometric potential of $g$ \cite{Bowen-Ruelle-SRB}:} The {\em geometric potential} of $g$
is the function $J:N_{\rm hyp}\to\R$ defined by
$$
J(x)=-\tfrac{d}{dt}\Big|_{t=0}\log \det\left[dg^t  \restriction_{E^u_x}\right]=
-\lim_{t\to 0}\tfrac{1}{t}\log\det\left[dg^t \restriction_{E^u_{\wh x}}\right].
$$

\medskip
Observe that $J$ is bounded, since $g$ is $C^{1+\beta}$.
%Even though $J$ is not globally defined, we can define its equilibrium measures.

\medskip
\noindent
{\sc Equilibrium measure for $tJ$:} An {\em equilibrium measure} for $tJ$ is a
$g$--invariant probability measure $\mu$ supported on $N_{\rm hyp}$ s.t.
$P_{\rm top}(tJ)=h_\mu(g^1)+\int tJ d\mu$, where 
$$
P_{\rm top}(tJ):=\sup\left\{
h_\mu(g^1)+\int_N tJ d\mu:
\begin{array}{c}
\mu\text{ is $g$--invariant probability}\\
\text{ measure with }\mu(N_{\rm hyp})=1
\end{array}
\right\}.
$$

\medskip
Since $J$ is bounded, $P_{\rm top}(tJ)<\infty$.
The proof of Theorem \ref{Thm-Flows-Bernoulli} is made exactly as in \cite{LLS-2016},
since it establishes the ``Bernoulli up to a period'' property at the symbolic extension and then projects
to the flow. The same reasoning can be made to prove Theorem \ref{Thm-Flows-periodic}, using 
its symbolic version \cite[Theorem 8.1]{Lima-Sarig}.

\section{Multidimensional billiards}\label{Section-Billiards}

We refer to the books \cite{Cornfeld-Fomin-Sinai,Chernov-Markarian} and the survey 
\cite{Szasz-survey} for general texts on the dynamics of billiards. 
Let $(N,\langle\cdot,\cdot\rangle)$ be a $C^\infty$ compact manifold, let $\{g^t\}_{t\in\R}$ be its
geodesic flow, and let $B_1,\ldots,B_k$ be open sets of $N$ s.t. their boundaries
$\partial B_1,\ldots,\partial B_k$ are $C^3$ submanifolds of codimension one.
The sets $B_1,\ldots,B_k$ are called {\em scatterers}.
For each $x\in \partial B_i$, let $n(x)\in T_xN$ be the unitary vector orthogonal to $T_x \partial B_i$ that points
outward $B_i$ (points in two boundaries have two normal vectors). If $N_0=N\setminus \bigcup_{i=1}^k B_i$,
then we define the {\em billiard flow} $\vf=\{\vf^t\}_{t\in\R}$ on $N$ with scatterers $B_1,\ldots,B_k$
as the restriction of the geodesic flow to the unit tangent bundle $T^1N$ with the identification 
$v\sim v-2\langle n(x),v\rangle n(x)$ for $(x,v)\in T^1N_0$ s.t. $x\in \bigcup_{i=1}^k \partial B_i$.
In other words, we flow under the geodesic flow until we hit one of the scatterers, and when this
occurs we reflect the vector according to the law of optics, i.e. the angle of incidence equals
the angle of reflection. The flow $\vf$ is well-defined in a subset of $T^1N_0$. Letting
$M=\{(x,v)\in T^1N_0:x\in \bigcup_{i=1}^k \partial B_i\text{ and }\langle v,n(x)\rangle \geq 0\}$,
the {\em billiard map} $f$ is the map of successive hits to the scatterers, i.e.
$f(x,v)=(x',v')$ where $(x,v)$ represents a collision point with a direction and $(x',v')$ represents
the next collision point with the new direction. Formally, there are closed sets $\mathfs S^+,\mathfs S^-\subset M$
s.t. $f:M\setminus\mathfs S^+\to M$ and $f^{-1}:M\setminus\mathfs S^-\to M$ are diffeomorphisms
onto their images, so that we can regard $f$ as an invertible map with singular set $\mathfs S=\mathfs S^+\cup\mathfs S^-$. When the second fundamental form of the scatterers is positive (the scatterers are strictly convex),
the billiard is called {\em dispersing}, and when it is nonnegative (the scatterers are convex),
then the billiard is called {\em semi-dispersing}. See Figure \ref{figure-billiards} for examples.
The map $f$ has a natural invariant Liouville measure $\mu_{\rm SRB}=\cos \angle(n(x),v)dxdv$.
\begin{center}
\begin{figure}[hbt!]
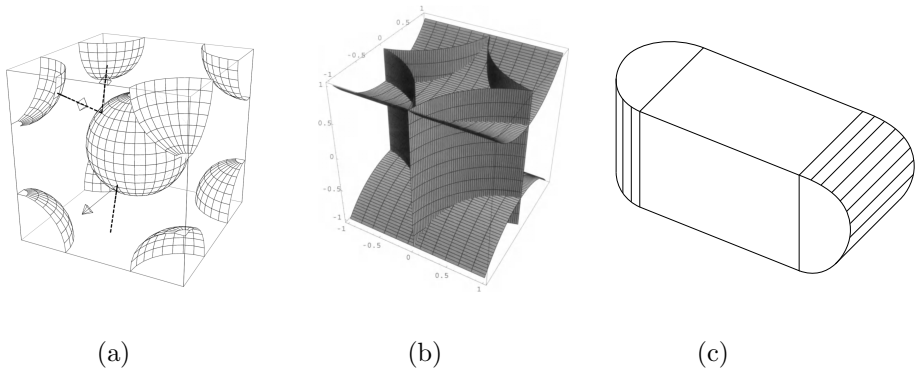

\centering
\def\svgwidth{12.3cm}

\caption{(a) Dispersing billiard in $\mathbb T^3$ (credits to Imre P\'eter T\'oth).
(b) Billiard of a hard balls system in $\mathbb T^3$ (credits to \cite{Posch-Hirschl}).
(c) Three-dimensional Bunimovich stadium.}
\label{figure-billiards}
\end{figure}
\end{center}

Billiard maps have been studied since the beginning of the twentieth century, one of the reasons
being their relation with statistical mechanics. Indeed, the dynamics associated to a
mechanical system of $m$ one-dimensional point particles in an interval with heavy walls 
is a billiard in a subset of $\R^m$. More generally, the famous Lorentz gas and hard balls systems
define billiards in some $\R^m$, see \cite[Chapter 6]{Cornfeld-Fomin-Sinai}.
The celebrated {\em Boltzmann-Sina{\u\i} hypothesis} states that in billiards of $N\geq 2$
hard balls of unit mass moving on the flat torus $\mathbb T^d=\R^d/\Z^d$, $d\geq 2$, the
measure $\mu_{\rm SRB}$ is ergodic after the obvious reductions coming from physical invariants \cite{Sinai-1963}. 
The billiards of classical Lorentz gases are dispersing, see Figure \ref{figure-billiards}(a).
For hard balls systems, the scatterers are cylinders and so the billiards are semidispersing,
see Figure \ref{figure-billiards}(b). 

There is a large literature on billiards, which includes proving the existence of invariant manifolds,
nonuniform hyperbolicity, ergodicity, the K property, the Bernoulli property, decay of correlations, etc.
The proofs of ergodicity follow Sina{\u\i}'s approach of first obtaining local ergodicity via the so-called
{\em fundamental theorem}. Many authors contributed to this question,
among them Sina\u{\i}, Bunimovich, Chernov, Kr\'amli, Sim\'anyi, Sz\'asz
\cite{Sinai-billiards,Bunimovich-Sinai-fundamental,Chernov-Sinai-1987,KSS-fundamental}.
While dynamical properties where checked for two-dimensional billiards, most of the
above contributions assume that these same properties hold for billiards in higher dimension.
For instance, it was commonly assumed that singularity submanifolds have bounded curvature,
which turned not to be always true \cite{Balint-et-al-Asterisque}.
Fortunately, the same authors showed that the theory remains true under a weaker condition called
{\em Lipschitz decomposability}, see \cite{Balint-et-al-IHP}. They also showed that algebraic semi-dispersing
billiards are Lipschitz decomposable (here, algebraic means that the scatterers are defined by algebraic equations).
Since billiards coming from hard balls systems are algebraic and $\mu_{\rm SRB}$ is hyperbolic
\cite{Simanyi-Szasz-completely-hyperbolic}, it was possible to establish the Boltzmann-Sina{\u\i} ergodic hypothesis,
whose contributions include but are not limited to
\cite{KSS-3-balls,KSS-4-balls,Simanyi-1992-II,Simanyi-1992-I,Simanyi-conditional,Simanyi-singularities}.
Actually, $\mu_{\rm SRB}$ is Bernoulli \cite{Chernov-Haskell}.

There is another class of billiards, known as {\em nowhere dispersing}, where
the boundary of scatterers is nowhere convex and yet $\mu_{\rm SRB}$ is hyperbolic.
In dimension two, the first examples are due to Bunimovich
\cite{Bunimovich-close-to-scattering,Bunimovich-ergodic-properties,Bunimovich-Nowhere-dispersing},
which include the famous Bunimovich stadium. There are some generalizations in higher dimension 
\cite{Bunimovich-Rehacek-NUH,Bunimovich-Rehacek-any-dimension,Bunimovich-delMagno-hyperbolicity},
and it is known that $\mu_{\rm SRB}$ is ergodic
\cite{Bunimovich-Rehacek-ergodicity,Bunimovich-delMagno-ergodicity}.
The example considered in \cite{Bunimovich-delMagno-hyperbolicity,Bunimovich-delMagno-ergodicity}
is depicted in Figure \ref{figure-billiards}(c), and can be regarded as a three-dimensional
version of the Bunimovich stadium.

Ledrappier \& Strelcyn have shown that the Pesin entropy formula holds for $\mu_{\rm SRB}$
for a class of systems that includes the billiards mentioned above \cite[Part III]{KSLP},
therefore in each of these cases $h=h_{\mu_{\rm SRB}}(f)>0$. The application
of Theorem \ref{Thm-Main} for multidimensional billiards deals with the growth rate on the
number of periodic orbits. Let ${\rm Per}_n(f)$ denote the number of periodic points of
period $n$ for $f$. Up to the authors' knowledge, the only available result about ${\rm Per}_n(f)$
in the literature is due to Stoyanov, who proved that for semi-dispersing billiards ${\rm Per}_n(f)$ grows
at most exponentially fast \cite{Stojanov-periodic-points}. Our result provides a sufficient
condition for the growth rate to be indeed exponential.

\begin{theorem}
Let $f$ be a billiard map in any dimension, and assume that
$\mu_{\rm SRB}$ is a hyperbolic mixing measure with entropy $h>0$.
Then there exists $C>0$ s.t. ${\rm Per}_n(f)\geq Ce^{hn}$ for all $n$ large enough.
\end{theorem}

This theorem is a direct consequence of Theorem \ref{Thm-Application-3}
(actually, if $\mu_{\rm SRB}$ is $\chi$--hyperbolic then $\liminf\limits_{n\to\infty}e^{-hn}{\rm Per}_{n,\chi}(f)\geq 1$).
It applies to all examples in Figure \ref{figure-billiards}. Note that,
the bigger $h$ is, the better is the estimate. The usual way of obtaining sharp
estimates is considering measures of maximal entropy, but in the broad setup considered here
we do not know if measures of maximal entropy exist nor if they are 
adapted and hyperbolic. The only situation where these are true is for a class of two-dimensional
dispersing billiards \cite{Baladi-Demers-MME}.
Using \cite{Lima-Matheus,Baladi-Demers-MME,Buzzi-2019},
it follows that ${\rm Per}_n(f)\geq Ce^{nh_*}$ for all $n$ large enough, where $h_*$ is the topological 
entropy of $f$. It would be interesting to extend this result to higher dimensions.

\section{Viana maps}\label{Section-Viana-maps}

We finish our applications proving Theorems~\ref{Thm-Viana-periodic} and \ref{Thm-Viana-Bernoulli}.
Recall that, for $a_0\in (1,2)$ s.t. $t=0$ is pre-periodic for the quadratic map $x\mapsto a_0-t^2$,
for $d\geq 16$ and $\alpha>0$, the associated Viana map is the skew product 
$f=f_{a_0,d,\alpha}:\mathbb S^1\times \mathbb{R}\to \mathbb S^1\times \mathbb{R}$ defined by 
$f(\theta,t)=(d\theta,a_0+\alpha \sin(2\pi \theta)-t^2)$, whose restriction to the annulus $\mathbb S^1\times I_0$
has an attractor. In the sequel, $f$ is considered as this restriction.
The map $f$ is $C^\infty$ with a critical set $\mathbb S^1\times\{0\}$ where $df$ is non-invertible and,
when $\alpha>0$ is small, it has a unique invariant probability measure $\mu_{\rm SRB}$
that is absolutely continuous with respect to Lebesgue \cite{Alves-SRB-2000,Alves-Viana-2002},
nonuniformly expanding \cite{Viana-maps}, and mixing \cite{Alves-Luzzato-Pinheiro}.
All these properties hold for a $C^3$ neighborhood of $f$. We claim that
assumptions (A1)--(A7) hold for every map in a $C^3$ neighborhood of $f$.
We also show that, for maps in this neighborhood, every measure with large entropy is adapted.
Once these are proved, Theorem \ref{Thm-Viana-periodic} follows from Theorem \ref{Thm-Application-3},
and Theorem \ref{Thm-Viana-Bernoulli} follows from Corollary \ref{Coro-Application-1}.

\begin{proposition}\label{prop.viana-conditions}
Every $f$ as above possesses a $C^2$ neighborhood $\mathfs U_f$ s.t. every map 
in $\mathfs U_f$ satisfies assumptions {\rm (A1)--(A7)}.
\end{proposition}

\begin{proof}
Start observing that $\mathbb S^1\times\R$ has zero curvature, hence (A1)--(A4) are automatic.
It remains to obtain (A5)--(A7).
In this proof, we will write $f_0:=f_{a_0,d,\alpha}$ and let $f$ be a map in a $C^2$ neighborhood of $f_0$.
While we will check (A5) directly, recall from Remark~\ref{remark-mult-constants} that (A6)--(A7) follow
from simpler estimates, allowing multiplicative constants and the Lipschitz constants of the derivatives
of inverse branches $g:E_x\to g(E_x)$ to be bounded by a multiple of $d(x,\mathfs S)^{-c}$ for some $c>0$.
%Since $f$ and its inverse branches are $C^2$, assumption (A7) will follow 
%once we prove that $\|dg\|_{C^0}\leq d(x,\mathfs S)^{-c}$ for some $c>0$ fixed,
%and this is guaranteed by (A6).

The map $f_0$ is $C^2$ with $df_0(\theta,t)=\begin{bmatrix}d & 0 \\ * & -2t\end{bmatrix}$, 
hence it has no discontinuities and its critical/singular set is equal to
$\mathfs S_{f_0}=\mathbb S^1\times \{0\}$. We fix a $C^2$ neighborhood $\mathfs U$ of $f_0$
s.t. for all $f\in\mathfs U$:
\begin{enumerate}[$\circ$]
\item $\frac{\partial}{\partial t}\det df_x<-2$ for $x=(\theta,t)$.
\item $\|df\|_{C^0}<\infty$ and $\min\limits_{x\in \mathbb S^1\times I_0}\|df_x\|>0$. 
\end{enumerate}
Using the first condition, the implicit function theorem implies that the singular set of $f$ is
a graph over $\mathbb S^1$, i.e. $\mathfs S_f=\{(\theta,t_\theta);\theta\in \mathbb S^1\}$.
 
Fix $f\in \mathfs U$, and write $\mathfs S=\mathfs S_f$ and $h(x)=|\det df_x|=\|df_x\|\cdot \|(df_x)^{-1}\|^{-1}$.
We start estimating $\|(df_x)^{-1}\|^{-1}$ from below in terms of $d(x,\mathfs S)$.
Fix $x\in \mathbb S^1\times I_0\setminus \mathfs S$.
By assumption, $\left|\tfrac{\partial h}{\partial t}\right|>2$ and so by the mean value theorem
$h(x)=h(\theta,t)-h(\theta,t_\theta)\geq 2|t-t_\theta|\geq 2d(x,\mathfs S)$, hence 
$$
\|(df_x)^{-1}\|^{-1}=\tfrac{h(x)}{\|df_x\|}\geq 2\|df\|_{C^0}^{-1}d(x,\mathfs S).
$$
%where $\mathfs L^{-1}=\|df\|_{C^0}=\max\limits_{y\in \mathbb S^1\times [-2,2]}\|df_y\|$.
By the quantitative version of the inverse function theorem, see e.g. \cite{Smart}, there
is a constant $K=K(f)$ s.t. $f$ is invertible in a ball centered at $x$ with radius
$K\|(df_x)^{-1}\|^{-1}\geq 2K\|df\|_{C^0}^{-1} d(x,\mathfs S)$.
Renaming and reducing $K$ if necessary, we can take $r_1(x):=Kd(x,\mathfs S)$ with $K<\tfrac{1}{2}$ and
assume that the restriction of $f$ to $B_x:=B(x,r_1(x))$ is a diffeomorphism. Now we prove (A5)--(A7).

\medskip
\noindent
(A5) If $y\in B_x$ then $d(y,\mathfs S)\geq \tfrac{1}{2}d(x,\mathfs S)$, hence
$\|(df_y)^{-1}\|^{-1}\geq 2\|df\|_{C^0}^{-1}d(y,\mathfs S)\geq \|df\|_{C^0}^{-1}d(x,\mathfs S)$.
Therefore $f(B_x)$ contains the ball with center $f(x)$ and radius
$\|df\|_{C^0}^{-1}d(x,\mathfs S)r_1(x)=K\|df\|_{C^0}^{-1}d(x,\mathfs S)^2$ and so
$\mathfrak r(x):=\tfrac{1}{2}K\|df\|_{C^0}^{-1}d(x,\mathfs S)^2$ satisfies (A5) up to multiplicative constants.

\medskip
\noindent
(A6) By hypothesis, $\|df\|_{C^0}<\infty$ and $\min\limits_{y\in \mathbb S^1\times I_0}\|df_y\|>0$
and so the assumption on $df$ is satisfied up to multiplicative constants.
Now let $g:E_x\to g(E_x)$ be the inverse branch of $f$
taking $f(x)$ to $x$. Let $z\in E_x$ and $g(z)=y\in B_x$. We have
$\|dg_z\|=\|(df_y)^{-1}\|\geq \tfrac{1}{\|df_y\|}\geq \|df\|_{C^0}^{-1}$ and
$$
\|dg_z\|=\|(df_y)^{-1}\|=\tfrac{\|df_y\|}{h(y)}\leq \tfrac{\|df\|_{C^0}}{2d(y,\mathfs S)}\leq \|df\|_{C^0}d(x,\mathfs S)^{-1},
$$
thus proving that $dg$ satisfied (A6) up to multiplicative constants.

\medskip
\noindent
(A7) The estimate on $df$ holds because $f$ has bounded $C^2$ norm. 
We now estimate the Lipschitz constant of $dg$. If $z,w\in E_x$ then
\begin{align*}
&\,\|dg_z-dg_w\|=\|(df_{g(z)})^{-1}-(df_{g(w)})^{-1}\|\\
&\leq\|(df_{g(z)})^{-1}\|\cdot\|(df_{g(w)})^{-1}\|\cdot\|df_{g(z)}-df_{g(w)}\|\\
&\leq \|d^2f\|_{C^0}\cdot\|(df_{g(z)})^{-1}\|\cdot\|(df_{g(w)})^{-1}\|\cdot d(g(z),g(w))\\
&\leq \left(\|df\|_{C^0}\|d^2f\|_{C^0}\|(df_{g(z)})^{-1}\| \|(df_{g(w)})^{-1}\| d(x,\mathfs S)^{-1}\right) d(z,w)\\
&\leq  \left(\|df\|_{C^0}^3\|d^2f\|_{C^0}d(x,\mathfs S)^{-3}\right) d(z,w),
\end{align*}
hence we obtain (A7) for $g$ up to multiplicative constants.
\end{proof}

%By \cite{Alves-SRB-2000} and \cite{Alves-Viana-2002} there exists a unique ergodic absolutely continuous invariant measure for any $\tilde{f}$ in a $C^3$ neighborhood of $f$. The following propositions shows that any measure with entropy higher than the absolutely continuous invariant measure is adapted.

\begin{proposition}\label{prop.viana-adapted}
If $f$ is as above, then every ergodic $f$--invariant measure $\mu$ with
$h_\mu(f)\geq h_{\mu_{\rm SRB}}(f)$ is hyperbolic and $f$--adapted.
The same applies to every map that is $C^3$ close to $f$.
\end{proposition}

\begin{proof}
Similarly to the previous proposition, write $f_0:=f_{a_0,d,\alpha}$ and let $f$
be a map in a $C^3$ neighborhood of $f_0$. %Start noting that
Given $(\theta_0,t_0)\in\mathbb S^1\times I_0$, 
we set $(\theta_\ell,t_\ell):=f_0^\ell(\theta_0,t_0)$. Thus
$
(df_0^n)_{(\theta_0,t_0)}=
\begin{bmatrix}
d^n & 0\\
b_n(\theta_0,t_0) & c_n(\theta_0,t_0)
\end{bmatrix}
$
where $c_n(\theta_0,t_0)=(-2)^n\prod_{\ell=0}^{n-1}t_\ell$ and 
$|b_n(\theta,t)|\leq  n d^n$ for all $n\geq1$, hence 
$\lim\limits_{n\to \infty}\frac{1}{n}\log \|(df_0^n)_x\|=\log(d)$ for all $x\in \mathbb S^1\times I_0$.
Moreover, if $\ve>0$ then there is $\delta>0$ s.t. every $f$ with $d_{C^1}(f,f_0)<\delta$ 
satisfies
$$
\log(d)-\ve\leq\liminf_{n\to \infty}\tfrac{1}{n}\log \| df^n_x\|\leq 
\limsup_{n\to \infty}\tfrac{1}{n}\log \|df^n_x\|\leq \log(d)+\ve.
$$
This implies that if $\lambda_1\geq \lambda_2$ are the Lyapunov exponents of $f$ with respect
to $\mu_{\rm SRB}$, then $\log(d)-\ve\leq \lambda_1\leq \log(d)+\ve$.
By \cite{Viana-maps}, there is a constant $c=c(f_0)>0$ s.t. if $f$ is $C^3$ close to $f_0$
then $\lambda_2>c$. In the sequel, we fix $\ve<\tfrac{c}{2}$ and take $f$ s.t.
both estimates on $\lambda_1,\lambda_2$ hold.

Let $\mu$ be an ergodic $f$--invariant probability measure with $h_\mu(f)\geq h_{\mu_{\rm SRB}}(f)$.
We first prove that $\mu$ is hyperbolic.
By the Pesin entropy formula (see \cite{Liu-entropy-endo} for a version for endomorphisms),
$h_{\mu_{\rm SRB}}(f)=\lambda_1+\lambda_2$ and by Ruelle inequality we have 
$h_\mu(f)\leq  \overline{\lambda}_1+\max\{0,\overline{\lambda}_2\}$,
where $\overline{\lambda}_1\geq\overline{\lambda}_2$ are the Lyapunov exponents of $f$ for
$\mu$. 

Let us prove that $\overline{\lambda}_2>0$. By contradiction assume $\overline{\lambda}_2\leq 0$. Recall that by assumption, $|\lambda_1-\overline{\lambda}_1|\leq 2\ve$
and so the condition $\overline{\lambda}_1\geq h_\mu(f)\geq h_{\mu_{\rm SRB}}(f)$ implies that
$$
\overline{\lambda}_1\geq \lambda_1+\lambda_2
$$
which gives that $0\geq \lambda_2-|\lambda_1-\overline{\lambda}_1|>c-2\ve>0$, a contradiction.
This proves that $\mu$ is hyperbolic and $\overline{\lambda}_2>0$. It remains to prove that it is $f$--adapted.
Since $f$ is $C^2$, there is a constant $C>0$ s.t. $\det df_x\leq C d(x,\mathfs S)$
and so it is enough to show that $\log(\det df_x)\in L^1(\mu)$. This is easy: 
first observe that $\log^+(\det df_x)$ is uniformly bounded, then 
$\log^+(\det df_x)\in L^1(\mu)$ so we can apply Kingman theorem to $\log(\det df_x)$. 
Recall that $\lim_{n\to\infty}\tfrac{1}{n}\log(\det df^n_x)=\overline{\lambda}_1+\overline{\lambda}_2$, 
$\mu$ almost everywhere. Kingman theorem implies that 
$\int \lim_{n\to\infty}\tfrac{1}{n}\log(\det df^n_x) d\mu=\int \log(\det df_x) d\mu$, hence  
$\int \log(\det df_x)d\mu\in (0,\infty)$. The proof is complete.
%On the other hand we have $\lambda_2(\mu)=\int \log\det(D\tilde{f})d\mu-\lambda_1(\mu)$ so in particular this implies that 
%$\log\det(D\tilde{f})$ belongs to $L_1(\mu)$, as $\tilde{f}$ is $C^2$ we have that $\det(D\tilde{f}(\theta,x))\leq C d((\theta,x),\mathfs{S})$ for some $C>0$ then
%$$
%-\infty <\int  \log\det(D\tilde{f})d\mu -\log C \leq \int \log d((\theta,x),\mathfs{S})d\mu.
%$$
\end{proof}

\section{Appendix: The graph transform method}

In this appendix, we explain how to define and construct the graph transforms
induced by a map that is $C^{1+\delta}$ close to a hyperbolic matrix, for some $\delta>0$.
The results proved here are used in Section \ref{Sec-graph-transform}.
We start recalling some definitions introduced in Section \ref{Section-preliminaries}.
Let $\norm{\cdot}$ be the euclidean norm of $\mathbb R^n$, induced by the canonical
inner product $\langle \cdot,\cdot\rangle_{\R^n}$ of $\R^n$.
Given a linear transformation $T:\R^n\to \R^n$, let $\|T\|=\sup\limits_{v\in\R^n\setminus\{0\}}\tfrac{\|Tv\|}{\|v\|}$.
Given an open bounded set $U\subset \R^n$ and $h:U\to \R^m$,
let $\|h\|_{C^0}:=\sup_{x\in U}\|h(x)\|$ denote the $C^0$ norm of $h$. 
When the domain $U$ is not clear in the context, we will write $\|h\|_{C^0(U)}$.
For $0<\delta\leq 1$, let $\Hol{\delta}(h):=\sup\frac{\norm{h(x)-h(y)}}{\|x-y\|^\delta}$ 
where the supremum ranges over distinct elements $x,y\in U$. 
We call $h$ a {\em contraction} if ${\rm Lip}(h)<1$, and an {\em expansion} if $h$ is invertible
and ${\rm Lip}(h^{-1})<1$. When $h$ is a linear transformation, then $\|h\|={\rm Lip}(h)$.
We also define $\norm{h}_{C^\delta}=\norm{h}_{C^0}+\Hol{\delta}(h)$.
If $h$ is differentiable then ${\rm Lip}(h)=\|dh\|_{C^0}$. Let
$\|h\|_{C^1}:=\|h\|_{C^0}+\|dh\|_{C^0}=\|h\|_{C^0}+{\rm Lip}(h)$ denote its $C^1$ norm,
and for $0<\delta\leq 1$ let
$$
\|h\|_{C^{1+\delta}}:=\|h\|_{C^1}+\Hol{\delta}(dh)=\|h\|_{C^0}+\|dh\|_{C^\delta}=
\|h\|_{C^0}+\|dh\|_{C^0}+\Hol{\delta}(dh)
$$
denote its $C^{1+\delta}$ norm. Below we list some basic properties of these norms.

\begin{lemma}\label{App-Lemma-Holder-norm}
The following holds:
\begin{enumerate}[{\rm (1)}]
\item If $\vf,\psi$ are maps s.t. $\vf\circ\psi$ is well-defined, then
$$
\left\{
\begin{array}{l}
\Hol{\delta}(\vf\circ\psi)\leq {\rm Lip}(\vf)\Hol{\delta}(\psi)\\
\\
\Hol{\delta}(\vf\circ\psi)\leq \Hol{\delta}(\vf){\rm Lip}(\psi)^\delta.
\end{array}
\right.
$$
In particular, if $\psi$ is a contraction then $\norm{\vf\circ\psi}_{C^\delta}\leq \norm{\vf}_{C^\delta}$.
\item {\rm (Perturbation of identity)} If $\norm{\vf}_{C^\delta}<1$, then $I+\vf$ is invertible and
$$
\norm{(I+\vf)^{-1}}_{C^\delta}\leq \tfrac{1}{1-\norm{\vf}_{C^\delta}}\cdot
$$
Similarly, if $\vf$ is linear and $\norm{\vf}<1$, then $I+\vf$ is invertible and
$$
\norm{(I+\vf)^{-1}}\leq \tfrac{1}{1-\norm{\vf}}\cdot
$$
\end{enumerate}
\end{lemma}
 
Obviously, the euclidean norm is equivalent to any other norm in $\R^n$, but the 
notions of contraction and expansion differ from the choice of the norm. 
On the other hand, statements claiming that a norm is small are not very sensitive
to the choice of the norm.
Given $r>0$, denote the ball of $\R^n$ with center 0 and radius $r$ by $B^n[r]$. 
Observe that if $0\leq d\leq n$ then $B^n[r]\subset B^d[r]\times B^{n-d}[r]\subset B^n[\sqrt{2}r]$.

In this appendix, we fix some parameters:
\begin{enumerate}[$\circ$]
\item $m>0$, the dimension of the euclidean space.
\item $0\leq d\leq m$, the dimension of one of the invariant subspaces. The other invariant subspace
has dimension $m-d$. 
\item $\delta\in (0,1)$, the regularity of derivatives.
\item $\chi>0$, that controls the minimal contraction and expansion.
\item $\ve>0$, the parameter of approximation.
\item $r>0$, the size of the domain where the maps will be defined.
\end{enumerate}
The reader should have in mind that the construction will be made for given parameters
$m,d,\delta,\chi$,
and that $\ve,r$ will be chosen to make the calculations work. The order of choice will be as follows: $\ve>0$ will
be arbitrarily small to make a finite number of inequalities to hold, and $r$ will be of the form
$\ve^{2/\delta}\times C$ with $C<1$. In particular, $2r^\delta<\ve$ for small $\ve>0$,
and we will use this inequality many times. Since $m$ will be fixed throughout the appendix,
we denote the ball of $\R^m$ with center 0 and radius $r$ simply by $B[r]$.

\subsection{Admissible graphs}

The introduction of admissible graphs requires fixing some extra parameters:
\begin{enumerate}[$\circ$]
\item $p\leq r/2$, the parameter that control the sizes of stable graphs.
\item $q\leq r/2$, the parameter that control the sizes of unstable graphs.
\end{enumerate}
All definitions in this section depend on the choice of these parameters. 
For simplicity of notation, the first $d$ coordinates in $\R^m$ will represent
the unstable subspace.

\medskip
\noindent
{\sc Graphs:} Given a $C^{1+\delta}$ function $G:B^d[p]\to \R^{m-d}$, let
$$
V^u=\{(v,G(v)): v\in B^d[p]\}
$$
denote the graph of $G$, and call it a {\em $u$--graph}.
Similarly, given a $C^{1+\delta}$ function $G:B^{m-d}[q]\to \R^{d}$, let
$$
V^s:=\{(G(w),w): w\in B^{m-d}[q]\}
$$
denote the graph of $G$, and call it an {\em $s$--graph}. 
We call $G$ the {\em representing function}.

\medskip
Among the $s/u$--graphs, we are interested in those that can eventually represent
stable/unstable manifolds, and for that we require some extra conditions on the
representing functions. We call them {\em admissible graphs}.
The definition requires some conditions on the representing function
that depends on the minimum $p\wedge q$, so that we can later relate $s$--graphs and
$u$--graphs. To simplify the exposition and ease the notation, we will work only with unstable graphs
with a given parameter $p$, and introduce an extra parameter $\eta\leq p$ that, for applications 
in Section \ref{Sec-graph-transform},
will be equal to $\eta=p\wedge q$, so that the results obtained in this and next section
work for both $s$--admissible and $u$--admissible graphs. We will also
just use the terminology graph and not mention that it is $u$ or $s$ graph.
From now on, fix $\eta\leq p\leq r/2$.

\medskip
\noindent
{\sc Admissible graph:} A graph $V=\{(v,G(v)): v\in B^d[p]\}$ is called an
{\em admissible graph} if its representing function $G$ satisfies the following conditions:
\begin{enumerate}
\item[(AM1)] $\|G(0)\|\leq 10^{-3}\eta$.
\item[(AM2)] $\|(dG)_0\|\leq \tfrac{1}{2}\eta^\delta$.
\item[(AM3)] $\norm{dG}_{C^\delta}=\|dG\|_{C^0}+{\rm Hol}_\delta(dG)\leq \tfrac{1}{2}$.
\end{enumerate}

\medskip
The estimate on the representing function is linear in $\eta$, while
the estimates on its derivative is of the order of $\eta^\delta$. This distinction
will be clear when we start establishing estimates in the next results. We start
with a basic lemma.

\begin{lemma}\label{App-Lemma-introductory}
If $G$ is the representing function of an admissible graph, then:
\begin{enumerate}[{\rm (1)}]
\item $\|dG\|_{C^0}<\ve$ where the norm is taken in $B^d[p]$.
\item $\|G\|_{C^0(B^d[t])}<10^{-2}t$, for every $t\in [\eta,p]$.
\end{enumerate}
\end{lemma}

\begin{proof}
(1) By (AM2)--(AM3), if $v\in B^d[p]$ then
$$
\|(dG)_v\| \leq \|(dG)_0\|+\tfrac{1}{2}\|v\|^\delta\leq \tfrac{1}{2}\eta^\delta+\tfrac{1}{2}p^\delta\leq r^\delta<\ve.
$$

\medskip
\noindent
(2) By the mean value inequality, the inequality above and (AM1) we obtain that
for $v\in B^d[t]$ we have
$$
\|G(v)\|\leq \|G(0)\|+\|dG\|_{C^0}\|v\|\leq 10^{-3}\eta+\ve t\leq (10^{-3}+\ve)t<10^{-2}t,
$$
where in the last inequality we require $\ve>0$ to be small.
\end{proof}

\medskip
We let $\M=\M^u_{p,\eta}$
denote the set of all admissible graphs as defined above.
We define a metric on $\M$,
comparing the representing functions. More specifically, for $V_1,V_2\in\M$ with
representing functions $G_1,G_2$ and for $i\geq 0$, define
$d_{C^i}(V_1,V_2):=\|G_1-G_2\|_{C^i}$ where the norm is taken in $B^d[p]$.

\subsection{Graph transforms}\label{App-Section-graph-transforms}

Now we introduce a dynamics for which we want to induce an action 
on admissible graphs. For that, consider a map $F$ that is close to a hyperbolic linear map.
We will apply the map $F$ to elements of $\M=\M^u_{p,\eta}$.
The idea is to see (a restriction of) each image as an admissible graph in $\wM=\M^u_{\wt p,\wt\eta}$,
where $\wt p,\wt\eta$ satisfy $\wt \eta\leq \wt p$ and 
are related to $p,\eta$. The relation that the four parameters, $p,\eta,\wt p,\wt\eta$,
must satisfy are consequences of the contraction/expansion properties required on $F$.
When properly implemented, we will have constructed an operator $\F:\M\to \wM$,
called {\em unstable graph transform} or simply {\em graph transform}.

We start fixing the four parameters that will define the spaces $\M$ and $\wM$:
\begin{enumerate}[$\circ$]
\item $0<\eta\leq p$, the parameters defining the domain $\M=\M^u_{p,\eta}$ of $\F$. 
\item $0<\wt\eta\leq \wt p$, the parameters defining the codomain
$\wM=\M^u_{\wt p,\wt\eta}$ of $\F$. 
\end{enumerate}
For $\ve>0$ small enough, we require the following:
\begin{enumerate}[ii..]
\item[(GT1)] $\wt p\leq e^\ve p$ and $\tfrac{\wt\eta}{\eta}=e^{\pm\ve}$.
\end{enumerate}
Now, let $F:B[r]\to\R^m$ be a $C^{1+\delta}$ map of the form $F=D+H$ with
$$
D=\left[\begin{array}{cc}D_1 & 0 \\
0 & D_2
\end{array}\right]
$$
s.t. the following conditions hold:
\begin{enumerate}[ii..]
\item[(GT2)] $D_1:\R^d\to\R^d$ and $D_2:\R^{n-d}\to\R^{n-d}$ are invertible linear maps
s.t. $\|D_1^{-1}\|<e^{-\chi}$ and $\|D_2\|<e^{-\chi}$.
\item[(GT3)] $H:B[r]\to\R^m$ satisfies:
\begin{enumerate}[(a)]
\item $\|H(0)\|<\ve\eta$.
\item $\|dH\|_{C^0(B[t])}<\ve t^\delta$, for every $t\in [\eta,2p]$.
\item ${\rm Hol}_\delta(dH)<\ve$, i.e. $\|dH_x-dH_y\|<\ve\|x-y\|^\delta$
for distinct $x,y\in B[r]$.
\end{enumerate}
\end{enumerate}
Let us explain conditions (GT1)--(GT3).
Conditions (GT2) and (GT3) state that $F$ is close to a hyperbolic matrix in the $C^{1+\delta}$ norm,
where the first $d$ coordinates expand and the last $m-d$ coordinates
contract\footnote{Recall that we have inverted the order
of $D_1,D_2$ in comparison to the main text.}.
To explain condition (GT1), consider an admissible graph $V$ with representing function
$G:B^d[p]\to\R^{m-d}$. We want to consider $F(V)$
as a new graph with representing function defined on $B^d[\wt p]$.
Since $F$ stretches $V$ horizontally roughly by $e^{\chi}$, if $\ve>0$ is small enough
then this growth is more than $e^{\ve}$.
By (GT1) we have $\wt p\leq e^\ve p$, therefore $F(V)$ contains an
admissible graph with representing function $\widetilde G:B^d[\wt p]\to \R^{m-d}$.
Moreover, by (GT1) the parameters $\eta,\wt\eta$ are comparable, and by the control on the error
term $H$ given by (GT2)--(GT3), we actually obtain that $\wt G$ satisfies
conditions (AM1)--(AM3) for the parameters $\wt p,\wt\eta$, thus defining an admissible graph
in $\wM$. The attentive reader might have observed the similarity
between (GT2)--(GT3) and the statement of Theorem \ref{Thm-non-linear-Pesin-2}.
It is indeed for $F_{\wh x,\wh y}^{\pm 1}$ that we will apply the results proved in this appendix.

Observe that, contrary to \cite{Sarig-JAMS,Ben-Ovadia-2019}, we do not require any uniform upper
bound on $\|D_1\|$ neither lower bound on $\|D_2\|$, so that the results can be applied
to the greater generality considered in the main text.
Indeed, classical works on the graph transform are set up similar to what we do, see e.g.
\cite{KSLP,Katok-Hasselblatt-Book}.
We start defining $\F$ properly.

\medskip
\noindent
{\sc Graph transform $\mathfs F$:} The {\em graph transform}
is the map $\F:\M\to\wM$ that sends an admissible graph $V\in \M$ with representing function $G$
to the unique admissible graph $\widetilde V$ with representing function
$\widetilde G:B^d[\wt p]\to\R^{m-d}$ s.t.
$\{(\wt v,\widetilde G(\wt v)):\wt v\in B^d[\wt p]\}\subset F(V)$.

\medskip
We already claim that $\wt V\in \wM$, but this requires
checking properties (AM1)--(AM3), that we will do in the sequel.
From now on write $H=(h_1,h_2)$, so that
$$
F(v,w)=(D_1v+h_1(v,w),D_2+h_2(v,w)).
$$

\subsubsection{Proof that $F(V)$ restricts to a graph}

We start showing that $F(V)$ can be restricted to a graph. This means that, in the proper
domain, the first coordinate of $F(V)$ is injective.

\begin{lemma}\label{App-Lemma-inverse-function}
The following holds for all $\ve>0$ small enough. The map $\Psi:B^d[p]\to \R^d$ defined by
\begin{align}\label{Def-Psi}
\Psi(v)=D_1v+h_1(v,G(v))
\end{align}
is injective, its image $\Psi(B^d[p])$ contains $B^d[e^{\chi-\sqrt{\ve}}p]$,
and its inverse $\Phi:\Psi(B^d[p])\to B^d[p]$  satisfies the following:
\begin{enumerate}[{\rm (1)}]
\item $\|d\Phi\|_{C^0}<e^{-\chi+\ve}$.
\item $\|\Phi(0)\|<2\ve\eta e^{-\chi+\ve}<2\ve\eta$.
\item $\norm{d\Phi}_{C^\delta}<e^{-\chi+3\ve}$.
\end{enumerate}
\end{lemma}

\begin{proof}
This lemma is a sole consequence of the expansion property of $D_1$.
We start proving that $\Psi$ is injective. By Lemma \ref{App-Lemma-introductory}(2),
if $v\in B^d[p]$ then $\|G(v)\|<10^{-2}p$,
and so $\norm{\begin{bmatrix} v \\ G(v)\end{bmatrix}}< 2p$. By (GT3),
we know that $\norm{dh_1}_{C^0(B[2p])}\leq \ve (2p)^\delta<\ve^2$.
Given $v_1,v_2\in B^d[p]$, we have
$$
\norm{(v_1,G(v_1))-(v_2,G(v_2))}=\norm{(v_1-v_2,G(v_1)-G(v_2))}\leq (1+\ve)\|v_1-v_2\|,
$$
since $\norm{G(v_1)-G(v_2)}\leq \|dG\|_{C^0}\norm{v_1-v_2}\leq \ve\norm{v_1-v_2}$.
Now assume that $\Psi(v_1)=\Psi(v_2)$, i.e.
$$
D_1(v_1-v_2)=h_1(v_2,G(v_2))-h_1(v_1,G(v_1)).
$$
We have $\norm{D_1(v_1-v_2)}\geq e^\chi\|v_1-v_2\|$, while the norm of the
right hand side of the above equality
is at most
$$
\norm{dh_1}_{C^0(B[2p])}\norm{(v_1,G(v_1))-(v_2,G(v_2))}\leq \ve^2(1+\ve)\|v_1-v_2\|,
$$
which implies that $v_1=v_2$, hence $\Psi$ is injective. Let 
$\Phi:\Psi(B^d[p])\to B^d[p]$ be the inverse of $\Psi.$

Before continuing, let $L(v)$ be the derivative of the map
$v\mapsto h_1(v,G(v))$, i.e.
$$
L(v)=(dh_1)_{(v,G(v))}\begin{bmatrix} I \\ dG_v\end{bmatrix}
$$
where $I$ is the $d\times d$ identity matrix.
By Lemma \ref{App-Lemma-introductory}(1), we have
$\norm{\begin{bmatrix} I \\ (dG)_v\end{bmatrix}}<1+\ve$ for all $v\in B^d[p]$, hence
$\norm{L}_{C^0}< 2\norm{dh_1}_{C^0(B[2p])}<2\ve^2$.

Now we show that the image of $\Psi$ contains $B^d[e^{\chi-\sqrt{\ve}}p]$.
We prove this using a fixed point theorem. Start observing that if $v\in B^d[p]$ then
$$
\norm{h_1(v,G(v))}\leq \norm{h_1(0,0)}+\norm{dh_1}_{C^0(B[2p])}\norm{\begin{bmatrix} v \\ G(v)\end{bmatrix}}
<\ve\eta+2\ve^2 p<2\ve p.
$$
If $\wt v\in B^d[e^{\chi-\sqrt{\ve}}p]$,
then $\Psi(v)=\wt v$ iff $v$ is a fixed point of the map $T:v\in B^d[p]\mapsto D_1^{-1}[\wt v-h_1(v,G(v))]$.
We have
$$
\norm{Tv}\leq \norm{D_1^{-1}}\left(\norm{\wt v}+\norm{h_1(v,G(v))}\right)
<e^{-\chi}\left(e^{\chi-\sqrt{\ve}}p+2\ve p\right)<\left(e^{-\sqrt{\ve}}+2\ve\right)p<p,
$$
where in the last passage we used that $e^{-\sqrt{\ve}}+2\ve<1-\sqrt{\ve}+3\ve<1$
if $\ve>0$ is small enough. Hence $T:B^d[p]\to B^d[p]$ and, since
$\norm{(dT)_v}=\norm{D_1^{-1}L(v)}\leq \norm{D_1^{-1}}\norm{L}_{C^0}<2\ve^2$,
the fixed point theorem implies that $T$ has a unique fixed point.
Since $\wt v\in B^d[e^{\chi-\sqrt{\ve}}p]$ is arbitrary, we get that 
$\Psi(B^d[p])$ contains $B^d[e^{\chi-\sqrt{\ve}}p]$.

The next step is to prove (1)--(3). We have $(d\Phi)_{\wt v}=[(d\Psi)_{v}]^{-1}$ where $v=\Phi(\wt v)$.
Letting $\wt L=L\circ\Phi$,
we have $(d\Psi)_v=D_1+\wt L(\wt v)=[I+\wt L(\wt v)D_1^{-1}]D_1$ and so 
$$
(d\Phi)_{\wt v}=D_1^{-1}[I+\wt L(\wt v)D_1^{-1}]^{-1}.
$$
We consider $I+\wt L(\wt v)D_1^{-1}$ as a perturbation of the
identity. We have
$\norm{\wt L D_1^{-1}}_{C^0}=\norm{L\Phi D_1^{-1}}_{C^0}\leq \norm{L}_{C^0}<2\ve^2$.
For a fixed $\wt v$, apply Lemma \ref{App-Lemma-Holder-norm}(2) to conclude that
$\norm{(d\Phi)_{\wt v}}<\tfrac{e^{-\chi}}{1-2\ve^2}<e^{-\chi+\ve}$,
where in the last inequality we used that $\tfrac{1}{1-2\ve^2}<1+\ve<e^\ve$
for $\ve>0$ small enough. Since $\wt v$ is arbitrary, we obtain (1).

To prove part (2), we start estimating $\norm{\Psi(0)}$. Since
$\begin{bmatrix} 0 \\ G(0)\end{bmatrix}\in B[\eta]$,
conditions (AM1) and (GT3) give that
\begin{align*}
&\,\norm{\Psi(0)}=\norm{h_1(0,G(0))}\leq
\norm{h_1(0,0)}+\norm{dh_1}_{C^0(B[\eta])} \norm{\begin{bmatrix} 0 \\ G(0)\end{bmatrix}}\leq
\ve\eta+\ve\eta^\delta 10^{-3}\eta\\
&\leq (\ve +10^{-3}\ve^2)\eta < 2\ve\eta,
\end{align*}
hence
$$
\norm{\Phi(0)}=\norm{\Phi(0)-\Phi(\Psi(0))}\leq e^{-\chi+\ve}\norm{\Psi(0)}<2\ve\eta e^{-\chi+\ve}.
$$

It remains to prove part (3). Again, we use that
$d\Phi=D_1^{-1}[I+\wt LD_1^{-1}]^{-1}$
and consider $I+\wt LD_1^{-1}$ as a perturbation of the
identity. We will estimate the values of $\Hol{\delta}(\wt LD_1^{-1}),\norm{\wt LD_1^{-1}}_{C^\delta}$ 
and, using this last estimate, apply Lemma \ref{App-Lemma-Holder-norm}(2) to conclude the proof.
Since $\wt LD_1^{-1}=L\Phi D_1^{-1}$ and $\Phi,D_1^{-1}$ are contractions,
the second estimate of Lemma \ref{App-Lemma-Holder-norm}(1) implies
that $\Hol{\delta}(\wt LD_1^{-1})\leq \Hol{\delta}(L)$, so it is enough to estimate 
$\Hol{\delta}(L)$. Recalling that $\norm{(v_1,G(v_1))-(v_2,G(v_2))}\leq (1+\ve)\|v_1-v_2\|$
for $v_1,v_2\in B^d[p]$, we have
\begin{align*}
&\,\norm{L(v_1)-L(v_2)}=
\norm{(dh_1)_{(v_1,G(v_1))}\begin{bmatrix} I \\ (dG)_{v_1}\end{bmatrix}-
(dh_1)_{(v_2,G(v_2))}\begin{bmatrix} I \\ (dG)_{v_2}\end{bmatrix}}\\
&\\
&\leq
\norm{(dh_1)_{(v_1,G(v_1))}\begin{bmatrix} I \\ (dG)_{v_1}\end{bmatrix}-
(dh_1)_{(v_1,G(v_1))}\begin{bmatrix} I \\ (dG)_{v_2}\end{bmatrix}}\ +\\
&  \\
& \ \ \ \
\norm{(dh_1)_{(v_1,G(v_1))}\begin{bmatrix} I \\ (dG)_{v_2}\end{bmatrix}-
(dh_1)_{(v_2,G(v_2))}\begin{bmatrix} I \\ (dG)_{v_2}\end{bmatrix}}\\
& \\
&\leq \norm{dh_1}_{C^0(B[2p])}\norm{(dG)_{v_1}-(dG)_{v_2}}\ +\\
& \ \ \ \
\norm{(dh_1)_{(v_1,G(v_1))}-(dh_1)_{(v_2,G(v_2))}}\norm{\begin{bmatrix} I \\ (dG)_{v_2}\end{bmatrix}}\\
&\leq \norm{dh_1}_{C^0(B[2p])}\Hol{\delta}(dG)\norm{v_1-v_2}^\delta+(1+\ve)^{1+\delta}\Hol{\delta}(dh_1)\norm{v_1-v_2}^\delta\\
&<\left[\norm{dh_1}_{C^0(B[2p])}\Hol{\delta}(dG)+(1+3\ve)\Hol{\delta}(dh_1)\right]\norm{v_1-v_2}^\delta,
\end{align*}
thus
$$
\Hol{\delta}(L)\leq \norm{dh_1}_{C^0(B[2p])}\Hol{\delta}(dG)+(1+3\ve)\Hol{\delta}(dh_1)<
\tfrac{1}{2}\ve^2+(1+3\ve)\ve<2\ve
$$
and so $\Hol{\delta}(\wt L D_1^{-1})<2\ve$.
We already know that $\norm{\wt L D_1^{-1}}_{C^0}<2\ve^2$, therefore
$\norm{\wt L D_1^{-1}}_{C^\delta}<2\ve^2+2\ve<2.5\ve$ for $\ve>0$ small enough.
Since $d\Phi=D_1^{-1}[I+\wt L D_1^{-1}]^{-1}$,
Lemma \ref{App-Lemma-Holder-norm}(2) gives that
$$
\norm{d\Phi}_{C^\delta}\leq \|D_1^{-1}\|\tfrac{1}{1-2.5\ve} < e^{-\chi}\tfrac{1}{1-2.5\ve}<e^{-\chi+3\ve},
$$
where in the last passage we used that $\tfrac{1}{1-2.5\ve}<1+3\ve<e^{3\ve}$ for small $\ve>0$.
\end{proof}

\begin{remark}
The derivative $(d\Psi)_v=D_1+\wt L(\wt v)$ has the form $A+B$, which could be written as
$(I+BA^{-1})A$ as we did or as $A(I+A^{-1}B)$ as did in \cite{Ben-Ovadia-2019}.
The problem with this second factorization is that the inverse map
has the form $(I+A^{-1}B)^{-1}A^{-1}$ and so, when estimating its $\delta$--H\"older constant,
the term ${\rm Lip}(A^{-1})^{\delta}$ appears and gives an estimate that depends on $\delta$ and 
is weaker than part (3).
\end{remark}

By the above lemma, $F(v,G(v))$ can be represented in terms of $\wt v=\Phi^{-1}(v)$ as 
$$
(D_1v+h_1(v,G(v)),D_2G(v)+h_2(v,G(v)))=(\wt v,\wt G(\wt v)),
$$
where $\wt G=[D_2G(\cdot)+h_2(\cdot,G(\cdot))]\circ\Phi=D_2G\Phi+h_2(\Phi,G\Phi)$. By hypothesis,
we have $\wt p\leq e^\ve p<e^{\chi-\sqrt{\ve}}p$ for $\ve>0$ small enough. Therefore, if $V\in\M$ is
an admissible graph with representing function $F$, then we can restrict $\wt G$ to $B^d[\wt p]$,
still denoting it by $\wt G:B^d[\wt p]\to\R^{m-d}$, and formally define $\F(V)$ as the graph
with representing function $\wt G$:
$$
\F(V)=\left\{(\wt v,\wt G(\wt v)):v\in B^d[\wt p]\right\}.
$$

\subsubsection{Proof that $\F(V)$ is admissible}

We let $V\in\M$ with representing function $G$, and 
$\F(V)$ with representing function $\wt G$. In the previous section we 
explicitly described $\wt G$. Using this, we will now prove that $\wt G$ satisfies
conditions (AM1)--(AM3) with respect to the parameters $\wt p,\wt n$, thus proving
that $\F(V)\in\wt\M$. The proofs rely on the contraction property of $D_2$.

\begin{lemma}\label{App-Lemma-condition-AM1}
$\norm{\wt G(0)}< e^{-\chi}\left(\norm{G(0)}+\sqrt{\ve}\eta\right)$, and so
$\norm{\wt G(0)}<10^{-3}\wt\eta$.
\end{lemma}

\begin{proof}
We have $\wt G(0)=D_2G(\Phi(0))+h_2(\Phi(0),G(\Phi(0)))$. Observe that
$\norm{G(\Phi(0))}\leq \norm{G(0)}+\norm{dG}_{C^0}\norm{\Phi(0)}$ which,
by (AM2) and Lemma \ref{App-Lemma-inverse-function}(2), satisfies
$$
\norm{G(\Phi(0))}\leq 10^{-3}\eta+2\ve^2\eta=(10^{-3}+2\ve^2)\eta. 
$$
In particular, $\norm{G(\Phi(0))}<\eta$ and so
$\norm{\begin{bmatrix} \Phi(0) \\ G(\Phi(0))\end{bmatrix}}<2\eta$.
Therefore
\begin{align*}
&\,\norm{\wt G(0)}\leq \norm{D_2G(\Phi(0))}+\norm{h_2(\Phi(0),G(\Phi(0)))}\\
&\leq \norm{D_2G(\Phi(0))}+\norm{h_2(0,0)}+\norm{dh_2}_{C^0(B[2\eta])}
\norm{\begin{bmatrix} \Phi(0) \\ G(\Phi(0))\end{bmatrix}}\\
&\leq \norm{D_2}(\norm{G(0)}+\norm{dG}_{C^0}\norm{\Phi(0)})+\norm{h_2(0,0)}+
\norm{dh_2}_{C^0(B[2\eta])}\norm{\begin{bmatrix} \Phi(0) \\ G(\Phi(0))\end{bmatrix}}\\
&\leq e^{-\chi}(\norm{G(0)}+2\ve^2\eta)+\ve\eta+2\ve^2\eta=e^{-\chi}\left(\norm{G(0)}+2\ve^2\eta+
e^\chi[\ve\eta+2\ve^2\eta]\right).
\end{align*}
If $\ve>0$ is small enough then
$2\ve^2\eta+e^\chi\left(\ve\eta+2\ve^2\eta\right)=(2\ve^2+e^\chi[\ve+2\ve^2])\eta<\sqrt{\ve}\eta$, and so
$$
\norm{\wt G(0)}\leq e^{-\chi}\left(\norm{G(0)}+\sqrt{\ve}\eta\right).
$$
Using (AM1), in particular we have that
$$
\norm{\wt G(0)}\leq e^{-\chi}(10^{-3}\eta+\sqrt{\ve}\eta)=
e^{-\chi}(10^{-3}+\sqrt{\ve})\eta\leq e^{-\chi+\ve}(10^{-3}+\sqrt{\ve})\wt\eta<10^{-3}\wt\eta,
$$
since $e^{-\chi+\ve}(10^{-3}+\sqrt{\ve})<10^{-3}$ if $\ve>0$ is small enough.
\end{proof}

\begin{lemma}\label{App-Lemma-condition-AM2}
$\norm{(d\wt G)_0}< e^{-2\chi+\ve}\left(\norm{(dG)_0}+2\ve^\delta\eta^\delta\right)$, and so
$\norm{(d\wt G)_0}<\tfrac{1}{2}\wt\eta^\delta$.
\end{lemma}

\begin{proof}
We will use that $\norm{\begin{bmatrix} \Phi(0) \\ G(\Phi(0))\end{bmatrix}}<2\eta$, as
proved Lemma \ref{App-Lemma-condition-AM1}. Since
$\wt G=[D_2G(\cdot)+h_2(\cdot,G(\cdot))]\circ\Phi$, we have 
that
$$(d\wt G)_{\wt v}=\left(D_2(dG)_{\Phi(\wt v)}+(dh_2)_{(\Phi(\wt v),G(\Phi(\wt v))}
\begin{bmatrix} I \\ (dG)_{\Phi(\wt v)}\end{bmatrix}\right) (d\Phi)_{\wt v}
$$
and so 
$$(d\wt G)_0=\left(D_2(dG)_{\Phi(0)}+(dh_2)_{(\Phi(0),G(\Phi(0))}
\begin{bmatrix} I \\ (dG)_{\Phi(0)}\end{bmatrix}\right) (d\Phi)_0.
$$
By (AM3) and Lemma \ref{App-Lemma-inverse-function}(2),
$$
\norm{(dG)_{\Phi(0)}}\leq \norm{(dG)_0}+\tfrac{1}{2}\norm{\Phi(0)}^\delta< 
\norm{(dG)_0}+\tfrac{1}{2}(2\ve\eta)^\delta<\norm{(dG)_0}+\ve^\delta\eta^\delta,
$$
hence
\begin{align*}
&\,\norm{(d\wt G)_0}< \left(\norm{D_2}\norm{(dG)_{\Phi(0)}}+\norm{dh_2}_{C^0(B[2\eta])}\right)\norm{(d\Phi)_0}\\
&<\left(e^{-\chi}\norm{(dG)_{\Phi(0)}}+2\ve\eta^\delta\right)\norm{(d\Phi)_0}=
e^{-\chi}\norm{(d\Phi)_0}\left(\norm{(dG)_{\Phi(0)}}+2e^\chi \ve\eta^\delta\right)\\
&<e^{-2\chi+\ve}\left(\norm{(dG)_0}+\ve^\delta\eta^\delta+ 2e^\chi \ve\eta^\delta\right)
<e^{-2\chi+\ve}\left(\norm{(dG)_0}+2\ve^\delta\eta^\delta\right)
\end{align*}
where in the last passage we used Lemma \ref{App-Lemma-inverse-function}(1) and 
that $2e^\chi \ve<\ve^\delta$ for $\ve>0$ small enough. Using (AM2), in particular we obtain that
\begin{align*}
&\,\norm{(d\wt G)_0}<e^{-2\chi+\ve}\left(\tfrac{1}{2}\eta^\delta+2\ve^\delta\eta^\delta\right)=
e^{-2\chi+\ve}\left(\tfrac{1}{2}+2\ve^\delta\right)\eta^\delta\\
&\leq e^{-2\chi+2\ve}\left(\tfrac{1}{2}+2\ve^\delta\right)\wt\eta^\delta<\tfrac{1}{2}\wt\eta^\delta,
\end{align*}
since $e^{-2\chi+2\ve}\left(\tfrac{1}{2}+2\ve^\delta\right)<\tfrac{1}{2}$ if $\ve>0$ is small enough. 
\end{proof}

\begin{lemma}\label{App-Lemma-condition-AM3}
$\norm{d\wt G}_{C^\delta}<e^{-\chi}\left(\norm{dG}_{C^\delta}+\sqrt{\ve}\right)$, and so
$\norm{d\wt G}_{C^\delta}<\tfrac{1}{2}$.
\end{lemma}

\begin{proof}
We have 
$$(d\wt G)_{\wt v}=\left(D_2(dG)_{\Phi(\wt v)}+(dh_2)_{(\Phi(\wt v),G(\Phi(\wt v))}
\begin{bmatrix} I \\ (dG)_{\Phi(\wt v)}\end{bmatrix}\right) (d\Phi)_{\wt v}.
$$
Since $\norm{d\Phi}_{C^\delta}<1$, Lemma \ref{App-Lemma-Holder-norm}(1) implies that
\begin{align*}
&\norm{d\wt G}_{C^\delta}\leq \norm{D_2(dG)_{\Phi(\wt v)}+(dh_2)_{(\Phi(\wt v),G(\Phi(\wt v))}
\begin{bmatrix} I \\ (dG)_{\Phi(\wt v)}\end{bmatrix}}_{C^\delta}\\
&\leq \norm{D_2}\norm{dG}_{C^\delta}+\norm{(dh_2)_{(\Phi(\wt v),G(\Phi(\wt v))}
\begin{bmatrix} I \\ (dG)_{\Phi(\wt v)}\end{bmatrix}}_{C^\delta}
\end{align*}
Call $\wt L=(dh_2)_{(\Phi(\wt v),G(\Phi(\wt v))}\begin{bmatrix} I \\ (dG)_{\Phi(\wt v)}\end{bmatrix}$.
Then $\wt L=L\circ\Phi$, with $L,\wt L$ defined as in the proof of Lemma
\ref{App-Lemma-inverse-function}, with $h_1$ replaced by $h_2$. Since we have the same estimates
for $h_1,h_2$, we obtain the same estimates for $L$, i.e. $\norm{L}_{C^0}<2\ve^2$ and
$\Hol{\delta}(L)<2\ve$. Since $\Phi$ is a contraction, we conclude that
$\norm{\wt L}_{C^\delta}\leq\norm{L}_{C^\delta}<2\ve^2+2\ve<3\ve$ for $\ve>0$ small enough.
Therefore, 
$$
\norm{d\wt G}_{C^\delta}\leq e^{-\chi}\norm{dG}_{C^\delta}+3\ve =e^{-\chi}\left(\norm{dG}_{C^\delta}+3\ve e^\chi\right)
<e^{-\chi}\left(\norm{dG}_{C^\delta}+\sqrt{\ve}\right)
$$
for $\ve>0$ small enough. In particular, using (AM3) we get that
$\norm{d\wt G}_{C^\delta}<e^{-\chi}\left(\tfrac{1}{2}+\sqrt{\ve}\right)<\tfrac{1}{2}$
when $\ve>0$ is small enough.
\end{proof}

By Lemmas \ref{App-Lemma-inverse-function}, \ref{App-Lemma-condition-AM1},
\ref{App-Lemma-condition-AM2}, \ref{App-Lemma-condition-AM3}, we finally
obtain that the graph transform $\F:\M\to\wM$ is well-defined. 

\begin{remark}\label{App-remark-extension}
The proofs of Lemmas \ref{App-Lemma-condition-AM1},
\ref{App-Lemma-condition-AM2}, \ref{App-Lemma-condition-AM3}
rely solely on our underlying assumptions and on Lemma \ref{App-Lemma-inverse-function}.
Therefore, the obtained estimates hold for $\wt G$ defined in the larger domain $\Psi(B^d[p])$
and, in particular, $\norm{d\wt G}_{C^0}<\ve$.
\end{remark}

\subsection{Estimates on the norm of $\F$}

In this section we obtain estimates of $\F$ in the $C^0$ and $C^1$ norms. 

\begin{proposition}\label{App-Prop-graph-transform}
The following holds for $\ve>0$ small enough. If
$V_1,V_2\in \M$ then:
\begin{enumerate}[{\rm (1)}]
\item $ d_{C^0}(\F(V_1),\F(V_2))\leq e^{-\chi/2} d_{C^0}(V_1,V_2)$.
\item $ d_{C^1}(\F(V_1),\F(V_2))\leq e^{-\chi/2}( d_{C^1}(V_1,V_2)+ d_{C^0}(V_1,V_2)^\delta)$.
\end{enumerate}
\end{proposition}

\begin{proof}
Let $V_1,V_2$ with representing functions $G_1,G_2$, and let $\wt G_1,\wt G_2$ be the representing
functions of $\F(V_1),\F(V_2)$. With respect to $G_1,G_2$, we consider the functions $\Psi_1,\Psi_2$
as defined in the proof of Lemma \ref{App-Lemma-inverse-function}, as well as their inverses
$\Phi_1,\Phi_2$. We also let
\begin{align*}
\Delta_1&=D_2G_1(\cdot)+h_2(\cdot,G_1(\cdot))\\
\Delta_2&=D_2G_2(\cdot)+h_2(\cdot,G_2(\cdot)),
\end{align*}
so that $\wt G_1=\Delta_1\Phi_1$ and $\wt G_2=\Delta_2\Phi_2$ (compositions).

\medskip
\noindent
(1) We want to prove that $\norm{\wt G_1-\wt G_2}_{C^0}\leq e^{-\chi/2}\norm{G_1-G_2}_{C^0}$. We have
\begin{align*}
&\,\norm{\wt G_1-\wt G_2}_{C^0}
\leq \norm{\Delta_1\Phi_1-\Delta_1\Phi_2}_{C^0}+\norm{\Delta_1\Phi_2-\Delta_2\Phi_2}_{C^0}\\
&\leq \norm{d\Delta_1}_{C^0}\norm{\Phi_1-\Phi_2}_{C^0}+\norm{\Delta_1-\Delta_2}_{C^0},
\end{align*}
so we estimate each of the terms above.
\begin{enumerate}[$\circ$]
\item $\norm{d\Delta_1}_{C^0}<2\ve$: to prove this, note that $d\Delta_1=D_2\circ dG_1+L_1$,
where $L_1$ is a function just like the one considered in Lemma \ref{App-Lemma-inverse-function},
so in particular $\norm{L_1}_{C^0}<2\ve^2$. Together with Lemma \ref{App-Lemma-introductory}(1),
we obtain that
$$
\norm{d\Delta_1}_{C^0}\leq \norm{D_2}\norm{dG_1}_{C^0}+\norm{L_1}_{C^0}\leq e^{-\chi}\ve +2\ve^2<2\ve.
$$ 
\item $\norm{\Delta_1-\Delta_2}_{C^0}\leq (e^{-\chi}+\ve^2)\norm{G_1-G_2}_{C^0}$: to prove this,
just observe that 
\begin{align*}
&\,\norm{\Delta_1-\Delta_2}_{C^0}=
\norm{D_2(G_1-G_2)+\left[h_2(\cdot,G_1(\cdot))-h_2(\cdot,G_2(\cdot))\right]}_{C^0}\\
&\leq \norm{D_2}\norm{G_1-G_2}_{C^0}+\norm{dh_2}_{C^0(B[2p])}\norm{G_1-G_2}_{C^0}\\
&\leq (e^{-\chi}+\ve^2)\norm{G_1-G_2}_{C^0}.
\end{align*}
\item $\norm{\Phi_1-\Phi_2}_{C^0}\leq \ve^2\norm{G_1-G_2}_{C^0}$: start observing that
$$
\norm{\Psi_1-\Psi_2}_{C^0}=\norm{h_1(\cdot,G_1(\cdot))-h_1(\cdot,G_2(\cdot))}_{C^0}
\leq \ve^2\norm{G_1-G_2}_{C^0}.
$$
It is worth reminding that the norms above are calculated in $B^d[p]$, while the norm of $\Phi_1-\Phi_2$
is to be calculated in $B^d[e^{\chi-\sqrt{\ve}}p]$. Since $\Psi_1(B^d[p])\supset B^d[e^{\chi-\sqrt{\ve}}p]$,
\begin{align*}
&\,\norm{\Phi_1-\Phi_2}_{C^0}\leq \norm{\Phi_1\Psi_1-\Phi_2\Psi_1}_{C^0}=\norm{{\rm Id}-\Phi_2\Psi_1}_{C^0}
=\norm{\Phi_2\Psi_2-\Phi_2\Psi_1}_{C^0}\\
&\leq \norm{d\Phi_2}_{C^0}\norm{\Psi_2-\Psi_1}_{C^0}\leq \norm{\Psi_2-\Psi_1}_{C^0}\leq
\ve^2\norm{G_1-G_2}_{C^0},
\end{align*}
where we used Lemma \ref{App-Lemma-inverse-function}(1).
\end{enumerate}
Using the three estimates above, we get that
\begin{align*}
&\,\norm{\wt G_1-\wt G_2}_{C^0}\leq 2\ve\cdot \ve^2\norm{G_1-G_2}_{C^0}+
(e^{-\chi}+\ve^2)\norm{G_1-G_2}_{C^0}\\
&=(e^{-\chi}+\ve^2+2\ve^3)\norm{G_1-G_2}_{C^0}\leq (e^{-\chi}+3\ve^2)\norm{G_1-G_2}_{C^0}
\end{align*}
and so $\norm{\wt G_1-\wt G_2}_{C^0}\leq e^{-\chi/2}\norm{G_1-G_2}_{C^0}$ for $\ve>0$ small enough.

\medskip
\noindent
(2) In view of part (1), we need to estimate $\norm{d\wt G_1-d\wt G_2}_{C^0}$
in terms of $\norm{G_1-G_2}_{C^0}$ and $\norm{G_1-G_2}_{C^0}^\delta$. We will show that
$$
\norm{d\wt G_1-d\wt G_2}_{C^0} \leq \left(e^{-\chi}+3\ve^2\right)\norm{dG_1-dG_2}_{C^0}+
\left(e^{-\chi+\ve}+3\ve^2\right)\norm{G_1-G_2}_{C^0}^\delta.
$$
Note that the contribution of $\norm{G_1-G_2}_{C^0}^\delta$ is very small.
We have $d\wt G_1=(d\Delta_1)_{\Phi_1}\circ d\Phi_1$ and 
$d\wt G_2=(d\Delta_2)_{\Phi_2}\circ d\Phi_2$, hence for $\wt v \in B^d[\wt p]$:
\begin{align*}
&\,\norm{(d\wt G_1)_{\wt v}-(d\wt G_2)_{\wt v}}=
\norm{(d\Delta_1)_{\Phi_1(\wt v)}(d\Phi_1)_{\wt v}-(d\Delta_2)_{\Phi_2(\wt v)}(d\Phi_2)_{\wt v}}\\
&\leq \norm{(d\Delta_1)_{\Phi_1(\wt v)}(d\Phi_1)_{\wt v}-(d\Delta_1)_{\Phi_1(\wt v)}(d\Phi_2)_{\wt v}}+\\
&\hspace{.47cm}\norm{(d\Delta_1)_{\Phi_1(\wt v)}(d\Phi_2)_{\wt v}-(d\Delta_1)_{\Phi_2(\wt v)}(d\Phi_2)_{\wt v}}+\\
&\hspace{.47cm} \norm{(d\Delta_1)_{\Phi_2(\wt v)}(d\Phi_2)_{\wt v}-(d\Delta_2)_{\Phi_2(\wt v)}(d\Phi_2)_{\wt v}}\\
&\leq \norm{d\Delta_1}_{C^0}\norm{d\Phi_1-d\Phi_2}_{C^0}+
\Hol{\delta}(d\Delta_1)\norm{\Phi_1-\Phi_2}_{C^0}^\delta \norm{d\Phi_2}_{C^0}+\\
&\hspace{.47cm} \norm{d\Delta_1-d\Delta_2}_{C^0}\norm{d\Phi_2}_{C^0}.
\end{align*}
In the above expression, we need to estimate the values of
$\norm{d\Delta_1-d\Delta_2}_{C^0}$, $\Hol{\delta}(d\Delta_1)$ and $\norm{d\Phi_1-d\Phi_2}_{C^0}$.

\medskip
\noindent
{\sc Claim 1:}
$\norm{d\Delta_1-d\Delta_2}_{C^0}\leq (e^{-\chi}+\ve^2)\norm{dG_1-dG_2}_{C^0}+2\ve\norm{G_1-G_2}_{C^0}^\delta$.

\begin{proof}[Proof of Claim $1$.]
We have $d\Delta_1=D_2\circ dG_1+L_1$ and $d\Delta_2=D_2\circ dG_2+L_2$ and so
\begin{align}\label{App-Estimate-Delta}
\norm{d\Delta_1-d\Delta_2}_{C^0}\leq \norm{D_2}\norm{dG_1-dG_2}_{C^0}+\norm{L_1-L_2}.
\end{align}
We will prove that $\norm{L_1-L_2}\leq \ve^2\norm{dG_1-dG_2}_{C^0}+2\ve\norm{G_1-G_2}_{C^0}^\delta$.
Indeed, since $\norm{\begin{bmatrix} I \\ (dG_2)_{v}\end{bmatrix}}\leq 1+\|dG_2\|_{C^0}<1+\ve<2$,
we have
\begin{align*}
\norm{L_1(v)-L_2(v)}&=\norm{(dh_2)_{(v,G_1(v))}\begin{bmatrix} I \\ (dG_1)_{v}\end{bmatrix}-
(dh_2)_{(v,G_2(v))}\begin{bmatrix} I \\ (dG_2)_{v}\end{bmatrix}}\\
&\leq \norm{(dh_2)_{(v,G_1(v))}\begin{bmatrix} I \\ (dG_1)_{v}\end{bmatrix}-
(dh_2)_{(v,G_1(v))}\begin{bmatrix} I \\ (dG_2)_{v}\end{bmatrix}}+\\
&\ \ \ \ \norm{(dh_2)_{(v,G_1(v))}\begin{bmatrix} I \\ (dG_2)_{v}\end{bmatrix}-
(dh_2)_{(v,G_2(v))}\begin{bmatrix} I \\ (dG_2)_{v}\end{bmatrix}}\\
&\leq \norm{dh_2}_{C^0(B[2p])}\norm{(dG_1)_v-(dG_2)_v}+\\
&\ \ \ \ 2\Hol{\delta}(dh_2)\norm{G_1(v)-G_2(v)}^\delta\\
&\leq \norm{dh_2}_{C^0(B[2p])}\norm{dG_1-dG_2}_{C^0}+\\
&\ \ \ \ 2\Hol{\delta}(dh_2)\norm{G_1-G_2}_{C^0}^\delta
\end{align*}
and so 
$$
\norm{L_1-L_2}_{C^0}\leq \ve^2\norm{dG_1-dG_2}_{C^0}+2\ve\norm{G_1-G_2}_{C^0}^\delta.
$$
Using (\ref{App-Estimate-Delta}),
$\norm{d\Delta_1-d\Delta_2}_{C^0}\leq (e^{-\chi}+\ve^2)\norm{dG_1-dG_2}_{C^0}+2\ve\norm{G_1-G_2}_{C^0}^\delta$.
\end{proof}

\medskip
\noindent
{\sc Claim 2:} $\Hol{\delta}(d\Delta_1)<\tfrac{1}{2}e^{-\chi}+3\ve<1$.

\begin{proof}[Proof of Claim $2$.]
By Lemma \ref{App-Lemma-Holder-norm}(1),
$$
\Hol{\delta}(d\Delta_1)\leq \norm{D_2}\Hol{\delta}(dG_1)+\Hol{\delta}(L_1)<\tfrac{1}{2}e^{-\chi}+\Hol{\delta}(L_1),
$$
so it is enough to prove that $\Hol{\delta}(L_1)<3\ve$. Calculating $\norm{L_1(v_1)-L_1(v_2)}$ as
in the proof of Claim 1, we get that
\begin{align*}
&\,\norm{L_1(v_1)-L_1(v_2)}\\
&=\norm{(dh_2)_{(v_1,G_1(v_1))}\begin{bmatrix} I \\ (dG_1)_{v_1}\end{bmatrix}-
(dh_2)_{(v_2,G_1(v_2))}\begin{bmatrix} I \\ (dG_1)_{v_2}\end{bmatrix}}\\
&\leq \norm{(dh_2)_{(v_1,G_1(v_1))}\begin{bmatrix} I \\ (dG_1)_{v_1}\end{bmatrix}-
(dh_2)_{(v_1,G_1(v_1))}\begin{bmatrix} I \\ (dG_1)_{v_2}\end{bmatrix}}+\\
& \ \ \ \ \norm{(dh_2)_{(v_1,G_1(v_1))}\begin{bmatrix} I \\ (dG_1)_{v_2}\end{bmatrix}-
(dh_2)_{(v_2,G_1(v_2))}\begin{bmatrix} I \\ (dG_1)_{v_2}\end{bmatrix}}\\
& \leq \norm{dh_2}_{C^0(B[2p])}\norm{(dG_1)_{v_1}-(dG_1)_{v_2}}+2\Hol{\delta}(dh_2)\norm{\begin{bmatrix} v_1-v_2 \\ G_1(v_1)-G_1(v_2)\end{bmatrix}}^\delta\\
&\leq \left(\norm{dh_2}_{C^0(B[2p])}\Hol{\delta}(dG_1)+
2(1+\ve)^\delta\Hol{\delta}(dh_2)\right)\norm{v_1-v_2}^\delta
\end{align*}
and so $\Hol{\delta}(L_1)\leq \tfrac{\ve^2}{2}+2(1+\ve)^\delta\ve<3\ve$.
\end{proof}

\medskip
\noindent
{\sc Claim 3:} $\norm{d\Phi_1-d\Phi_2}_{C^0}\leq \ve^2\norm{dG_1-dG_2}_{C^0}+2\ve\norm{G_1-G_2}_{C^0}^\delta$.

\begin{proof}[Proof of Claim $3$.]
Fix $\wt v\in B^d[e^{\chi-\sqrt{\ve}}p]$ and let $v=\Phi(\wt v)$. We have
\begin{align*}
&\,\norm{(d\Phi_1)_{\wt v}-(d\Phi_2)_{\wt v}}\leq\norm{{\rm Id}-(d\Psi_1)_v(d\Phi_2)_{\wt v}}\\
&=\norm{(d\Psi_2)_v(d\Phi_2)_{\wt v}-(d\Psi_1)_v(d\Phi_2)_{\wt v}}\leq \norm{(d\Psi_2)_v-(d\Psi_1)_v}
\end{align*}
because $\norm{d\Phi_1},\norm{d\Phi_2}<1$. We have $(\Psi_2-\Psi_1)(v)=h_1(v,G_2(v))-h_1(v,G_1(v))$, thus
\begin{align*}
\norm{d(\Psi_2-\Psi_1)_v}&=\norm{(dh_1)_{(v,G_2(v))}\begin{bmatrix} I \\ (dG_2)_v\end{bmatrix}-
(dh_1)_{(v,G_1(v))}\begin{bmatrix} I \\ (dG_1)_v\end{bmatrix}}\\
&\leq \norm{(dh_1)_{(v,G_2(v))}\begin{bmatrix} I \\ (dG_2)_v\end{bmatrix}-
(dh_1)_{(v,G_2(v))}\begin{bmatrix} I \\ (dG_1)_v\end{bmatrix}}+\\
& \ \ \ \ \norm{(dh_1)_{(v,G_2(v))}\begin{bmatrix} I \\ (dG_1)_v\end{bmatrix}-
(dh_1)_{(v,G_1(v))}\begin{bmatrix} I \\ (dG_1)_v\end{bmatrix}}\\
&\leq \norm{dh_1}_{C^0(B[2p])}\norm{(dG_1)_v-(dG_2)_v}+2\Hol{\delta}(dh_1)\norm{G_2(v)-G_1(v)}^\delta\\
&\leq \norm{dh_1}_{C^0(B[2p])}\norm{dG_1-dG_2}_{C^0}+2\Hol{\delta}(dh_1)\norm{G_1-G_2}_{C^0}^\delta\\
&\leq \ve^2\norm{dG_1-dG_2}_{C^0}+2\ve\norm{G_1-G_2}_{C^0}^\delta,
\end{align*}
concluding the proof of Claim 3.
\end{proof}
Plugging all estimates together, we obtain that for $\ve>0$ small enough it holds
\begin{align*}
\norm{d\wt G_1-d\wt G_2}_{C^0}&\leq 2\ve\left(\ve^2\norm{dG_1-dG_2}_{C^0}+2\ve\norm{G_1-G_2}_{C^0}^\delta\right)+\\
&\ \ \ \ e^{-\chi+\ve}\left(\ve^2\norm{G_1-G_2}_{C^0}\right)^\delta+\\
&\ \ \ \ e^{-\chi+\ve}\left[(e^{-\chi}+\ve^2)\norm{dG_1-dG_2}_{C^0}+2\ve\norm{G_1-G_2}_{C^0}^\delta\right]\\
&\leq \left(e^{-\chi}+\ve^2+2\ve^3\right)\norm{dG_1-dG_2}_{C^0}+\\
&\ \ \ \ \left(e^{-\chi+\ve}+4\ve^2+2\ve\right)\norm{G_1-G_2}_{C^0}^\delta\\
&\leq \left(e^{-\chi}+3\ve^2\right)\norm{dG_1-dG_2}_{C^0}
+\left(e^{-\chi+\ve}+3\ve\right)\norm{G_1-G_2}_{C^0}^\delta.
\end{align*}
Finally, with the estimate of part (1) we conclude that $\norm{\wt G_1-\wt G_2}_{C^1}$ is at most
$$
(e^{-\chi}+3\ve^2)\norm{G_1-G_2}_{C^0}+\left(e^{-\chi}+3\ve^2\right)\norm{dG_1-dG_2}_{C^0}
+\left(e^{-\chi+\ve}+3\ve\right)\norm{G_1-G_2}_{C^0}^\delta
$$
which is bounded by $e^{-\chi/2}\left(\norm{G_1-G_2}_{C^1}+\norm{G_1-G_2}_{C^0}^\delta\right)$,
proving part (2).
\end{proof}

\subsection{Relation between $s$--graphs and $u$--graphs}

We finish this appendix providing two relations between $s$--graphs and $u$--graphs.
The first deals with intersections of graphs based at the same point, and the second
deals with intersections of images of $u$--graphs under $F$ with
$s$--graphs. Let $\mathfs M^u=\mathfs M^u_{p,\eta}$
and $\mathfs M^s=\mathfs M^s_{q,\eta}$.

\begin{lemma}\label{App-Lemma-admissible-graphs}
The following holds for $\ve>0$ small enough. For
every $V^s\in\mathfs M^s$ and $V^u\in\mathfs M^u$ it holds:
\begin{enumerate}[{\rm (1)}]
\item $V^s$ and $V^u$ intersect at a single point $w$, and $\|w\|\leq 50^{-1}\eta$.
%\marginpar{\small{If needed, (1) can be improved to $10^{-2}\eta$}}
\item $w=w(V^s,V^u)$ is a Lipschitz function of the pair $(V^s,V^u)$, with Lipschitz constant
$\tfrac{2}{1-\ve^2}$.
\end{enumerate}
\end{lemma}

\begin{proof}
(1) Fix $V^s,V^u$ with representing functions $G,J$. Observe that intersections 
of $V^s$ and $V^u$ correspond to points $(v_1,v_2)$ that satisfy the system:
$$
\left\{
\begin{array}{l}
v_1=J(v_2)\\
v_2=G(v_1).
\end{array}
\right.
$$
This system, in turn, corresponds to $v_1$ being a fixed point of $J\circ G$,
wherever this composition is well-defined (one direction is 
direct; for the other, if $(J\circ G)(v_1)=v_1$ then defining $v_2=G(v_1)$ we have
$(v_1,v_2)=(v_1,G(v_1))=(J(v_2),v_2)$). Hence we are reduced to analysing the 
fixed points of $J\circ G$.

Since $G,J$ are defined in different domains,
we first prove the existence of a fixed point
for their composition and later prove that it is unique. 
For the existence, we reduce their
domains to $B^d[\eta]$ and $B^{m-d}[\eta]$. By Lemma \ref{App-Lemma-introductory}(2), we have
$$
B^d[\eta] \ \xrightarrow{\ G\ } \ B^{m-d}[10^{-2}\eta] \subset B^{m-d}[\eta]\ \xrightarrow{\ J\ }\ B^d[10^{-2}\eta],
$$
so the restriction of $J\circ G$ to $B^d[\eta]$ is well-defined. We claim that this restriction is
a contraction. Indeed, in the respective domains we have $\|dG\|_{C^0},\|dJ\|_{C^0}<\ve$ and
so $\|d(J\circ G)\|_{C^0}<\ve^2<1$. This guarantees the existence of a fixed point $v_1\in B^d[10^{-2}\eta]$,
which guarantees the existence of an intersection point $w=(v_1,v_2)\in V^s\cap V^u$ with 
$\|w\|< 50^{-1}\eta$.

To prove uniqueness, we first extend $G,J$ to domains on the same size.
By the Kirszbraun theorem, we can extend $G,J$ to functions
$\widetilde G:B^d[p\vee q]\to \R^{m-d}$
and $\widetilde J:B^{m-d}[p\vee q]\to \R^d$ so that ${\rm Lip}(\widetilde G)={\rm Lip}(G)<\ve$
and ${\rm Lip}(\widetilde J)={\rm Lip}(J)<\ve$. The same calculation performed in
Lemma \ref{App-Lemma-introductory}(2) gives that
$$
B^d[p\vee q] \ \xrightarrow{\ \widetilde G\ } \ B^{m-d}[10^{-2}(p\vee q)]
\subset B^{m-d}[p\vee q]\ \xrightarrow{\ \widetilde J\ }\ B^d[10^{-2}(p\vee q)],
$$
and by definition the composition $\widetilde J\circ\widetilde G$ is a contraction. 
It therefore has the unique fixed point $w$.

\medskip
\noindent
(2) We start with a simple claim about contractions on metric spaces.

\medskip
\noindent
{\sc Claim:} Let $(X,d)$ be a complete metric space, let $f_1,f_2$ be contractions
on $X$, and let $x_1,x_2$ be the unique fixed points of $f_1,f_2$. If
$d(f_1,f_2):=\sup\{d(f_1(x),f_2(x)):x\in X\}$, then
$$
d(x_1,x_2)\leq \tfrac{d(f_1,f_2)}{1-{\rm Lip}(f_1)}\cdot
$$

\begin{proof}[Proof of the claim.]
We have
\begin{align*}
&\,d(x_1,x_2)=d(f_1(x_1),f_2(x_2))\leq d(f_1(x_1),f_1(x_2))+d(f_1(x_2),f_2(x_2))\\
&\leq {\rm Lip}(f_1)d(x_1,x_2)+d(f_1,f_2)
\end{align*}
and the claim follows.
\end{proof}

Let $V^s_1,V^s_2\in \mathfs M^s$ with representing functions $G_1,G_2$,
and $V^u_1,V^u_2\in\mathfs M^u$ with representing functions $J_1,J_2$.
Let $f_1=J_1\circ G_1$ and $f_2=J_2\circ G_2$, and 
let $w_i$ be the intersection of $V^s_i$ and $V^u_i$. By part (1), 
$w_1=(v_1,v_2)$ and $w_2=(t_1,t_2)$ where $v_1$ is the unique fixed point of $f_1$ and
$t_1$ is the unique fixed point of $f_2$. Observing that ${\rm Lip}(f_1) <\ve^2$ and that
\begin{align*}
&\,\|f_1-f_2\|_{C^0}\leq \|J_1\circ G_1-J_1\circ G_2\|_{C^0}+\|J_1\circ G_2-J_2\circ G_2\|_{C^0}\\
&\leq {\rm Lip}(J_1)\|G_1-G_2\|_{C^0}+\|J_1-J_2\|_{C^0}\leq \ve \|G_1-G_2\|_{C^0}+\|J_1-J_2\|_{C^0}\\
&\leq \|G_1-G_2\|_{C^0}+\|J_1-J_2\|_{C^0},
\end{align*}
the claim implies that $\|v_1-t_1\|\leq \tfrac{\|G_1-G_2\|_{C^0}+\|J_1-J_2\|_{C^0}}{1-\ve^2}$.
Proceeding similarly with $v_2,t_2$, we conclude that
$$
\|w_1-w_2\| \leq \tfrac{2(\|G_1-G_2\|_{C^0}+\|J_1-J_2\|_{C^0})}{1-\ve^2}=
\tfrac{2(d_{C^0}(V^s_1,V^s_2)+d_{C^0}(V^u_1,V^u_2))}{1-\ve^2}\cdot
$$
\end{proof}

For the second lemma, let $F$ and $\mathfs F$ be as in Section \ref{App-Section-graph-transforms},
with $\mathfs F:\mathfs M^u_{p,\eta}\to\mathfs M^u_{\wt p,\wt \eta}$.
We let $\wt q\geq \wt\eta$ be another parameter and consider $\mathfs M^s_{\wt q,\wt \eta}$.

\begin{lemma}\label{App-Lemma-graphs-intersection}
For every $V^u\in\mathfs M^u_{p,\eta}$ and $V^s\in \mathfs M^s_{\wt q,\wt \eta}$,
the intersection $F(V^u)\cap V^s$ consists of a single point.
\end{lemma}

\begin{proof}
If $G$ is the representing function of $V^u$ then
$$
F(V^u)=\left\{(\wt v,\wt G(\wt v)):\wt v\in \Psi(B^d[p])\right\},
$$
where $\wt G:\Psi(B^d[p])\to \R^{m-d}$ is defined as before by
$\wt G=D_2G\Phi+h_2(\Phi,G\Phi)$.
By Remark \ref{App-remark-extension}, $\norm{d\wt G}_{C^0}<\ve$.
Hence we can apply the same proof of Lemma \ref{App-Lemma-admissible-graphs}(1)
and conclude that $F(V^u)\cap V^s$ consists of a single point.
\end{proof}

\bibliographystyle{alpha}
\bibliography{bibliography}{}

\end{document}